\documentclass[a4paper,11pt]{amsart}
\pdfoutput=1
\usepackage[utf8]{inputenc}

\usepackage[english]{babel}
\usepackage[autostyle, english = british]{csquotes}

\usepackage{amsmath}
\usepackage{amsfonts}
\usepackage{amssymb}
\usepackage{amsthm}
\usepackage{thmtools}
\usepackage{mathrsfs}
\usepackage{mathabx}
\usepackage{stmaryrd}
\usepackage[all]{xy}
\usepackage{multirow}
\usepackage{graphicx}
\usepackage[usenames,dvipsnames,svgnames,table]{xcolor}
\usepackage{transparent}

\usepackage{tikz}
\usetikzlibrary{cd, fit, knots, positioning, arrows.meta, matrix, mindmap, trees, shadows, shapes, calc, fadings, decorations.pathreplacing, intersections, shapes.callouts, shapes.arrows, shapes.geometric, decorations.pathmorphing}

\usepackage[bookmarks=true]{hyperref}
\usepackage{cleveref}

\definecolor{azulmaneiro}{rgb}{0, 0.125, 0.666}
\definecolor{brownmaneiro}{rgb}{0.4, 0.2, 0}
\hypersetup{colorlinks=true, linkcolor=azulmaneiro, citecolor=brownmaneiro, urlcolor=red}

\numberwithin{equation}{section}

\theoremstyle{plain}
\newtheorem{mainthm}{Theorem}

\newtheorem{thm}{Theorem}[section]
\newtheorem{lem}[thm]{Lemma}
\newtheorem{prop}[thm]{Proposition}
\newtheorem{cor}[thm]{Corollary}

\newtheorem{fact}[thm]{Fact}

\theoremstyle{definition}
\newtheorem{defn}[thm]{Definition}
\newtheorem{ex}[thm]{Example}
\newtheorem{ques}[thm]{Question}

\newtheorem{conv}[thm]{Convention}

\theoremstyle{remark}
\newtheorem{rem}[thm]{Remark}

\newcommand{\Z}{\mathbb{Z}}

\newcommand{\Q}{\mathbb{Q}}

\newcommand{\R}{\mathbb{R}}

\newcommand{\QG}{\mathbb{Q}[G]}

\DeclareMathOperator{\hH}{H} 
\DeclareMathOperator{\Tor}{Tor}

\DeclareMathOperator{\dH}{dH}
\newcommand{\dExt}{{\operatorname{d}}{\operatorname{Ext}}}

\newcommand{\dis}{\mathbf{dis}}
\newcommand{\QGdis}{{}_{\QG}\dis}

\newcommand{\RGdis}{{}_{RG}\dis}

\DeclareMathOperator{\ccd}{cd} 

\newcommand{\FP}{\mathtt{FP}} 
\newcommand{\CP}{\mathtt{CP}}
\newcommand{\tF}{\mathtt{F}} 
\newcommand{\tC}{\mathtt{C}}

\newcommand{\ob}{\mathrm{ob}}
\DeclareMathOperator{\spE}{\operatorname{E}}

\DeclareMathOperator{\iid}{id}

\DeclareMathOperator{\Hom}{Hom}

\DeclareMathOperator{\stab}{stab}
\DeclareMathOperator{\Aut}{Aut}

\DeclareMathOperator{\conn}{conn}
\DeclareMathOperator{\Lk}{Lk}

\DeclareMathOperator{\CAT}{CAT}

\DeclareMathOperator{\HNN}{HNN} 
\DeclareMathOperator{\GL}{GL}

\newcommand{\csbgp}{\leq_c} 
\newcommand{\osbgp}{\leq_o} 
\newcommand{\cns}{\trianglelefteq_c} 
\newcommand{\caG}{\mathcal{G}} 
\newcommand{\caT}{\mathscr{T}}
\newcommand{\caV}{\mathcal{V}}
\newcommand{\caE}{\mathscr{E}}

\newcommand{\ind}{\mathrm{ind}}

\newcommand{\phee}{\varphi} 
\newcommand{\veps}{\varepsilon} 
\newcommand{\quer}[1]{\overline{#1}} 
\newcommand{\wtil}[1]{\widetilde{#1}} 
\newcommand{\LL}{\langle\!\langle} 
\newcommand{\GG}{\rangle\!\rangle} 
\newcommand{\RR}{\rangle\!\rangle} 
\newcommand{\spans}[1]{\langle {#1} \rangle} 

\newcommand{\leer}{\varnothing} 

\newcommand{\mc}{\mathcal}

\newcommand{\mscr}{\mathscr}

\newcommand{\into}{\hookrightarrow}
\newcommand{\onto}{\twoheadrightarrow}

\newcommand{\caO}{\mathcal{O}}
\newcommand{\caA}{\mathcal{A}}
\newcommand{\caF}{\mathcal{F}}
\newcommand{\caH}{\mathcal{H}}

\newcommand{\caN}{\mathcal{N}}
\newcommand{\caR}{\mathcal{R}}
\newcommand{\caC}{\mathcal{C}}
\newcommand{\caB}{\mathcal{B}}
\newcommand{\caK}{\mathcal{K}}
\newcommand{\caW}{\mathcal{W}}

\newcommand{\apsys}{\mathscr{A}}

\newcommand{\argu}{\hbox to 1.5ex{\hrulefill}}  

\usepackage[style=alphabetic,sorting=nyt,maxbibnames=7,maxcitenames=5,block=space,backend=bibtex]{biblatex}
\title[]{Generalisable presentations and compactness properties of locally compact right-angled Artin groups}

\author[I. Castellano]{Ilaria Castellano}
\address{
Politecnico di Torino\\
Department of Mathematical Sciences (DISMA),  Corso Duca degli Abruzzi, 24, 10129 Torino, Italy}
\email{ilaria.castellano@polito.it}

\author[B. Marchionna]{Bianca Marchionna}
\address{
Université catholique de Louvain\\
Research Institute in Mathematics and Physics\\
Chemin du Cyclotron 2, 1348 Louvain-la-Neuve, Belgium}
\email{bianca.marchionna@uclouvain.be}

\author[B. Nucinkis]{Brita Nucinkis}
\address{
Royal Holloway, University of London\\ 
Department of Mathematics, McCrea Building, TW20~0EX Egham, UK}
\email{brita.nucinkis@rhul.ac.uk}

\author[Y. Santos]{Yuri Santos Rego}
\address{
University of Lincoln\\ 
Charlotte Scott Research Centre for Algebra \\ 
College of Health and Science, Brayford Pool, LN6 7TS Lincoln, UK}
\email{ysantosrego@lincoln.ac.uk}

\begin{document}

\begin{abstract}
We propose the systematic study of presentations that can be {\em generalised over a continuous open group monomorphism $\phee$}. Presentations with this property can turn some abstract presentations, such as standard presentations of orientable surface groups, Artin groups, and some Thompson groups, into topological groups with a prescribed open subgroup.  
Later we focus on right-angled Artin groups (RAAGs) and introduce a notion of {\em topological RAAGs}. Our approach differs from lattice envelopes and produces examples of locally compact (LC) groups that contain RAAGs as discrete subgroups, but generally not as lattices. We investigate some geometric aspects of topological RAAGs, with a special emphasis on compactness properties of LC ones. This includes a study of universal Salvetti-type complexes which may be of independent interest. These complexes share some properties with buildings. Although in some cases they are $\CAT(0)$ cube complexes and provide models for classifying spaces, in other cases they are not even uniquely geodesic. For a large class of examples we establish high connectivity properties for these complexes. This yields novel examples of 
 LC groups with prescribed compactness properties or rational cohomological dimension. We note that the Bestvina--Brady machinery does not automatically generalise to this setting; nevertheless, we extend the Bieri--Stallings construction to obtain totally disconnected locally compact (TDLC) groups of type~$\mathtt{FP}_n$ but not~$\mathtt{FP}_{n+1}$. Along the way we record counterparts of cohomological results, such as a Mayer--Vietoris sequence and K\"unneth formula in discrete (co)homology for TDLC groups, which have not appeared elsewhere in the literature. Despite our non-discrete LC focus we obtain, as by-product, new examples of discrete groups with controlled finiteness properties including, for every $n \in \mathbb{Z}_{\geq 1}$, a Thompson-like Bieri--Stallings group of type $\mathtt{F}_n$ but not $\mathtt{F}_{n+1}$. 
\end{abstract}

\maketitle

\tableofcontents

\section{Introduction} \label{sec:intro}
Although the theory of finiteness properties for discrete groups is well established, a comparable theory of compactness properties for (non-compact, non-discrete) locally compact groups is at a relatively early, though rapidly evolving, stage~\cite{abels-tiemeyer,CastellanoWeigel0,ccc,chanfi-witzel,BonnSauer,BonnGiersbach}. As is known to experts, the study of the compactness properties $\CP_n(R)$, where $R$ a commutative ring with unit, and $\tC_n$ for locally compact (LC) groups can be reduced to the totally disconnected locally compact (TDLC) case~\cite{chanfi-witzel}. When working exclusively with TDLC groups, we will follow~\cite{ccc} and use the term ``finiteness properties" instead of ``compactness properties" and  will use~$\FP_n(R)$ and~$\tF_n$ instead of $\CP_n(R)$ and $\tC_n.$

A primary motivation for this work is the construction of new families of LC~groups with prescribed (co)homological properties. For instance, groups 
\emph{separated by homological} (resp.~\emph{homotopical}) \emph{compactness properties}, i.e.~groups that are of type~$\CP_n(R)$ but not of type~$\CP_{n+1}(R)$ (resp.~of type~$\tC_n$ but not of type~$\tC_{n+1}$) for some positive integer~$n$, or LC groups with a good model for their classifying spaces. In the discrete setting, there is a rich body of work about separating groups according to their finiteness properties. 
Some of the first (and most famous) families of discrete groups separated by finiteness properties are Bieri--Stallings groups~$SB_n$, $n\geq 1$ \cite{StallingsSB2,bieri} and Bestvina--Brady groups~$BB_n$, $n\geq 1$~\cite{BestvinaBrady}. Other classical examples are recalled in \Cref{rem:finprop}.

The following theorem illustrates, with a wide family of examples, some applications of our results.

\begin{mainthm}\label{application}
For every finitely generated right-angled Artin group (RAAG)~$A_\Gamma$ and every non-discrete TDLC group~$U$, there exists a ``topological RAAG''~$\caA_\Gamma=\caA_\Gamma(U)$ with the following properties.
\begin{enumerate}
    \item \label{Appl1} $\mc{A}_\Gamma$ is non-discrete, TDLC, and acts cocompactly by cubical isometries on a finite-dimensional $\CAT(0)$-cube complex~$\wtil{S}_\Gamma$.
    \item \label{Appl2} $\mc{A}_\Gamma$ contains~$A_\Gamma$ as a discrete subgroup and $U$ as an open subgroup.
    \item \label{Appl3} There is a normal form for elements of~$\mc{A}_\Gamma$ induced by normal forms from~$A_\Gamma$ and from coset representatives~$\caO\backslash U$ of an edge stabiliser~$\caO$ of the action~$\mc{A}_\Gamma \curvearrowright \wtil{S}_\Gamma$.
    \item \label{Appl4} $\mc{A}_\Gamma$ is of type~$\FP_n(R)$ (resp.~$\tF_n$) if and only if~$U$ is so. 
\end{enumerate}
Further specialising the choices of~$U$ and~$A_\Gamma$, the following hold.
\begin{enumerate}
\setcounter{enumi}{4}
    \item \label{Appl5} If~$U$ is profinite, then additionally~$\wtil{S}_\Gamma$ is locally finite, the action~$\mc{A}_\Gamma \curvearrowright \wtil{S}_\Gamma$ is geometric (in the sense of \cite[Definition~4.C.1]{cdlh:metric}), the TDLC group~$\mc{A}_\Gamma$ is of type~$\tF_\infty$, and the $1$-skeleton of~$\wtil{S}_\Gamma$ is a Cayley--Abels graph for~$\mc{A}_\Gamma$. Moreover, in this case~$\wtil{S}_\Gamma$ is itself a model of minimal dimension for the classifying space~$\underline{\operatorname{E}}\caA_\Gamma$ for proper actions, and~$\ccd_\Q(\caA_\Gamma)=\ccd_\Q(A_\Gamma) = \dim \wtil{S}_\Gamma$.
    \item \label{Appl6} If~$U$ is not of type~$\FP_m(R)$ for some~$m$ and~$R$ and~$\Gamma$ is a finite connected irreducible graph with more than one vertex, or if~$U$ is profinite and~$A_\Gamma$ is one-ended, then $\caA_\Gamma$~is neither a lattice envelope of, nor is quasi-isometric to,~$A_\Gamma$.
\end{enumerate}
\end{mainthm}
For instance, choosing $A_\Gamma = \Z^n$ and $U = \Z_p$, the ring of $p$-adic integers, yields a non-discrete TDLC RAAG $\caA_\Gamma$ of type $\tF_\infty$ containing $\Z^n$ discretely (but not cocompactly), and equipped with a concrete $n$-dimensional $\underline{\operatorname{E}}\caA_\Gamma$. Or, choosing the matrix group $U = \mathbf{A}_4(\Q_p) \leq \GL_4(\Q_p)$ of Abels \cite{Abels} and any RAAG $A_\Gamma$, the corresponding non-discrete TDLC RAAG $\mc{A}_\Gamma$ is compactly presented but not of type $\CP_3(\Z)$, and is not a lattice envelope of $A_\Gamma$. A full argument for \Cref{application} is provided at the end of Section~\ref{sus:mini-intro-RAAGs}.

Abstract RAAGs~$A_\Gamma$ are always of type $\tF_\infty$ while Bieri--Stalings groups~$SB_n$ and Bestvina--Brady groups~$BB_n$ can have varying finiteness properties. In our case the situation is different: the ``topological RAAGs" themselves may already exhibit varying compactness properties, and an analogue of the Bestvina--Brady machinery is generally not possible; see \Cref{s:BiSta} for a discussion. Nevertheless, we adapt the Bieri--Stallings construction to build further TDLC groups of type $\FP_n$ but not $\FP_{n+1}$; see \Cref{thmE} further below.

We elucidate in Section~\ref{sus:mini-intro-presentations} how our general construction of topological groups works, and discuss the main results for topological RAAGs in Section~\ref{sus:mini-intro-RAAGs}. 

\subsection{Generalisable presentations} \label{sus:mini-intro-presentations}
One attempt to obtain non-discrete LC groups with prescribed  
properties would be to start off with an abstract group presentation $\langle X\mid R\rangle$ and ask whether it admits an adequate topology. The broad problem of equipping groups with a non-discrete Hausdorff topology was posed by Markov in the 1940s; see \cite{markov}. Here we diverge from Markov's route while still starting off with presentations. 
Specifically, we develop a general construction of possibly non-discrete topological groups from a given group presentation $\langle X\mid R\rangle$ 
and a continuous open monomorphism~$\varphi\colon \caO\hookrightarrow U$ of Hausdorff topological groups~$\caO\leq_o U$. (The notation ``$\leq_o$" is shorthand for ``$\caO$ is an open subgroup of $U$''.) The key idea is to replace the free group~$F(X)$ with an \emph{iterated HNN-extension} 
$$\caH_X(\varphi):=\big(U\ast F(X)\big)/\LL t\omega t^{-1}\varphi(\omega)^{-1} \mid t\in X,\,\omega \in \caO\GG$$ with set of stable letters~$X$ and with respect to the continuous open monomorphism $\varphi\colon \caO\hookrightarrow U$ above.
{Throughout, $X$ will always be finite.}
An iterative application of~\cite[Proposition~8.10]{cdlh:metric} implies that there is a unique Hausdorff group topology on~$\caH_X(\varphi)$ that makes the canonical embedding~$U\hookrightarrow \caH_X(\varphi)$ open and continuous.
The group we construct is the quotient 
\begin{equation}\label{eq:Gphi}
    G_\varphi:=\caH_X(\varphi)/\LL R\GG
\end{equation}
endowed with the quotient topology. We are then interested in presentations~$\langle X\mid R\rangle$ that \emph{can be generalised over~$\varphi$} --- that is, presentations for which the canonical projection~$\caH_X(\varphi)\twoheadrightarrow G_\varphi$ is injective once restricted to~$U$; see~\Cref{defn:genpres}. This amounts to requiring that, up to isomorphisms of topological groups, $U$ embeds as an open subgroup of $G_\varphi$.
This gives us control over the topology of~$G_\phee$. 
\begin{mainthm}[\protect{Propositions~\ref{prop:topologyofArtingps} and~\ref{prop:connCompGP}}]\label{thmA}
 Let~$\langle X\mid R\rangle$ be a presentation {with~$X$ finite and} that can be generalised over~$\phee\colon\mc{O} \into U$ with $U$ Hausdorff. Then~$G_\varphi$ is Hausdorff and the following hold.
\begin{enumerate}
\item $G_\phee$ is discrete if and only if~$U$ is discrete.
\item$G_\phee$ is connected if and only if~$U$ is connected and~$X=\leer$.
\item $G_\phee$ is locally compact if and only if~$U$ is locally compact.
\item $G_\phee$ is totally disconnected if and only if~$U$ is totally disconnected.
\item $G_\phee$ is compact if and only if~$U$ is compact and~$X = \leer$.
\item $G_\phee$ is compactly generated whenever~$U$ is so.
\item If $U$ is compactly presented,  $\mc{O}$ is compactly generated and~$R$ is finite, then~$G_\phee$ is compactly presented.
\item The connected component of the identity in $G_\phee$ coincides with the connected components~$U_0=\caO_0=\varphi(\caO)_0$ of the identity in~$U$,~$\caO$ and~$\varphi(\caO)$, respectively. Moreover, $G_\varphi/U_0$ is topologically isomorphic to~$G_{\wtil{\varphi}}$, where~$\wtil{\varphi}\colon \caO/U_0\hookrightarrow U/U_0$ is given by~$\wtil{\varphi}(\omega U_0)=\varphi(\omega)U_0$.
\end{enumerate}
\end{mainthm}
Note that if~$U=\{1\}$, in particular $\varphi=\iid_U$, then~$G_\varphi$ is canonically isomorphic to $F(X)/\LL R\GG \cong \langle X\mid R\rangle$ endowed with the discrete topology.

Not every presentation $\langle X\mid R\rangle$ can be generalised over an arbitrary $\phee\colon\caO\hookrightarrow U$. We highlight two extreme types of presentations. 
On one hand, presentations whose relators are {\em balanced} turn out to be generalisable over every $\phee\colon\caO\hookrightarrow U$; see~\Cref{thm:retract}. For the concept of {\em presentations with balanced relators} we refer to \Cref{defn:balanced}. The standard finite presentations of Artin groups (in particular RAAGs and braid groups), fundamental groups of orientable surfaces of positive genus, and Thompson's group~$F$ have balanced relators (see~\Cref{ex:generalisablepresentations}). 
On the other hand, presentations~$\langle X\mid R\rangle$ containing a relator of the form $x^n$, for some~$x\in X$ and~$n\geq 2$, 
 cannot be generalised over any open continuous non-surjective monomorphism~$\varphi\colon U\hookrightarrow U$; see~\Cref{prop:morethanCoxeter}. 
 
 The above observations also imply that ``being generalisable over $\phee$" is not a group property. That is, a group~$G$ may admit one presentation that is generalisable over every~$\phee$ and another presentation that is not generalisable over some $\phee$; see~\Cref{ex:diffpres}. 

The first broad question that arises is how the parameters in the definition of~$G_\phee$ affect its algebro-topological properties.

\begin{ques}\label{ques:generalpropertiesofgeneralisedpresentations}
    Suppose a presentation $\langle X\mid R \rangle$ with $X$ finite can be generalised over $\varphi: \caO \hookrightarrow U$. Which properties of the topological group $G_\varphi$ are determined by those of~$\caO$, $U$ and the abstract group $F(X)/\LL R\GG$?
\end{ques}

{Note that the quotient construction in Equation~\eqref{eq:Gphi} also allows us to plug in discrete groups $\caO$ and $U$ and thus to construct interesting new examples of discrete groups. For example, in~\Cref{ex:ThompsonV} we shall use our machinery with Thompson's group~$V$ playing the role of the base group~$U$ to produce novel Thompson-like examples separated by finiteness properties; see \cite{skipper-witzel-zaremsky} for context and first examples of this kind. 

More generally, the construction of $G_\phee$ is the key of our machinery to build LC groups separated by compactness properties. To this end the first problem that needs to be tackled is the following specialisation of \Cref{ques:generalpropertiesofgeneralisedpresentations}. 

\begin{ques}\label{ques:compactness}
    Let $G_\varphi$ be an LC group arising from a generalised presentation of $\langle X\mid R \rangle$ over $\varphi: \caO \hookrightarrow U$, with~$X$ finite. 
    Which compactness properties does $G_\phee$ inherit from those of $\caO$ and $U$ and the finiteness conditions of~$F(X)/\LL R\GG$?
\end{ques}

 In the simplest case~$R=\emptyset$, answers to \Cref{ques:compactness} are well-known; see for instance~\Cref{cor:HpheeisFPn}. Another straightforward scenario would present itself if the abstract group~$G:= F(X)/\LL R\RR$ happens to be a cocompact (or more generally quasi-isometric) subgroup of~$G_\phee$. In this case, $G_\phee$ would inherit finiteness properties of~$G$~\cite[Corollary~5.7]{ccc}. 
As it turns out, $G$ always embeds as a discrete subgroup in~$G_\varphi$. Whether, additionally,~$G$ is a (possibly cocompact) lattice is a topic in the prominent theory of \emph{lattice envelopes}, which is used to create non-discrete LC groups containing a given abstract group as a lattice; see~\cite{bafusa} and references therein. 

\begin{ques}\label{ques:envelopes}
    Let a presentation $\langle X\mid R \rangle$ with $X$ finite be generalisable over $\varphi: \caO \hookrightarrow U$ with $U$ locally compact. When is $G_\phee$ a lattice envelope of the abstract group $G:=F(X)/\LL R\GG$?
\end{ques}

 We shall resolve Questions~\ref{ques:compactness} and~\ref{ques:envelopes} in case $\langle X \mid R\rangle$ is the standard presentation of a right-angled Artin group (RAAG). In fact, we show that in the general case, i.e.~$G=F(X)/\LL R\GG$ not necessarily a RAAG, it is hardly ever cocompact in $G_\phee$; see \Cref{prop:cocomp}. Consequently, combining \Cref{prop:cocomp} and work of Caprace and De Medts~\cite[Proposition~5.4]{CDM:raag}} one concludes that, if~$G_\varphi$ is induced by a RAAG presentation $G = A_\Gamma$ induced by a finite connected irreducible graph~$\Gamma$ with more than one vertex and~$U$ is non-discrete {TD}LC, then~$A_\Gamma$ is never a lattice in~$G_\varphi$ except in the ``trivial" case where~$\caO=U=\varphi(\caO)$ and moreover~$U$ is compact. 
 
 With~$A_\Gamma$ not cocompact in~$G_\phee$, one does not obtain compactness properties for free. In~\Cref{s:FPRAAGs} we exploit RAAGs and adequate base groups~$\caO \osbgp U$ to build non-discrete LC groups $G_\phee$ whose compactness properties are completely determined, despite~$G_\phee$ not containing RAAGs as lattices. Our results will also show that often~$G=A_\Gamma$ is not even quasi-isometric to~$G_\varphi$.

Back to general presentations, the proof of \Cref{prop:cocomp} mentioned above makes heavy use of the following topological splitting arising from any generalised presentation~$\langle X\mid R\rangle$ over~$\varphi\colon \caO \hookrightarrow U$: 
\begin{equation*}
    G_\varphi\cong \LL U\GG_{G_\varphi}\rtimes G;
\end{equation*}
see~\Cref{prop:artin}. 
Unless~$\varphi(\caO)=\caO=U$, the normal closure~$\LL U\GG_{G_\varphi}$ actually strictly contains~$U$ (see \Cref{prop:normU}) and $G_\varphi$ is not simply a group extension of~$U$ and~$G$. In other words, one obtains an interesting topological group $G_\phee$ when we are given a non-surjective continuous open monomorphism $\varphi\colon \caO\hookrightarrow U$ with $\caO\leq_o U$.  
The easiest, yet still meaningful, case is obtained by taking~$\caO$ as any proper open subgroup of~$U$ and let $\varphi$ be the canonical inclusion. By van Dantzig's theorem, this construction is always possible if~$U$ is a TDLC group. More elaborated discrete and non-discrete examples are offered in~\Cref{sus:exMonosO=U}, where we construct non-surjective continuous open monomorphisms~$\varphi\colon U\hookrightarrow U$. Restricting to the case~$\caO=U$ naturally leads to study of topological groups~$U$ with strong rigidity properties, in particular those that fail to be open co-Hopfian, 
see~\Cref{def:opencoHopfian}. Note that being open co-Hopfian for LC groups is implied, for instance, by the Howe--Moore property~\cite{HM}; see \Cref{sus:exMonosO=U}.

\subsection{Topological right-angled Artin groups (topological RAAGs)} \label{sus:mini-intro-RAAGs} 
Right-angled Artin presentations given by a graph $\Gamma$ have balanced relators and can then be generalised over every continuous open monomorphism $\varphi\colon\caO\hookrightarrow U$ with~$\caO\leq_o U$. We will call the resulting topological groups {\em topological RAAGs} and will denote them  by $\caA_\Gamma(\varphi)$; see \Cref{s:topRAAGs}. 

We start a systematic investigation of topological RAAGs and establish several properties. For example, in case the underlying map $\phee$ satisfies~$\phee(\caO)=\caO\leq_o U$, topological RAAGs admit a normal form; see~Theorem~\ref{thm:NF}. 

When addressing their cocompactness properties, see \Cref{ques:compactness}, we can easily answer some  very specific cases, such as when $\caO=U=\varphi(\caO)$ or when the defining graph $\Gamma$ is of very specific shapes amenable to standard inductive arguments; see~\Cref{thm:FPc}.
However, for general $\Gamma$, deducing compactness properties of $\mc{A}_\Gamma(\phee)$ requires a lot more effort. To this end, in Section~\ref{s:Salvetti} we construct a universal Salvetti-type complex $\wtil{S}_\Gamma(\phee)$ for our topological RAAG~$\caA_\Gamma(\phee)$. This turns out to be a cubed complex showing building-like behaviour, see Sections~\ref{sus:aparts} and~\ref{sus:build}, and which
is covered by isometric copies of the universal Salvetti complex~$\wtil{\Sigma}_\Gamma$ of the abstract RAAG~$A_\Gamma.$ When $U=\{1\}$ and hence $\caA_\Gamma(\varphi)\cong A_\Gamma,$ the complex~$\wtil{S}_\Gamma(\varphi)$ recovers~$\wtil{\Sigma}_\Gamma.$ If~$\Gamma=(S,\emptyset)$ is totally disconnected, then~$\caA_\Gamma(\varphi)\cong \caH_S(\varphi)$ and $\wtil{S}_\Gamma(\varphi)$ is the universal Bass--Serre tree of the topological HNN-extension $\caH_S(\varphi)$; see \Cref{ex:basicExSalv}.

The complex~$\wtil{S}_\Gamma(\varphi)$ has several good properties that are analogous to those of~$\wtil{\Sigma}_\Gamma$. For instance, $\caA_\Gamma(\varphi)$ acts cocompactly by cubical isometries on~$\wtil{S}_\Gamma(\varphi)$ with vertex-stabilisers isomorphic to~$U$ and, if~$\varphi(\caO)\subseteq \caO$, with pointwise stabilisers of positive dimensional cells isomorphic to~$\caO$; see~\Cref{prop:stabs}. In~\Cref{prop:basicS} we record further  basic properties of~$\wtil{S}_\Gamma(\varphi)$, such as an explicit formula of its dimension, a characterisation of local finiteness, and a characterisation of when its $1$-skeleton is a Cayley--Abels graph for~$\caA_\Gamma(\varphi)$.

The major differences between~$\wtil{S}_\Gamma(\varphi)$ and~$\wtil{\Sigma}_\Gamma$ emerge when we compare their metric structures and homotopy types. While~$\wtil{\Sigma}_\Gamma$ is always a $\CAT(0)$ cube complex and hence contractible, for~$\wtil{S}_\Gamma(\varphi)$ we show the following:

\begin{mainthm}[Metric, homotopical and homological properties of~$\wtil{S}_\Gamma(\varphi)$]\label{thmB} 
\,\\For every finite graph~$\Gamma$ and every continuous open monomorphism~$\varphi\colon \caO\hookrightarrow U$ with~$\caO\leq_o U$, the cubed complex~$\wtil{S}_\Gamma(\varphi)$ is a complete geodesic metric space with respect to the piecewise Euclidean metric (see~\Cref{prop:basicS}\eqref{prop:basicS5}). Furthermore,
  \begin{enumerate}
        \item \label{thmB2} For $\varphi(\caO)=\caO$, $\wtil{S}_\Gamma(\varphi)$ is a $\CAT(0)$ cube complex (see~\Cref{thm:CAT0}).
      \item \label{thmB1} If $\Gamma$ is connected with at least one edge and $\varphi(\caO)\subseteq\caO$, then the metric space $\wtil{S}_\Gamma(\varphi)$ is uniquely geodesic if and only if $\varphi(\caO)=\caO$ (see~\Cref{thm:uniqgeod} and~\Cref{cor:uniqgeod}).
      \item \label{thmB3} If $\Gamma$ is connected and $\caO=U\neq \varphi(\caO)$, the homotopical connectivity $\conn_\pi(\wtil{S}_\Gamma(\varphi))$ and homological connectivity $\conn_h(\wtil{S}_\Gamma(\varphi))$ of $\wtil{S}_\Gamma(\varphi)$ satisfy
      \begin{equation*}
\conn_\pi(\wtil{S}_\Gamma(\varphi))=\conn_h(\wtil{S}_\Gamma(\varphi)) \geq \conn_h(L(\Gamma))+1,
      \end{equation*}
      where $L(\Gamma)$ is the clique complex of~$\Gamma$ (see~\Cref{thm:connPhi(U)=U}). In particular, $\wtil{S}_\Gamma(\varphi)$ is contractible whenever~$L(\Gamma)$ is $\Z$-acyclic (see~\Cref{cor:connPhi(U)=U}).
  \end{enumerate}
\end{mainthm}
The proof of~\Cref{thmB} exploits, in various ways, the fact that~$\wtil{S}_\Gamma(\varphi)$ is covered by $\caA_\Gamma(\varphi)$-translated copies of~$\wtil{\Sigma}_\Gamma$, which we call \emph{apartments} of~$\wtil{S}_\Gamma(\varphi)$. In particular, for the proofs of~\Cref{thmB}\eqref{thmB2}--\eqref{thmB3}, the crux is to understand the homotopy and the homology types of the intersections of apartments, which we manage under the respective hypotheses on~$\varphi$ and~$\Gamma$ above. 

In particular, when $\varphi(\caO)=\caO$, apartments are either pairwise disjoint or intersect only at a vertex (see~\Cref{cor:intApphi(O)=O}). As a consequence,~$\wtil{S}_\Gamma(\varphi)$ has the homotopy type of the nerve complex of the relevant covering by apartments, which we prove, thanks to the normal form of~$\caA_\Gamma(\varphi)$ provided in~\Cref{thm:NF}, to be the clique complex of a chordal graph. The latter is well-known to be contractible. The favourable shape of the intersections of apartments allows us to also prove that~$\wtil{S}_\Gamma(\varphi)$ is a cube complex satisfying Gromov's link condition (see~\Cref{thm:CAT0}). 
This yields that~$\wtil{S}_\Gamma(\varphi)$ is a $\CAT(0)$ cube complex. Note that being a cube complex is a property not necessarily guaranteed if~$\varphi(\caO)\neq \caO$, see~\Cref{ex:pockets}.

If instead~$\Gamma$ is connected and~$\caO=U\neq\varphi(\caO)$, we show that all possible finite intersections of apartments have the homotopy (resp. homology) type of the sublevel sets in~$\wtil{\Sigma}_\Gamma$ considered by Bestvina--Brady~\cite{BestvinaBrady}, which in turn are expressed in terms of the respective connectivities of~$L(\Gamma)$. 
This, together with the contractibility of the nerve of the covering by apartments, implies \Cref{thmB}\eqref{thmB3}.

Thanks to the higher connectivity of~$\wtil{S}_\Gamma(\varphi)$ and the favourable action of~$\caA_\Gamma(\varphi)$ on it, we are able to deduce the following. 

\begin{mainthm}[\protect{Theorems~\ref{thm:FPa} and~\ref{thm:FPb}}]\label{thmC}
   Let~$\Gamma$ be a finite graph and~$\varphi\colon \caO \hookrightarrow U$ a continuous open monomorphism of LC~groups~$\caO\leq_o U$. For every~$n\in \Z_{\geq 1}$ and every unital ring~$R$ the following hold.
   \begin{enumerate}
       \item \label{thmC1} Suppose~$\varphi(\caO)=\caO$. If~$\caO$ is of type~$\CP_{n-1}(R)$ and~$U$ is of type~$\CP_n(R)$, then~$\caA_\Gamma(\varphi)$ is of type~$\CP_n(R)$.
 Moreover, if~$\caO$ is type~$\CP_{n}(R)$ then
 \begin{equation*}
     U \text{ is of type }\CP_n(R)\,\Longleftrightarrow\,\caA_\Gamma(\varphi)\text{ is of type }\CP_n(R).
 \end{equation*}
       \item \label{thmC2} Assume~$\Gamma$ is connected, $\caO=U\neq\varphi(\caO)$, and $1\leq n\leq \conn_h(L(\Gamma))+2$, where $\conn_h(L(\Gamma))$ is the homological connectivity of the clique complex~$L(\Gamma)$ of~$\Gamma$. If~$U$ is of type~$\CP_n(R)$ then~$\caA_\Gamma(\varphi)$ is of type~$\CP_n(R)$. 
   \end{enumerate}
   The same statements hold by replacing~$\CP_\bullet(R)$ with~$\tC_\bullet$ where applicable.
\end{mainthm}
By~\Cref{lem:FPpropTDLC}, the proof of~\Cref{thmC} reduces to the case when~$U$ hence~$\caO$ and~$\caA_\Gamma(\varphi)$ are~TDLC. For this reason, both~\Cref{thm:FPa} and~\Cref{thm:FPb} are results on finiteness conditions of TDLC groups. Here we can rely on the richer and better-understood theory of compactness, or finiteness,  properties of TDLC groups developed by Castellano--Corob~Cook in~\cite{ccc}. In particular, given~\Cref{thmB}, the proofs of~\Cref{thmC} follow from the TDLC version of Brown's celebrated criterion \cite{BrownFP} and the fact that finiteness properties are inherited by group retracts of TDLC groups.

Although the proof of~\Cref{thmC} only makes use of the higher connectivity properties of~$\wtil{S}_\Gamma(\varphi)$, the fact that~$\wtil{S}_\Gamma(\varphi)$ is~$\CAT(0)$ if~$\varphi(\caO)=\caO$ allows us to deduce the following: 

\begin{mainthm}[\protect{see~\Cref{thm:EFGspace}}]\label{thmD}
    Let~$\Gamma$ be a finite graph and~$\varphi\colon \caO\hookrightarrow U$ a continuous open monomorphism of LC groups~$\caO\leq_oU$ such that~$\varphi(\caO)=\caO$ and~$|U:\caO|<\infty$. 
    
    Let~$\caF$ be any family of~$\caA_\Gamma(\varphi)$ containing~$U$ and consisting of closed subgroups~$H\leq \caA_\Gamma(\varphi)$ such that~$|H:H\cap gUg^{-1}|<\infty$ for some~$g\in \caA_\Gamma(\varphi)$. 
    Then~$\wtil{S}_\Gamma(\varphi)$ is a model for the classifying $\caA_\Gamma(\varphi)$-CW complex~$\spE_\caF\caA_\Gamma(\varphi)$ of~$\caA_\Gamma(\phee)$ with respect to~$\caF$ (in the sense of~\cite{luck:survey}).
\end{mainthm}
Recall that a \emph{family} of an LC group~$G$, following \cite{luck:survey}, is any collection of closed subgroups of~$G$ closed under finite intersections and conjugations of its elements.
\Cref{thmD} applies, in particular, when~$U$ and hence~$\caO$ are compact and~$\caF=\caK$ is the family of all compact subgroups of~$\caA_\Gamma(\varphi)$. In this setting $\wtil{S}_\Gamma(\varphi)$ is  a model for~$\underline{\operatorname{E}}\caA_\Gamma(\varphi)={\operatorname{E}}_\caK\caA_\Gamma(\varphi)$, see~\Cref{cor:ECOGspace}. Furthermore, if~$U$ is profinite, we also deduce in~\Cref{thm:cdQRAAG} that the dimension of~$\wtil{S}_\Gamma(\varphi)$ equals the rational discrete cohomological dimension of~$\caA_\Gamma(\varphi)$. Hence, in this case,~$\wtil{S}_\Gamma(\varphi)$ is a model for~$\underline{\operatorname{E}}\caA_\Gamma(\varphi)$ of minimal dimension, see~\Cref{thm:cdQRAAG}. This extends the analogous statement for abstract RAAGs proved by Charney and Davis in \cite[Corollary~3.2.2]{CharneyDavisKpi1s}.

 The reader familiar with the Bestvina--Brady machinery will naturally inquire whether, given the construction of~$\wtil{S}_\Gamma(\varphi)$, we are able to obtain analoga of Bestvina--Brady groups, which would potentially lead to another family of LC groups with prescribed compactness properties. We observe, however, that this cannot be done, at least not with our current constructions; see~\Cref{obs:BBgps}.

 Despite the absence of Bestvina--Brady subgroups we directly generalise the earlier Bieri--Stallings construction in Section~\ref{s:BiSta}. We thus obtain novel examples of TDLC~groups of type $\mathtt{FP}_n$ but not $\mathtt{FP}_{n+1}$. 

 \begin{mainthm}[\protect{\Cref{thm:BieriStallings}}]\label{thmE} 
Let~$R \in \{\Z, \Q\}$ and let $\phee : U \into U$ be a continuous open monomorphism of TDLC groups. If the base group~$U$ is $\Q$-acyclic and of type~$\FP_n(R)$, then the \emph{$n$-th generalised Bieri--Stallings group}~$SB_n(\phee)$ is of type~$\FP_n(R)$ but not of type $\FP_{n+1}(R)$.
 \end{mainthm}

To prove \Cref{thmE} we have to establish TDLC counterparts of fundamental (co)homological results that have not yet been recorded in the literature, such as a Mayer--Vietoris sequence  and a K\"unneth formula; see Theorems~\ref{thm:mv} and~\ref{thm:kunneth}, respectively. 

To conclude, note that \Cref{application} stated earlier in the introduction is an immediate consequence of several of the more general statements mentioned above. For the sake of completeness, let us briefly outline an argument.

\begin{proof}[Proof of \Cref{application}] 
Given $A_\Gamma$ and $U$ as in the statement, van Dantzig's theorem implies that $U$ contains some compact open subgroup $\caO \osbgp U$. Take $\phee : \caO \into U$ to be the inclusion map and set $\caA_\Gamma := \caA_\Gamma(\phee)$ and $\wtil{S}_\Gamma := \wtil{S}_\Gamma(\phee)$. 

Now item~\eqref{Appl1} is covered by \Cref{thmA} and~\Cref{thmB}. Item~\eqref{Appl2} follows from \Cref{thm:retract} and \Cref{prop:artin}. Item~\eqref{Appl3} is just \Cref{thm:NF}. Item~\eqref{Appl4} is a consequence of \Cref{thmC}\eqref{thmC1}. 

In case~$U$ is profinite, we may take~$\caO \osbgp U$ any proper open subgroup, which then has finite index. Recalling that profinite groups enjoy all finiteness properties (i.e., are of type~$\tF_\infty)$, item~\eqref{Appl5} now follows from \Cref{prop:stabs}, \Cref{prop:basicS}, \Cref{thmC}\eqref{thmC1}, \Cref{thmD} and \Cref{thm:cdQRAAG}. 

As argued at the end of Section~\ref{sus:mini-intro-presentations} (see \Cref{prop:cocomp} and~\cite[Proposition~5.4]{CdM}), if~$U$ is non-discrete TDLC, $\caO\neq U$ and~$\Gamma$ is a finite connected irreducible graph with more than one vertex, then~$\caA_\Gamma$ is not a lattice envelope of~$A_\Gamma$. If both groups were quasi-isometric, they would have the same finiteness properties \cite[Corollary~5.7]{ccc}, which cannot happen if~$U$ is not~$\FP_m(R)$. (Recall that~$A_\Gamma$ is always~$\tF_\infty$.) Assuming again~$U$ profinite and~$\caO\neq U$, if~$\caA_\Gamma$ and~$A_\Gamma$ were quasi-isometric then~$\wtil{S}_\Gamma$ would be quasi-isometric to the universal Salvetti complex~$\wtil{\Sigma}_\Gamma$ for the abstract RAAG~$A_\Gamma$. If~$A_\Gamma$ is $1$-ended, so is~$\wtil{\Sigma}_\Gamma$. In contrast, $\caA_\Gamma$~retracts onto the HNN-extension~$\mc{H}_{\{t\}}(\phee)$ (see \Cref{thm:retract}) which is infinitely-ended (as seen from its Bass--Serre tree). Thus~$\wtil{S}_\Gamma$ also has infinitely many ends, yielding a contradiction. Item~\eqref{Appl6} is thus proved. 
\end{proof}

\subsection*{Acknowledgement} 
The four authors are indebted to Raphael Appenzeller, Jos\'e Burillo, Indira Chatterji, Ged Corob Cook, Sam Corson, Auguste Hébert, Dawid Kielak, Waltraud Lederle, Oscar Ocampo and Stefan Witzel for helpful conversations and remarks throughout the course of this work. We are particularly grateful for Raphael Appenzeller for providing, and allowing us to reproduce, 3D-printed models of (pieces of) our Salvetti complexes. 

This project was initiated during the MFO Workshop ``Homological aspects for TDLC groups" in November 2023. All authors would like to thank the \emph{Mathematisches Institut Oberwolfach} (MFO) and all their staff for hosting us, as well as all participants of the workshop for stimulating discussions. 
Work on this project was continued across multiple research visits. The authors would like to thank for their hospitality: the Charlotte Scott Research Centre for Algebra of the University of Lincoln (for IC, BM, BN in February 2024), Bielefeld University (for BN and YSR in September 2024 and for IC and YSR in March 2026), the University of Heidelberg (for YSR in May 2025), the University of Galway (for IC, BN and YSR in May 2025), and Royal Holloway, University of London (for BM and YSR in May 2026). 

This research was partially supported by the \emph{Deutsche Forschungsgemeinschaft} (DFG, German Research Foundation) CRC-TRR~358, Project ID 491392403. IC and BM are members of the INdAM -- GNSAGA group. {IC~coordinates the ``INdAM - GNSAGA Project'' CUP E53C25002010001.} BM is a chargée de recherche of the Fonds de la Recherche Scientifique -- FNRS, and was partially supported by that funding and well as by the RTG 2229 ``Asymptotic Invariants and Limits of Groups and Spaces'' of the DFG. BM and BN were partially supported by LMS Scheme 4 grant 42517. YSR received pump-priming support from the research committees of SEPS and CoHS at the University of Lincoln.


\section{Preliminaries and conventions} \label{sec:prelim}

\subsection{Abstract and topological groups} \label{sec:basicTopGroups}
Throughout this work {\bf topological groups will always be Hausdorff} unless specified otherwise.
We amend standard group-theoretic notation to take topological properties into account: for example, ``$N \cns G$'' means ``$N$ is a closed normal subgroup of~$G$'' and ``$U \leq_o G$'' means ``$U$ is an open subgroup of~$G$''. The connected component of~$1$ in a topological group~$G$ is denoted by~$G_0$.
As usual, topological closures are denoted by putting a vertical bar on top of the given set.  
Given two topological groups~$H$ and~$G$ and a continuous $G$-action~$\Phi\colon G\times H\to H$ on~$H$, the \emph{topological semi-direct product} of~$H$ and~$G$ with respect to~$\Phi$ is the group-theoretical semi-direct product~$H\rtimes_\Phi G$ endowed with the product topology (which is a group topology for~$H\rtimes_\Phi G$ by~\cite[Proposition~10.12(b)]{strop}).

Tensor products are usually taken over~$\Z$ or over group rings $R[G]$, in which case we usually write $\otimes_G$ when the ring of coefficients $R$ for $R[G]$-group rings is understood from context.

A free group on a set~$X$ is denoted by~$F(X)$. Whenever handling group presentations~$\spans{X \mid R}$, we sometimes commit the common abuse of notation by writing defining relators $r \in R \subseteq F(X)$ as equations. That is, if the word~$r \in R \subseteq F(X)$ can be written as~$r = w_1 w_2^{-1}$ with~$w_1, w_2 \in F(X)$, then in the presentation~$\spans{X \mid R}$ the relator~$r$ may be written as~$w_1 = w_2$ instead of~$r$. 

For a subset $X$ of an abstract group $G$, we denote by $\LL X \RR$ the \emph{normal closure} of $X$ in $G$. If the parent group $G$ needs to be made explicit, we write $\LL X\RR_G$. In case $G$ is a topological group, note that the \emph{topological normal closure} of a subset $X \subseteq G$, i.e., the intersection of all closed normal subgroups containing $X$, coincides with $\quer{\LL X \RR}$, the topological closure of the (abstract) normal closure of $X$. 

For future reference we recall the following classical fact. We include a proof as we could not find a reference.

\begin{fact}\label{fact:ConCompTG}
    Let~$G$ be a topological group with connected component of the unit~$G_0$. Then the group~$G/G_0$ is totally disconnected with respect to the quotient topology and, for every open subgroup~$H\leq G$, one has~$G_0=H_0$.
\end{fact}

\begin{proof}
    By~\cite[III.2.2, Proposition~7]{bou:top}, the connected component of every~$g\in G$ in~$G$ is the coset~$gG_0$. Hence, by~\cite[Lemma~2.9(d)]{strop}, the group~$G/G_0$ is totally disconnected with the quotient topology.

    Let now~$H\leq_o G$. By~\cite[Lemma~4.10]{strop}, $G_0$ is contained in every open subgroup of~$G$ and in particular~$G_0\subseteq H$. Note that~$G_0H_0/H_0$ is the image of the connected subgroup~$G_0\leq H$ in~$H/H_0$. Since the canonical projection~$H\twoheadrightarrow H/H_0$ is continuous, $G_0H_0/H_0$ is a connected subgroup of~$H/H_0$. But~$H/H_0$ is totally disconnected, so~$G_0H_0=H_0$ and hence~$G_0\subseteq H_0$. In turn, arguing analogously as before,~$H_0/G_0$ is a connected subgroup in the totally disconnected group~$G/G_0$. Therefore, $H_0=G_0$. 
\end{proof}

\subsection{Cell complexes and their connectivity}\label{sus:basicCellComplx}
When dealing with cell complexes, we follow standard treatment such as in \cite{FritschPiccinini,GeoBook}, and in the case of metric cell complexes we closely follow the exposition of \cite[Chapter~I.7]{BridsonHaefliger}. We freely make use of well-known facts about such spaces without further comments (e.g., unions and intersections of subcomplexes are subcomplexes, or constructions of common metrics on cell complexes). 

Recall that a non-empty, path-connected space~$X$ is \emph{$n$-connected} when its $i$-th homotopy group~$\pi_i(X)$ is trivial for every~$i \in \{0, 1, \ldots, n\}$. Similarly, we say that~$X$ is \emph{$n$-acyclic} when it is non-empty and the $i$-th integral reduced homology group~$\wtil{\hH}_i(X)$ is zero for all $i \in \{0, 1, \ldots, n\}$. Conventionally, every non-empty space is \emph{$(-1)$-connected} and~\emph{$(-1)$-acyclic}.

The \emph{nerve complex of a covering}~$\{X_\alpha\mid \alpha\in I\}$ of a set $X = \bigcup_{\alpha \in I}X_\alpha$ by non-empty subsets is the simplicial complex whose $k$-simplices consist of all possible sets of $k+1$ pairwise distinct elements $X_{\alpha_0}$, $\ldots$, $X_{\alpha_{k}}$ of the cover such that $\bigcap_{i=0}^k X_{\alpha_i} \neq \leer$. The following well-known result allows us to use nerves to check connectivity. For a proof in the homotopical case, see \cite[Theorem~6]{BjornerNerve}. Note that Bj\"orner's arguments carry over verbatim to yield the homological version of the theorem.

\begin{thm}[{Nerve theorem}]\label{thm:nerve}
    Let~$X$ be a regular cell complex\footnote{A cell complex is \emph{regular} when attaching maps of closed $n$-cells are homeomorphisms of the whole (closed) cell. Note that polyhedral complexes (in particular standard realisations of simplicial complexes) are regular.} and let $\Delta:=(\Delta_i)_{i\in I}$ be a family of subcomplexes of~$X$. Denote by $\caN(\Delta)$ the {nerve} of the covering~$\Delta$. 
    \begin{enumerate}
        \item \label{thm:nerve1} Assume that every non-empty finite intersection~$\Delta_{i_1}\cap \cdots \cap \Delta_{i_t}$, $t\geq 1$, is contractible. Then~$X$ is homotopy equivalent to~$\caN(\Delta)$.
        \item \label{thm:nerve2} 
        Let~$k\geq 1$ and assume that every non-empty finite intersection $\Delta_{i_1}\cap \cdots \cap \Delta_{i_t}$, $1\leq t\leq k+2$, is $(k-t+1)$-connected (resp.~$(k-t+1)$-acyclic). Then $X$ is $k$-connected (resp.~$k$-acyclic) if and only if $\caN(\Delta)$ is $k$-connected (resp.~$k$-acyclic).
    \end{enumerate}
\end{thm}

Graphs treated in this paper are simplicial unless stated otherwise, and we typically realise them metrically and cellularly without further remarks. To avoid trivial statements we always assume that \textbf{vertex sets of our graphs are non-empty}. We denote by  $\Gamma = (S,E)$ a graph $\Gamma$ with vertex set $V(\Gamma) = S$ and edge set $E(\Gamma) = E$. Given a graph $\Gamma = (S,E)$, an \emph{induced subgraph} (or \emph{full subgraph}) is any subgraph $\Lambda = (T,F)$ of $\Gamma$ with $T \subseteq S$, $F \subseteq E$ and such that, for every edge $\{s,t\}\in E$ with endpoints $s,t\in T$, one has $\{s,t\}\in F$. Given a subset~$T\subseteq S$, the subgraph~$\ind_\Gamma(T)$ induced by~$T$ is exactly the induced subgraph of $\Gamma$ whose vertex set is $T$. 
A \emph{clique} in a graph $\Gamma$ is a complete subgraph.
The \emph{join} of two disjoint graphs~$\Gamma_1 = (S_1,E_1)$ and~$\Gamma_2 = (S_2,E_2)$ is the graph~$\Gamma_1 \vee \Gamma_2$ (also denoted~$\Gamma_1 + \Gamma_2$ in the literature) whose vertex set is $S_1 \sqcup S_2$ and whose edge set is the union 
\[E_1 \sqcup E_2 \sqcup \{ \, \{s,\wtil{s}\} \, \mid s \in S_1, \wtil{s} \in S_2\}.\]
The (graph) join operation easily extends to finitely many graphs, and is readily seen to be associative and commutative.

\subsection{Right-angled Artin groups}\label{sus:basicRAAGs}
Right-angled Artin groups (RAAGs) are an important subfamily of Artin groups. The reader is referred to~\cite{CharneySurvey} for an introduction.

Given a finite graph~$\Gamma=(S,E)$, the \emph{right-angled Artin group} of type~$\Gamma$ is defined by the following presentation:
\[A_\Gamma = \spans{ S \mid [s,t],\,\forall\,\{s,t\}\in E },\]
where~$[s,t]=sts^{-1}t^{-1}$. 

One typically says that RAAGs interpolate between finitely generated free groups and finitely generated free abelian groups due to the extremal cases for~$\Gamma$. Indeed, from the presentation we read off that~$A_\Gamma \cong \Z^n$ when~$\Gamma$ is a complete graph on~$n$ vertices, whereas~$A_\Gamma \cong F_n$ when $\Gamma$ is a totally disconnected graph on $n$ vertices. 

A \emph{standard parabolic subgroup} (sometimes called \emph{special subgroup}) of a RAAG~$A_\Gamma$ is any subgroup generated by a subset of vertices of~$\Gamma$. 
It is a fundamental result, established by van der Lek~\cite{vanderLek} in the more general Artin case, that for~$\Lambda = \ind_\Gamma(T)$ the subgroup~$\spans{T} \leq A_\Gamma$ 
is canonically isomorphic to the RAAG~$A_\Lambda$ of type~$\Lambda$. This yields an obvious embedding~$A_\Lambda \into A_\Gamma$. 

Let also
\begin{align*}
\mc{V}^f & := \{T \subseteq S \mid \ind_\Gamma(T) \text{ is a clique in } \Gamma\}\sqcup\{\emptyset\}.
\end{align*}
The reader familiar with Coxeter groups will recognise that~$\caV^f$ is the collection of all subsets~$T\subseteq S$ which generate a finite special subgroup in the right-angled Coxeter group of defining graph~$\Gamma$.

Given a graph~$\Gamma = (S,E)$, there are some obvious ``visual" ways to split the RAAG $A_\Gamma$ into smaller parabolics. For instance, suppose $\Gamma$ is itself a disjoint union of full subgraphs $\Gamma_1 = (S_1, E_1)$, $\ldots$, $\Gamma_n = (S_n, E_n)$. That is, $S = S_1 \sqcup \ldots \sqcup S_n$ and $E = E_1 \sqcup \ldots \sqcup E_n$. Then the RAAG $A_\Gamma$ canonically splits as a free product 
\[A_\Gamma \cong A_{\Gamma_1} \ast \ldots \ast A_{\Gamma_n}\]
of standard parabolics. Similarly, suppose $\Gamma$ splits as a join of full subgraphs $\Gamma = \Gamma_1 \vee \ldots \vee \Gamma_n$. Then the RAAG $A_\Gamma$ canonically splits as a direct product 
\[A_\Gamma \cong A_{\Gamma_1} \times \ldots \times A_{\Gamma_n}\]
of standard parabolics. It is straightforward to identify vertex subsets $T \subseteq S$ for which~$A_{\ind_\Gamma(T)}$ is a direct factor of~$A_\Gamma$, $\Gamma=(S,E)$.
Namely, given two subsets of vertices~$T,V\subseteq S$ let
\begin{equation*}
    [T,V]:=\{tvt^{-1}v^{-1}\mid t\in T,\,v\in V\}\subseteq A_\Gamma.
\end{equation*}
Note that
\begin{equation}\label{eq:[T,V]}
    [T,V]=\{1\}\,\Longleftrightarrow\,\forall t\in T,v\in V,\,\{t,v\}\in E.
\end{equation}
In particular, if~$T\subseteq S$ satisfies~$[T,S\setminus T]=\{1\}$, then~$A_\Gamma$ is canonically isomorphic to~$A_{\ind_\Gamma(T)}\times A_{\ind_\Gamma(S\setminus T)}$. 
In this set-up, we call the canonical projection
\begin{equation}\label{eq:canRetr}
    \pi_T\colon A_\Gamma\cong A_{\ind_\Gamma(T)}\times A_{\ind_\Gamma(S\setminus T)}\longrightarrow A_{\ind_\Gamma(T)}
\end{equation}
the \emph{canonical retraction of $A_\Gamma$ onto $A_{\ind_\Gamma(T)}$}. We shall revisit these constructions later on in \Cref{sus:parab} in the case of topological RAAGs.


\subsection{Iterated HNN-extensions and topological groups}
\label{sus:HNNtop}
Let~$U$ be a group with a subgroup~$\caO \leq U$ and a monomorphism (i.e., injective homomorphism) $\varphi\colon \caO\hookrightarrow U$. Given a finite set $X$, recall that the \emph{iterated HNN-extension} associated to (or over)~$\varphi$ and with stable letters in~$X$ is the group
 \begin{equation*}
        \HNN(X,\phee)=\caH_X(\varphi):=\big(U\ast F(X)\big)/\LL t\omega t^{-1}\varphi(\omega)^{-1} \mid t\in X,\,\omega \in \caO\GG.
    \end{equation*}
We adopt the convention that~$\caH_{\leer}(\phee) = U$ and notice that in this case only the codomain~$U$ of~$\phee$ plays a role. Note also that~$\caH_{\{t\}}(\varphi)$ coincides the usual HNN-extension
\[\caH_{\{t\}}(\varphi) = \HNN(\phee, t)= U\ast_{\phee}^t\] 
associated to~$\varphi$ and with a single stable letter~$t$. Regardless of~$|X|$, we simplify speech calling any group of the form~$\caH_X(\varphi)$ an HNN-extension.

\smallskip

The next lemma collects some fundamental, well-known facts about HNN-extensions that will be used throughout. 
We remark that item~\eqref{fact:HNNtopTopology} below is the technical starting point for our construction of topological groups given by generalised presentations (see~\Cref{s:genpres}). 

\begin{lem} \label{fact:HNNtop}
Let $X$ be a finite non-empty set and~$\phee\colon \caO \into U$ a monomorphism of groups~$\caO\leq U$. Then the following claims hold.
\begin{enumerate}
\item \label{fact:HNNtopGraphofgroups} The group~$\caH_X(\varphi)$ is the fundamental group of the following graph of groups~$(\mathscr{G},\Lambda)$ in the sense of the Bass--Serre theory: $\Lambda$ is a bouquet of loops based at the vertex~$v_0$ with edge set~$\{t, \quer{t}\mid t\in X\}$; the vertex group~$\mathscr{G}_{v_0}$ equals~$U$; the edge groups~$\mathscr{G}_t=\mathscr{G}_{\overline{t}}$, $t\in X$, are all equal to~$\caO$; and for every~$t\in X$, the two monomorphisms from~$\mathscr{G}_t=\mathscr{G}_{\bar t}$ into~$\mathscr{G}_{v_0}$ are the inclusion $\mscr{G}_t = \caO\hookrightarrow U = \mscr{G}_{v_0}$ and $\varphi : \mscr{G}_{\quer{t}} = \mc{O} \into U = \mscr{G}_{v_0}$, respectively. In particular, there is a canonical group monomorphism $\iota\colon U\hookrightarrow\caH_X(\varphi)$.

\item \label{fact:HNNtopBassSerretree} 
The group~$\caH_X(\varphi)$ acts by left translation on its associated universal Bass--Serre tree~$T=T_X(\varphi)$, which is a tree whose vertex set~$V(T)$ and set of undirected edges~$E(T)$ are as follows:
\[V(T)=\caH_X(\varphi)/U, \text{ the set of all left $U$-cosets in}~\mc{H}_X(\phee),\]
and 
\[E(T)=\big\{\{gU,gtU\} \mid g\in \caH_X(\varphi), \, t\in X\big\}.\]
Moreover, for every $t\in X$ and $n\geq 1$, one has
\begin{align*}
U\cap tUt^{-1}=\varphi(\caO), & \quad U\cap t^{-1}U t=\caO, \\
U\cap t^nUt^{-n}\subseteq \varphi(\caO) , & \, \text{ and } \, U\cap t^{-n}U t^n \subseteq \caO.
\end{align*}

\item \label{fact:HNNtopUniversalproperty} The following universal property holds: For every group homomorphism $\alpha\colon U\to G$ and sequence of elements $(g_t)_{t\in X}$ of $G$ such that
\[\alpha(\varphi(\omega))=g_t\alpha(\omega)g_t^{-1} \quad \forall \, t\in X, \forall\,\omega\in \caO\] 
there exists a unique group homomorphism $\beta\colon \caH_X(\varphi) \longrightarrow G$
satisfying $\beta\circ \iota=\alpha$ and $\beta(t) = g_t$ for all $t \in X$, where $\iota\colon U\hookrightarrow \caH_X(\varphi)$ is the canonical inclusion from~\eqref{fact:HNNtopGraphofgroups}.

\item \label{fact:HNNtopPresentation} Let~$\spans{Y \mid R}$ be an abstract group presentation of~$U$, and for each~$\omega \in \mc{O}$ fix a choice of words~$\omega_Y, \quer{\omega} \in F(Y)$ representing $\omega,\phee(\omega)\in U$ in the presentation~$\spans{Y \mid R}$, respectively. Then $\mc{H}_X(\phee)$ admits a group presentation
\[\big\langle X \sqcup Y \mid R \cup \{ t \omega_Y t^{-1} = \quer{\omega} \mid t \in X, \omega \in \mc{O}\} \big\rangle.\]

\item \label{fact:HNNtopTopology} Suppose~$U$ is topological group, $\mc{O}$ is an open subgroup, and that $\phee$ is a continuous open monomorphism, i.e., $\phee : \mc{O} \into U$ is an injective group homomorphism that is continuous and open. Then there exists a unique group topology on~$\mc{H}_X(\phee)$ that turns the canonical inclusion~$\iota \colon U \into \mc{H}_X(\phee)$ from~\eqref{fact:HNNtopGraphofgroups} into a continuous open monomorphism.
\end{enumerate}
\end{lem}

\begin{proof}
Part~\eqref{fact:HNNtopGraphofgroups} follows from \cite[\S I.5.1]{ser:trees}. Part~\eqref{fact:HNNtopBassSerretree} follows from~\eqref{fact:HNNtopGraphofgroups}, \cite[{\S I.5.3}]{ser:trees} and the following observation. 
For every~$n\geq 1$, $U\cap t^{\pm n}Ut^{\mp n}$ is the pointwise stabiliser of the geodesic path in~$T$ connecting the vertices~$1U$ and~$t^{\pm n}U$. Since~$t^{\pm}U$ is a vertex of this geodesic path, we have~$U\cap t^nUt^{-n}\subseteq U\cap tUt^{-1}=\varphi(\caO)$ and~$U\cap t^{-n}Ut^n\subseteq U\cap t^{-1}Ut=\caO$.

Parts~\eqref{fact:HNNtopUniversalproperty} and~\eqref{fact:HNNtopPresentation} are well-known in the case~$|X| = 1$ (see, for instance, \cite[Chapter~IV{\S}2]{LyndonSchupp:cgt} or \cite[Chapter~8.B]{cdlh:metric} and in particular \cite[Proposition~8.B.10]{cdlh:metric}), 
but both can be extended in the obvious way to an arbitrary finite non-empty~$X$ by iterating the HNN construction: namely,~$\mc{H}_{\{t_1,\ldots,t_{n+1}\}}(\phee) \cong \HNN(\phee \circ \iota_n, t_{n+1})$ where $\iota_n: U \into \mc{H}_{\{t_1,\ldots,t_{n}\}}(\phee)$ is the canonical embedding from~\eqref{fact:HNNtopGraphofgroups} (see e.g.~\cite[Lemma~11.77]{RotmanGroups}).

Part~\eqref{fact:HNNtopTopology} follows from \cite[Proposition~8.B.10]{cdlh:metric}, again iterating the HNN construction as in the previous paragraph.
\end{proof}

Thanks to~\Cref{fact:HNNtop}\eqref{fact:HNNtopTopology} we can adopt some convenient terminology.

\begin{conv} \label{conv:topHNNext}
Given a finite set~$X$ and  a continuous open monomorphism~$\phee \colon \mc{O} \into U$ of (possibly non-Hausdorff) topological groups,
the \emph{topological iterated HNN-extension}~$\mc{H}_X(\phee)$ is the abstract iterated HNN-extension over~$\phee$ with stable letters in~$X$ endowed with the unique group topology from~\eqref{fact:HNNtopTopology}. Due to~\Cref{fact:HNNtop}, we may (and always do) identify~$U$, $\mc{O}$ and~$X$ with their canonical one-to-one images in~$\mc{H}_X(\phee)$. 
The codomain~$U$ of~$\phee$ is called the \emph{base group} of~$\mc{H}_X(\phee)$ and it is identified with an open subgroup of~$\mc{H}_X(\phee)$. 
\end{conv}

\begin{ex}[{Free groups as topological iterated HNN-extensions}]\label{ex:easyHNN}
  Let $\varphi=\iid_{\{1\}}$ be the identity map of the trivial group~$\{1\}$ endowed with the discrete topology. Then~$\caH_X(\iid_{\{1\}})$ is canonically isomorphic to the free group~$F(X)$ with the discrete topology. Indeed, the discrete topology is the unique group topology that makes the  embedding~$\{1\}=U\hookrightarrow \caH_X(\varphi)$ open and continuous. 
\end{ex}
We now record some basic examples of continuous open monomorphisms~$\caO\hookrightarrow U$ between topological subgroups~$\caO\leq_o U$. For further examples when~$\caO=U$ see~\Cref{sus:exMonosO=U}. 
\begin{ex}\label{ex:monoOU}
    \begin{enumerate}
        \item \label{ex:monoOU1} Obviously, a topological automorphism is an open continuous surjective monomorphism.
        \item \label{ex:monoOU2} Let~$U$ be a topological group and~$\caO\leq_o U$ be an open subgroup. The inclusion map~$\caO\hookrightarrow U$ is a continuous open monomorphism.
        \item \label{ex:monoOU3} 
         Let $G$ be an LC~group and let~$\alpha\colon G\to G$ be a topological automorphism which is \emph{contractive}, i.e., 
   $\lim_{n\to +\infty}\alpha^n(g)=1$ for every~$g\in G$. By \cite[Proposition~2.1]{wang84}, for all compact open subgroups~$\caO_1,\caO_2\leq G$ there is a~$k\in \Z$ such that $\alpha^k(\caO_1)\subseteq \caO_2$. Hence $\alpha^k:\caO_1\hookrightarrow\caO_2$ is an open continuos monomorphism.
   For instance, one can consider the contractive automorphism $\alpha\colon x\in \Q_p\mapsto px\in \Q_p$ of the additive group~$\Q_p$ of $p$-adic integers, and build open continuous monomorphisms between its compact open subgroups, which are of the form $p^n\Z_p$ for some $n\in\Z$.
    \end{enumerate}
\end{ex}

In the following result we collect some topological properties of iterated HNN-extensions. These are, as expected, very much dependent on the properties of the base group $U.$

\begin{lem} \label{lem:topologyofHNNextensions}
Let~$\phee:\caO\hookrightarrow U$ be a continuous open monomorphisms between (possibly non-Hausdorff) topological groups. Let $\mc{H}_X(\phee)$ be a topological HNN-extension over $\phee$. Then the following hold.
\begin{enumerate}
\item \label{rem:HNNdisc} $\mc{H}_X(\phee)$ is discrete if and only if $U$ is discrete.
\item \label{lem:topologyofHNNextensions1} $\mc{H}_X(\phee)$ is Hausdorff if and only if $U$ is Hausdorff.
\item \label{lem:topologyofHNNextensions2} $\mc{H}_X(\phee)$ is locally compact if and only if $U$ is locally compact.
\item \label{lem:topologyofHNNextensions3} $\caH_X(\varphi)_0=U_0$ (see~\Cref{sec:basicTopGroups}). In particular, $\mc{H}_X(\phee)$ is totally disconnected if and only if $U$ is totally disconnected. Moreover, $\mc{H}_X(\phee)$ is connected if and only if $U$ is connected and $X=\leer$.
\item \label{lem:topologyofHNNextensions4} If $\mc{H}_X(\phee)$ is Hausdorff, then it is compact if and only if $U$ is compact and $X = \leer$.
\item \label{lem:topologyofHNNextensions5} $\mc{H}_X(\phee)$ is compactly generated whenever $U$ is.
\item \label{lem:topologyofHNNextensions6} If $U$ is compactly presented, then $\mc{H}_X(\phee)$ is compactly presented if and only if $\mc{O}$ is compactly generated.
\end{enumerate}
\end{lem}

\begin{proof}
We will use the fact that~$U$ embeds as an open (hence clopen) subgroup of~$\mc{H}_X(\phee)$, see~\Cref{fact:HNNtop}\eqref{fact:HNNtopTopology}, repeatedly and will not remark on it further.

\eqref{rem:HNNdisc} If~$\mc{H}_X(\phee)$ is discrete, so is its subspace~$U$. For the converse, we refer the reader to~\cite[\S I.3.1]{bou:top}. 

\eqref{lem:topologyofHNNextensions1} If~$U$ is Hausdorff then the singleton~$\{1\}$ is closed in~$U$ and therefore in~$\mc{H}_X(\phee)$. Thus~$\mc{H}_X(\phee)$ is Hausdorff~\cite[Ch.~III, I.2, Corollary]{bou:top}. The converse is obvious.

\eqref{lem:topologyofHNNextensions2} Closed subspaces of locally compact spaces are themselves locally compact. Conversely, $U$ being locally compact implies that $1 \in U$ has some open neighbourhood with compact closure in~$U$ and hence in~$\mc{H}_X(\phee)$. It follows that the whole of~$\mc{H}_X(\phee)$ is locally compact~\cite[Corollary~3.12]{strop}.

\eqref{lem:topologyofHNNextensions3} 
The equality that $\mc{H}_{X}(\phee)_0 = U_0$ follows from \Cref{fact:ConCompTG} and the fact that~$U\leq_o\mc{H}_{X}(\phee)$. In particular, $\caH_X(\varphi)$ is totally disconnected if and only if~$U$ is. Moreover, since $U\leq_o\caH_X(\varphi)$,~$\caH_X(\varphi)$ is connected if and only if~$\caH_X(\varphi)=\caH_X(\varphi)_0=U_0$ or, equivalently, $\caH_X(\varphi)=U=U_0$. In turn, the latter is equivalent to asking~$X=\emptyset$ and~$U$ connected.

\eqref{lem:topologyofHNNextensions4} If~$X=\leer$, then~$\mc{H}_X(\phee) = U$ by design. Conversely, assume~$\mc{H}_X(\phee)$ is compact. In this case both~$U$ and~$\mc{O}$ are also compact, being clopen subgroups of the Hausdorff group~$\mc{H}_X(\phee)$. 
Moreover,~$\caH_X(\phee)$ acts on its universal Bass--Serre tree~$T$ with vertex stabilisers conjugate to~$U$ (see~\Cref{fact:HNNtop}\eqref{fact:HNNtopBassSerretree}). By the orbit-stabiliser theorem all~$\caH_X(\phee)$-orbits on~$V(T)$ are finite. By~\cite[Corollary~II.2.8(1)]{BridsonHaefliger},  the compact group~$\caH_X(\phee)$ fixes a vertex and therefore~$\caH_X(\varphi)=U$. 

\eqref{lem:topologyofHNNextensions5} This is straightforward as $\mc{H}_X(\phee)$ is generated by~$U$ and~$X.$ Recall~$|X|$ is finite and $U$ is generated by a union of a compact and a finite subset.

\eqref{lem:topologyofHNNextensions6} 
Recall that $\mc{H}_{\{t_1,\ldots,t_{n+1}\}}(\phee) \cong \HNN(\phee \circ \iota_n, t_{n+1})$, where $\iota_n: U \into \mc{H}_{\{t_1,\ldots,t_{n}\}}(\phee)$ is the canonical embedding from~\Cref{fact:HNNtop}\eqref{fact:HNNtopGraphofgroups}. Now apply~\cite[Proposition~8.B.10(4)]{cdlh:metric} iteratively. 
\end{proof}

\begin{rem} 
Observe that even when~$\mc{O}$ is compactly generated, the compact presentability of~$U$ in~\Cref{lem:topologyofHNNextensions}\eqref{lem:topologyofHNNextensions6} is not always necessary to obtain a compactly presented HNN-extension $\mc{H}_X(\phee)$. See, for example, the discrete examples occurring as  by-products of  proofs of Higman's Embedding Theorem; see~\cite[pp.~450--461]{RotmanGroups}. See also~\cite[p.~198, Exercise~2]{brown:coho} for further examples. 
\end{rem}

An extension of~\Cref{lem:topologyofHNNextensions}\eqref{lem:topologyofHNNextensions3} is the following.
\begin{prop}\label{prop:connComp}
Let~$U$ be a topological group and~$\caO\leq_o U$.
Then, every continuous open monomorphism~$\varphi\colon \caO\hookrightarrow U$ induces, by passing to the quotient, a continuous open monomorphism
\[\begin{tikzcd}        
    \widetilde{\varphi}\colon \widetilde{\caO}:=\caO/\caO_0 \arrow[hook]{r} & \wtil{U}:=U/U_0, \qquad \wtil{\varphi}(\omega \caO_0)=\varphi(\omega)U_0 \quad\forall\,\omega\in \caO. 
\end{tikzcd}\]
    Moreover, for every finite set~$X$ one has a short exact sequence of topological groups
     \begin{equation}\label{eq:sesconn}  
  \xymatrix{1\ar[r] & U_0 \ar@{->}[r]^-{i} & \caH_X(\varphi) \ar@{->}[r]^-{\eta} & \caH_X({\wtil{\varphi}})\ar[r] & 1}
     \end{equation}
  where~$i$ is the inclusion and~$\eta$ is the continuous open epimorphism induced by the canonical projection~$\pi_U\colon U\twoheadrightarrow \widetilde{U}$ and the identity~$\iid_X\colon X\to X$. 

   In particular,~$\caH_X(\wtil{\varphi})$ is totally disconnected.
\end{prop}
\begin{proof}
  By~\Cref{fact:ConCompTG} we have~$U_0=\caO_0=\varphi(\caO)_0$.
    Hence~$\wtil{\varphi}$ is a well-defined group homomorphism. It is continuous and open because~$\varphi$ is continuous and open and both~$\widetilde{\caO}$ and~$\widetilde{U}$ carry the quotient topology. 
    Moreover, since~$\varphi$ is an isomorphism of topological groups onto~$\varphi(\caO)$ and by~\Cref{fact:ConCompTG}, we have
    \begin{equation}\label{eq:conncompphi}
       \varphi(\caO_0)=\varphi(\caO)_0=U_0.
    \end{equation}
    By Equation~\eqref{eq:conncompphi} and injectivity of~$\varphi$, we deduce that~$\wtil{\varphi}$ is injective and therefore a continuous open monomorphism. 

   Let~$X$ be a finite set. By~\Cref{fact:HNNtop}\eqref{fact:HNNtopUniversalproperty}, $\pi_U$ and~$\iid_X$ induce a unique continuous group epimorphism
   \begin{equation*}
       \eta\colon \caH_X(\varphi)\longrightarrow \caH_X(\wtil{\varphi}).
   \end{equation*}
   
   We first prove that~$\eta$ is open. 
   Recall that $U$ is an open subgroup of~$\caH_X(\varphi)$ (see~\Cref{fact:HNNtop}\eqref{fact:HNNtopTopology}). Since~$\eta\vert_U=\pi_U$ is an open homomorphism, so is~$\eta$.

   We now show that
   \begin{equation}\label{eq:conn1}
       \ker(\eta)=U_0.
   \end{equation}
   The crucial point is to prove that~$\ker(\eta)\subseteq U$. Indeed, from the latter inclusion one concludes that
     $$\ker(\eta)=\ker(\eta\vert_U)=\ker(\pi_U)=U_0.$$
   We prove that~$\ker(\eta)\subseteq U$ via a geometric argument. Let~$T$ and~$T'$ be the universal Bass--Serre trees of~$\caH_X(\varphi)$ and~$\caH_X(\wtil{\varphi})$, respectively (see~\Cref{fact:HNNtop}\eqref{fact:HNNtopBassSerretree}).
  Since~$\eta(U)=\wtil{U}$ and~$\eta\vert_X=\iid_X$, the map~$\eta$ induces a surjective $\eta$-equivariant simplicial map 
  \begin{equation*} 
      \wtil{\eta}\colon T\to T', \qquad \wtil{\eta}(gU):=\eta(g)\wtil{U},\,\forall gU\in V(T).
  \end{equation*}
  
   We claim that~$\wtil{\eta}$ is an isomorphism of trees. 
   By~\cite[\S 1.3]{bass}, it suffices to prove that~$\wtil{\eta}$ is locally bijective. That is, since~$\wtil{\eta}$ is $\eta$-equivariant, it suffices to show that~$\wtil{\eta}$ restricts to a bijective map 
   $$\widetilde{\eta}_{St}\colon\mathrm{St}_T(1U)\to \mathrm{St}_{T'}(1\wtil{U}),$$
   where~$\mathrm{St}_T(1U)$ and~$\mathrm{St}_{T'}(1\wtil{U})$ are the sets of neighbour vertices of~$1U$ in~$T$ and of~$1\wtil{U}$ in~$T'$, respectively. Notice that
   \begin{equation*}
       \mathrm{St}_T(1U)=\bigsqcup_{x\in X\sqcup X^{-1}}UxU/U\quad\text{and}\quad \mathrm{St}_{T'}(1\wtil{U})=\bigsqcup_{x\in X\sqcup X^{-1}}\wtil{U}x\wtil{U}/\wtil{U}.
   \end{equation*}
   For all~$u\in U$ and~$x\in X\sqcup X^{-1}$, notice that
   \begin{equation*}
       \wtil{\eta}(uxU)=\eta(ux)\wtil{U}=\pi_U(u)x\wtil{U}.
   \end{equation*}
   Hence, since~$\pi_U$ is surjective, we have~$\wtil{\eta}_{St}(\mathrm{St}_T(1U))=\mathrm{St}_{T'}(1\wtil{U})$.
   Moreover, for all~$u_1,u_2\in U$ and~$x_1,x_2\in X\sqcup X^{-1}$ satisfying~$\pi_U(u_1)x_1\wtil{U}=\pi_U(u_2)x_2\wtil{U}$, we have
   \begin{equation}\label{eq:conn2}
       \pi_U(u_2^{-1}u_1)\in \wtil{U}\cap x_2\wtil{U}x_1^{-1}. 
   \end{equation}
   The normal form of iterated HNN-extensions (see~\cite[\S I.5.2, Exercise]{ser:trees} or apply iteratively~\cite[Theorem~14.3]{bogo}) yields~$x_1=x_2$. Since~$U_0=(\caH_X(\varphi))_0=x_1U_0x_1^{-1}$ is contained in both~$U$ and~$x_1Ux_1^{-1}$, from Equation~\eqref{eq:conn2} we deduce that
   \begin{equation*}
       \begin{split}
           u_2^{-1}u_1U_0& =\pi_U(u_2^{-1}u_1)\in\wtil{U}\cap x_1\wtil{U}x_1^{-1} =U/U_0\cap x_1(U/U_0)x_1^{-1}\\
           & =U/U_0\cap (x_1Ux_1^{-1})/U_0=(U\cap x_1Ux_1^{-1})/U_0
       \end{split}
   \end{equation*}
 and therefore~$u_1x_1U_0=u_2x_1U=u_2x_2U_0$.
   We conclude that~$\wtil{\eta}\colon T\to T'$ is an isomorphism of trees, and hence maps geodesic paths of~$T$ to geodesic paths of~$T'$. 
   
   Let~$g\in \ker(\eta)$ and denote by~$[1U,gU]$ the geodesic path in~$T$ from~$1U$ to~$gU$. Since~$\wtil{\eta}([1U,gU])$ is the geodesic in~$T'$ connecting~$1\wtil{U}$ to~$\eta(g)\wtil{U}=1\wtil{U}$, it has no edges. This implies that also~$[1U,gU]$ has no edges, i.e.,~$gU=U$. Therefore, $\ker(\eta)\subseteq U$ and the proof of Equation~\eqref{eq:conn1} is concluded.
   
   Finally, from the exactness of~\eqref{eq:sesconn}, $\caH_X(\wtil{\varphi})$ is topologically isomorphic to~$\caH_X(\varphi)/U_0$. By~\Cref{lem:topologyofHNNextensions}\eqref{lem:topologyofHNNextensions3} and~\Cref{fact:ConCompTG}, we conclude that~$\caH_X(\wtil{\varphi})$.
\end{proof}

\subsection{(Co)homological background  for TDLC~groups}\label{sus:backCoho}
Here we collect all necessary background on rational discrete cohomology for TDLC groups needed later and which is well know due to mainly Castellano--Weigel~\cite{CastellanoWeigel0} and Castellano--Corob Cook~\cite{ccc}.
Let~$G$ be a TDLC group and~$R$ be any unital commutative ring. 
Denote by~$\RGdis$ the category of \emph{discrete $R[G]$-modules},  which consists of all left $R[G]$-modules for which the left $G$-action is continuous when they carry the discrete topology. The category~$\RGdis$ is abelian with enough injectives. 
In case~$R$ is a field of characteristic zero, $\RGdis$ has also enough projectives~\cite[Proposition~3.2]{CastellanoWeigel0}. Of particular interest are the \emph{proper discrete permutation modules} $Q \in \ob(\RGdis),$ i.e.,
$$Q \cong \bigoplus_{i\in I} R[G/U_i],$$
where~$I$ is some non-empty indexing set and the~$U_i$'s are compact open subgroups of~$G.$
If~$R$ is a field of characteristic zero, these are the building blocks for the projectives. Namely, a $\RGdis$-module $P$ is projective if and only if it is a direct summand of a proper discrete permutation $R[G]$-module $Q$~\cite[Corollary 3.3]{CastellanoWeigel0}.

Following~\cite{CastellanoWeigel0}, for every unital commutative ring~$R$, every $M\in ob(\RGdis)$ and~$n\geq 0$, let 
$$\dExt^n_G(M,-): \RGdis \to {}_R\mathbf{mod}$$
be the right derived functor of~$\Hom_{\RGdis}(M,-)$. The \emph{discrete cohomology} of $G$ over~$\RGdis$ is defined as
$$\dH^n(G,-)=\dExt^n_G(R,-) \quad \forall\,n\geq 0.$$
In particular, the \emph{discrete cohomological dimension} of~$G$ over~$R$ is
\begin{equation*}
    \ccd_R(G):=\sup\{n\in \Z_{\geq 0}\mid \dH^n(G,-)\neq 0\}.
\end{equation*}
For~$R=\Q$ (or more generally a field of characteristic zero) the reader may refer to~\cite[Lemma~3.6]{CastellanoWeigel0} for characterisations of~$\ccd_{\Q}(G)$ in terms of projectives and  to~\cite[Equation (4.6)]{CastellanoWeigel0} for the definition of the rational discrete homology $\dH_n(G,-)$ of $G.$

\subsubsection{Compactness and finiteness conditions} We collect the background on various, equivalent, compactness conditions for locally compact and some basic facts needed later. If not stated differently,~$R$ denotes a unital commutative ring. 

The first attempt to generalise the well-known notions of~$\FP_n(R)$ and~$\tF_n$ for discrete groups to locally compact groups is due to Abels and Tiemeyer~\cite{abels-tiemeyer}. They give compactness properties~$\CP_n(R)$ and~$\tC_n$ via simplicial sets and relying on Brown's finiteness criterion as in~\cite{BrownFP}. We recall an equivalent formulation in the next~\Cref{thm:at}.

Recall that a direct system of groups $(\mathrm H_i)_{i\in I}$ is \emph{essentially trivial} if for every $i \in I$ there is a $k \in I, k>i$ such that the map $\mathrm  H_i \to \mathrm H_k$ is trivial. 
\begin{thm}[\protect{\cite[Theorem 3.2.2]{abels-tiemeyer}}]\label{thm:at} 
Let $G$ be an LC~group and $T$ be an LC~space on which $G$~acts properly. Let $\operatorname{E}T$ be the free simplicial set on $T$ and $\underline A  = \{ G \cdot {\spE}K \,|\, K \subseteq T \mbox{ compact} \}$ a filtration of $|\operatorname{E}T|$. Then 
\begin{enumerate} 
\item $G$ is type $\CP_n$ if and only if the direct system $\mathrm H_i(\underline A)$ is essentially trivial for all $i<n.$
\item $G$ is type $\tC_n$ if and only if the direct system $\pi_i(\underline A)$ is essentially trivial for all $i<n.$
\end{enumerate}
\end{thm} 

\begin{fact}[\protect{\cite[Theorems 2.2.2 and~2.3.3]{abels-tiemeyer}}]
An LC group is compactly generated (resp.~compactly presented) if and only if it is of type $\tC_1$ (resp.~$\tC_2$).
\end{fact}

In~\cite{ccc} the authors give a intuitive definition for the conditions~$\FP_n(R)$ and~$\tF_n$ for~TDLC~groups more closely aligned to the definition for discrete groups and prove a TDLC~analogue to Brown's criterion for finiteness conditions.

\begin{defn}[\protect{\cite[\S3.3]{ccc}}]\label{defn:FPn}
Let $G$ be a TDLC~group and $M \in \ob(\RGdis)$.

\begin{enumerate}
  \item For~$n\in \Z_{\geq 0}$, $M$ is \emph{of type~$\FP_n(R)$} if there is a resolution by proper discrete permutation modules $P_\ast \twoheadrightarrow M$ 
  such that $P_k$ is finitely generated for all $k \leq n$. Moreover, $M$ is \emph{of type~$\FP_\infty(R)$} if it is of type~$\FP_n(R)$ for every~$n\geq 0$.  
  \item For~$n\in \Z_{\geq 0}\cup\{\infty\}$, $G$ is \emph{of type~$\FP_n(R)$} if the trivial $\RGdis$-module $R$ is of type $\FP_n(R)$. Note that every~$G$ is of type~$\FP_0(R)$.
\end{enumerate}
\end{defn}

There is also an intuitive definition of $\tF_n$ for TDLC groups:
\begin{defn}[\protect{\cite[\S3.1]{ccc}}]\label{defn:F_n}
Let $G$ be a TDLC~group and~$n \in \Z_{\geq 0} \cup \{\infty\}.$ We say $G$ is \emph{of type~$\tF_n$} if there exists a contractible proper discrete $G$-CW-complex whose $n$-skeleton has finitely many $G$-orbits. Moreover, $G$ is \emph{of type~$\tF_\infty$} if it is of type~$\tF_n$ for every~$n\geq 0$.
\end{defn}

\begin{fact}\label{fact:FnimplFPn}
  Let~$n\in \Z_{\geq 0}$. Every locally compact group of type~$\tC_n$ is also of type~$\CP_n(R)$. Moreover, every TDLC~group of type~$\tF_n$ is also of type~$\FP_n(R)$.
\end{fact}

Theorem~\ref{thm:at} together with the TDLC~version of Brown's criterion~\cite[Theorem 4.7]{ccc} implies the well-known fact that $\tF_n$ and $\tC_n$, as well as $\FP_n(R)$ and $\CP_n(R)$, are equivalent conditions for TDLC groups.

\begin{fact}[\protect{\cite[Theorem A and Theorem~6.3]{chanfi-witzel}}]\label{fact:LCvsTDLC}
Let $G$ be an LC group and let $G_0$ be the connected component of the identity. For every~$n\in \Z_{\geq 0}$, the following hold.
\begin{enumerate}
    \item\label{fact:CnFn} $G$ is of type $\tC_n$ if and only if $G/G_0$ is of type $\tF_n.$
    \item\label{fact:CPnFPn} $G$ is of type $\CP_n(R)$ if and only if $G/G_0$ is of type $\FP_n(R)$
\end{enumerate}
    \end{fact}

From now on we shall use the more well-known notation of $\FP_n(R)$ and $\tF_n$ when talking about TDLC~groups. 
Standard results about abstract groups of type $\FP_n(R)$ and $\tF_n$ can be carried over to the TDLC~setting, see~\cite[\S3]{ccc}. For $R=\Z$ this poses some difficulties as the category ${}_{\Z[G]}\dis$ does not have enough projectives. Hence the authors of~\cite{ccc} take a detour via algebraic objects in the category of $k$-spaces. These are compactly generated weakly Hausdorff spaces where the corresponding module category has enough projectives~\cite[\S1.7]{ccc}. One can thus consider analogues to $\FP_n(R)$ and show that these are equivalent to the originally defined conditions $\FP_n(R)$, see~\cite[Theorem 3.10]{ccc}. This gives us the following analogue of a standard fact for abstract groups:

\begin{fact}[\protect{\cite[Corollary~3.18]{ccc}}]\label{fact:FPnInduction} Let $H$ be an open subgroup of the TDLC~group $G$ and let $A \in \ob({{}_{R[H]}\dis}).$ Then $A$ is of type $\FP_n(R)$ if and only if $R[G]\otimes_H A \in \ob(\RGdis)$ is of type $\FP_n(R).$
    \end{fact}

\begin{prop}[\protect{\cite[Theorem 1.3]{BonnGiersbach}}]\label{prop:graphofgrpsFP}
    Suppose~$G$ is a TDLC group which splits topologically\footnote{That is, $G$ is topologically isomorphic to the fundamental group $\pi_1(\caG,\Gamma)$ of a finite graph of TDLC groups endowed with the topology described, for instance, in~\cite[\S 5.1]{cmw:unimod}.} over a finite graph of TDLC~groups~$(\caG,\Gamma)$ over the graph~$\Gamma$.  Suppose $\mathcal{G}_e$ is of type~$\FP_n(R)$ for all~$e\in E(\Gamma)$. Then~$G$ is of type~$\FP_n(R)$ if and only if, for all $v\in V(\Gamma)$, $\mathcal{G}_v$ is of type~$\FP_{n}(R)$. The same statement holds by replacing~$\FP_n(R)$ with~$\tF_n$.
\end{prop}

Note that the result in~\cite{BonnGiersbach} is stated for~$R =\Z$ only. However, their proof can be carried over verbatim for any unital commutative ring~$R$.

Here we record a slightly stronger version of one direction applied to $\caH_X(\varphi)$ with LC base group~$U$.

\begin{cor}\label{cor:HpheeisFPn}
    Let $\caH_X(\varphi)$ be a topological HNN-extension over a continuous open monomorphism $\varphi: \caO \into U$ of LC groups~$\caO\leq_o U$ and let~$n \geq 1$. Assume $\caO$ is of type $\CP_{n-1}(R)$
    and $U$ is of type $\CP_n(R).$ Then $\caH_X(\varphi)$ is of type $\CP_n(R).$

The same statement holds by replacing~$\CP_n(R)$ with~$\tC_n$.
\end{cor}

\begin{proof}
 By~\Cref{prop:connComp} and~\Cref{fact:LCvsTDLC} it suffices to show that $\caH_X(\widetilde\varphi)$ is of type $\FP_n(R)$ (resp.~$\tF_n$). 
 We first prove the statement for~$\FP_n(R)$.
 The Bass--Serre tree~$T$ for $\caH_X(\widetilde\varphi)$ (see~\Cref{fact:HNNtop}\eqref{fact:HNNtopBassSerretree}) is $n$-good for $G$ over~$R$ in the sense of Brown~\cite{BrownFP} and has finitely many $\caH_X(\varphi)$-orbits. In particular it is $R$-acyclic and the pointwise stabilisers of $p$-cells are of type $\FP_{n-p}(R)$, for every~$p$. Hence the TDLC version of Brown's criterion~\cite[Proposition 4.5]{ccc} applies, implying $\caH_X(\widetilde\varphi)$ is of type $\FP_n(R)$.

 The statement for~$\tF_n$ is now immediate: the case~$n=1$ is clear, the case~$n=2$ is due to~\cite[Theorem~4.9]{ccc} and the fact that~$T$ has contractible geometric realisation, and the case~$n\geq 3$ follows from \Cref{lem:topologyofHNNextensions}\eqref{lem:topologyofHNNextensions6}, \cite[Proposition~3.13]{ccc} and the statement for~$\FP_n(\Z)$.
\end{proof}

\begin{rem}\label{rem:finprop}
Below we recall some well-known examples of non-discrete TDLC groups of type $\FP_n(\Z)$ but not $\FP_{n+1}(\Z)$, which are extensions of constructions for abstract groups.
\begin{enumerate}
    \item[(a)] Tiemeyer~\cite{tiemeyer} presents a local-to-global principle to extend the examples of Abels \cite{Abels} of abstract soluble $S$-arithmetic groups. Combining \cite{tiemeyer} and \cite{abels-brown}, one obtains for every prime $p$ and every $n \geq 2$ a TDLC soluble matrix group $\mathbf{A}_n(\Q_p)$ of type~$\tF_{n-2}$ but not~$\tF_{n-1}$.
    \item[(b)] Bonn and Sauer~\cite{BonnSauer} present a method to construct TDLC~groups separated by their finiteness conditions via Schlichting completions~$G$ of abstract groups~$\Gamma$ with respect to a commensurated subgroup~$\Lambda.$ These examples depend on the finiteness conditions for both~$\Gamma$ and~$\Lambda$. They recover Abels's matrix groups over~$\Q_p$ as a particular case~\cite[Example 3.7]{BonnSauer}.
    \item[(c)] Bonn and Giersbach~\cite[Theorem 1.1]{BonnGiersbach} use Smith's box product \cite{SimonDuke} to construct, for every positive integer~$n$, simple non-discrete TDLC groups of type~$\tF_{n-1}$ but not~$\tF_n$. These can be regarded as non-discrete counterparts of the Thompson-like examples presented in~\cite{skipper-witzel-zaremsky}, although the two constructions are fairly different.
\end{enumerate}
\end{rem}
\subsubsection{A Mayer--Vietoris sequence and K\"unneth Formula in rational discrete (co)homology} 
The arguments in the following proofs are classical but we include them here for reader's convenience. Note that $\Q$ can be replaced by any field of characteristic zero, rational discrete cohomology is most suitable for our set--up.
\begin{thm}[Mayer--Vietoris Sequence]\label{thm:mv} Let $G$ be a TDLC~group and {let} $\mathcal T$ be a tree acted on by $G$ with open stabilisers. Then, for every $M\in \ob(\QGdis)$, there are long exact sequences 
\begin{equation}\label{eq:mv coho}
    \cdots \to \prod_{v\in \caV} 
\dH^n(G_v,M) \to \prod_{e\in \caE} \dH^n(G_e,M) \to \dH^{n+1}(G,M) \to \cdots
\end{equation}

\begin{equation}\label{eq:mv ho}
    \cdots \to \bigoplus_{e\in \caE}
\dH_n(G_e,M) \to  \bigoplus_{v\in \caV}\dH_n(G_v,M) \to \dH_{n-1}(G,M) \to \cdots
\end{equation}
where~$\caV$ and~$\caE$ are arbitrary sets of $G$-orbits on
vertices and edges of~$\caT$, respectively, and~$G_v$ and~$G_e$ denotes the stabilisers of~$v$ and~$e$ in~$G$, respectively. 
\end{thm}
\begin{proof} 
    Let $\mathcal T=(V,E)$ and consider the associated augmented chain complex
    \begin{equation*}
        0\to\Q[E]\to\Q[V]\to\Q\to 0,
    \end{equation*}
    which is a short exact sequence in~$\QGdis$. 
    Note that the above sequence rewrites as
    \begin{equation}\label{eq:ses tree2} 0\to\bigoplus_{e\in \caE}\Q[G/G_e]\to\bigoplus_{v\in \caV}\Q[G/G_v]\to\Q\to 0.
    \end{equation}

The long exact sequence for $\textup{dExt}$ that arise from the short exact sequence~\eqref{eq:ses tree2} is 
\begin{equation*}
    \cdots \to \prod_{v\in \caV} 
\textup{dExt}^n(\Q[G/G_v],M) \to \prod_{e\in \caE} \textup{dExt}^n(\Q[G/G_e],M) \to \textup{dExt}^{n+1}(\Q,M) \to \cdots
\end{equation*}
and the Eckmann--Shapiro type lemma~\cite[\S 2.9]{CastellanoWeigel0} yields the Mayer--Vietoris sequence~\eqref{eq:mv coho}.
    
  On the other hand, since~\eqref{eq:ses tree2} is a sequence of free $\Q$-modules, after tensoring it with~$M$ over $\Q$ we obtain a short exact sequence in~$\QGdis$
  \begin{equation*}
  0\to\bigoplus_{e\in \caE}\Q[G]\otimes_{G_e}M\to\bigoplus_{v\in \caV}\Q[G]\otimes_{G_v}M\to M\to 0,
    \end{equation*}
  by diagonal $G$-action. 
The long exact sequence in homology then yields 
    \begin{equation*}
    \cdots \to 
\bigoplus_v\dH_n(G, \Q[G]\otimes_{G_v}M) \to \dH_{n}(G,M) \to \bigoplus_e \dH_{n-1}(G,\Q[G]\otimes_{G_e}M)\to  \cdots
\end{equation*}
By a standard argument (see for example~\cite[II.6.2]{brown:coho}),  for every $k\geq 0$ and $\mathcal O\leq_o G$, one obtains $\dH_k(G,\QG\otimes_{\mathcal O} M)\cong\dH_k(\mathcal O,M)$. One thus deduces the sequence~\eqref{eq:mv ho}.
\end{proof}

\begin{rem}
    The Mayer--Vietoris sequence in cohomology can be obtained also for an arbitrary unital commutative ring $R$.
\end{rem}

We remind the reader that, in this paper, the product of two topological groups is always endowed with the product topology.

\begin{thm}[K\"unneth Formula]\label{thm:kunneth}  
Let $G$ and $H$ be TDLC~groups, let $\Q$ be endowed with trivial $G$- and $H$-actions,
and let $n \in\Z_{\geq 0}$. Then
\begin{align*}
    \dH_n(G\times H,\Q)\cong_{\Q} &\bigoplus_{p=0}^n \dH_p(G,\Q)\otimes \dH_{n-p}(H,\Q)\oplus\\
   & \oplus\bigoplus_{p=0}^{n-1}
\mathrm{Tor}^\Q_1\big(\dH_p(G,\Q), \dH_{n-1-p}(H, \Q)\big)\nonumber
\end{align*}
\end{thm}
\begin{proof}
    The proof of~\cite[Corollary 3.2.23]{loh} can be transferred verbatim. The reader should keep in mind that in the TDLC~context the discrete projective resolutions $C_\ast(G)$ and $ C_\ast(H)$ consist of proper discrete permutation modules
over~$\QG$ and~$\Q [H]$, respectively (see~\cite[\S 1.1]{ccc}).
\end{proof}
A TDLC~group~$G$ is said to be \emph{$\Q$-acyclic} if $\dH_n(G,\Q)=\{0\}$ for all $n\geq 1$.
\begin{cor}\label{cor:Qacyc}
    If $G$ and $H$ are $\Q$-acyclic TDLC groups, then so is $G\times H$.
\end{cor}
\begin{proof}
   Recall that $\dH_0(G,\Q)=\dH_0(H,\Q)=\Q$. To prove the claim it suffices to observe what follows:
   \begin{equation}\label{eq:Qacyc}
   \begin{split}
       & V\otimes \{0\}\simeq\{0\}\otimes V=\{0\}\text{ for every }\Q\text{-vector space }V;\\
       & \Tor_1^\Q(\{0\},\Q)=\Tor_1^{\Q}(\Q,\{0\})=\Tor_1^\Q(\Q,\Q)=\{0\}. \qedhere
   \end{split}
   \end{equation}
\end{proof}
For further reference we also include the following classical fact, which follows directly from~\cite[Fact~5.1]{ccc} and~\cite[Proposition~3.7(a)]{CastellanoWeigel0}.
\begin{fact}\label{fact:ProfQacyc}
    Profinite groups are $\Q$-acyclic.
\end{fact}

Another observation to be used later is the following analogue of a known isomorphism in group (co)homology. 

\begin{lem} \label{lem:cohomologyiso}
   Let $G$ be a TDLC group and~$A$ a trivial discrete left $\QG$-module~$A$. Then, for every $n\geq 0$ there is a $\Q$-linear isomorphism
    $$\dH^n(G,A)\cong\Hom_\Q\big(\dH_{n}(G,\Q),A\big).$$
    In particular, for all $n\geq 0$, $\dH^n(G,A)=0$ if and only if $\dH_n(G,A)=0$.
\end{lem}

\begin{proof}
    Let~$(\mathcal{P}_\bullet,\partial_\bullet)$ be a projective resolution of the trivial 
    discrete $\QG$-bimodule~$\Q$ and denote by~$(\mathcal{K}_\bullet,d_\bullet)$ the complex of right $\Q$-vector spaces that is obtained by tensoring $(\mathcal{P}_\bullet,\partial_\bullet)$ with~$\Q$ over~$\QG$, i.e., $\mathcal{K}_\bullet =\mathcal{P}_\bullet\otimes_G \Q$. Here the right $\Q$-action on $\mathcal{K}_\bullet$ is provided by the right trivial action on~$\Q$. By the Universal Coefficient Theorem in Cohomology~\cite[Theorem~7.59(ii)]{rot}, 
    \begin{equation*}
\mathrm{H}^n\big(\Hom_\Q\big((\mathcal{K}_\bullet,d_\bullet),A\big)\big)\cong_\Q\Hom_\Q\big(\mathrm{H}_{n}(\mathcal{K}_\bullet,d_\bullet),A\big)\oplus \mathrm{Ext}^1_\Q\big(\mathrm{H}_{n-1}(\mathcal{K}_\bullet,d_\bullet),A\big).
    \end{equation*}

For every~$n$, notice that~$\mathrm{H}_{n}(\mathcal{K}_\bullet,d_\bullet)=\dH_n(G,\Q)$ by construction. Moreover, $\mathrm{H}^n\big(\Hom_{\Q}((\mathcal{K}_\bullet,d_\bullet),A)\big)\cong\dH^n(G,A)$. Indeed, one has that $\Hom_\Q(\mathcal{P}_\bullet\otimes_G \Q,A)\cong_\Q \Hom_{\QG}(\mathcal{P}_\bullet,\Hom_\Q(\Q,A))$ by the hom-tensor adjunction~\cite[Proposition~2.6.3]{weibel}. 
Finally,~$\mathrm{H}_{n-1}(\mathcal{K}_n,d_n)$ is projective as a $\Q$-module (or, better, free) and hence~$\mathrm{Ext}^1_\Q\big(\mathrm{H}_{n-1}(\mathcal{K}_\bullet,d_\bullet),A\big)=0$. 
\end{proof}

When establishing that a group does not have certain finiteness properties, a common trick is to check whether it has infinitely generated (co)homology. 

\begin{lem}\label{lem:fin dim}
    If $G$ is a TDLC~group of type $\FP_n(\Q)$ then~$\dH_k(G,\Q)$ and~$\dH^k(G,\Q)$ have finite dimension over~$\Q$ for every~$k\in\{0,\ldots,n\}$.
\end{lem}

\begin{proof}
    Since~$G$ is of type~$\FP_n(\Q)$, there is a projective resolution~$P_\bullet\to \Q$ such that $P_i$ is a proper discrete permutation $\Q[G]$-module for every $i\in\{0,\ldots,n\}$. Thus $\Q\otimes_G P_\bullet$ and $\Hom_G(P_\bullet,\Q)$ are complexes whose $\Q$-vector spaces have finite dimension in degree $i\in\{0,\ldots,n\}$. Indeed, for a compact open subgroup $K$ of $G$, one has
\begin{equation*}
    \Hom_G(\Q[G/K],\Q)\cong \Hom_K(\Q,\Q)\cong\Q\cong\Q\otimes_K\Q\cong\Q\otimes_G\Q[G/K].\qedhere
\end{equation*}
\end{proof}

\section{Presentations that can be generalised over $\phee$}\label{s:genpres} 

In this section we define what it means for a presentation $\langle X\mid R\rangle$ to be generalisable, and establish basic properties.

\begin{defn}\label{defn:genpres}
Let $\phee : \mc{O} \into U$ be a continuous open monomorphism of Hausdorff 
topological groups with $\caO\leq_o U$. Given an abstract group presentation $ \langle X \mid R \rangle$ with $X$ finite, we define the topological group
\begin{equation*}
    G_\varphi:=\mc{H}_X(\phee) / {{\LL R\RR}}_{\mc{H}_X(\phee)},
\end{equation*} 
where $\LL R\RR_{\mc{H}_X(\phee) }$ is the normal closure of (the obvious image of) $R$ in $\mc{H}_X(\phee)$, and the group topology we consider on~$G_\varphi$ is the quotient topology induced by the topology on~$\caH_X(\varphi)$ as in~\Cref{fact:HNNtop}\eqref{fact:HNNtopTopology}.

We say that $\langle X \mid R \rangle$ \emph{can be generalised over $\phee$} (or $\langle X \mid R \rangle$ is {\em generalisable over $\varphi$}) when $U$ embeds into $G_\phee$ as an open subgroup in the obvious way. More precisely, the composition $\phi$ of the canonical injection $U \into \mc{H}_X(\phee)$ as in \Cref{fact:HNNtop}\eqref{fact:HNNtopGraphofgroups} and the canonical projection $\mc{H}_X(\phee) \onto G_\phee$ is injective. 
\[\xymatrix{U\ \ \ar@/_1pc/[rr]_\phi\ar@{^{(}->}[r]&\mc{H}_X(\phee)\ar@{->>}[r]&G_\phee}\]
\end{defn}

\begin{rem}\label{rem:topXgen}
In \Cref{defn:genpres}, the finiteness hypothesis on~$X$ can be removed. All we need for the definition is to have a group topology on~$\caH_X(\varphi)$ that makes the canonical embedding~$U\hookrightarrow \caH_X(\varphi)$ open and continuous. Such a group topology on~$\caH_X(\varphi)$ exists, see~\cite[\S 5.1]{cmw:unimod} for every set~$X$. Unless $X$ is finite, it is not known whether this group topology is unique. Hence {\bf we always consider finite generating sets $X$} throughout.
\end{rem}
\begin{rem}\label{rem:closcomm}
We can also define $G_\phee$ over non-Hausdorff topological groups by setting 
\[    G_\varphi:=\mc{H}_X(\phee) / \overline{{\LL R\RR}}_{\mc{H}_X(\phee)}.\]
The above definition reverts to the one in \Cref{defn:genpres} because, when $U$ is Hausdorff, then so is
 $\mc{H}_X(\phee)$ by \Cref{lem:topologyofHNNextensions}\eqref{lem:topologyofHNNextensions1}, in which case $\LL R\RR_{\caH_X(\varphi)}$ is closed in $\mc{H}_X(\phee)$. To see why, note that $\LL R\RR_{\caH_X(\varphi)}\cap U=\{1\}$ and so $\LL R\RR_{\caH_X(\varphi)}$ is a discrete subgroup. Then~\cite[\S III.2.1, Proposition~5]{bou:top} applies.
\end{rem}

The following is a straightforward application of \Cref{fact:HNNtop}\eqref{fact:HNNtopPresentation}, hence we omit the proof.

\begin{prop}[{Abstract presentation}] \label{prop:abspres} 
 Let $\langle X\mid R\rangle$ be a presentation that can be generalised over $\phee : \mc{O} \into U$. 
Let $\langle Y\mid S\rangle$ be an abstract presentation~of the group~$U$. For every $\omega\in\mc{O}$ let $\omega_Y$ and~$\quer{\omega}$  be words in~$F(Y)$ representing $\omega$ and $\phee(\omega)$ in $\langle Y \mid S\rangle$, respectively. Then
     \begin{equation*} 
G_\varphi\cong_{\text{\tiny{abstr}}}\Big\langle X\sqcup Y \,\Big|  R \cup S\cup \{x\omega_Y x^{-1}\quer{\omega}^{-1} \mid x\in X, \omega \in \mc{O}\} \Big\rangle.
    \end{equation*}
\end{prop}

\begin{prop}[{Basic topological properties}] \label{prop:topologyofArtingps}
 Let $\langle X\mid R\rangle$ be a presentation that can be generalised over $\phee : \mc{O} \into U$, with $U$ a (not necessarily Hausdorff) topological group. Then the following hold.
\begin{enumerate}
\item \label{prop:topArtin1} $G_\phee$ is Hausdorff if and only if~$U$ is Hausdorff.
\item \label{rem:topArtindisc} $G_\phee$ is discrete if and only if~$U$ is discrete.
\item \label{prop:topArtin0} $G_\phee$ is connected if and only if~$U$ is connected and~$X=\leer$.
\item \label{prop:topArtin2} $G_\phee$ is locally compact if and only if~$U$ is locally compact.
\item \label{prop:topArtin3} $G_\phee$ is totally disconnected if and only if~$U$ is totally disconnected.
\item \label{prop:topArtin4} If $G_\phee$ is Hausdorff, then it is compact if and only if~$U$ is compact and~$X = \leer$.
\item \label{prop:topArtin5} $G_\phee$ is compactly generated whenever~$U$ is so.
\item \label{prop:topArtin6} Suppose $U$ (equivalently, $G_\phee$) is Hausdorff. If $U$ is compactly presented,  $\mc{O}$ is compactly generated, and~$R$ is finite, then~$G_\phee$ is compactly presented.
\end{enumerate}
\end{prop}

\begin{proof}
\eqref{prop:topArtin1}--\eqref{prop:topArtin5} are straightforward because $U$ is embedded as an open (and hence closed) subgroup of $G_\phee$ (see~\Cref{lem:topologyofHNNextensions}). 

\eqref{prop:topArtin6} By \Cref{lem:topologyofHNNextensions}\eqref{lem:topologyofHNNextensions6}, $\mc{H}_X(\phee)$ is compactly presented. 
By \Cref{rem:closcomm},  { $\LL R\RR_{\caH_X(\varphi)}$} coincides with the kernel of the canonical projection $\pi\colon\caH_X(\varphi)\to G_\varphi$. As a normal subgroup, $\LL R\RR_{\caH_X(\varphi)}$ is generated by a finite subset. 
The claim thus follows from~\cite[2.1.~Satz]{AbelsKD}.
\end{proof}

Now we go back to our standing assumption that \textbf{our base topological groups $\caO$ and $U$ are always Hausdorff}. Consequently, all topological groups encountered from now on in this paper will also be Hausdorff. We shall thus omit this condition from our hypotheses.

\medskip

The universal property of iterated HNN-extensions (see~\Cref{fact:HNNtop}\eqref{fact:HNNtopUniversalproperty}) generalises to~$G_\varphi$ as follows.

\begin{prop}[{Universal property}]\label{prop:univ} Let $ \langle\, X \mid R\, \rangle$ be a presentation that can be generalised over $\phee : \mc{O} \into U$.
    For every continuous group homomorphism~$\alpha\colon U\to G$ and every sequence of elements~$(g_x)_{x\in X}$ in $G$ satisfying the relations in $R$ (with $g_x$ in place of $x$) and such that
    \begin{equation*}
    \alpha(\varphi(\omega))=g_x\alpha(\omega)g_x^{-1} \quad \forall\,\omega\in\caO,\,\forall\,x\in X
    \end{equation*}
    there is a unique continuous group homomorphism $\overline{\alpha}\colon G_\varphi\to G$ such that~$\overline{\alpha}\vert_U=\alpha$ and $\overline{\alpha}(\bar x)=g_x$ for every~$x\in X$. Here $\bar x$ denotes the image of $x$ in the quotient group $G_\varphi$.
\end{prop}
\begin{proof}
By~\Cref{fact:HNNtop}\eqref{fact:HNNtopUniversalproperty}, there is a unique group homomorphism $\wtil{\alpha}\colon \caH_X(\varphi)\to G$ such that $\wtil{\alpha}\vert_U=\alpha$ and $\wtil{\alpha}(x)=g_x$ for every~$x\in X$. Since $\wtil{\alpha}\vert_U=\alpha$, the morphism $\tilde\alpha$ is continuous on the open neighbourhood $U$ of the identity. Thus $\wtil{\alpha}$ is continuous and, in particular, $\ker(\wtil{\alpha})$ is closed in~$\caH_X(\varphi)$. As the  elements~$(g_x)_{x\in X}$ satisfy the relations in $R$, $\LL R\RR_{\caH_X(\varphi)}\subseteq \ker(\wtil{\alpha})$. Hence, $\tilde\alpha$ induces a continuous group homomorphism~$\overline{\alpha}\colon G_\varphi\to G$ that makes the diagram
$$\xymatrix{U\ar[r]^-{\iota}\ar[rd]_-\alpha&\caH_X(\varphi)\ar[r]^-\pi\ar[d]_-{\tilde\alpha}&G_\varphi\ar[ld]^-{\bar\alpha}\\
&G&
}$$
commute. In the diagram, $\iota$ and~$\pi$ denote the canonical embedding and the canonical projection, respectively.
Finally, the uniqueness of $\widetilde{\alpha}$ implies that $\overline{\alpha}$ is the unique  continuous morphism such that $\overline{\alpha}\vert_U=\alpha$ and $\overline{\alpha}(\bar x)=g_x$ for every~$x\in X$.
\end{proof}
\begin{rem}\label{rem:normU}
In arbitrary topological groups, since open subgroups are closed, one has that $\LL A\RR=\overline{\LL A\RR}$ for every open subset $A$.
\end{rem}
The next proposition generalises a well-known result for HNN-extensions.
\begin{prop}[Splitting property]\label{prop:artin}
   Let~$\langle X \mid R \rangle$ be a presentation that can be generalised over~$\phee : \mc{O} \into U$ and let~$G:=F(X)/\LL R\RR$. Denote by~$i\colon U\hookrightarrow G_\varphi$ the inclusion map and by~$p\colon G_\varphi\to G$ the epimorphism induced by the morphism $\tilde p\colon\caH_X(\varphi)\to F(X)$ mapping $u\mapsto 1$ and~$x\mapsto x$ for all~$u\in U$ and all $x\in X$. 
   The sequence of topological groups
   \[
  \xymatrix{1\ar[r] & \LL U\GG_{G_\varphi}\ar[r]^-i & G_\varphi\ar[r]^-{p} & G\ar[r] & 1}
   \]
  is split exact.   In particular, there is a topological isomorphism 
 $$G_\varphi\cong\LL U\GG_{G_\varphi}\rtimes_{\Phi} G,$$
    where $\Phi\colon G\to \mathrm{Aut}(\LL U\RR_{G_\varphi})$ describes the $G$-action on~$\LL U\RR_{G_\varphi}$ by conjugation, and~$\LL U\RR_{G_\varphi}\rtimes_{\Phi} G$ is endowed with the product topology (here~$\LL U\RR_{G_\varphi}$ carries the subspace topology and~$G$ the discrete topology).
\end{prop}
\begin{proof}
 First we check that~$p$ is well-defined by passage to quotients. The kernel of the composition $\caH_X(\varphi)\stackrel{\tilde p}{\to} F(X)\twoheadrightarrow G$ is a closed normal subgroup containing all the relators given by $R$.  
 Hence~$\LL R\RR_{\caH_X(\varphi)}\subseteq \ker(\tilde{p})$.  Since~$\ker(p)\supseteq U$, the epimorphism~$p$ is continuous. 
 
    Let~$A$ be the subgroup of~$G_\varphi$ generated by the images~$\overline{x}$ of the stable letters $x\in \caH_X(\varphi)$. We show that
    \begin{equation}\label{eq:prod}
        G_\varphi=\LL U\RR_{G_\varphi}\cdot A=\{n\cdot a\mid n\in \LL U\RR_{G_\varphi},\,a\in A\}.
    \end{equation}
   Indeed, as~$G_\varphi$ is (abstractly) generated by $U\cup\{\bar{x}\mid x\in X\}$,  every $g\in G_\varphi$ can be written as $a_1u_1\cdots a_nu_n$, for some $n\geq 1$, $a_1,\ldots, a_n\in A$ and $u_1,\ldots, u_n\in U$. Therefore 
    \begin{equation*} 
       g=\Big((a_1u_1a_1^{-1})(a_1a_2u_2a_2^{-1}a_1^{-1})\cdots(a_1\cdots a_n u_na_n^{-1}\cdots a_1^{-1})\Big)\cdot a_1\cdots a_n.
    \end{equation*}
    
    A (continuous) right-inverse of~$p$ is the homomorphism~$s\colon G\to G_\varphi$ defined by mapping the image of~$x\in X$ in $G$ to $\bar x\in G_\varphi$. In particular,~$s(G)=A$ and $p\vert_A$ is injective.
    Finally, to check that $\ker(p)=\LL U\RR_{G_\varphi}$, let~$n\in \LL U\RR_{G_\varphi}$ and~$a\in A$ such that~$p(na)=1$ (see Equation~\eqref{eq:prod}). {Since~$\ker(p)\supseteq \LL U\GG_{G_\varphi}$}, we have~$p(a)=1$ and hence~$a=1$ because~$p\vert_A$ is injective.

    The second part of the statement follows from~\cite[Theorem~10.14]{strop}.
\end{proof}
The latter result motivates the following convention.
\begin{conv}\label{conv}
    Given a presentation $\langle X\mid R\rangle$ that can be generalised over $\phee\colon \mc{O} \into U$, we identify the abstract group~$G:=F(X)/\LL R\RR$ with its isomorphic image in~$G_\varphi:=\caH_X(\varphi)/\LL R\RR_{\caH_X(\varphi)}$. This means that $G$ shall be regarded as a discrete (hence closed) subgroup of~$G_\varphi$. {In particular, we regard~$X$ as a subset of~$G_\varphi$.}
\end{conv}

For the next \Cref{prop:XRsplit}, we deal with the following generalisation of free products of groups. Namely, given topological groups~$U$ and $G_1,\ldots, G_n$ with continuous open monomorphisms~$\iota_i\colon U\hookrightarrow G_i$, $1\leq i\leq n$, we denote by~$\bigast_U (G_i,\iota_i)$ (or simply~$\bigast_U G_i$ whenever the maps $\iota_i$ are clear) the \emph{free product of~$(G_i,\iota_i)_{i\in\{1, \ldots, n\}}$ amalgamated along~$U$} as defined in~\cite[\S I.1.2]{ser:trees}. 
There is a standard way~\cite[\S I.4.4, Example~(c)]{ser:trees} to  express~$\bigast_U G_i$ as the fundamental group of a finite star of groups. Given this and following~\cite[p.~30]{Cohen},~$\bigast_UG_i$ can be described as an iterated free product with amalgamation of the form 
\begin{equation}\label{eq:iteramalg}
    \textstyle{\bigast_U}G_i\cong (((G_1\ast_U G_2)\ast_U G_3)\cdots)\ast_UG_n.
\end{equation}
Hence, by~\cite[Proposition~8.B.9]{cdlh:metric}, there is a unique group topology on~$\bigast_U G_i$ that turns the canonical embeddings of every~$G_i$ and of~$U$ into~$\bigast_UG_i$ into continuous open monomorphisms.
\begin{prop}\label{prop:XRsplit}
    Let~$\langle X\mid R\rangle$ be a presentation. Assume that there is a partition~$X=X_1\sqcup \cdots \sqcup X_n$, $n\geq 1$, that induces a partition~$R=R_1\sqcup \ldots \sqcup R_n$ where, for every~$1\leq i\leq n$, the set~$R_i$ is the collection of all elements of~$R$ formed only by letters in~$X_i\sqcup X_i^{-1}$. Then, for every continuous open monomorphism~$\varphi\colon \caO\hookrightarrow U$, the following hold:
    \begin{enumerate}
        \item \label{prop:XRsplit1} If $\langle X\mid R\rangle$ can be generalised over~$\varphi$, then~$\langle X_i\mid R_i\rangle$ can be generalised over~$\varphi$ for every~$1\leq i\leq n$.
        \item \label{prop:XRsplit2} Suppose that~$\langle X\mid R\rangle$ can be generalised over~$\varphi$. Let~$G_\varphi:=\caH_X(\varphi)/\LL R\GG_{\caH_X(\varphi)}$ and~$G_{\varphi,i}:=\caH_X(\varphi)/\LL R\GG_{\caH_{X_i}(\varphi)}$ for all~$1\leq i\leq n$. Then there is a canonical continuous open isomorphism
        \begin{equation*}
            \textstyle{\bigast_U} (G_{\varphi,i},j_i)\stackrel{\cong}{\longrightarrow} G_\varphi,
        \end{equation*}
       where~$j_i\colon U\hookrightarrow G_{\varphi,i}$ is the canonical embedding for every~$1\leq i\leq n$.
    \end{enumerate}
\end{prop}
\begin{proof}
    \eqref{prop:XRsplit1} Once proved that each~$H_i:=\caH_{X_i}(\varphi)$ embeds in~$H:=\caH_X(\varphi)$ as an open subgroup containing~$U\leq H$, one has that~$\LL R_i\GG_{H_i}=\langle hrh^{-1}\mid h\in H_i,\,r\in R_i\rangle\leq \LL R\GG_H$ and the statement follows. We claim that~$H$ is canonically isomorphic to~$\bigast_U (H_i,\ell_i)$, where~$\ell_i\colon U\hookrightarrow H_i$ is the obvious embedding for every~$1\leq i\leq n$. Indeed, if~$\langle Y\mid S\rangle$ is any presentation of~$U$, by \Cref{fact:HNNtop}\eqref{fact:HNNtopPresentation} each~$H_i$ is presented by 
    \begin{equation*} 
   \langle X_i\sqcup Y\mid R\sqcup \{t\omega_Y t^{-1}=\overline{\omega}\mid t\in X_i,\omega\in \caO\}\rangle
   \end{equation*}
   and~$H$ is presented by 
    \begin{equation}\label{eq:XRsplit1}
        \langle X\sqcup Y\mid R\sqcup \{t\omega_Y t^{-1}=\overline{\omega}\mid t\in X,\omega\in \caO\}\rangle.
    \end{equation}
   Since~$X=\bigsqcup_{i=1}^nX_i$ and~$R=\bigsqcup_{i=1}^nR_i$, the presentation in Equation~\eqref{eq:XRsplit1} is a presentation for~$\bigast_U (H_i,\ell_i)$; see Equation~\eqref{eq:iteramalg}. Therefore, $\bigast_U (H_i,\ell_i)$ is canonically isomorphic to~$H$.
   
   \eqref{prop:XRsplit2} The statement follows from \Cref{prop:abspres}.
\end{proof}

\noindent A typical situation in which~\Cref{prop:XRsplit} applies is the Artin presentation.
\begin{ex}[{Splitting of topological Artin groups}]\label{ex:ArtinSplit}
    Let~$\Gamma=(S,E)$ be a finite graph with a function~$m\colon E\to \Z_{\geq 2}$, where we denote $m_{s,t}:=m(\{s,t\})$. The \emph{Artin presentation of type~$(\Gamma,m)$} is the group presentation~$\langle S\mid R\rangle$ where
     \begin{equation*}
         R=\left\{\underbrace{sts\cdot \ldots}_{m_{s,t} \text{ factors}} = \underbrace{tst\cdot \ldots}_{m_{s,t} \text{ factors}} \,\Bigg \vert\, \{s,t\} \in E\right\}.
     \end{equation*}
    The group given by this presentation is called the \emph{Artin group of type~$(\Gamma,m)$}.
     
     Let~$\Gamma_i=(S_i,E_i)$, $1\leq i\leq n$, be the connected components of~$\Gamma$ and let~$\langle S_i\mid R_i\rangle$ be the Artin presentation of type $(\Gamma_i, m_i)$ where $m_i:=m\vert_{E_i}$. 
     Once we know that this presentation $\langle S \mid R\rangle$ is generalisable over $\phee$, which we shall prove later on in \Cref{thm:retract}\eqref{thm:retractb}, then \Cref{prop:XRsplit} provides us with a canonical isomorphism of topological groups
     \begin{equation*}
         \caA_{(\Gamma,m)}(\varphi)\cong \textstyle{\bigast_U} \caA_{(\Gamma_i, m_i)}(\varphi),
     \end{equation*}
     where~$\caA_{(\Gamma,m)}(\varphi):=\caH_S(\varphi)/\LL R\GG$ and~$\caA_{(\Gamma_i,m_i)}(\varphi):=\caH_{S_i}(\varphi)/\LL R_i\GG$ for every~$1\leq i\leq n$.
\end{ex}

The following result is the analogue of \Cref{prop:connComp} for~$G_\varphi$.
\begin{prop}[The connected component of~$1$ in~$G_\varphi$]\label{prop:connCompGP}
     For every continuous open monomorphism~$\varphi\colon \caO\hookrightarrow U$, let
   \[\begin{tikzcd}        
    \widetilde{\varphi}\colon \widetilde{\caO}:=\caO/\caO_0 \arrow[hook]{r} & \wtil{U}:=U/U_0 
\end{tikzcd}\]
    be the continuous open monomorphism as in Proposition~\ref{prop:connComp}.
    Then every presentation~$\langle X\mid R\rangle$ that can be generalised over~$\varphi$ can be also generalised over~$\wtil{\varphi}$. Moreover, we have a short exact sequence of topological groups
     \begin{equation}\label{eq:sesconn2}  
  \xymatrix{1\ar[r] & U_0\ar[r]^-{i} & G_\varphi\ar[r]^-{\overline{\eta}} & G_{\wtil{\varphi}}\ar[r] & 1}
     \end{equation}
  where~$i$ is the inclusion and~$\overline{\eta}$ is the continuous open epimorphism induced by the canonical projection~$\pi_U\colon U\twoheadrightarrow \widetilde{U}$ and the identity~$\iid_X\colon X\to X$. 

   In particular, $(G_\varphi)_0=U_0$ and~$G_{\wtil{\varphi}}$ is a totally disconnected group.
\end{prop}
\begin{proof}
  Let~$\langle X\mid R\rangle$ be a group presentation that can be generalised over~$\varphi$, that is~$\LL R\GG_{\caH_X(\varphi)}\cap U=\{1\}$. Moreover,~$\wtil{U}$ (and hence~$\caH_X(\wtil{\varphi})$) is Hausdorff. Hence $\langle X\mid R\rangle$ can be generalised over~$\wtil{\varphi}$ if and only if
  \begin{equation}\label{eq:connGP0}
      \LL R\GG_{\caH_X(\wtil{\varphi})}\cap \wtil{U}=\{1\}.
  \end{equation} 

 We may focus on proving~\eqref{eq:connGP0}.
  Let~$\eta\colon \caH_X(\varphi)\to \caH_X(\wtil{\varphi})$ be the continuous open epimorphism as in \Cref{prop:connComp}. Recall that~$\eta$ is a surjective homomorphism satisfying~$\eta\vert_X=\iid_X$. By \Cref{prop:connComp} we have
  \begin{equation}\label{eq:connGP1}
      \eta(\LL R\GG_{\caH_X(\varphi)})=\LL R\GG_{\caH_X(\wtil{\varphi})}.
  \end{equation}
(To see it, recall that the normal closure of a set~$Y$ in a group~$H$ is the subgroup generated by~$\bigcup_{h\in H}hYh^{-1}$, and then use the fact that~$\eta$ is a surjective homomorphism.)
Moreover,~$\eta$ induces an isomorphism of topological groups
\begin{equation*}
    \caH_X(\wtil{\varphi})\cong \caH_X(\varphi)/U_0
\end{equation*}
that makes the subgroup~$\LL R\GG_{\caH_X(\wtil{\varphi})}\leq\caH_X(\wtil{\varphi})$ correspond to~$\LL R\GG_{\caH_X(\varphi)}U_0/U_0\leq \caH_X(\varphi)/U_0$.
To prove~\eqref{eq:connGP0} it remains to show that
\begin{equation}\label{eq:connGP2}
    \big(\LL R\GG_{\caH_X(\varphi)}U_0/U_0 \big)\cap \wtil{U}=\{1\}.
\end{equation}
But~$U\cap \LL R\GG_{\caH_X(\varphi)}=\{1\}$ yields $U\cap \LL R\GG_{\caH_X(\varphi)}U_0=U_0$ and hence~\eqref{eq:connGP2}. This proves~\eqref{eq:connGP0} and concludes the proof that~$\langle X\mid R\rangle$ can be generalised over~$\wtil{\varphi}$.

To show the exactness of the sequence~\eqref{eq:sesconn2}, we first observe that~$p$ coincides the group homomorphism~$\overline{\eta}\colon G_\varphi\to G_{\wtil{\varphi}}$ induced --- by passing to the quotients --- by the continuous open epimorphism~$\eta\colon \caH_X(\varphi)\to \caH_X(\wtil{\varphi})$ used in the first part of the proof. Such an~$\overline{\eta}$ well-defined because of~\eqref{eq:connGP1}.
Having~$\eta\vert_U=\pi_U$ and~$\eta\vert_X=\iid_X$ (see~\Cref{prop:connComp}) and by the uniqueness part of \Cref{prop:univ}, we deduce that~$p=\overline{\eta}$.

The latter equality helps us determining~$\ker(p)$. Namely, since~$\ker(\eta)=U_0$ by \Cref{prop:connComp}, we have
   \begin{equation*}
       \ker(p)=\ker(\overline{\eta})=(\ker(\eta)\LL R\GG_{\caH_X(\varphi)})\big/\LL R\GG_{\caH_X(\varphi)}=(U_0\LL R\GG_{\caH_X(\varphi)})\big/\LL R\GG_{\caH_X(\varphi)}.
   \end{equation*}
That is, $\ker(p)$ is the image of~$U_0$ via the canonical embedding~$\iota\colon U\hookrightarrow G_\varphi, \,u\mapsto u\LL R\GG_{\caH_X(\varphi)}$. Hence, the sequence in~\eqref{eq:sesconn2} has all the claimed properties. 
   
   The last statement of the lemma now follows from~\Cref{fact:ConCompTG}.
\end{proof}

\subsection{Presentations with balanced relators}\label{ex:balanced}

In this section we provide a sufficient condition for presentations to be generalisable over any continuous open monomorphism~$\phee : \caO \into U$.

\begin{defn}\label{defn:balanced}
 A presentation $\langle X\mid R\rangle$ is said to have {\em balanced relators} if every relator $r \in R$ belongs to the kernel of the map $F(X)\to \Z$ given by  $x_1^{\varepsilon_1}\cdots x_n^{\varepsilon_n}\mapsto\varepsilon_1+\cdots+\varepsilon_n$.
\end{defn}

\begin{ex} \label{ex:generalisablepresentations}
The following groups, which are extensively studied in the literature, admit presentations with balanced relators.
\begin{enumerate}
    \item {(\emph{Orientable surface groups}).} For a positive integer~$g$ the group $$\Gamma_g = \langle a_1, b_1, a_2, b_2,\ldots, a_g, b_g\mid [a_1, b_1],\ldots, [a_g, b_g] = 1\rangle$$
    is the fundamental group~$\Gamma_g \cong \pi_1(\Sigma_g)$ of the closed, orientable, genus-$g$ surface 
$\Sigma_g$.
    \item {(\emph{Thompson's group~$F$}).}  Thompson's group~$F$ is the group of piecewise linear homeomorphisms of the interval~$[0,1]$ with break points in~$\Z[\frac12]$ and slopes powers of~$2$. Among many other presentations \cite{CannonFloydParry},~$F$ also admits the following finite presentation: 
$$\langle x_0,x_1 \,|\,[x_0x_1^{-1}, x_0^{-1}x_1x_0]=1, [x_0x_1^{-1} , x_0^{-2}x_1x_0^2]=1 \rangle.$$
\item {(\emph{Artin groups}).} As in \Cref{ex:ArtinSplit}, the Artin group of type~$(\Gamma,m)$ is defined by the presentation 
\[\spans{ S \mid \{\underbrace{sts\cdot \ldots}_{m_{s,t} \text{ factors}} = \underbrace{tst\cdot \ldots}_{m_{s,t} \text{ factors}} \mid \{s,t\} \text{ is an edge in } \Gamma \} }.\]
\item {(\emph{Baumslag--Solitar groups $B(n,n)$}).} For an integer~$n$, let
\begin{equation*}
    B(n,n)=\langle a,b\mid ba^nb^{-1}=a^n\rangle
\end{equation*}
be the virtually abelian Baumslag--Solitar group with integral parameters~$(n,n)$.
\end{enumerate}
\end{ex}
Although all the cases in~\Cref{ex:generalisablepresentations} deal with torsion-free groups, groups with balanced presentation might have torsion; see \Cref{ex:diffpres}.

\begin{thm}\label{thm:retract}
    Let $X\neq \leer$, let $\langle X\mid R\rangle$ be a  presentation with balanced relators and let $\phee : \mc{O} \into U$ be a continuous open monomorphism of topological groups $\caO\leq_o U$. 
    Denote by~$ \iota\colon U\hookrightarrow \caH_{\{t\}}(\phee)$ the canonical embedding. 
    Then the following hold:
    \begin{enumerate}
    \item\label{thm:retracta} There is a continuous group epimorphism 
    \begin{equation*}
       \overline{\psi}\colon G_\varphi\to \caH_{\{t\}}(\varphi)\ \text{such that}\ u\mapsto \iota(u), x\mapsto t\ \text{for all}\ u\in U,x\in X;
    \end{equation*}
    \item\label{thm:retractb} $\langle X\mid R\rangle$ can be generalised over $\varphi;$
  \item\label{thm:retractc} For every $x\in X$ the map $\overline\psi$ has a right-inverse~$j_x\colon \caH_{\{t\}}(\varphi)\to G_\varphi$, which is a continuous open  monomorphism satisfying~$j_x\vert_U=\textup{id}_U$ and~$j_x(t)=x$;
  
  \item\label{thm:retractd} The subgroup~$j_x(\caH_{\{t\}}(\varphi))$ is topologically isomorphic to~$\caH_{\{t\}}(\varphi)$ and is a group retract\footnote{A \emph{group retract} of a topological group~$G$ is a closed subgroup~$H\leq G$ for which there is a continuous epimorphism~$\rho\colon G\to H$ that satisfies~$\rho\vert_H=\iid_H$; see~\cite[Definition~8.A.11]{cdlh:metric}).} of~$G_\varphi$.
  \end{enumerate}
\end{thm}
\begin{proof}
As usual we regard~$U$ as an open subgroup of~$\caH_X(\varphi)$; see \Cref{fact:HNNtop}. To ease notation let~$H=\caH_{\{t\}}(\varphi)$.

\eqref{thm:retracta}-\eqref{thm:retractb} 
We first build the continuous group epimorphism~$\psi\colon \caH_X(\varphi)\to H$ induced by the assignments~$u\mapsto \iota(u)$ and~$x\mapsto t$ for all~$u\in U$ and~$x\in X.$ Existence and uniqueness of~$\psi$ follow from Lemma~\ref{fact:HNNtop}\eqref{fact:HNNtopUniversalproperty} observing that, for all~$x\in X$ and~$u\in U$, 
    \begin{equation*}
    \psi(xux^{-1})=\psi(x)\psi(u)\psi(x)^{-1}=t \iota(u)t^{-1}=\iota(\varphi(u))=\psi(\varphi (u)).
    \end{equation*}
    Note that~$\psi$ is continuous since~$\psi{\vert_U}=\iota$ is continuous and $U$ is an open subgroup of~$\caH_X(\varphi)$. 
As $\psi$ maps every balanced relator $r$ to $1_H$, ${\LL R\GG}_{\caH_X({\varphi})}
\subseteq \ker(\psi)$ and  there is a continuous homomorphism~$\overline\psi$ making the  diagram
   \begin{equation*}
  \xymatrix{
  {\caH_X(\varphi)}\ar@{->>}[rr]^{\psi} \ar@{->>}[d]_{\pi} && H\\
      G_\varphi\ar@{-->>}[urr]_{\overline{\psi}}
  }
  \end{equation*}
  commute, where~$\pi$ is the canonical quotient map. In particular, $\ker(\pi\vert_U) \subseteq \ker(\psi\vert_U) = \ker(\iota) =\{1\}$ and~$\langle X \mid R \rangle$ can be generalised over~$\phee$.

  \eqref{thm:retractc} By Lemma~\ref{fact:HNNtop}\eqref{fact:HNNtopUniversalproperty}, for every~$x\in X$ there is a unique group homomorphism~$\ell_x\colon H\to \caH_X(\varphi)$ induced by the assignments~$u\mapsto u$ for every~$u\in U,$ and $t\mapsto x.$
   Arguing as in case~\eqref{thm:retracta}, one shows that~$\ell_x$ is continuous and open because~$\ell_x\vert_U$ is continuous and open onto its image, which is the open subgroup~$U$ of~$\caH_X(\varphi)$. Then~$\psi\circ \ell_x$ is a continuous group homomorphism that is the identity on~$U\cup\{t\}$. By Lemma~\ref{fact:HNNtop}\eqref{fact:HNNtopUniversalproperty}, 
 $\psi\circ\ell_x=\textup{id}_{\caH_{\{t\}}(\varphi)}.$
Set
   \begin{equation*}
       j_x:=\pi\circ \ell_x\colon \caH_{\{t\}}(\varphi)\to G_\varphi,
   \end{equation*}
   where~$\pi\colon \caH_X(\varphi)\twoheadrightarrow G_\varphi$ is the canonical projection. Since both~$\pi$ and~$\ell_x$ are open continuous group homomorphisms, so is~$j_x$. Moreover,
   \begin{equation}\label{eq:retr2}
      \overline{\psi}\circ j_x=(\overline{\psi}\circ \pi)\circ\ell_x=\psi\circ \ell_x=\textup{id}_{\caH_{\{t\}}(\varphi)}.
   \end{equation}
   In particular, $j_x$ is injective. Therefore, $j_x(\caH_{\{t\}}(\varphi))$ is an open (hence closed) subgroup of~$G_\varphi$ and it is topologically isomorphic to~$\caH_{\{t\}}(\varphi)$ via~$j_x$.
   
   \eqref{thm:retractd} To conclude that~$j_x(\caH_{\{t\}}(\varphi))$ is a group retract observe that 
   \begin{equation*}
       j_x\circ\bar\psi\colon G_\varphi\to j_x(\caH_{\{t\}}(\varphi))
   \end{equation*}
   is a continuous group epimorphism that is the identity on~$j_x(\caH_{\{t\}}(\varphi))$. 
\end{proof}

\begin{cor}\label{cor:Ut^U} Let $\langle X\mid R\rangle$  be a presentation with balanced relators with $X \neq \leer$ and let $\phee : \mc{O} \into U$ be a continuous open monomorphism with $\caO\leq_o U$. 
In~$G_\varphi$ one has  $U\cap  xU x^{-1}=\varphi(\caO)$ and~$U\cap  x^{-1}U x=\caO$ for every~$x\in X$.
\end{cor}
\begin{proof}
    Given~$x\in X$ let~$j=j_x\colon \caH_{\{t\}}(\varphi)\to G_\varphi$ be the open continuous monomorphism of Theorem~\ref{thm:retract}\eqref{thm:retractc}. 
    Since~$j\vert_U=\textup{id}_U$ and~$j(t)=x$, one has
    \begin{equation*}
        j(U\cap t^{\pm 1}Ut^{\mp 1})=j(U)\cap j(t)^{\pm 1}j(U)j(t)^{\mp 1}=U\cap  x^{\pm 1}U x^{\pm 1}
    \end{equation*}
    as well as $j(\caO)=\caO$ and $j(\varphi(\caO))=\varphi(\caO)$.
    The claim follows from \Cref{fact:HNNtop}\eqref{fact:HNNtopBassSerretree}.
\end{proof}
\begin{prop}[Normality of $U$ in $G_\phee$]\label{prop:normU}
 Let $\langle X\mid R\rangle$ be a presentation with balanced relators and $\phee : \mc{O} \into U$ a continuous open monomorphism with $\caO\leq_o U$. Thus
\begin{equation*}
     \caO=U\ \text{and $\varphi$ is surjective} \Longleftrightarrow\,\LL U\GG_{G_\varphi}=U\ \text{in~$G_\varphi$}
\end{equation*}
\end{prop}
\begin{proof}
   Clearly, $\LL U\RR_{G_\varphi}=U$ if and only if~$U\unlhd G_\varphi$.
Since $G_\varphi$ is generated by the image of~$U\cup X^{\pm 1}$, one has that
\begin{equation*}
    U\unlhd G_\varphi\,\Longleftrightarrow\, xU x^{-1}=U,\quad\forall\,x\in X.
\end{equation*}
(Note that, $xU x^{-1}=U$ implies also~$ x^{-1}U x=U$.)
It remains to show that the following holds in~$G_\varphi$:
\begin{equation}\label{eq:tUt=U}
     xU x^{-1}=U,\quad\forall\,x\in X\,\Longleftrightarrow\, \varphi(\caO)=\caO=U.
\end{equation}
By Corollary~\ref{cor:Ut^U}, for every~$x\in X$ one has
\begin{equation*}
    \varphi(\caO)= x\caO  x^{-1},\quad U\cap  xU x^{-1}=\varphi(\caO)\quad\text{and}\quad U\cap x^{-1}U x=\caO\text{ in }G_\varphi.
\end{equation*}
Therefore,
\begin{equation*}
\begin{split}
    U= xU x^{-1}& \Longleftrightarrow\, |U:U\cap  xU x^{-1}|=1=|U: U\cap x^{-1}U x|\\
     & \Longleftrightarrow\, |U:\varphi(\caO)|=1=|U:\caO|
\end{split}
\end{equation*}
and~\eqref{eq:tUt=U} follows. 
\end{proof}
\begin{thm}[{Cocompactness of obvious subgroups}]\label{prop:cocomp}
 Let $\langle X\mid R\rangle$ be a presentation with balanced relators and $\phee : \mc{O} \into U$ a continuous open monomorphism with $\caO\leq_o U$ and $\phee(\mc{O}) \subseteq \mc{O}$. Then the following hold:
\begin{enumerate}
\item \label{prop:cocompU} The base group $U \csbgp G_\phee$ is cocompact if and only if $X=\leer$. 
\item \label{prop:cocompAGamma} The discrete group~$G:=F(X)/\LL R\RR\csbgp G_\phee$ is cocompact if and only if 
\begin{itemize}
\item $X=\leer$  and $U$ is compact; or
\item $\phee(\mc{O}) = \mc{O} = U$ and $U$ is compact. 
\end{itemize} 
\end{enumerate} 
\end{thm}

\begin{proof}
\eqref{prop:cocompU} If $X=\leer$ then $G_\phee = U$ by definition. Conversely, assume that the orbit space $G_\phee / U$ is compact and suppose for a contradiction that there is some $x \in X$. From \Cref{thm:retract} recall the group retraction $\quer{\psi} : G_\phee \to \mc{H}_{\{t\}}(\phee)$, which is injective on $U$, and consider the  commutative diagram of topological spaces 
\[
\xymatrix{
{G_\varphi} \ar@{->>}[rr]^{p_G} \ar@{->>}[d]_{\quer{\psi}} && {G_\phee/U}\ar@{-->>}[d]^{\overline\psi_U} &&\\
{\mc{H}_{\{t\}}(\varphi)}\ar@{->>}[rr]_{p_{\caH}} \ar@/_2.0pc/[rrrr]_{p_{\LL U\RR}}&& {\mc{H}_{\{t\}}(\varphi)/U}\ar@{-->>}[rr]_-\pi&&  \mc{H}_{\{t\}}(\varphi)/\LL U\RR_{\mc{H}_{\{t\}}(\varphi)}\cong_{top}\Z
}
\]
where  $\overline\psi_U$ is the continuous map induced by $\overline\psi$ and the  canonical projections $p_G,p_\caH$, and the continuous map $\pi$ is induced by the quotient map $p_{\LL U\RR}\colon \caH_{\{t\}}(\varphi)\twoheadrightarrow \caH_{\{t\}}(\varphi)/\LL U\GG$. 
One sees that $\Z$ with the discrete topology would then be the image (see~\Cref{prop:artin}) of the compact space $G_\phee / U$ through $\pi\circ\overline\psi_U$, a contradiction. 

\eqref{prop:cocompAGamma} Let $X=\leer$. Clearly, the trivial group is cocompact in $G_\phee$ if and only if $G_\varphi=U$ is compact. Assume  that $X\neq \leer$. Let 
$p \colon G_\varphi\twoheadrightarrow G_\varphi/G$
be the quotient map onto the coset space $G_\varphi/G$. By \Cref{prop:artin}, 
\[
G_\varphi=\LL U\GG_{G_\phee}\cdot G \quad \text{ and } \quad \LL U\GG_{G_\phee}\cap G=\{1\}.
\]
Since $\LL U\GG_{G_\phee}$ is open in $G_\varphi$ and $p$ is open and continuous, the restriction $p\vert_{\LL U \GG}$ yields a homeomorphism
between $\LL U\GG_{G_\phee}$ and $G_\phee/G$. 
Thus, $G$ is cocompact in $G_\phee$ if and only if $\LL U\GG_{G_\phee}$ is compact.

Suppose first that $\varphi(\caO)=\caO=U$ and $U$ is compact. Then  $\LL U\GG_{G_\phee}=U$ is compact by \Cref{prop:normU}.
Conversely, suppose that $\LL U\GG_{G_\phee}$ is compact. Since $\caO$ and~$U$ are open (and so closed) subgroups of $\LL U\GG_{G_\phee}$, both $\caO$ and~$U$ are compact. It remains to show that~$\varphi(\caO)=\caO=U$. For what follows, we fix an arbitrary $t\in X$. We first argue that $\varphi(\caO)=\caO$. Since $\varphi(\caO)=t\caO t^{-1}\subseteq \caO$ by hypothesis, the subspace 
\[N_\caO :=\bigcup_{n\in\Z_{\geq 0}}t^{-n}\caO t^{n} \subseteq G_\phee\]
is an increasing sequence of open subgroups of $\LL U\GG_{G_\phee}$. Therefore, $N_\caO$ is an open (hence closed) subgroup of $\LL U\GG_{G_\phee}$, which implies that $N_\caO$ is compact. Being an increasing union of open subgroups, there is $N\geq 0$ for which $t^{-N}\caO t^{N}=t^{-N-1}\caO t^{N+1}$, which in turn implies that $\varphi(\caO)=t\caO t^{-1}=\caO$.

Let us now check that $\mc{O} = U$. Assume, for a contradiction, that $\caO \subsetneq U$. Recall the retraction~$\quer{\psi} : G_\phee \onto \mc{H}_{\{t\}}(\phee)$ from \Cref{thm:retract} with $\quer{\psi}\vert_U=\iid_U$ and $\quer{\psi}(x)=t$ for all $x \in X$. It is clear that $\quer{\psi}(\LL U \RR_{G_\phee}) = \LL U \RR_{\mc{H}_{\{t\}}(\phee)}$. Thus $\LL U \RR_{\mc{H}_{\{t\}}(\phee)}$ is compact in $\mc{H}_{\{t\}}(\phee)$. 
By \Cref{prop:artin}, $\caH_{\{t\}}(\varphi)$ is topologically isomorphic to the inner semidirect product~$\LL U\GG_{\caH_{\{t\}}(\varphi)}\rtimes \Z$. Hence, $\caH_{\{t\}}(\varphi)/\LL U\GG$ is homeomorphic to the discrete space~$\Z$, which is non-compact.
\end{proof}

We conclude this subsection with some auxiliary results that help us analyse intersections of conjugates of $U$ and $\caO$ inside $G_\phee$. 

\begin{defn}\label{defn:exponent}
Let $\langle X\mid R\rangle$ be a presentation with balanced relators.
The \emph{exponent map} associated to $\langle X\mid R\rangle$ is the group homomorphism $e\colon F(X)/\LL R\RR\to \Z$ induced by the assignments $x\mapsto 1$ for every $x\in X$.
\end{defn}

Note that, for every open continuous monomorphism~$\phee : \mc{O} \into U$  with $\caO\leq_o U$, the exponent~$e$ extends to the whole group~$G_\phee$ by sending~$U$ to~$\{0\}$, allowing us to define the extended exponent~$\quer{e}$. This maps~$\quer{e}$ will be used later on in the case where~$G_\phee$ is a topological RAAG; see \Cref{obs:exponentisMorsefunction}. 

\begin{lem}\label{lem:UwU}
 Let $\langle X\mid R\rangle$ be a presentation with balanced relators and $\varphi\colon\caO\hookrightarrow U$ a continuous open monomorphism with $\caO\leq_o U$. Denoting $G:=F(X)/\LL R\RR$ and $G_\varphi:=\caH_X(\varphi)/{\LL R\RR}_{\caH_X(\varphi)}$ the following hold:
    \begin{enumerate}
        \item \label{lem:UwU1}
    for every~$g\in G$ with \mbox{$e(g)>0$,} one has $U\cap gUg^{-1}\subseteq \varphi(\caO)$ in~$G_\phee$. 
    \item \label{lem:UwU2} for every~$g\in G$ with~$e(g)<0$, one has $U\cap gUg^{-1}\subseteq \caO$ in $G_\varphi$.
    \item \label{lem:UwU3} if~$\varphi(\caO)\subseteq \caO$, then  $U\cap gUg^{-1}=\caO\cap g\caO g^{-1}$ for every~$g\in G$ with~$e(g)\neq 0$.
    \end{enumerate}
\end{lem}
\begin{proof} \eqref{lem:UwU1}-\eqref{lem:UwU2} Let $g\in G$. We prove the statement for~$e(g)>0$, as the argument is analogous for $e(g)<0$.
   Let~$\overline{\psi}\colon G_\varphi\to \caH_{\{t\}}(\varphi)$ be the continuous homomorphism given by Theorem~\ref{thm:retract}, and notice that $\overline{\psi}(g)=t^{e(g)}$. As usual, we keep denoting the images of~$U$ and~$\varphi(\caO)$ in~$\caH_{\{t\}}(\varphi)$ by $U$ and~${\varphi(\caO)}$, respectively. By~\Cref{fact:HNNtop}\eqref{fact:HNNtopBassSerretree}, one has
   \begin{equation*}\label{eq:UwU1}
       \overline{\psi}(U\cap gUg^{-1})\subseteq \overline{\psi}(U)\cap \overline{\psi}(gUg^{-1})={U}\cap t^{e(g)}{U} t^{-e(g)}\subseteq {\varphi(\caO)}\subseteq U.
   \end{equation*}
  Hence, in $G_\varphi$ one has
   \begin{equation*}
       U\cap gUg^{-1}\subseteq (\overline{\psi}\vert_U)^{-1}(\overline{\psi}(U\cap gUg^{-1}))\subseteq (\overline{\psi}\vert_U)^{-1}({\varphi(\caO)})=\varphi(\caO)
   \end{equation*}
   because $\overline{\psi}\vert_U$ maps bijectively onto the image of~$U$ in~$\caH_{\{t\}}(\varphi)$.
   
\eqref{lem:UwU3}
    Since~$\caO\subseteq U$, we have $\caO\cap g\caO g^{-1}\subseteq U\cap gUg^{-1}$. To prove the reverse inclusion, observe that~\eqref{lem:UwU1} and~\eqref{lem:UwU2} imply
    \begin{equation*}
        U\cap gUg^{-1}\subseteq \caO\supseteq U\cap g^{-1}Ug
    \end{equation*}
    because~$\varphi(\caO)\subseteq \caO$. 
    Since~$U\cap gUg^{-1}=g(U\cap g^{-1}Ug)g^{-1}$, we conclude that $U\cap gUg^{-1}\subseteq \caO\cap g\caO g^{-1}$.
\end{proof}

Under a stronger hypothesis (and a much simpler proof), we get a stronger conclusion than that of \Cref{lem:UwU}\eqref{lem:UwU3}. 

\begin{lem}\label{lem:HDnotes2}
    Keeping the notation from \Cref{lem:UwU}, assume that $\varphi(\caO)\subseteq \caO$ and that $U\cap gUg^{-1}\subseteq \caO$ for all $g\in G \setminus \{1\}$. Then 
    \begin{equation*}
        U\cap gUg^{-1}=\caO\cap g\caO g^{-1}
    \end{equation*}
in $G_\phee$ for all $g \in G \setminus\{1\}$. 
\end{lem}

\begin{proof}
Given $g \in G \setminus \{1\}$ we have $U\cap gUg^{-1}\subseteq \caO$ by hypothesis, so that $U\cap gUg^{-1} = g(U\cap g^{-1}Ug) g^{-1}\subseteq g\caO g^{-1}$ as well. Thus $U\cap gUg^{-1}\subseteq \caO \cap g\caO g^{-1}$. Since $\caO\subseteq U$, the reverse inclusion clearly holds. 
\end{proof}

\subsection{Some non-generalisable presentations} \label{sus:topCox} 

Here we present a type of presentation that cannot be generalised over any non-surjective~$\phee : U \into U$ with~$U$ an arbitrary topological (Hausdorff, by default) group.

\begin{prop} \label{prop:morethanCoxeter}
    Let~$U$ be a topological group admitting a continuous open non-surjective monomorphism~$\varphi\colon U\hookrightarrow U$. Suppose $\spans{X \mid R}$ is an abstract presentation with $X$ finite but with~$R$ containing a relator~$x^n$ for some~$n \in \Z_{\geq 2}$ and~$x \in X$. (That is, the presentation $\spans{X \mid R}$ stipulates that some generator has finite order.) Then $\spans{X \mid R}$ is \textbf{not} generalisable over~$\phee$.
\end{prop} 

\begin{proof}
Let~$\pi : \mc{H}_X(\phee) \onto G_\phee$ be the canonical projection from the topological HNN-extension~$\mc{H}_X(\phee)$ onto the topological group quotient~$G_\phee := \mc{H}_X(\phee) / {\LL R \RR}_{\mc{H}_x(\phee)}$. By definition, we need to show that~$\pi$ is not injective when restricted to the natural copy of~$U$ in $\mc{H}_X(\phee)$. 

By hypothesis, there exists some~$u \in U \setminus \phee(U)$. In particular, for the given $x \in X$, one has $u \neq \phee^n(u) = x^n u x^{-n}$ in~$\mc{H}_X(\phee)$. On the other hand, the word $x^n$ is contained in~$R$, hence 
\[\pi(\phee^n(u)) = \pi(x^n u x^{-n}) = \pi(x)^n \pi(u) \pi(x)^{-n} = \pi(u),\]
whence~$\pi\vert_U$ is not injective.
\end{proof}

In stark contrast with \Cref{thm:retract}\eqref{thm:retractb} and \Cref{ex:generalisablepresentations} 
we have the following infinite series of non-examples. 

\begin{ex}
Let~$\phee : U \into U$ be continuous open self-monomorphism of a topological group~$U$ with $\phee(U) \neq U$. (Concretely we may take, e.g., the $p$-adics $U = \Z_p$ and $\phee : \Z_p \into \Z_p$ being multiplication by $p$.) Then: 
\begin{enumerate}
    \item[(a)] 
\emph{Coxeter presentations}, namely group presentations of the form
$$\langle S\mid s^2,\,(vw)^{m(\{v,w\})},\,s\in S,\,\{v,w\}\in E\rangle$$
for some finite graph~$\Gamma=(S,E)$ with a function~$m\colon E\to \Z_{\geq 2}$, are not generalisable over $\phee$. 

\item[(b)] The standard presentation of \emph{Thompson's group~$V$} given in \cite[Lemma 6.1]{CannonFloydParry} has one generator of order $2$ and hence cannot be generalised over $\phee$. 
\end{enumerate}
\end{ex}

A given finitely generated group might admit a generalisable and a non-generalisable presentations, as shown in the following example or in~\cite[Remark, p.~640]{bami}.
\begin{ex}\label{ex:diffpres}
    Let~$G$ be the group given by the balanced presentation
    \begin{equation}\label{eq:diff1}
        \langle a,b,c \mid a^2 b^{-2} c a b^{-2} c a^{-1}\rangle.
    \end{equation}
    Note that~\eqref{eq:diff1} has the form~$\langle a,b,c\mid r(a,b,c)^2\rangle$, where~$r(a,b,c):=a^2 b^{-2} c a^{-1}$. However, it is easily checked that~$G$ is also isomorphic to the group given by the following non-balanced presentation:
    \begin{equation}\label{eq:diff2}
        \langle a, b, c, w \mid w^2,\, a^2 b^{-2} c a^{-1} w^{-1}\rangle.
    \end{equation}
    Let~$\varphi\colon \caO\hookrightarrow U$ be any continuous open monomorphism of topological groups~$\caO\leq_o U$.
    By~\Cref{thm:retract}\eqref{thm:retractb}, the presentation in~\eqref{eq:diff1} is generalisable over~$\varphi$, while by~\Cref{prop:morethanCoxeter} the one in~\eqref{eq:diff2} is \emph{not} generalisable over~$\varphi$ as long as~$\caO=U\neq \varphi(\caO)$.
\end{ex}


\subsection{Compactness properties and generalised presentations}\label{sus:FPgenpres} 

Let~$\langle X\mid R\rangle$ be a group presentation with~$X$ finite and that can be generalised over a continuous open monomorphism~$\varphi\colon \caO\hookrightarrow U$ of LC groups~$\caO\leq_o U$.
Below we collect some general results on the compactness properties for the LC group~$G_\varphi=\caH_X(\varphi)/\LL R\GG$. 

\smallskip

Firstly, the following lemma allows us to focus only on the case when~$U$ is~TLDC.
\begin{lem}\label{lem:FPpropTDLC}
    Let~$\langle X\mid R\rangle$ and~$\varphi\colon \caO\hookrightarrow U$ be as at the beginning of the section, and let~$\wtil{\varphi}$ be as in Proposition~\ref{prop:connComp}. Then, for every~$n\geq 1$ and every unital ring~$R$,
    \begin{enumerate}
        \item $G_\varphi$ is of type~$\CP_n(R)$ if and only if~$G_{\wtil{\varphi}}$ is of type~$\FP_n(R)$;
        \item $G_\varphi$ is of type~$\tC_n$ if and only if~$G_{\wtil{\varphi}}$ is of type~$\tF_n$.
    \end{enumerate}
\end{lem}
\begin{proof}
    Combine~\Cref{prop:connCompGP} with~\Cref{fact:LCvsTDLC}.
\end{proof}

Another possible simplification is the following.
\begin{lem}\label{lem:FPpropSQcup}
Let~$\langle X\mid R\rangle$ and~$\varphi\colon \caO\hookrightarrow U$ be as at the beginning of the section. Assume there exists a partition~$X=X_1\sqcup \cdots \sqcup X_n$,~$n\geq 1$, that induces a partition of relators~$R=R_1\sqcup \cdots \sqcup R_n$ by taking, for every~$1\leq i\leq n$, $R_i$ as the set of all words of~$R$ formed only by letters in~$X_i$. Let~$G:=\caH_X(\varphi)/\LL R\GG$ and, for every~$1\leq i\leq n$, let~$G_i:=\caH_{X_i}(\varphi)/\LL R_i\GG$.
    Then, for every~$n\geq 1$ and every unital ring~$R$, the following hold:
      \begin{enumerate}
        \item\label{lem:FPpropSQ} Let~$U$ and~$\caO$ be TDLC groups of type~$\FP_n(R)$ and~$\FP_{n-1}(R)$, respectively. Then $G$ is of type~$\FP_n(R)$ if and only if, for every~$1\leq i\leq n$,~$G_i$ is of type~$\FP_n(R)$;
        \item\label{lem:FpropSQ} Let~$U$ and~$\caO$ be TDLC groups of type~$\tF_n$ and~$\tF_{n-1}$, respectively. Then~$G$ is of type~$\tF_n$ if and only if, for every~$1\leq i\leq n$,~$G_i$ is of type~$\tF_n$.
    \end{enumerate}
\end{lem}

\begin{proof}
    By \Cref{prop:XRsplit}, we have~$G\cong \bigast_U G_i$ as TDLC~groups. Hence, both~\eqref{lem:FPpropSQ} and~\eqref{lem:FpropSQ} follow from~\Cref{prop:graphofgrpsFP}. 
\end{proof}

\begin{rem}
    Note that~\Cref{lem:FPpropTDLC} allows us, for instance, to reduce the study of the finiteness properties of TDLC~Artin groups to the case when the defining graph is connected; see~\Cref{ex:ArtinSplit}.
\end{rem}

\begin{lem}\label{lem:FPpropRetr}
    Let~$\varphi\colon \caO\hookrightarrow U$ be as at the beginning of the section, assume~$U$ is~TDLC and let~$\langle X\mid R\rangle$ be a presentation with $X$ finite and with balanced relators.
   Then, for every~$n\geq 1$ and every unital ring~$R$,
   \begin{enumerate}
       \item if~$G_\varphi$ is of type~$\FP_n(R)$ and~$\caO$ is of type~$\FP_{n-1}(R)$, then~$U$ is of type~$\FP_n(R)$;
       \item if~$G_\varphi$ is of type~$\tF_n$ and~$\caO$ is of type~$\tF_{n-1}$, then~$U$ is of type~$\tF_n$. 
   \end{enumerate}
\end{lem}
\begin{proof}
   Without loss of generality let~$X\neq \emptyset$. Let~$t\in X$.
 By \Cref{thm:retract}\eqref{thm:retractd}, $\caH_{\{t\}}(\varphi)$ is a group retract of~$G_\varphi$. By~\cite[Corollary~5.6]{ccc}, if~$G_\varphi$ is of type~$\FP_n(R)$ (resp.~$\tF_n$), then so is~$\caH_{\{t\}}(\varphi)$. Now the claims follow from \Cref{prop:graphofgrpsFP}.
\end{proof}

\section{Topological right-angled Artin groups}\label{s:topRAAGs} 
For the section, $\Gamma=(S,E)$ will typically be a finite (as usual: simplicial, unlabelled) graph with vertex set~$S \neq \leer$ and edge set~$E\subseteq\big\{\{s,t\}\subseteq S\mid s\neq t\big\}$.

Presentations of right-angled Artin groups have balanced relators (in the sense of~\Cref{defn:balanced}) and hence, by \Cref{thm:retract}\eqref{thm:retractb}, they can be generalised over any continuous open monomorphism~$\varphi\colon \caO\hookrightarrow U$ of topological groups with~$\caO\leq_o U$. 
This motivates the following definition. 

\begin{defn}[{Topological RAAGs}]\label{defn:topRAAG}
    Let~$\Gamma=(S,E)$ be a finite graph and let~$\varphi\colon \caO\hookrightarrow U$ be a continuous open monomorphism of topological (Hausdorff, by default) groups~$\caO\leq_o U$. Write 
    \begin{align*}
    R_\Gamma:=\big\{sts^{-1}t^{-1}\mid \{s,t\}\in E\big\} \subseteq F(S).
    \end{align*} 
    The \emph{topological right-angled Artin group} (or, for short, \emph{topological RAAG}) \emph{of type~$(\Gamma,\varphi)$} is the group 
    \begin{equation}
        \caA_\Gamma(\varphi):=\caH_S(\varphi)/\LL R_\Gamma\GG_{\caH_S(\varphi)}
    \end{equation}
    associated to the standard presentation of the right-angled Artin group $A_\Gamma$ (see \Cref{defn:genpres}), 
    endowed with the quotient topology induced from the group topology on~$\caH_S(\varphi)$ as in \Cref{conv:topHNNext}.
\end{defn}

All the properties discussed in Section~\ref{s:genpres} clearly apply to~$\caA_\Gamma(\varphi)$.
\begin{ex}\label{ex:easyRAAG}
Below are some extremal examples of topological RAAGs.
  \begin{itemize}
      \item[(i)] \label{ex:easyRAAG1} Let~$\caO=U=\{1\}$ and~$\varphi=\mathrm{id}_{\{1\}}$. By Example~\ref{ex:easyHNN}, $\caH_S(\varphi)$ is canonically isomorphic to~$F(S)$ with the discrete topology. Hence, 
      \begin{equation*}
          \caA_\Gamma(\varphi)=\caH_S(\varphi)/\LL R_\Gamma\GG_{\caH_S(\varphi)}\cong F(S)/\LL R_\Gamma\GG_{F(S)}=A_\Gamma,
      \end{equation*}
     where~$A_\Gamma$ is the discrete RAAG of type~$\Gamma$.
      \item[(ii)] \label{ex:easyRAAG2} Let~$E=\emptyset$, i.e., let~$\Gamma$ be totally disconnected. Then
     \begin{equation*}
         \caA_\Gamma(\varphi)=\caH_S(\varphi)/\{1\}=\caH_S(\varphi).
     \end{equation*}
    \end{itemize}
\end{ex}

\subsection{A normal form for topological RAAGs if~$\varphi(\caO)=\caO$}\label{sus:NF} 
As is well known~\cite{heme:nf}, words in abstract RAAGs admit a normal form. We establish a similar result for topological RAAGs in the case where $\varphi(\caO)=\caO$. This will be used in Section~\ref{s:Salvetti}, but may also be of independent interest. We start with a useful observation.

\begin{rem} \label{obs:Oconj}
For subsequent arguments, when~$\phee(\caO) = \caO$ we note that in~$\caA_\Gamma(\varphi)$ one has 
\begin{equation*} 
    a\caO a^{-1}=\caO \qquad\forall\,a\in A_\Gamma\subseteq \caA_\Gamma(\varphi).
\end{equation*}
This is because $A_\Gamma$ is generated by $S$ and, for every $s\in S$, one has $s\caO s^{-1}=\varphi(\caO)=\caO$ and hence $s^{\pm 1}\caO s^{\mp 1}=\caO$; see~\Cref{cor:Ut^U} and \Cref{lem:UwU}\eqref{lem:UwU3}. 
\end{rem}

To simplify notation, we denote by~$N = \LL U \RR_{\caA_\Gamma(\phee)}$ the normal closure of~$U$ in~$\caA_\Gamma(\varphi)$. 
Let also~$\caR$ be any set of representatives of the right cosets of~$U$ modulo~$\caO$ such that~$1\in \caR$.
\begin{defn}\label{defn:NF}
   Let~$\varphi(\caO)=\caO$.
   The set of \emph{normal sequences}~$\caW=\caW_\Gamma(\varphi)$ is the collection of all sequences of the form
\begin{equation}\label{eq:sigma}
    \sigma=(u_0,a_1,u_1,\ldots, a_n,u_n)
\end{equation}
for some~$n\geq 0$, $u_0,\ldots, u_n\in U$, $a_1,\ldots, a_n\in A_\Gamma$ such that, whenever~$n\geq 1$,
\begin{itemize}
    \item[(i)] $u_i\in \caR\setminus\{1\}$ for all~$1\leq i\leq n-1$ and~$u_n\in\caR$; and
    \item[(ii)] $a_1,\ldots, a_n\in A_{\Gamma}\setminus\{1\}$.
\end{itemize}
The word~$g_\sigma=u_0a_1u_1\cdots a_nu_n$ represents an element of~$\caA_\Gamma(\varphi)$ that, if there is no ambiguity, we will denote again by~$g_\sigma$.

A \emph{normal form} for~$g\in \caA_\Gamma(\varphi)$ is any word $g_\sigma$ such that $g=g_\sigma$ in~$\caA_\Gamma(\varphi)$ and $\sigma=(u_0,a_1,u_1,\ldots, a_n,u_n)\in \caW$. We refer to the integer $n$ as the \emph{length of a normal form}~$g_\sigma$ of~$g$.
\end{defn}

\smallskip

The main goal of this section is to prove that, for every $g\in \caA_\Gamma(\varphi)$, there is a \emph{unique} $\sigma\in \caW$ for which $g=g_\sigma$; see~\Cref{thm:NF} below. In view of this, we introduce an $\caA_\Gamma(\varphi)$-action on~$\caW$ such that, for every~$\sigma\in \caW$, $g_\sigma\cdot (1)=\sigma$. Our strategy is inspired by the one adopted in~\cite{bogo} to prove that HNN-extensions have a normal form~\cite[Theorem~14.3]{bogo}.

\medskip

Consider an arbitrary~$\sigma=(u_0,a_1,u_1,\ldots, a_n,u_n)\in \caW$ and let~$u_0=\omega\check{u}_0$ for some~$\omega\in \caO$ and~$\check{u}_0\in \caR$. For all $u\in U$ and~$t\in S$, define
\begin{equation}\label{eq:preact}
   \begin{array}{rl}
       u\cdot \sigma & =(uu_0,a_1,u_1,\ldots, a_n,u_n) \\
       t^{\pm 1}\cdot \sigma & =\left\{
       \hspace{-0.2cm}\begin{array}{lll}
       (\varphi^{\pm 1}(\omega), t^{\pm 1},\check{u}_0,a_1,u_1,\ldots, a_n,u_n), & \text{if }n=0\text{ or }\check{u}_0\neq 1 & (1)\\
       (\varphi^{\pm 1}(u_0)u_1,a_2,u_2,\ldots, a_n,u_n), & \text{if }\check{u}_0=1,\,a_1=t^{\mp 1} & (2)\\
       (\varphi^{\pm 1}(u_0),t^{\pm 1}a_1,u_1,\ldots, a_n,u_n), & \text{if }\check{u}_0=1,\,a_1\neq t^{\mp 1}.  & (3) 
       \end{array}
       \right.
   \end{array} 
\end{equation}
Conventionally,~$\varphi^{+1}=\varphi$.
In cases~(2) and~(3), notice that~$u_0=\omega$.
Note also that~$u\cdot \sigma, t^{\pm 1}\cdot \sigma\in \caW$.

\begin{rem}
The idea behind the definition in~\eqref{eq:preact} is to mimic how left multiplication by~$u$ or by~$t^{\pm 1}$ behaves in $\caA_\Gamma(\varphi)$. More precisely, for every $\sigma=(u_0,a_1,u_1,\ldots, a_n,u_n)\in \caW$ one has 
\begin{equation}\label{eq:words}
    ug_\sigma=uu_0a_1u_1\cdots a_nu_n\quad\text{and}\quad t^{\pm 1}g_\sigma=t^{\pm 1}u_0a_1u_1\cdots a_nu_n.
\end{equation}
Clearly, $ug_\sigma=g_{\tau}$ where~$\tau=u\cdot \sigma\in \caW$.
Moreover, upon applying the relations 
\[t^{\pm 1}\omega=\varphi^{\pm 1}(\omega)t^{\pm 1} \qquad\forall\,\omega\in \caO\]
to the very left-hand side of the word $t^{\pm 1}g_\sigma$ in~\eqref{eq:words} and simplifying $t^{\pm 1}t^{\mp 1}$ if needed, one deduces that $t^{\pm 1}g_\sigma=g_{\tau}$ where $\tau=t^{\pm 1}\cdot \sigma\in\caW$.
\end{rem}

\begin{lem}\label{lem:HNNact}
    Let~$\varphi(\caO)=\caO$. Then the assignments in~\eqref{eq:preact} induce a $\caH_S(\varphi)$-action on~$\caW$.
\end{lem}
\begin{proof}
Recall that~$\caH_S(\varphi)=(U\ast F(S))/\LL tk t^{-1}\varphi(k)^{-1}\mid t\in S,\,k\in \caO\GG$.
If, for all $\sigma\in \caW$ and~$t\in S$, one proves that
\begin{equation}\label{eq:t+-inv}
    t^{\mp 1}\cdot (t^{\pm 1}\cdot \sigma)=\sigma,
\end{equation}
then~\eqref{eq:preact} induces a $(U\ast F(S))$-action on~$\caW$. If in addition, for all~$\sigma\in\caW$, $t\in S$ and~$k\in \caO$, one proves that
\begin{equation}\label{eq:tuphi}
    tk\cdot \sigma=\varphi(k)t\cdot \sigma,
\end{equation}
then the relevant $(U\ast F(S))$-action on~$\caW$ induces a~$\caH_S(\varphi)$-action on~$\caW$ by passing to the quotient.

Let~$\sigma=(u_0,a_1,u_1,\ldots, a_n,u_n)\in \caW$ and let~$u_0=\omega\check{u}_0$ for some~$\omega\in \caO$ and~$\check{u}_0\in \caR$. We first prove~\eqref{eq:t+-inv}, referring to the cases~(1)-(3) as in~\eqref{eq:preact}. 
If~$\sigma$ is as in case~$(1)$, then
\begin{equation*}
    \begin{split}
        t^{\mp 1}\cdot (t^{\pm 1}\cdot \sigma) & = t^{\mp 1}\cdot (\varphi^{\pm 1}(\omega), t^{\pm 1}, \check{u}_0,a_1,u_1,\ldots, a_n, u_n)\\
        & \stackrel{(2)}{=} (\omega\check{u}_0,a_1,u_1,\ldots, a_n, u_n)=\sigma.
    \end{split}
\end{equation*}
If~$\sigma$ is as in case~$(2)$, then
\begin{equation*}
     \begin{split}
        t^{\mp 1}\cdot (t^{\pm 1}\cdot\sigma) & = t^{\mp 1}\cdot (\varphi^{\pm 1}(u_0)u_1, a_2,u_2,\ldots, a_n, u_n)\\
        & \stackrel{(1)}{=} (u_0,t^{\mp 1},u_1,\ldots, a_n, u_n)=\sigma.
    \end{split}
\end{equation*}
Finally, if~$\sigma$ is as in case~$(3)$, then
\begin{equation*}
     \begin{split}
        t^{\mp 1}\cdot (t^{\pm 1}\cdot\sigma) & = t^{\mp 1}\cdot (\varphi^{\pm 1}(u_0), t^{\pm 1}a_1,u_1,\ldots, a_n, u_n)\\
        & \stackrel{(3)}{=} (u_0,t^{\mp 1}t^{\pm 1}a_1,u_1,\ldots, a_n, u_n)=\sigma.
    \end{split}
\end{equation*}

We now prove~\eqref{eq:tuphi}. Note that
\begin{equation}\label{eq:tuphi1}
    t\cdot (k\cdot \sigma)=\left\{
    \begin{array}{ll}
      (\varphi(k\omega),t,\check{u}_0,a_1,u_1,\ldots, a_n,u_n), & \text{if }n=0\text{ or }\check{u}_0\neq 1;\\
      (\varphi(k u_0)u_1,a_2,u_2,\ldots, a_n,u_n), & \text{if }\check{u}_0=1,\,a_1=t^{-1};\\
      (\varphi(k u_0), ta_1,u_1,\ldots, a_n, u_n), & \text{if }\check{u}_0=1,\,a_1\neq t^{-1}.
    \end{array}
    \right.
\end{equation}
Since $\varphi$ is multiplicative, by comparing~\eqref{eq:tuphi1} with~\eqref{eq:preact} we conclude that~$t\cdot (k\cdot \sigma)=\varphi(k)\cdot (t\cdot \sigma)$.
\end{proof}

\begin{lem}\label{lem:Aact}
    Assume that~$\varphi(\caO)=\caO$. 
    Then~$\LL [s,t]\mid \{s,t\}\in E\GG\unlhd \caH_S(\varphi)$ is contained in the kernel of the $\caH_S(\varphi)$-action on~$\caW$ as in Lemma~\ref{lem:HNNact}. In particular, the latter $\caH_S(\varphi)$-action induces an $\caA_\Gamma(\varphi)$-action on~$\caW$ by passing to the quotient, and~$g_\sigma\cdot (1)=\sigma$ for every~$\sigma\in \caW$.
\end{lem}
\begin{proof} 
   The second part of the statement follows from the first part.
    Hence, it suffices to prove that
    \begin{equation}\label{eq:stact}
        (st)\cdot \sigma=(ts)\cdot \sigma,\quad\forall\{s,t\}\in E,\,\sigma\in \caW.
    \end{equation}

    Let~$\sigma=(u_0,a_1,u_1,\ldots, a_n,u_n)\in \caW$, where $u_0=\omega\check{u}_0$ for some~$\omega\in \caO$ and~$\check{u}_0\in \caR$. Recall that~$[s,t]=1$ in~$A_\Gamma$ and that $u_1\in \caR$, with $u_1=1$ if and only if~$n=1$. Given these two observations, one checks that
    \begin{equation}\label{eq:stact1}
    s\cdot (t\cdot \sigma)=\left\{
    \begin{array}{ll}
      (\varphi^2(\omega),st,\check{u}_0,a_1,u_1,\ldots, a_n,u_n), & \text{if }n=0\text{ or }\check{u}_0\neq 1 
      \\
      & \text{(apply (1) and then (3));}\\
      (\varphi^2(u_0), s, u_1,a_2,u_2,\ldots, a_n,u_n), & \text{if }\check{u}_0=1,\,a_1=t^{-1} 
      \\
      & \text{(apply (2) and then~(1));}\\
       (\varphi^2(u_0), t, u_1,a_2,u_2,\ldots, a_n,u_n), & \text{if }\check{u}_0=1,\,a_1=s^{-1} 
       \\
       & \text{(apply (3) and then (3))};\\
      (\varphi^2(u_0)u_1,a_2,u_2,\ldots, a_n,u_n), & \text{if }\check{u}_0=1,\,a_1=s^{-1}t^{-1}
      \\
      & \text{(apply (3) and then (2));}\\
       (\varphi^2(u_0),sta_1,u_1, \ldots, a_n, u_n), & \text{if }\check{u}_0=1,\,a_1\not\in\{s^{-1},t^{-1},s^{-1}t^{-1}\}\\
       & \text{(apply (3) and then (3)).}
    \end{array}
    \right. 
    \end{equation}
    Exchanging the roles of~$s$ and~$t$ in~\eqref{eq:stact1}, one has
     \begin{equation}\label{eq:stact2}
       t\cdot (s\cdot \sigma)=\left\{
    \begin{array}{ll}
      (\varphi^2(\omega),ts,\check{u}_0,a_1,u_1,\ldots, a_n,u_n), & \text{if }n=0\text{ or }\check{u}_0\neq 1;\\
      (\varphi^2(u_0), t, u_1,a_2,u_2,\ldots, a_n,u_n), & \text{if }\check{u}_0=1,\,a_1=s^{-1};\\
      (\varphi^2(u_0), s, u_1,a_2,u_2,\ldots, a_n,u_n), & \text{if }\check{u}_0=1,\,a_1=t^{-1};\\
      (\varphi^2(u_0)u_1, a_2,u_2,\ldots, a_n, u_n), & \text{if }\check{u}_0=1,\,a_1=t^{-1}s^{-1};\\
       (\varphi^2(u_0),tsa_1,u_1, \ldots, a_n, u_n), & \text{if }\check{u}_0=1,\,a_1\not\in\{s^{-1},t^{-1},t^{-1}s^{-1}\}.
    \end{array}
    \right. 
    \end{equation}
    Since~$st=ts$ in~$A_\Gamma$, \eqref{eq:stact} follows by comparing~\eqref{eq:stact1} with~\eqref{eq:stact2}.
\end{proof}

\begin{thm}[A normal form for~$\caA_\Gamma(\varphi)$]\label{thm:NF}
    Let~$\Gamma$ be a finite graph and assume that~$\varphi(\caO)=\caO$. Then the map
    $$\caW\longrightarrow \caA_\Gamma(\varphi), \qquad\sigma\longmapsto g_\sigma$$
    is a bijection.
\end{thm}
\begin{proof}
   Since~$g_\sigma\cdot (1)=\sigma$ for every~$\sigma\in \caW$ by~\Cref{lem:Aact}, the map is injective. It remains to show that it is also surjective. By~\Cref{prop:artin}, $\caA_\Gamma(\varphi)$ is generated by~$U\cup A_\Gamma$. Therefore, for every~$g\in \caA_\Gamma(\varphi)$ there are~$u_0,\ldots, u_n\in U$ and~$a_1,\ldots, a_n\in A_\Gamma$ ($n\geq 0$) such that
   \begin{equation}\label{eq:gsigma1}
       g=u_0a_1u_1\cdots a_nu_n.
   \end{equation}
   The goal now is to modify the right-hand side of~\eqref{eq:gsigma1} to find~$\sigma\in\caW$ such that~$g=g_\sigma$.
   We proceed by induction on the minimum integer~$n=n(g)\geq 0$ such that~$g$ admits a form as in~\eqref{eq:gsigma1}.
   
   If~$n(g)=0$, then $g\in U$ and~$g=g_{\sigma}$ for~$\sigma=(g)\in \caW$. Assume now that~$n(g)\geq 1$ and let~$g=u_0a_1u_1\cdots a_nu_n$ as in~\eqref{eq:gsigma1}.  By the minimality of~$n$, we have $a_1\neq 1$.
   Let
   \begin{equation*}
       h:=u_1a_2u_2\cdots a_nu_n\in \caA_{\Gamma}(\varphi).
   \end{equation*}
   By definition, $n(h)\leq n-1$. More precisely, since~$n=n(g)$ and~$g=u_0a_1h$, we have~$n(h)=n-1$. By induction, there is~$\tau=(v_1,b_2,v_2,\ldots, b_{n},v_{n})\in \caW$ such that~$h=h_\tau$. Hence, 
   \begin{equation}\label{eq:gsigma2}
       g=u_0a_1h=u_0a_1v_1b_2v_2\cdots b_nv_n.
   \end{equation}
   From~\eqref{eq:gsigma2}, we cannot directly conclude that $g=g_\sigma$ for some~$\sigma\in \caW$, as one might have~$v_1\not\in \caR$.
   Let~$v_1=\omega\check{v}_1$ for some~$\omega\in \caO$ and~$\check{v}_1\in \caR$. We claim that
   \begin{equation*}
       \check{v}_1\in \caR\setminus\{1\}.
   \end{equation*}
   Indeed, assume for a contradiction that~$\check{v}_1=1$ and let~$\omega':=a_1\omega a_1^{-1}\in \caO$; see \Cref{obs:Oconj}. By~\eqref{eq:gsigma2}, 
   \begin{equation*}
       g=(u_0\omega')a_1\check{v}_1b_2\cdots b_nv_n=(u_0\omega')(a_1b_2)v_2\cdots b_nv_n 
   \end{equation*}
  but~$n(g)=n$, impossible.
   
   Hence,~$\check{v}_1\in \caR\setminus\{1\}$ and by~\eqref{eq:gsigma2} we have
   \begin{equation*}
       g=(u_0\omega')a_1\check{v}_1b_2v_2\cdots b_nv_n.
   \end{equation*}
  Therefore, $g=g_\sigma$ where $\sigma=(u_0\omega',a_1,\check{v}_1, b_2,v_2,\ldots, b_n,v_n)\in \caW$.
\end{proof}

\subsection{On the conjugates of~$\caO$ and~$U$ in~$\caA_\Gamma(\varphi)$}\label{sus:conjOU}
As a follow-up of \Cref{lem:UwU}, in this section we establish some technical results on the conjugates of the subgroups~$\caO$ and~$U$ in~$\caA_\Gamma(\varphi)$. We mainly focus on two cases: when the defining graph is connected and when~$\varphi(\caO)=\caO$.
These results will be auxiliary to the proofs of Section~\ref{s:Salvetti}.

\begin{lem}\label{lem:invers}
   Let~$\Gamma=(S,E)$ be a connected graph. Then, for all~$s,t\in S$ and~$\veps\in\{\pm 1\}$, the following holds in~$\caA_\Gamma(\varphi)$:
   \begin{equation}\label{eq:invers}
       s^{\veps}\caO s^{-\veps}=t^{\veps}\caO t^{-\veps}.
   \end{equation}
\end{lem}
\begin{proof}
   Note that~\eqref{eq:invers} holds by definition if~$\veps=1$ (and both groups are equal to~$\varphi(\caO)$). It remains to prove it for~$\veps=-1$. 
   
   Since every two vertices in~$\Gamma$ can be connected by a path, it suffices to prove~\eqref{eq:invers} with~$\veps=-1$ for all~$s,t\in S$ with~$\{s,t\}\in E$. But, given~$\{s,t\}\in E$, we have~$st=ts$ in~$A_{\Gamma}\subseteq \caA_\Gamma(\varphi)$ and hence
   \begin{equation*}
       ts^{-1}\caO st^{-1}=s^{-1}(t\caO t^{-1})s=s^{-1}\varphi(\caO) s=s^{-1}(s\caO s^{-1})s=\caO. \qedhere
   \end{equation*}
\end{proof}

\medskip

Let~$e\colon A_\Gamma\to \Z$ be the exponent map of~$A_\Gamma\cong \langle S\mid R_\Gamma\rangle$; see~\Cref{defn:exponent}.

\begin{lem}\label{lem:expAGamma}
    Let~$\Gamma=(S,E)$ be a connected graph. For every~$a\in A_\Gamma\subseteq \caA_\Gamma(\varphi)$ with $e(a)\geq 0$, one has~$a\caO a^{-1}=\varphi^{e(a)}(\caO)$ in $\caA_\Gamma(\varphi)$. 
\end{lem}
\begin{proof}
    Since~$A_\Gamma=\langle S\rangle$ and~$e(a)\geq 0$, there are $s_1,\ldots, s_k\in S$ and~$\varepsilon_1,\ldots, \varepsilon_k\in \{\pm 1\}$ such that
    \begin{equation}\label{eq:proda}
        a=s_1^{\varepsilon_1}\cdots s_k^{\varepsilon_k}\quad\text{and}\quad e(a)=\sum_{i=1}^n\varepsilon_i\geq 0.
    \end{equation}

     By Lemma~\ref{lem:invers}, for all~$s,t\in S$ and $\varepsilon,\eta\in\{\pm 1\}$ we obtain that
    \begin{equation}\label{eq:expA2}
        s^\varepsilon t^\eta \caO t^{-\eta}s^{-\varepsilon}=s^{\varepsilon+\eta}\caO s^{-\varepsilon-\eta}.
    \end{equation}
    Hence, given an arbitrary~$s\in S$, by Equations~\eqref{eq:proda} and~\eqref{eq:expA2} we deduce that
    \begin{equation*}
        a\caO a^{-1}=s^{e(a)}\caO s^{-e(a)}.
    \end{equation*}
On the other hand, since $e(a)\geq 0$ and~$\varphi(\caO)=s\caO s^{-1}$, we also have
\begin{equation*}
\varphi^{e(a)}(\caO)=s^{e(a)}\caO s^{-e(a)}.\qedhere
\end{equation*}
    \end{proof}
     \begin{cor}\label{cor:expU1}
        Let~$\Gamma=(S,E)$ be a connected graph and let~$\varphi(\caO)\subseteq \caO$. Then, for all~$a,b\in A_\Gamma$ with $e(a)\leq e(b)$, one has~$a\caO a^{-1}\supseteq b\caO b^{-1}$ in $\caA_\Gamma(\varphi)$.
    \end{cor}
\begin{proof}
     Since~$e\colon A_\Gamma\to \Z$ is a group homomorphism, we have~$e(a^{-1}b)= e(b)-e(a)\geq 0$. Hence, Lemma~\ref{lem:expAGamma} yields $\caO\supseteq \varphi^{e(a^{-1}b)}(\caO)=a^{-1}b\caO b^{-1}a$.
\end{proof}

 \begin{cor}\label{cor:expU2}
    Let $\Gamma=(S,E)$ be a connected graph and assume~$\varphi(\caO)\subseteq \caO$. Then, for all~$s\in S$ and~$a\in A_\Gamma$, one has
        \begin{equation}\label{eq:conjU}
            a\caO a^{-1}=s^{e(a)}\caO s^{-e(a)}\ \text{in $\caA_\Gamma(\varphi)$}.
        \end{equation}
\end{cor}
    \begin{proof}
         Observe that~$e(a)=e(s^{e(a)})$. Hence Corollary~\ref{cor:expU1} applies. 
    \end{proof}
\begin{prop}\label{prop:normUconn}
     Let~$\Gamma=(S,E)$ be a connected graph and assume that~$\caO=U$. Then, for every~$s\in S$, 
     \begin{equation*}
         \LL U\GG_{\caA_\Gamma(\varphi)}=\bigcup_{n\in \Z}s^nUs^{-n}.
     \end{equation*}
\end{prop}
\begin{proof}
    Clearly, $N:=\LL U\GG_{\caA_\Gamma(\varphi)}\supseteq\bigcup_{n\in \Z}s^nUs^{-n}$. To prove the reverse inclusion, recall that~$N$ is generated by~$\{gug^{-1}\mid g\in \caA_\Gamma(\varphi),\,u\in U\}$. Since~$\caA_\Gamma(\varphi)$ is generated by~$U\cup A_\Gamma$, we observe that~$N$ is in fact generated by
    \begin{equation*}
        \{aua^{-1}\mid a\in A_\Gamma,\,u\in U\}.
    \end{equation*}
    In turn, by \Cref{cor:expU2}, for all~$a_1,\ldots, a_n\in A_\Gamma$ and~$u_1,\ldots, u_n\in U$, one has
    $$a_iu_ia_i^{-1}\in s^{e(a_i)}Us^{-e(a_i)},\qquad\forall\,i\geq 1.$$
   Hence, Corollary~\ref{cor:expU1} implies that
    \begin{equation*}
        (a_1u_1a_1^{-1})\cdot\ldots \cdot (a_nu_na_n^{-1})\in s^{\min_i e(a_i)}Us^{-\min_i e(a_i)}
    \end{equation*}
   and the claim follows.
\end{proof}

Finally, we collect two consequences of the normal form of~$\caA_\Gamma(\varphi)$ if~$\varphi(\caO)=\caO$ (see~\Cref{thm:NF}) that are useful in the proofs in~\Cref{s:phi(O)=O}.
\begin{lem}\label{lem:UaUa-1}
    Let~~$\varphi(\caO)=\caO$.
    Then, for all~$a\in A_\Gamma\setminus\{1\}$, one has $U\cap aUa^{-1}=\caO=\caO\cap a\caO a^{-1}$.
\end{lem}
\begin{proof}
By \Cref{obs:Oconj}, $\caO=a\caO a^{-1}$. Hence, $\caO=\caO\cap a\caO a^{-1}$ and~$\caO$ is contained in~$U\cap aUa^{-1}$.
The reverse inclusion holds if~$e(a)\neq 0$ by \Cref{lem:UwU}. Below, we give an alternative proof of it that includes also the case when~$a\neq 1$ but~$e(a)=0$. 

Let~$u\in U\setminus\caO$, say $u=\omega\check{u}$ for some~$\omega\in \caO$ and~$\check{u}\in \caR$. Since~$u\not\in \caO$, we have~$\check{u}\neq 1$. Let~$\omega':=a\omega a^{-1}\in \caO$ (see~\Cref{obs:Oconj}) and observe that
    \begin{equation}\label{eq:aua-1}
        aua^{-1}=\omega' a\check{u}a^{-1}.
    \end{equation}
Since~$a\neq 1$ and~$\check{u}\neq 1$, the right-hand side of~\eqref{eq:aua-1} gives the normal form for~$aua^{-1}$ by the sequence~$\sigma=(\omega',a,\check{u},a^{-1},1)\in\caW$ (see~\Cref{thm:NF}). Notice that~$\sigma$ has length~$2$. Hence~$g=aua^{-1}\not\in U$, as the normal form of every element of~$U$ has associated sequence of length~$0$. We conclude that if~$aua^{-1}\in U$ then $u\in \caO$ and hence~$aua^{-1}\in a\caO a^{-1}=\caO$. This proves that~$U\cap aUa^{-1}\subseteq \caO$.
\end{proof}

\begin{lem}\label{lem:prechor}
    Let~$\varphi(\caO)=\caO$.
    Then, for all~$u_1,u_2\in U\setminus \caO$ and~$a_1,a_2\in A_\Gamma$, one has
    \begin{equation*}
        (a_1u_1a_1^{-1})(a_2u_2a_2^{-1})\in \bigcup_{b\in A_\Gamma}bUb^{-1}\,\Longleftrightarrow\, a_1=a_2.
    \end{equation*}
\end{lem}
\begin{proof}
   The implication~($\Leftarrow$) is clear. Assume now that $a_1\neq a_2$ and let
   \begin{equation}\label{eq:prech0}
       x:=(a_1u_1a_1^{-1})(a_2u_2a_2^{-1}).
   \end{equation}
   
   Recall that~$\caO=a_1\caO a_1^{-1}=a_2\caO a_2^{-1}$; see \Cref{obs:Oconj}. We claim that there are~$\omega\in \caO$ and~$\check{u}_1,\check{u}_2\in \caR\setminus\{1\}$ such that
   \begin{equation}\label{eq:prech1}
       x=\omega a_1\check{u}_1 (a_1^{-1}a_2) \check{u}_2 a_2^{-1}.
   \end{equation}
   In detail, let~$u_2=h\check{u}_2$ for some~$h\in \caO$ and~$\check{u}_2\in \caR$. Since~$u_2\in U\setminus \caO$, we have $\check{u}_2\neq 1$. Hence, recalling \Cref{obs:Oconj},
   \begin{equation}\label{eq:prech11}
       a_1^{-1}a_2u_2=ka_1^{-1}a_2\check{u}_2,\qquad\text{for }k:=a_1^{-1}a_2ha_2^{-1}a_1\in \caO.
   \end{equation}
   In turn, since~$u_1\in U\setminus \caO$, we have~$u_1k\in U\setminus \caO$. Hence, there are~$l\in \caO$ and~$\check{u}_1\in \caR\setminus\{1\}$ such that~$u_1k=l\check{u}_1$. Let~$\omega:=a_1la_1^{-1}\in \caO$ (see~\Cref{obs:Oconj}) and observe that
   \begin{equation}\label{eq:prech12}
       a_1u_1k=a_1l\check{u}_1=\omega a_1\check{u}_1.
   \end{equation}
   Therefore,
   \begin{equation*}
       x=a_1u_1(a_1^{-1}a_2u_2)a_2^{-1}\stackrel{\eqref{eq:prech11}}{=}a_1(u_1k)(a_1^{-1}a_2)\check{u}_2a_2^{-1}\stackrel{\eqref{eq:prech12}}{=}\omega a_1\check{u}_1(a_1^{-1}a_2)\check{u}_2a_2^{-1}
   \end{equation*}
   and~\eqref{eq:prech1} is proved.
   
  By~\eqref{eq:prech1}, we deduce~$x=g_\sigma$ for some~$\sigma\in \caW$ whose length is either~$2$ or~$3$. In particular, since~$a_1\neq a_2$, the length of~$\sigma$ is~$2$ if and only if either $a_1=1$ or~$a_2=1$.
   
   If by contradiction~$x\in bUb^{-1}$ for some~$b\in A_\Gamma$, then there are~$\omega\in b\caO b^{-1}=\caO$ (see~\Cref{obs:Oconj}) and~$\check{u}\in \caR$ such that
   \begin{equation*}
       x=\omega b\check{u}b^{-1}.
   \end{equation*}
   Then~$x=g_\tau$ for some~$\tau\in \caW$ whose length is at most~$2$. In particular, the length of~$\tau$ is~$2$ if and only if $b\neq 1$ and~$\check{u}\in \caR\setminus\{1\}$.  

   By Theorem~\ref{thm:NF}, we have $\sigma=\tau$. In particular, both~$\sigma$ and~$\tau$ have length~$2$. In that case, $\tau=(\omega, b,\check{u}, b^{-1},1)$ and both~$b\neq 1$ and~$\check{u}\neq 1$. In particular, both the first and the latter entries of~$\tau$ belong to~$\caO$. Moreover, since~$\sigma$ has length~$2$ and~$a_1\neq a_2$, either $\sigma=(\omega\check{u}_1,a_2,\check{u}_2, a_2^{-1},1)$ (if $a_1=1$) or $\sigma=(\omega, a_1, \check{u}_1,a_1^{-1}, \check{u}_2)$ (if $a_2=1$). In the first case, we have $\omega\check{u}_1\in U\setminus\caO$ because $\omega\in \caO$ and~$\check{u}_1\in U\setminus \caO$. In the second case, we have $\check{u}_2\in U\setminus \caO$. In any case $\sigma\neq \tau$, a contradiction. This proves the implication~$(\Rightarrow)$.
\end{proof}

\subsection{On standard parabolic subgroups of topological RAAGs}\label{sus:parab} 

In this section, we initiate the study of analogues of parabolic subgroups for topological RAAGs; see \Cref{sus:basicRAAGs} for the abstract case. Our study is far from being complete; below we only establish some partial results in view of \Cref{sus:FPRAAGs}.

\begin{prop}\label{prop:parab}
    Let~$\Gamma=(S,E)$ be a connected graph and let~$\caO=U$. Then, for every connected induced subgraph~$\Lambda=(T,F)$ of~$\Gamma$, the group homomorphism
    \begin{equation*}
        \iota=\iota_{T,\Gamma,\varphi}\colon \caA_\Lambda(\varphi)\longrightarrow \caA_\Gamma(\varphi)
    \end{equation*}
    induced by the identity map~$\iid_U\colon U\to U$ and the canonical inclusion~$T\hookrightarrow S$ (see~\Cref{prop:univ}) is a continuous open monomorphism. In particular, 
    $\caA_\Lambda(\varphi)$ is topologically isomorphic to its image via~$\iota$, which is the open subgroup~$\LL U\GG_{\caA_\Gamma(\varphi)}\cdot A_\Lambda$.
\end{prop}
\begin{proof}
  Let~$N_1=\LL U\GG_{\caA_\Lambda(\varphi)}$ and~$N_2=\LL U\GG_{\caA_\Gamma(\varphi)}$. Since both~$\Lambda$ and~$\Gamma$ are connected, by Proposition~\ref{prop:normUconn} for every~$t\in T$ we have
  \begin{equation}\label{eq:parab0}
      N_1=\bigcup_{n\in \Z}t^nUt^{-n}\leq \caA_\Lambda(\varphi)\quad\text{and}\quad N_2=\bigcup_{n\in \Z}t^nUt^{-n}\leq \caA_\Gamma(\varphi).
  \end{equation}
 Hence, the homomorphism mapping~$N_1$ identically onto $N_2$, i.e.,
 \begin{equation*}
     \iota\vert_{N_1}\colon N_1\longrightarrow N_2
 \end{equation*}
 is a continuous open isomorphism. Being~$N_1$ an open subgroup of~$\caA_\Lambda(\varphi)$, one deduces that $\iota$ is a continuous open homomorphism.

  It remains to prove that~$\iota$ is injective. By Proposition~\ref{prop:artin} we have
  \begin{equation}\label{eq:parab1}
      \caA_\Lambda(\varphi)=N_1\cdot A_\Lambda\quad\text{and}\quad N_1\cap A_\Lambda=\{1\}
  \end{equation}
  as well as
 \begin{equation}\label{eq:parab2}
      \caA_\Gamma(\varphi)=N_2\cdot A_\Gamma\quad\text{and}\quad N_2\cap A_\Gamma=\{1\}.
  \end{equation}
  Hence, 
  \begin{equation*}
    \iota(\caA_\Lambda(\varphi))=\iota(N_1)\iota(A_\Lambda)=N_2\cdot A_\Lambda.
  \end{equation*}
 Since both~$\iota\vert_{N_1}$ and~$\iota\vert_{A_\Lambda}$ are monomorphisms,~\eqref{eq:parab1} and~\eqref{eq:parab2} yield that~$\iota$.
\end{proof}

For the next results, it might be useful to recall the notation~$[T,V]$, for~$T,V\subseteq S$ introduced in~\eqref{eq:[T,V]}.
\begin{cor}\label{cor:retrparab}
    Let~$\Gamma=(S,E)$ be a finite connected graph and assume~$\caO=U$. Let~$\Lambda=(T,F)$ be any connected induced subgraph of~$\Gamma$ such that~$[T,S\setminus T]=\{1\}$. Consider the continuous group homomorphism
    \begin{equation*}
       p=p_{T,\Gamma,\varphi}\colon \caA_\Gamma(\varphi)\longrightarrow \caA_\Lambda(\varphi)
    \end{equation*}
    induced by the identity~$\iid_U\colon U\to U$ and the canonical retraction~$\pi_T\colon A_\Gamma\to A_\Lambda$ (see~\eqref{eq:canRetr} and~\Cref{prop:univ}). Then
\begin{equation*}
p_{T,\Gamma,\varphi}\circ\iota_{T,\Gamma,\varphi}=\iid_{\caA_\Lambda(\varphi)}.
\end{equation*}
\end{cor}
\begin{proof}
Notice that
   \begin{equation*}
       p\vert_{A_\Lambda}=\pi_T\vert_{A_\Lambda}=\iid_{A_\Lambda}.
   \end{equation*}
Let~$t\in T$. For every~$u\in U$ notice that
    \begin{equation*}
        p(tut^{-1})=\pi_T(t)u\pi_T(t)^{-1}=tut^{-1}. 
    \end{equation*}
Hence, by~\eqref{eq:parab0}, 
\begin{equation*}
    p(\iota_{T,\Gamma,\varphi}(n))=n\qquad\forall\,n\in \LL U\GG_{\caA_\Lambda(\varphi)}.
\end{equation*}
The decomposition of~$\caA_\Lambda(\varphi)$ as in~\eqref{eq:parab1} now yields the claim.
\end{proof}

\begin{lem}\label{lem:ontoRAAG}
    Let~$\Gamma=(S,E)$ be the join of two induced subgraphs~$\Lambda_1=(S_1,E_1)$ and~$\Lambda_2=(S_2,E_2)$, i.e., $S=S_1\sqcup S_2$ and $[S_1,S_2]=\{1\}$. 
    Consider the continuous group homomorphism
    \begin{equation*}
        \eta=\eta_{S_1,S_2}\colon \caA_\Gamma(\varphi)\longrightarrow A_{\Lambda_2}
    \end{equation*}
    induced (see~\Cref{prop:univ}) by the constant map~$U\to A_{\Lambda_1},\,u\mapsto 1$ and the canonical retraction~$\pi_{S_1}\colon A_\Gamma=A_{\Lambda_1}\times A_{\Lambda_2}\to A_{\Lambda_1}$ onto~$A_{\Lambda_1}$ (see~\eqref{eq:canRetr}). Then~$\eta$ is surjective and 
    \begin{equation*}
        \ker(\eta)=\LL U\GG_{\caA_\Gamma(\varphi)}\cdot A_{\Lambda_2}.
    \end{equation*}
\end{lem}
\begin{proof}
   Let~$N=\LL U\GG_{\caA_\Gamma(\varphi)}$. By Proposition~\ref{prop:artin}, one has
   \begin{equation*}
       \caA_\Gamma(\varphi)=N\cdot A_\Gamma\quad\text{and}\quad N\cap A_\Gamma=\{1\}.
   \end{equation*}
  Notice also that
  \begin{equation*}
  N\subseteq \ker(\eta)\quad \text{and}\quad
  \eta\vert_{A_\Gamma}=\pi_{S_1}.
  \end{equation*}
  In particular, $\eta$ is surjective. Moreover, for every~$g\in \caA_\Gamma(\varphi)$, say~$g=na$ for some~$n\in N$ and~$a\in A_\Gamma$, observe that
  \begin{equation*}
      \eta(g)=\eta(n)\eta(a)=\pi_{S_1}(a).
  \end{equation*}
  Therefore, $\ker(\eta)=N\cdot \ker(\pi_{S_1})=N\cdot A_{\Lambda_2}$.
\end{proof}

\begin{prop}[A normal series of standard parabolic subgroups]\label{prop:chainTopParab}
    Let~$\Gamma=(S,E)$ be a finite connected graph and assume~$\caO=U$. Consider any strictly increasing chain of sets~$T_1\subset \ldots \subset T_n=S$, with~$n\geq 2$, satisfying
    \begin{equation*}
        [T_{i-1},T_{i}\setminus T_{i-1}]=\{1\},\qquad\forall\,i\geq 2.
    \end{equation*}
    For every~$i\geq 1$, let~$\Lambda_i:=\ind_\Gamma(T_i)$ and assume that~$\Lambda_1$ is connected. Then the continuous monomorphisms 
    $$\iota_i\colon \caA_{\Lambda_{i-1}}(\varphi)\hookrightarrow \caA_{\Lambda_i}(\varphi),\qquad\forall\,i\geq 1$$
    as in Proposition~\ref{prop:parab} induce a strictly increasing chain of open subgroups of~$\caA_\Gamma(\varphi)$
    \begin{equation*}
        \caA_{\Lambda_1}(\varphi) \leq \caA_{\Lambda_2}(\varphi) \leq \ldots \leq \caA_{\Lambda_n}(\varphi)=\caA_\Gamma(\varphi)
    \end{equation*}
   such that, for every~$i\geq 2$, $\caA_{\Lambda_{i-1}}(\varphi)$ is normal in~$\caA_{\Lambda_i}(\varphi)$ and
    \begin{equation*}
        \caA_{\Lambda_i}(\varphi)/\caA_{\Lambda_{i-1}}(\varphi)\cong A_{\ind_\Gamma(T_i\setminus T_{i-1})}.
    \end{equation*}
\end{prop}

The proof of \Cref{prop:chainTopParab} requires the following  
general 
fact about abstract RAAGs. We provide a proof for completeness. 
\begin{lem}\label{lem:parabker}

    For~$T\subseteq S$, let 
    $\Lambda:=\ind_\Gamma(T)$ and $\Xi:=\ind_\Gamma(S\setminus T)$. 
    \begin{enumerate}
        \item \label{lem:parabker1} The subgroup~$\langle T\rangle\leq A_\Gamma$ is normal if and only if 
        $\Gamma=\Lambda\vee \Xi$ 
        or, equivalently, if and only if $T=\leer$ or $[T,S\setminus T]=\{1\}$.
        \item \label{lem:parabker2}If~$\langle T\rangle\leq A_\Gamma$ is normal, then the inclusions 
        $A_{\Lambda}\into A_\Gamma$ and~$A_{\Xi}\into A_\Gamma$
        induce a group isomorphism 
        $A_{\Lambda}\times A_{\Xi}\to A_\Gamma$. 
    \end{enumerate}
\end{lem}
\begin{proof}
\eqref{lem:parabker1} Without loss of generality, $T\neq \emptyset$. Clearly, 
$\Gamma=\Lambda\vee \Xi$
if and only if~$[T,S\setminus T]=\{1\}$. If~$[T,S\setminus T]=\{1\}$, then~$\langle T\rangle$ is normalised by~$S$ and hence is normal in~$A_\Gamma$. 

Assume conversely that~$\langle T\rangle$ is normal in~$A_\Gamma$, and let~$s\in S\setminus T$ and~$t\in T$. Then there is~$w\in \langle T\rangle$,~$w\neq 1$, satisfying
\begin{equation}\label{eq:wst=s}
    wst=s.
\end{equation}
We aim to deduce that~$[s,t]=1$ from~\eqref{eq:wst=s}. 

Our argument is based on the normal form of RAAGs established in~\cite[\S 3]{heme:nf}, that we now recall. Set a total ordering~$\preceq$ on~$S$ and, for every~$t\in S\sqcup S^{-1}$, let~$s(t):=t$ if~$t\in S$ and~$s(t):=t^{-1}$ if~$t\in S^{-1}$ (in other words, $s(t)$ is the vertex in~$\Gamma$ correspondent to~$t$). By~\cite[Proposition~3.2]{heme:nf}, every~$\gamma\in A_\Gamma$ admits a \emph{normal form}, i.e., it can be expressed as a word~$t_1\cdots t_n$ in the letters~$S\sqcup S^{-1}$ with the property that, for all~$1\leq i<j\leq n$ for which~$s(t_j)\prec s(t_j)$ and~$[s(t_i),s(t_j)]=1$, there is~$i<k<j$ such that~$[s(t_k),s(t_j)]\neq 1$. The sequence~$(s(t_1),\ldots,s(t_n))$ is uniquely determined by~$\gamma$ and is called the \emph{type} of~$\gamma$. Every word in~$S\sqcup S^{-1}$ can be reduced to a word in normal form via successive shuffling and simplifications of elements with their inverses~\cite[pp.~237-238]{heme:nf}.

Let~$w=t_1\cdots t_n$ be the normal form of~$w$ in~$\langle T\rangle\cong A_{\Lambda_T}$ and hence in~$A_\Gamma$. By~\eqref{eq:wst=s}, we have
\begin{equation*}
    t_1\cdots t_n st=s.
\end{equation*}
and the element~$t_1\cdots t_nst\in A_\Gamma$ has type~$(s)$. This means that the word~$t_1\cdots t_nst$ is not in normal form. Since~$t_1\cdots t_n$ is in normal form, the letter~$t$ needs to be shuffled from right to left in the word~$t_1\cdots t_nst$. In particular,~$[s,t]=1$.

\eqref{lem:parabker2} The statement is an immediate consequence of~\eqref{lem:parabker1}.
\end{proof}

\begin{cor}\label{cor:chainparab}
   Every strictly increasing chain of vertex subsets $$T_1\subset T_2\subset \ldots \subset T_n=S,\,n\geq 2$$ satisfying
    \begin{equation*}
        [T_{i-1},T_{i}\setminus T_{i-1}]=\{1\},\qquad\forall\,i\geq 2
    \end{equation*}
    induces a strictly increasing chain of parabolic subgroups of~$A_\Gamma$
    \begin{equation*}
        \langle T_1\rangle \leq \langle T_2\rangle \leq \ldots \leq \langle T_n\rangle=A_\Gamma
    \end{equation*}
   such that, for every~$i\geq 2$, $\langle T_{i-1}\rangle$ is normal in~$\langle T_i\rangle$ and
    \begin{equation*}
        \langle T_{i}\rangle /\langle T_{i-1}\rangle\cong \langle T_i\setminus T_{i-1}\rangle\leq A_\Gamma.
    \end{equation*}
\end{cor}
\begin{proof}
    Given~$T\subseteq S$, recall the standard fact that the subgroup~$\langle T\rangle \leq A_\Gamma$ is canonically isomorphic to~$A_{\ind_\Gamma(T)}$~\cite[Ch.~II, Theorem~4.13]{vanderLek}. Then~\Cref{lem:parabker} applies.
\end{proof}

\begin{proof}[Proof of \Cref{prop:chainTopParab}]
   Let~$2\leq i\leq n$. Notice that
   $$\Lambda_{i}=\Lambda_{i-1}\vee \ind_\Gamma(T_{i}\setminus T_{i-1}).$$ 
   Being~$\Lambda_1$ connected, one inductively deduces that~$\Lambda_i$ is connected. By \Cref{prop:parab}, $\caA_{\Lambda_{i-1}}(\varphi)$ embeds as an open subgroup in~$\caA_{\Lambda_i}(\varphi)$ with image~$\LL U\GG_{\caA_{\Lambda_i}(\varphi)}\cdot A_{\Lambda_{i-1}}$. 
   Moreover, by Lemma~\ref{lem:ontoRAAG}, the subgroup~$\LL U\GG_{\caA_{\Lambda_i}(\varphi)}\cdot A_{\Lambda_{i-1}}$ is normal in~$\caA_{\Lambda_i}(\varphi)$ and
   \begin{equation*}
       \caA_{\Lambda_{i}}(\varphi)/(\LL U\GG_{\caA_{\Lambda_i}(\varphi)}\cdot A_{\Lambda_{i-1}})\cong A_{\ind_\Gamma(T_i\setminus T_{i-1})}.\qedhere
   \end{equation*}
\end{proof}


\section{An analogue of the universal cover of the Salvetti complex} \label{s:Salvetti}

Throughout this section, $\Gamma = (S,E)$ denotes as usual a (simplicial) graph with $S \neq \leer$ and~$\varphi\colon \caO\hookrightarrow U$ a continuous open monomorphisms of topological (as usual Hausdorff) groups~$\caO\leq_o U$. 
 
Our goal here is to construct a piecewise Euclidean cell complex $\wtil{S}_\Gamma(\phee)$ 
on which $\caA_{\Gamma}(\varphi)$ acts cocompactly, cellularly, and with controlled stabilisers. 
Such a construction can be compared to the Bass--Serre tree of HNN-extensions, and to the universal cover of the Salvetti complex of abstract RAAGs. A key special case, and our guiding example, is when the base group~$U$ is profinite, implying that the action $\caA_{\Gamma}(\varphi) \curvearrowright \wtil{S}_\Gamma(\phee)$ is geometric.

By varying $\phee : \mc{O} \into U$ we are able to determine some geometric and topological properties of the space $\wtil{S}_\Gamma(\phee)$ in interesting cases. This puts us in a position to invoke (the TDLC version of) Brown's criterion to deduce compactness properties for locally compact~RAAGs. Of course, some other cohomological information can be obtained; we refer to \Cref{s:FPRAAGs} further below for these applications. 

We begin by recalling some basic metric geometry. We denote by $I^k := [0,1]^k \subset \R^k$ the standard $k$-dimensional Euclidean unit cube, which we always endow with the Euclidean metric. 

\begin{defn}[Cubical realisation of posets]
Given a finite set~$A$ with~$k$ elements, its power set~$2^A = \{B \mid B \subseteq A\}$ is naturally a poset with respect to set-theoretic inclusion. Such a poset can be metrically realized as a $k$-dimensional Euclidean cube~$[0,1]^k \subseteq \R^k$ as follows:
\begin{itemize}
\item The elements $B \in 2^A$ are in bijection with the vertices of $[0,1]^k$, with $\leer$ being mapped to the origin $0 \in \R^k$;
\item A pair $B,C \in 2^A$ spans an edge (which is isometric to $[0,1]$) if $B \subsetneq C$ and $|C \setminus B| = 1$;
\item A pair $B,E \in 2^A$ with $B \subsetneq E$ and $E \setminus B = \{x,y\}$ spans a Euclidean square (which is isometric to $[0,1] \times [0,1]$) with $B$, $B \cup \{x\}$, $B \cup \{y\}$, and $E = B \cup \{x,y\}$ as vertices and $\{B, B \cup \{x\}\}$, $\{B, B \cup \{y\}\}$, $\{B \cup \{x\}, E\}$, and $\{B \cup \{y\}, E\}$ as edges;
\item Inductively, a pair $B,E \in 2^A$ with $B \subsetneq E$ spans a Euclidean cube of dimension $|E|-|B|$, with each pair $C,D$ with $B \subseteq C \subsetneq D \subseteq E$ spanning a $(|D|-|C|)$-dimensional face of such cube.
\end{itemize}

The construction above will be called \emph{cubical realisation} of the poset $2^A$.
\end{defn}
 We now define our version of the ``universal" Salvetti complex for $\caA_{\Gamma}(\varphi)$. By ``universal" we mean that we are attempting to construct the simply-connected version of the complex, not the analogue of the $K(\pi,1)$ of a RAAG. As our space is made up of cubical pieces, our exposition follows that of $M_\kappa$-polyhedral complexes by Bridson--Haefliger \cite[Chapter~I.7]{BridsonHaefliger}.

\begin{defn}[{Generalised universal Salvetti complex}] \label{def:Salvetti} 
Let $\Gamma = (S,E)$ be a finite graph with vertex set $V(\Gamma) = S \neq \leer$ and edge set $E(\Gamma) = E$. Let $\phee : \mc{O} \into U$ be a continuous open monomorphism of topological groups~$\caO\leq_o U$.   The \emph{generalised universal Salvetti complex} for the topological RAAG $\caA_\Gamma(\varphi)$ of type $(\Gamma,\phee)$ is the $M_0$-polyhedral (i.e., piecewise Euclidean) complex~$\wtil{S}_\Gamma(\varphi)$ constructed as follows: 
\begin{enumerate}
    \item \label{def:SalvettiVertices} The vertex set is 
    \[\wtil{S}_\Gamma(\varphi)^{(0)} := \{gU \mid g \in \caA_\Gamma(\varphi)\},\] 
    the coset space of the base group $U$ in $\caA_\Gamma(\varphi)$.
\end{enumerate}
To construct $k$-dimensional cells of $\wtil{S}_\Gamma(\varphi)$ with $k > 0$, we need an intermediate step. Given a $k$-element subset 
\[T=\{t_{1}, \ldots, t_{k}\}\in \caV^f = \{ T' \subseteq S \mid \ind_\Gamma(T') \text{ is a clique in } \Gamma \}\sqcup\{\emptyset\} \] 
and $g \in \caA_\Gamma(\varphi)$, we identify the $2^k$-element set 
\[\big\{ g t_1^{\veps_1} \cdots t_k^{\veps_k} U \in \wtil{S}_\Gamma(\varphi)^{(0)} \,\big\vert\, \veps_i \in \{0,1\},\, \forall i\big\}\]
with the poset $2^T$ in the obvious way: an element $X \in 2^T$ is mapped to the coset $g t_1^{\veps_1} \cdots t_k^{\veps_k} U$ with~$(\veps_i)_{i=1}^k\subseteq\{0,1\}^k$ having support~$X$. (If $T = \leer$, we have the singleton $\{gU\}$ being identified with the $0$-cube $\{0\}$.)
\begin{enumerate}
\setcounter{enumi}{1}
    \item \label{def:SalvettiCubes} For every $T = \{t_{1}, \ldots, t_{k}\}\in \caV^f$ with $|T|=k$ and for every $g \in \caA_\Gamma(\varphi)$, attach the cubical realisation of the poset $2^T$ to the $2^k$-element vertex subset 
    \[\{g t_1^{\veps_1} \cdots t_k^{\veps_k} U \mid \veps_i \in \{0,1\},\,\forall\,i\} \subset \wtil{S}_\Gamma(\varphi)^{(0)}.\] 
     The resulting $k$-dimensional cube is denoted by~$gQ_T$, and $T$ is said to be the \emph{type} of~$gQ_T$. 
     Accordingly, we set $g Q_{\leer} = gU$. Conventionally, $1Q_T=Q_T$ for every~$T\in \caV^f$.
\end{enumerate}
Note that a 1-dimensional cell corresponds to~$\{gU,gtU\}$, with $g \in \caA_\Gamma(\varphi)$ and~$t\in S$.

The resulting space $\wtil{S}_\Gamma(\varphi)$ is a \emph{cubed complex}, i.e., a cell complex whose cells are cubes endowed with the Euclidean metric, and the attaching maps are Euclidean isometries~\cite[\S I.7, Example~7.40(4)]{BridsonHaefliger}. 

The Euclidean metric on each cube of $\wtil{S}_\Gamma(\varphi)$ extends piecewise to the length pseudo-metric on the whole $\wtil{S}_\Gamma(\varphi)$~\cite[Definition~I.7.38]{BridsonHaefliger}.  We shall see later in \Cref{prop:basicS}\eqref{prop:basicS5} that the length pseudo-metric on $\wtil{S}_\Gamma(\varphi)$ is indeed a metric.
\end{defn}

\begin{figure}
    \centering
    \includegraphics[width=0.75\linewidth]{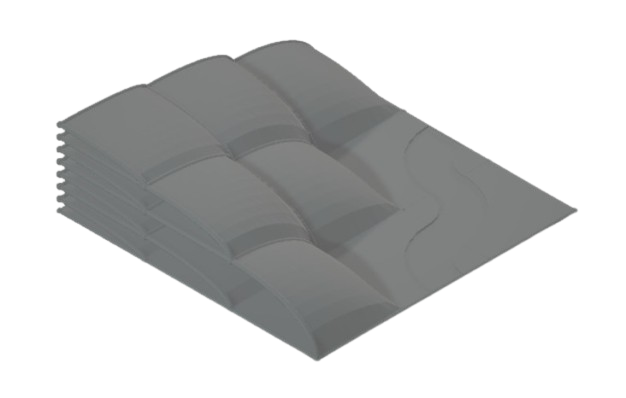}
    \caption{$\wtil{S}_\Gamma(\varphi)$ for~$\Gamma=\bullet$---$\bullet$ and~$\varphi\colon z\in \Z_2\mapsto 2z\in \Z_2$. For more information and a 3D visualisation of this model, see Raphael Appenzeller's Thingiverse page:\\ \url{https://www.thingiverse.com/thing:7360375}.}
    \label{fig:Raphael}
\end{figure}

\begin{ex}\label{ex:basicExSalv} 
The generalised universal Salvetti complex is indeed a generalisation. 
   \begin{enumerate}
       \item \label{ex:basicExSalv1} If the base group $U$ is trivial, then $\mc{A}_\Gamma(\phee)$ is canonically isomorphic to the abstract RAAG $A_\Gamma$ and $\wtil{S}_\Gamma(\varphi)$ is canonically isomorphic to the \emph{universal Salvetti complex} $\wtil{\Sigma}_\Gamma$ for $A_\Gamma$; see~\cite[Section~3.6]{CharneySurvey}. 
			
       \item \label{ex:basicExSalv2} On the other extreme, assume that $\Gamma$ is totally disconnected, i.e., $\Gamma=(S,\leer)$, and~$U$ is Hausdorff. In this case, $\caA_\Gamma(\varphi)=\caH_S(\varphi)$ is the iterated HNN-extension over $\varphi$ with stable letters $S$. By design $\wtil{S}_\Gamma(\varphi)$ is $1$-dimensional and is canonically isomorphic to the Bass--Serre tree of $\caH_S(\varphi)$; see~\Cref{fact:HNNtop}\eqref{fact:HNNtopBassSerretree}. 
       \item \label{ex:basicExSalv3} 
       Assume~$\caO=U$ and that $\varphi$ is surjective. By Propositions~\ref{prop:artin} and~\ref{prop:normU}, $\caA_\Gamma(\varphi)$ is topologically isomorphic to the topological semidirect product to $U\rtimes A_\Gamma$. In particular, the canonical isomorphism $A_\Gamma \cong \caA_\Gamma(\varphi) / U$, $a \mapsto aU$,
induces an isomorphism of cubed complexes between $\wtil{S}_\Gamma(\varphi)$ and the universal Salvetti complex $\wtil{\Sigma}_\Gamma$ for $A_\Gamma$. 
   \end{enumerate}
\end{ex}

\subsection{Basic properties of~$\wtil{S}_\Gamma(\varphi)$}\label{sus:SalvBasics}
In this section, we discuss some general properties of the complex~$\wtil{S}_\Gamma(\varphi)$ that follow quite directly from its definition. Throughout this section we keep the notation set in \Cref{def:Salvetti}. 

\begin{prop}[The $\mc{A}_\Gamma(\phee)$-action on $\wtil{S}_\Gamma(\varphi)$] \label{prop:stabs}
 The natural action of the topological RAAG $\mc{A}_\Gamma(\phee)$ on the vertex set $\wtil{S}_\Gamma(\varphi)^{(0)}$ of the generalised universal Salvetti complex $\wtil{S}_\Gamma(\phee)$ extends to an action by isometries on the whole~$\wtil{S}_\Gamma(\varphi)$. In particular, the following properties hold.
\begin{enumerate}
    \item \label{prop:stabs1} For every $T\in \caV^f$, the group $\mc{A}_\Gamma(\phee)$ acts transitively on the cubes of~$\wtil{S}_\Gamma(\varphi)$ of type~$T$. In particular, $\mc{A}_\Gamma(\phee)$ acts cocompactly on~$\wtil{S}_\Gamma(\varphi)$. 
		
    \item \label{prop:stabs2} For all $g\in \mc{A}_\Gamma(\phee)$, the stabiliser $\stab_{\mc{A}_\Gamma(\phee)}(gU)$ of the vertex~$gU\in \wtil{S}_\Gamma(\varphi)^{(0)}$ is given by 
    \begin{equation*}\label{eq:stabgU}
       \stab_{\mc{A}_\Gamma(\phee)}(gU)=gUg^{-1}.
    \end{equation*}
    In particular, $\stab_{\mc{A}_\Gamma(\phee)}(gU)$ is topologically isomorphic to the codomain $U$ of $\phee$.
		
    \item \label{prop:stabs3} For every~$T\in \caV^f$ with $|T|=k\geq 1$, the pointwise stabiliser $\stab_{\mc{A}_\Gamma(\phee)}(gQ_T)$ of the $k$-cell $gQ_T$ is the conjugate 
		\[ \stab_{\mc{A}_\Gamma(\phee)}(gQ_T)=g\cdot\stab_{\mc{A}_\Gamma(\phee)}(Q_T)\cdot g^{-1}.\]
    Additionally, if $\varphi(\caO)\subseteq \caO$, then 
    \begin{equation}\label{eq:stabQT2}
         \stab_{\mc{A}_\Gamma(\phee)}(gQ_T)=g\cdot\varphi^k(\caO)\cdot g^{-1},
    \end{equation}
    in which case $\stab_{\mc{A}_\Gamma(\phee)}(gQ_T)$ is topologically isomorphic to the domain $\caO$ of $\phee$.
\end{enumerate}
\end{prop}

\begin{proof}
For this proof set $G = \mc{A}_\Gamma(\phee)$. The claims in~\eqref{prop:stabs1} and~\eqref{prop:stabs2} are straightforward. To prove~\eqref{prop:stabs3}, let~$T=\{t_1,\ldots, t_k\}\in \caV^f$ be of size~$k\geq 1$.
     By definition, every closed cell in~$\wtil{S}_\Gamma(\varphi)$ is uniquely determined by its vertices. Hence,~$\stab_G(Q_T)$ equals the intersection of the stabilisers of all vertices of~$Q_T$, that is
    \begin{equation}\label{eq:stabs1}
      \stab_G(Q_T)=\bigcap_{\veps\in\{0,1\}^k} t_1^{\veps_1}\cdots t_k^{\veps_k}U t_k^{-\veps_k}\cdots t_1^{-\veps_1}.
    \end{equation}
    Moreover, by a similar reasoning,
    \begin{equation*}
        \stab_G(gQ_T)=g\cdot\stab_G(Q_T)\cdot g^{-1}.
    \end{equation*}
    
    We shall now focus on~$\stab_G(Q_T)$ and assume that $\phee(\mc{O}) \subseteq \mc{O}$. 
   In the following, we adopt the convention that the products of elements of~$G$ over an empty set of indices is~$1$. 
    For every~$\veps=(\veps_i)\in \{0,1\}^k$ with $\veps\neq (0,\ldots,0)$, let $i=i(\veps)$ be the  highest index with $1\leq i\leq k$ for which $\veps_{i}=1$. 
One then has 
    \begin{equation*}
        t_1^{\veps_1}\cdots t_k^{\veps_k}Ut_k^{-\veps_k}\cdots t_1^{-\veps_1} = t_1^{\veps_1}\cdots t_{i}^{\veps_{i}} U t_{i}^{-\veps_{i}} \cdots t_1^{-\veps_1}.
    \end{equation*} 
    Hence
    \begin{align}\label{eq:stabs2}
		\begin{split}
		 &\Big(t_1^{\veps_1}\cdots t_{i-1}^{\veps_{i-1}}Ut_{i-1}^{-\veps_{i-1}}\cdots t_1^{-\veps_1}\Big)\cap \Big(t_1^{\veps_1}\cdots t_k^{\veps_k}U t_k^{-\veps_k}\cdots t_1^{-\veps_1}\Big) = \\
		= \quad & t_1^{\veps_1}\cdots t_{i-1}^{\veps_{i-1}}\big(U\cap t_{i}Ut_{i}^{-1}\big)t_{i-1}^{-\veps_{i-1}}\cdots t_1^{-\veps_1} \\ 
        =\quad & \, t_1^{\veps_1}\cdots t_{i-1}^{\veps_{i-1}}\big(\varphi(\caO))t_{i-1}^{-\veps_{i-1}}\cdots t_1^{-\veps_1}\quad\text{(by \Cref{cor:Ut^U})} \\ 
        =\quad & \, \varphi^{\veps_1+\ldots +\veps_{i-1}+1}(\caO) \quad\hspace{2.3cm}\text{(because $\phee(\mc{O})\subseteq \mc{O}$)}\\
        =\quad & \, \varphi^{\veps_1+\ldots +\veps_k}(\caO)\hspace{3.35cm} \text{(by definition of $i$)}
		\end{split}
    \end{align}
    By~\eqref{eq:stabs1} and~\eqref{eq:stabs2}, we conclude that
    \begin{equation*}
        \begin{split}
            & \stab_G(Q_T)=\\
            = & \, U\cap \underset{\veps \neq (0,\ldots,0)}{\bigcap_{\veps\in\{\pm 1\}^k}}\Big(t_1^{\veps_1}\cdots t_{i-1}^{\veps_{i-1}}Ut_{i-1}^{-\veps_{i-1}}\cdots t_1^{-\veps_1}\Big)\cap \Big(t_1^{\veps_1}\cdots t_k^{\veps_k}U t_k^{-\veps_k}\cdots t_1^{-\veps_1}\Big)\\
            = & \, \underset{\veps \neq (0,\ldots,0)}{\bigcap_{\veps\in\{\pm 1\}^k}}\varphi^{\veps_1+\ldots +\veps_k}(\caO).
        \end{split}
    \end{equation*}
Again using that $\varphi(\caO)\subseteq \caO$, we have $\varphi^{m+1}(\caO)\subseteq \varphi^m(\caO)$ for every~$m\in \Z_{\geq 0}$. Hence the last equality implies~\eqref{eq:stabQT2}.
\end{proof}

\Cref{prop:stabs}\eqref{prop:stabs3}, specifically Equation~\eqref{eq:stabQT2}, gives us a mild, but natural condition to have full control of the stabilisers of the $\mc{A}_\Gamma(\phee)$-action on $\wtil{S}_\Gamma(\phee)$. The assumption ``$\phee(\caO) \subseteq \caO$" will therefore be included in many statements about~$\wtil{S}_\Gamma(\phee)$ throughout the remainder of the paper. 
In cases where $\phee(\caO)=\caO$, we might obtain even further properties about the $\mc{A}_\Gamma(\phee)$-action, as the following lemma shows.

\begin{lem}\label{lem:Ufinorbs}
    Assume that~$\varphi\colon \caO\hookrightarrow U$ satisfies~$\varphi(\caO)=\caO$. Then the $U$-action on~$\wtil{S}_\Gamma(\varphi)$ has finite orbits if and only if~$|U:\caO|<\infty$. 
\end{lem}
\begin{proof}
    Recall that~$\caA_\Gamma(\varphi)=\langle U\cup A_\Gamma\rangle$.
    For every~$g\in \caA_\Gamma(\varphi)\setminus U$ define
    \begin{equation*}
        n(g):=\min\left\{k\geq 1\,\Bigg\vert\, 
        \begin{array}{c}
          g=a_1u_1\cdots u_{k-1}a_k\\
          \text{for some }u_1,\ldots, u_{k-1}\in U,\,a_1,\ldots, a_k\in A_\Gamma\setminus\{1\}
        \end{array}\right\}.
    \end{equation*}
    We claim that, for all~$g_1,\ldots, g_\ell\in \caA_\Gamma(\varphi)\setminus U$ and~$\ell\geq 1$, one has
    \begin{equation}\label{eq:indexU}
        \big|U:U\cap \textstyle{\bigcap_{i=1}^\ell g_iUg_i^{-1}}\big|\leq |U:\caO|^{\sum_{i=1}^\ell n(g_i)}
    \end{equation}
    and the equality in~\eqref{eq:indexU} holds if~$\ell=1$. The statement of the lemma is now a direct consequence of the previous claims. Indeed, for every cell~$\sigma$ of~$\wtil{S}_\Gamma(\varphi)$, its pointwise stabiliser~$U_\sigma$ in~$U$ is the intersection of~$U$ and finitely many conjugates of~$U$ in~$\caA_\Gamma(\varphi)$ and, by the orbit-stabiliser theorem, $|U\cdot \sigma|=|U:U_\sigma|$.

   It remains to prove~\eqref{eq:indexU}, with equality attained if~$\ell=1$.
    Firstly, observe that
    \begin{equation*}
         \big|U:U\cap \textstyle{\bigcap_{i=1}^\ell g_iUg_i^{-1}}\big|\leq \prod_{i=1}^\ell|U:U\cap g_iUg_i^{-1}|.
    \end{equation*}
    Hence, it remains to prove that
    \begin{equation}\label{eq:indexU3}
        |U:U\cap gUg^{-1}|\leq |U:\caO|^{n(g)}, \quad \forall\, g\in \caA_\Gamma(\varphi).
    \end{equation}
    We proceed by induction on~$n(g)\geq 1$. The case~$n(g)=1$ is given by the fact that, for every~$a\in A_\Gamma\setminus\{1\}$, we have
    \begin{equation*}
        |U:U\cap aUa^{-1}|=|U:\caO \cap a\caO a^{-1}|=|U:\caO|
    \end{equation*}
    by~\Cref{lem:UaUa-1} and \Cref{obs:Oconj}.
    Let~$n\geq 2$ and assume the claim true for all~$h\in \caA_\Gamma(\varphi)$ with~$n(h)=n-1$. Let~$g\in \caA_\Gamma(\varphi)$ have~$n(g)=n$ and let~$g=a_1u_1\cdots u_{n-1}a_n$ for some~$u_1,\ldots, u_{n-1}\in U$ and~$a_1,\ldots, a_n\in A_\Gamma\setminus\{1\}$. Then~$h:=a_2u_2\cdots u_{n-1}a_n$ satisfies~$n(h)=n-1$. Since~$g=a_1u_1h$, we have
    \begin{equation*}
    \begin{array}{rlr}
        |U:U\cap gUg^{-1}|& \leq |U:\caO\cap a_1u_1hUh^{-1}u_1^{-1}a_1^{-1}| & \text{(as $\caO\leq U$)}\\
        & =|U:\caO|\cdot |\caO:\caO \cap a_1u_1hUh^{-1}u_1^{-1}a_1^{-1}| & \\
        & =|U:\caO|\cdot |\caO:\caO \cap u_1hUh^{-1}u_1^{-1}| & \text{(as $\caO=a_1\caO a_1^{-1}$)}\\
        & \leq |U:\caO|\cdot |U:U\cap hUh^{-1}| & \text{(as $\caO\leq U$)}
        \end{array}
    \end{equation*}
    and~\eqref{eq:indexU3} follows by induction.
\end{proof}

\begin{prop}[Basic properties of~$\wtil{S}_\Gamma(\varphi)$]\label{prop:basicS}
Let~$\wtil{\Sigma}_\Gamma$ denote the universal Salvetti complex of the abstract RAAG~$A_\Gamma$. Then:
    \begin{enumerate}
         \item \label{prop:basicS1} The map
        \begin{equation}\label{eq:AGGUgen}
            A_\Gamma \longrightarrow \mc{A}_\Gamma(\phee)/U, \qquad a\longmapsto aU
        \end{equation}
        induces an embedding of cubed complexes $\wtil{\Sigma}_\Gamma \into \wtil{S}_\Gamma(\varphi)$. 
        \item \label{prop:basicS2} The space $\wtil{S}_\Gamma(\varphi)$ is a (path-)connected cubed complex of finite dimension~$d_\Gamma$, where 
        \begin{equation*}
            d_\Gamma=\max\{ \, |T| \, \colon \, T \in \caV^f \,\}.
        \end{equation*}
        In particular, $\wtil{S}_\Gamma(\varphi)$ has the same dimension as~$\wtil{\Sigma}_\Gamma$.
				
         \item \label{prop:basicS3} The $1$-skeleton of~$\wtil{S}_\Gamma(\varphi)$ is locally finite if and only if both indices~$|U:\caO|$ and~$|U:\varphi(\caO)|$ are finite.
				
         \item \label{prop:basicS4} Suppose the codomain~$U$ of~$\phee$ is a locally compact group. Then the $1$-skeleton of $\wtil{S}_\Gamma(\varphi)$ is a Cayley--Abels graph for $\mc{A}_\Gamma(\phee)$ if and only if~$U$ is compact.
        \item \label{prop:basicS5} The 
				length pseudo-metric on~$\wtil{S}_\Gamma(\varphi)$ (see~\Cref{def:Salvetti}) is a metric. Moreover, endowed with this metric, $\wtil{S}_\Gamma(\varphi)$ is a complete geodesic metric space. 
    \end{enumerate}
\end{prop}

\begin{proof} 
Denote $G = \mc{A}_\Gamma(\phee)$. 

\eqref{prop:basicS1} By \Cref{prop:artin}, $G$ splits as the semidirect product of $N=\LL U\GG$ and $A_\Gamma$, with the latter group being canonically identified with the subgroup of $G$ generated by $S$. In other words, $G=N\cdot A_\Gamma=\{na\mid n\in N, a\in A_\Gamma\}$ and~$N\cap A_\Gamma=\{1\}$. In particular, $U\cap A_\Gamma=\{1\}$ and hence the map in~\eqref{eq:AGGUgen} is a well-defined injection. Recall that $\wtil{\Sigma}_\Gamma=\wtil{S}_\Gamma(\iid_{\{1\}})$, where~$\iid_{\{1\}}$ is the identity map on the trivial group~$\{1\}$; see~\Cref{ex:basicExSalv}\eqref{ex:basicExSalv1}. Hence the given injection $A_\Gamma\hookrightarrow G/U$ induces an injective, cube-preserving map $\wtil{\Sigma}_\Gamma\hookrightarrow \wtil{S}_\Gamma(\varphi)$.

\eqref{prop:basicS2} Since $G=\langle U\cup S\rangle$, it is straightforward to check that the $1$-skeleton of $\wtil{S}_\Gamma(\varphi)$ is a connected graph. Hence $\wtil{S}_\Gamma(\varphi)$ is path-connected; see, e.g.,~\cite[Corollary~1.4.5]{GeoBook} recalling that, for non-empty CW complexes, path-connectedness is equivalent to connectedness~\cite[Proposition~1.2.21]{GeoBook}. The claim on the dimension is an immediate consequence of \Cref{def:Salvetti} (and our default assumption that~$\Gamma$ is finite). 

\eqref{prop:basicS3} For every $g\in G$, the set $\mathrm{St}(gU)$ of all vertices of $\wtil{S}_\Gamma(\varphi)$ that are adjacent to $gU$ is given by 
   \begin{equation*}
       \mathrm{St}(gU)=\{gutU\mid u\in U,\,t\in S\}\sqcup\{gut^{-1}U\mid u\in U,\,t\in S\}.
   \end{equation*}
   In particular, left translation by~$g$ induces a bijection $\mathrm{St}(1U)\to \mathrm{St}(gU)$ and hence the cardinality of $\mathrm{St}(gU)$ is independent of $g$. It remains to show that
   \begin{equation}\label{eq:locfinit}
       |\mathrm{St}(1U)|<\infty\,\Longleftrightarrow\,|U:\caO|<\infty\text{ and }|U:\varphi(\caO)|<\infty.
   \end{equation}
Arguing as in part~\eqref{prop:basicS1}, one has that $U\cap A_\Gamma=\{1\}$ and hence~$s^{\pm 1}U\neq t^{\pm 1}U$ for all~$s,t\in S$ with~$s\neq t$. 
By the orbit-stabiliser theorem, we get the following bijections:
   \begin{equation*}
   \begin{array}{rcl}
       \{utU\mid u\in U,\,t\in S\} & \longrightarrow & \bigsqcup_{t\in S}U/(U\cap tUt^{-1}), \\
       \{ut^{-1}U\mid u\in U,\,t\in S\} & \longrightarrow & \bigsqcup_{t\in S}U/(U\cap t^{-1}Ut).
   \end{array}
   \end{equation*}
   Moreover, by \Cref{cor:Ut^U}, for every~$t\in S$ one has 
   \begin{equation*}
       U\cap tUt^{-1}=\varphi(\caO)\quad\text{and}\quad  U\cap t^{-1}Ut=\caO.
   \end{equation*} 
   This proves~\eqref{eq:locfinit} and hence item~\eqref{prop:basicS3}.

\eqref{prop:basicS4} Recall that a \emph{Cayley--Abels graph} for~$G$ is any locally finite connected simplicial graph on which~$G$ acts (by simplicial automorphisms) vertex-transitively and with compact-open vertex stabilisers. By construction and \Cref{prop:stabs}\eqref{prop:stabs1} and by item~\eqref{prop:basicS3} above, the $1$-skeleton~$\wtil{S}_\Gamma(\varphi)^{(1)}$ of~$\wtil{S}_\Gamma(\varphi)$ is a connected simplicial graph on which~$G$ acts vertex-transitively. Moreover, by~\Cref{prop:stabs}\eqref{prop:stabs2}, the vertex stabilisers are conjugates of~$U$. Hence~$\wtil{S}_\Gamma(\varphi)^{(1)}$ is a Cayley--Abels graph for~$G$ then, and only then, when~$\wtil{S}_\Gamma(\varphi)$ is locally finite and~$U$ is compact. However, if~$U$ is compact then $|U:\caO|<\infty$ and~$|U:\varphi(\caO)|<\infty$ as both~$\caO$ and~$\varphi(\caO)$ are open subgroups of~$U$, which also yields local finiteness by item~\eqref{prop:basicS3}.

\eqref{prop:basicS5} Note that~$\caV^f$ is finite as we assume $\Gamma$ to be finite. Hence, by~\eqref{prop:basicS1}, $G$ acts on~$\wtil{S}_\Gamma(\varphi)$ with finitely many orbits of cells. Since~$\wtil{S}_\Gamma(\varphi)$ is finite-dimensional by part~\eqref{prop:basicS2}, this implies that~$\wtil{S}_\Gamma(\varphi)$ has finitely many shapes of cells. The claim now follows from~\cite[Theorem~I.7.50]{BridsonHaefliger}.
\end{proof}

At this point the reader familiar with RAAGs and Artin groups will rightfully wonder whether the generalised universal Salvetti complex $\wtil{S}_\Gamma(\phee)$ is a \emph{cubical} complex (in the sense of \cite[Definition~I.7.32]{BridsonHaefliger}). Or, even stronger, a CAT$(0)$-cube complex (i.e., a cubical complex that satisfies the CAT$(0)$ condition on triangles). In the following \Cref{ex:pockets}, we show that~$\wtil{S}_\Gamma(\varphi)$ is, in general, not cubical, lacking thereafter further strong metric properties like $\CAT$(0)-ness, injectivity, or the presence of unique convex geodesic bicombings. 
We thank Indira Chatterji and Raphael Appenzeller for helpful remarks that led us to the following. 

\begin{ex}[{Some non-properties of~$\wtil{S}_\Gamma(\varphi)$}]\label{ex:pockets}
    Let~$\Gamma$ be the $1$-segment graph with vertices~$s$ and~$t$. Let~$U$ be any topological group with a continuous open group monomorphism~$\varphi\colon U\hookrightarrow U$ satisfying~$\varphi(U)\neq U$. For instance, if~$U=\Z_p$ is the additive group of~$p$-adic integers, $\varphi$ may be the multiplication by~$p$. Let~$G=\caA_\Gamma(\varphi)$. 
    
    For~$u\in U\setminus \varphi(U)$, let~$g=sus^{-1}=tut^{-1}\in G$, and let~$Q$ be the $2$-cell (i.e., full square) of $\wtil{S}_\Gamma(\varphi)$ with vertices $\{1U,sU,tU,stU{=tsU}\}$. Note that
    \begin{equation*}
        gU=U,\quad gsU=sU,\quad gtU=tU \quad \text{ and } \quad gstU=sutU\neq stU.
    \end{equation*}
   
    In particular, the cubes~$Q$ and~$gQ$ are distinct but share two adjacent edges, creating a sort of \emph{pocket}, as depicted in \Cref{fig:pocket}. 
		
\begin{figure}[h]
    \centering
    \def\svgwidth{0.9 \linewidth}
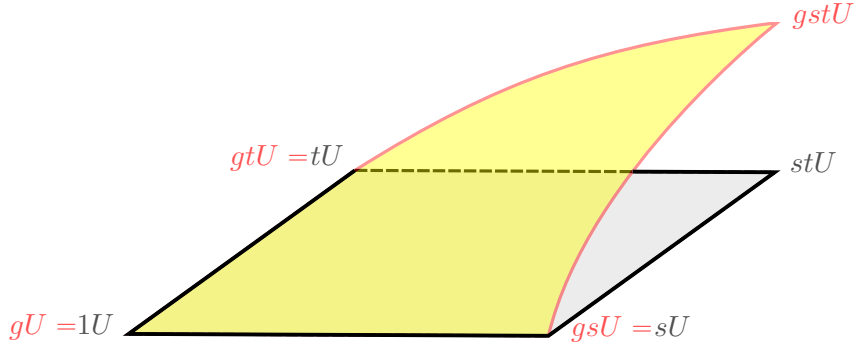
\caption{A pocket in a generalised universal Salvetti complex.} \label{fig:pocket}
\end{figure}
   
    The presence of pockets in $\wtil{S}_\Gamma(\varphi)$ has the following consequences:
    \begin{enumerate}
		
   \item The cubed complex~$\wtil{S}_\Gamma(\varphi)$ is \textbf{not} a cubical complex (in the sense of~\cite[Definition~I.7.32]{BridsonHaefliger}), since~$Q\cap gQ$ is not a face of both~$Q$ and~$gQ$. 
 
 \item The metric space $\wtil{S}_\Gamma(\varphi)$, endowed with its piecewise Euclidean length metric (see~\Cref{def:Salvetti} and \Cref{prop:basicS}\eqref{prop:basicS5}), is \textbf{not} uniquely geodesic. As a consequence, it is neither $\CAT(0)$ nor convexly uniquely bicombable~\cite[Definition~1.1]{HaettelCUB}. Indeed, the diagonals $\gamma_1$ and $\gamma_2$ connecting~$sU$ and~$tU$ in~$Q$ and~$gQ$, respectively, are two distinct geodesics in the given complex with the same endpoints. 

\begin{figure}[h]
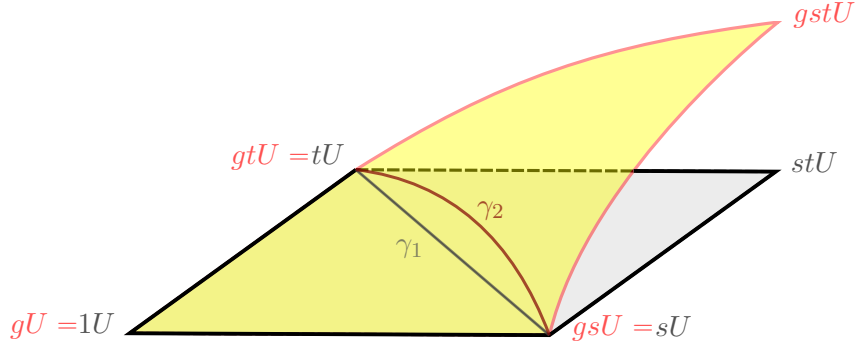

    \centering
    \def\svgwidth{0.9 \linewidth}

\caption{Distinct geodesics $\gamma_1$, $\gamma_2$ with same endpoints.} \label{fig:nonuniquegeodesics}
 \end{figure}

\item The space~$\wtil{S}_\Gamma(\varphi)$, endowed with the piecewise $\ell^{\infty}$-metric from its squares, is \textbf{not} injective~\cite[Section~2]{LangInjectivity}. 
        Indeed, consider two orthogonal medians in each of~$Q$ and~$gQ$, and select three closed balls~$B_1$, $B_2$, $B_3$ as in \Cref{fig:notHelly}.

 \begin{figure}[h]
    \centering
    \def\svgwidth{0.9 \linewidth}
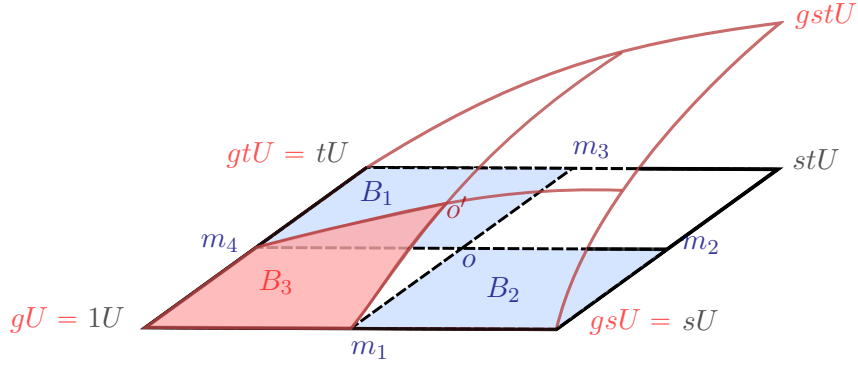
\caption{Closed balls $B_1$, $B_2$, $B_3$ without the Helly property.} \label{fig:notHelly}
      \end{figure}
   
        Then $B_i\cap B_j\neq \leer$ for all $1\leq i<j\leq 3$, but $B_1\cap B_2\cap B_3=\leer$. 
    \end{enumerate}
\end{ex}

In view of~\Cref{ex:pockets}, are there mild conditions under which $\wtil{S}_\Gamma(\varphi)$ is highly connected or in fact a $\CAT(0)$ cube complex? Of course, there are trivial answers to these questions as seen earlier: $\wtil{S}_\Gamma(\phee)$ is a $\CAT(0)$ cube complex when~$U=\{1\}$ or when~$\Gamma$ is totally disconnected; see~\Cref{ex:basicExSalv}. 

We shall prove that $\wtil{S}_\Gamma(\varphi)$ is CAT$(0)$ in yet another situation. In another setting, we prove that it is highly connected. Both cases will allow us to deduce important cohomological information about $\caA_{\Gamma}(\varphi)$; see~\Cref{s:phi(O)=O} and~\Cref{s:O=U}.

\subsection{Apartments in $\wtil{S}_\Gamma(\varphi)$ and their intersections}\label{sus:aparts}
To analyse connectivity of the generalised universal Salvetti complex $\wtil{S}_\Gamma(\varphi)$, we draw inspiration from the theory of buildings. 

By \Cref{prop:basicS}\eqref{prop:basicS1}, the universal Salvetti complex of the abstract RAAG $A_\Gamma$ canonically embeds into $\wtil{S}_\Gamma(\varphi)$. We shall denote this subcomplex by $\Sigma_0$. 
Throughout this section, we also let $N=\LL U\GG_{\caA_\Gamma(\varphi)}$ be the normal closure of $U$ in the topological RAAG $\caA_\Gamma(\varphi)$. Recall that $\caA_\Gamma(\varphi)=N\cdot A_\Gamma$ and~$N\cap A_\Gamma=\{1\}$ by \Cref{prop:artin}. 

\begin{lem} \label{lem:Nact}
    Let~$n\in N$. For all closed cells $\sigma,\tau$ of~$\Sigma_0$, if $n\cdot \sigma=\tau$ setwise, then $\sigma=\tau$ and $n$ fixes~$\sigma$ pointwise.
\end{lem}
\begin{proof}

Every closed cell in~$\wtil{S}_\Gamma(\varphi)$ is uniquely determined by its vertices. Hence, to prove the claim, if~$n\cdot \sigma=\tau$, it suffices to observe that the vertex sets of~$\sigma$ and~$\tau$ coincide, and that $n$ fixes every vertex of $\sigma$. 
It is easy to see that these properties are a consequence of the following: 
    \begin{equation}\label{eq:Nact}
        \text{For all } a,b \in A_\Gamma \leq \caA_\Gamma(\varphi), \, naU=bU \, \Longrightarrow\, a=b.
    \end{equation}
Let us prove the implication~\eqref{eq:Nact}. Suppose $a,b\in A_\Gamma$ are such that $naU=bU$. Since~$N$ is a normal subgroup of $\caA_\Gamma(\varphi)$, there is some $h\in N$ such that $b^{-1}n=hb^{-1}$. Hence, $naU=bU$ implies~$U=b^{-1}naU=hb^{-1}aU$, hence then $b^{-1}a\in h^{-1}U\subseteq N$. But~$b^{-1}a\in A_\Gamma$ and~$N\cap  A_\Gamma=\{1\}$; see \Cref{prop:artin}. Therefore~$a=b$.
\end{proof}

\begin{defn}[{Apartments}]\label{defn:apart}
   Recall that $\Sigma_0$ denotes the canonical image of the universal Salvetti complex $\wtil{\Sigma}_\Gamma$ of the abstract RAAG $A_\Gamma$ in $\wtil{S}_\Gamma(\varphi)$; see \Cref{prop:basicS}\eqref{prop:basicS1}. The subcomplex $\Sigma_0$ is called the \emph{fundamental apartment} of $\wtil{S}_\Gamma(\varphi)$. More generally, an \emph{apartment} of $\wtil{S}_\Gamma(\varphi)$ is any subcomplex of $\wtil{S}_\Gamma(\varphi)$ of the form $g\Sigma_0$  for some~$g\in \caA_\Gamma(\varphi)$. 
The \emph{apartment system} of $\wtil{S}_\Gamma(\varphi)$ is the collection 
\[\apsys:=\{g\Sigma_0\mid g\in \caA_\Gamma(\varphi)\}\]
of all its apartments. 
\end{defn}

The following is straightforward from the definitions. 

\begin{lem}\label{lem:basicApartm}
    Every $a\in A_\Gamma$ stabilises $\Sigma_0$ setwise. In particular, for all $n \in N$ and $g\in nA_\Gamma\subseteq \caA_\Gamma(\varphi)$ we have
    \begin{equation*}
        g\Sigma_0=n\Sigma_0
    \end{equation*}
    and hence
    \begin{equation*}
        \apsys=\{n\Sigma_0\mid n\in N\}.
    \end{equation*}
    In particular, 
\begin{equation*}
    \wtil{S}_\Gamma(\varphi)=\bigcup_{n\in N}n\Sigma_0.
\end{equation*}
\end{lem}

\begin{rem}\label{rem:|A|=1}
Note that $\wtil{S}_\Gamma(\varphi)$ contains only the standard apartment $\Sigma_0$ in a ``trivial" case since 
\begin{equation*}
    |\apsys|=1 \, \Longleftrightarrow\, \varphi(\caO)=\caO=U.
\end{equation*}
Indeed, the implication ``$\Longleftarrow$'' follows from \Cref{ex:basicExSalv}\eqref{ex:basicExSalv3}. 
Moreover, by \Cref{lem:Nact} if $n\in N$ satisfies $n\Sigma_0=\Sigma_0$, then $n$ is contained in the pointwise stabiliser of $\Sigma_0$ in $\mc{A}_\Gamma(\phee)$, hence in the (pointwise) stabiliser of any of its cells. By \Cref{prop:stabs}\eqref{prop:stabs2}, we conclude that $n$ belongs to $U$. Since $\apsys=\{n\Sigma_0\mid n\in N\}$ by \Cref{lem:basicApartm}, we deduce that 
$|\apsys|=1$ implies $N=U$. This yields $\varphi(\caO)=\caO=U$ in view of \Cref{prop:normU}.
\end{rem}

In the remainder of this section we focus our study on intersections of apartments, i.e., subcomplexes of the form $n_1\Sigma_0\cap n_2\Sigma_0$ for $n_1,n_2\in N$. 
Note that
\begin{equation}\label{eq:redn1=1}
    n_1\Sigma_0\cap n_2\Sigma_0=n_1(\Sigma_0\cap n_1^{-1}n_2\Sigma_0)
\end{equation}
and recall that $n_1$ acts cubically on~$\wtil{S}_\Gamma(\varphi)$. Hence, without loss of generality, we may focus only on the intersections of the form
\[\Sigma_0\cap n\Sigma_0,\quad n\in N.\]

\begin{lem}\label{lem:intersAp} 
Let $n\in N$. Then, for all $a\in A_\Gamma \leq \mc{A}_\Gamma(\phee)$, 
    \begin{equation}\label{eq:interspt}
        naU\in \Sigma_0\cap n\Sigma_0\,\Longleftrightarrow\,n\in\stab_{\caA_\Gamma(\varphi)}(aU)=aUa^{-1}.
    \end{equation}
    Moreover, for every cube $Q_T\subseteq \Sigma_0$ of positive dimension (see~\Cref{def:Salvetti}\eqref{def:SalvettiCubes}), 
    \begin{equation}\label{eq:intersQT'}
        naQ_T\subseteq \Sigma_0\cap n\Sigma_0\,\Longleftrightarrow\,n\in \stab_{\caA_\Gamma(\varphi)}(aQ_T),
    \end{equation} 
    where~$\stab_{\caA_\Gamma(\varphi)}(aQ_T)$ is the pointwise stabiliser of~$aQ_T$ in~$\caA_\Gamma(\varphi)$.
		
    If in addition $\varphi(\caO)\subseteq \caO$, then
    \begin{equation}\label{eq:intersQT}
        naQ_T\subseteq \Sigma_0\cap n\Sigma_0\,\Longleftrightarrow\,n\in a\varphi^{|T|}(\caO) a^{-1}.
    \end{equation}
\end{lem}

\begin{proof}
   By \Cref{lem:Nact}, a closed cell $\sigma$ belongs to $\Sigma_0\cap n\Sigma_0$ if and only if it is fixed by~$n$ pointwise. Hence the equivalences~\eqref{eq:interspt} and~\eqref{eq:intersQT'} hold. Finally, if $\varphi(\caO)\subseteq \caO$, then \Cref{prop:stabs} yields $\stab_G(aQ_T)=a\varphi^{|T|}(\caO) a^{-1}$, so \eqref{eq:intersQT} follows. 
\end{proof}

We introduce a further subcomplex to better describe how apartments intersect.

\begin{defn}[{Valleys}] \label{defn:valley}
Given $a\in A_\Gamma$ and $n\in a\caO a^{-1}$, the \emph{valley $\caV_{a,n}$ in $\Sigma_0$ through (the vertex) $aU$ at $n\in N$} is the subcomplex of $\Sigma_0$ given by the union of cubes 
   \[\caV_{a,n} = \underset{nbQ_T \subseteq \Sigma_0 \cap n\Sigma_0}{\bigcup_{b\in A_\Gamma, \, T\in \caV^f,}} bQ_T.\]
More generally, a \emph{valley in $\wtil{S}_\Gamma(\varphi)$} is any subcomplex of the form~$g\caV_{a,n}$, where $g\in \caA_\Gamma(\varphi)$, $a\in A_\Gamma$, and $n \in N$ is such that $n\in a\caO a^{-1}$.
\end{defn}

The terminology we chose for valleys is inspired by some concrete examples to be seen later on, but the following instructive observations might support our word choice.

\begin{rem}\label{rem:valley} 
We first note that, in view of \Cref{lem:intersAp}, a valley $\caV_{a,n}$ can be equally described as 
   \[\caV_{a,n} = \underset{n \in \stab_{\caA_\Gamma(\phee)}(bQ_T) \cap a \caO a^{-1}}{\bigcup_{b\in A_\Gamma, \, T\in \caV^f,}} bQ_T.\]
Moreover, as the definition suggests, $\caV_{a,n}$ does indeed ``pass through the vertex $aU$". In effect, $\stab_{\caA_\Gamma(\phee)}(aU) = aUa^{-1}$ by \Cref{prop:stabs}\eqref{prop:stabs2}, hence 
\[ \stab_{\caA_\Gamma(\phee)}(aU) \cap a \caO a^{-1} = a\caO a^{-1}.\]
Thus, given any $n \in a \caO a^{-1}$ as required in \Cref{defn:valley}, it is clear that $\caV_{a,n}$ contains the $0$-cube $aU$. 
    \begin{enumerate}
        \item \label{rem:valley1} More generally, invoking \Cref{prop:stabs} and \Cref{lem:UwU} and recalling the exponent map $e : A_\Gamma \to \Z$ from \Cref{defn:exponent}, the vertex set of the valley $\caV_{a,n}$ is given by 
	\begin{equation*}
    \begin{split}
        \caV_{a,n}^{(0)} & =\{bU\mid b\in A_\Gamma \text{ is such that } nbU \in \Sigma_0 \cap n\Sigma_0\}\\
        & =\{bU\mid b\in A_\Gamma \text{ with } e(b)=e(a), \, n\in a\caO a^{-1}\cap bUb^{-1}\}\\
        & \phantom{=:}\sqcup\{bU\mid b\in A_\Gamma \text{ with } e(b)\neq e(a), \, n\in a\caO a^{-1}\cap b\caO b^{-1}\}.
       \end{split}
   \end{equation*}
   Moreover, in case $\varphi(\caO)\subseteq \caO$, then the set of $k$-dimensional cells of $\caV_{a,n}$  with $k\geq 1$ is given by 
   \begin{equation*}
       \caV_{a,n}^{(k)} = \{bQ_T\mid T\in \caV^f, \, b \in A_\Gamma \text{ with } |T|=k \text{ and } n\in a\caO a^{-1}\cap b\varphi^{|T|}(\caO)b^{-1}\}.
   \end{equation*}
	
   \item \label{rem:valley2} The terminology ``at $n \in N$" in the definition of a valley $\caV_{a,n}$ concerns the use of valleys as filtrations of apartments, inspired by sublevel sets of height functions. (See also \Cref{lem:Valleysconnectedcase} further below.) Here we note that part~\eqref{rem:valley1} implies that, for every $n\in \bigcap_{a\in A_\Gamma}a\caO a^{-1}$, the apartment $\Sigma_0$ is the union 
\begin{equation*}
    \Sigma_0=\bigcup_{a\in A_\Gamma}\caV_{a,n}.
\end{equation*}
    \end{enumerate}
\end{rem}

The following technical result will be a fundamental tool to determine connectivity properties of the generalised universal Salvetti complex.

\begin{thm}[Intersection of apartments]\label{thm:intersApart}
    Assume that $\varphi(\caO)\subseteq \caO$. Then, given an arbitrary $n \in N$, exactly one of the following (mutually exclusive) conditions is satisfied:
    \begin{enumerate}
        \item \label{thm:intersApart1} $\Sigma_0\cap n\Sigma_0=\emptyset$. Moreover, this case holds if and only if $n\not\in \bigcup_{a\in A_\Gamma}aUa^{-1}$.
				
        \item \label{thm:intersApart2} $\Sigma_0\cap n\Sigma_0=(\Sigma_0\cap n\Sigma_0)^{(0)}\neq \emptyset$. (That is, $\Sigma_0\cap n\Sigma_0$ consists of at least one vertex and contains no cells of positive dimension.) Moreover, this case holds if and only if 
				\[|\{e(a)\mid aU\in (\Sigma_0\cap n\Sigma_0)^{(0)}\}|=1.\] 
				The above is also equivalent to 
				\[n\in \Big(\bigcup_{a\in A_\Gamma}aUa^{-1}\Big) \setminus \Big(\bigcup_{a\in A_\Gamma}a\caO a^{-1}\Big).\]
        \item \label{thm:intersApart3} The subset $\caB =\{a\in A_\Gamma\mid n\in a\caO a^{-1}\} \subseteq A_\Gamma$ is non-empty and $\Sigma_0\cap n\Sigma_0=\bigcup_{a\in \caB}\caV_{a,n}$. Moreover, this case holds if and only if $n\in \bigcup_{a\in A_\Gamma}a\caO a^{-1}$.
    \end{enumerate}
\end{thm}

\begin{proof}
   Since $\Sigma_0\cap n\Sigma_0$ is a subcomplex of $\wtil{S}_\Gamma(\varphi)$, it is empty if and only if it has no vertices. Hence, part~\eqref{thm:intersApart1} follows from Equivalence~\eqref{eq:interspt} in \Cref{lem:intersAp}.
	
   Now assume that the intersection $\Sigma_0 \cap n\Sigma_0$ contains cells of positive dimension. In view of \Cref{rem:valley}\eqref{rem:valley1} and the assumption $\phee(\caO) \subseteq \caO$, we invoke \Cref{lem:intersAp} and \Cref{defn:valley} and conclude that part~\eqref{thm:intersApart3} of the statement must occur. 

It remains to prove that item~\eqref{thm:intersApart2} holds in case $\Sigma_0 \cap n\Sigma_0$ is non-empty but does not have positive dimensional cells. Again since $\varphi(\caO)\subseteq \caO$, we have $\varphi^k(\caO)=t^k\caO t^{-k}$ for all $k\geq 1$ and $t\in S$. Hence \Cref{lem:intersAp}\eqref{eq:intersQT} yields 
	\begin{align*}
       \Sigma_0\cap n\Sigma_0=(\Sigma_0\cap n\Sigma_0)^{(0)}\,\Longleftrightarrow\,n\not\in\bigcup_{a\in A_\Gamma}a\caO a^{-1}.
	\end{align*}
   Note that, for all $a_1$, $a_2 \in A_\Gamma$ with $e(a_1) \neq e(a_2)$, one has $e(a_2^{-1}a_1) \neq 0$ and hence 
   \begin{equation}\label{eq:abUint}
   U \cap a_2^{-1}a_1 U a_1^{-1} a_2 = \caO \cap a_2^{-1}a_1 \caO a_1^{-1} a_2
   \end{equation}
   by \Cref{lem:UwU}\eqref{lem:UwU3}. Hence, again by \Cref{lem:intersAp}\eqref{eq:interspt}, if we had moreover $a_1U$, $a_2 U \in \Sigma_0 \cap n\Sigma_0$, we would have $n \in a_1\caO a_1^{-1} \cap a_2\caO a_2^{-1}$, contradicting the fact that $n\not\in\bigcup_{a\in A_\Gamma}a\caO a^{-1}$. Thus $\Sigma_0\cap n\Sigma_0=(\Sigma_0\cap n\Sigma_0)^{(0)} \neq \leer$ is equivalent to 
   \[|\{e(a)\mid aU\in (\Sigma_0\cap n\Sigma_0)^{(0)}\}|=1.\] 
   It remains to prove that
   \begin{align*} 
       n\in \Big(\bigcup_{a\in A_\Gamma}aUa^{-1}\Big)\setminus\Big(\bigcup_{a\in A_\Gamma}a\caO a^{-1}\Big)\,\Longleftrightarrow\,|\{e(a)\mid aU\in (\Sigma_0\cap n\Sigma_0)^{(0)}\}|=1. 
   \end{align*}
   Now if~$n\in \Big(\bigcup_{a\in A_\Gamma}aUa^{-1}\Big)\setminus \Big(\bigcup_{a\in A_\Gamma}a\caO a^{-1}\Big)$, then Equation~\eqref{eq:abUint} and \Cref{lem:intersAp} together imply that there are no $a,b\in A_\Gamma$ with $e(a)\neq e(b)$ such that $aU, bU\in \Sigma_0\cap n\Sigma_0$.
  Now assume that~$n\in a\caO a^{-1}$ for some~$a\in A_\Gamma$. By \Cref{cor:Ut^U}, for all~$t\in S$ we have
    \begin{equation*}
        U\cap t^{-1}Ut=\caO
    \end{equation*} 
  and~$n\in aUa^{-1}\cap at^{-1}Uta^{-1}$. By \Cref{lem:intersAp}\eqref{eq:interspt}, we conclude that~$aU, at^{-1}U\in \Sigma_0\cap n\Sigma_0$ and~$e(a)\neq e(at^{-1})$. This completes the proof of part~\eqref{thm:intersApart2}.
\end{proof}

\begin{cor}\label{cor:nS=S}
    Assume that $\phee(\caO) \subseteq \caO$. For all $n\in N$, one has 
    \begin{align*}
    n\Sigma_0=\Sigma_0\,\Longleftrightarrow\, n\in \bigcap_{a\in A_\Gamma}a\caO a^{-1}\,\Longleftrightarrow\, \Sigma_0\cap n\Sigma_0=\bigcup_{a\in A_\Gamma}\caV_{a,n}.
\end{align*}
\end{cor}

\begin{proof}
\Cref{thm:intersApart}\eqref{thm:intersApart3} already yields the second equivalence in the statement. 
It remains to show the first equivalence. 

If $n\Sigma_0=\Sigma_0$, then $naQ_T\subseteq \Sigma_0\cap n\Sigma_0$ for all $T\in \caV^f$ and $a\in A_\Gamma$. Since $\varphi(\caO)\subseteq \caO$ we deduce by \Cref{lem:intersAp}\eqref{eq:intersQT} that $n\in \bigcap_{a\in A_\Gamma}a\caO a^{-1}$. Conversely, if $n\in \bigcap_{a\in A_\Gamma}a\caO a^{-1}$ then $n\in \bigcap_{a\in A_\Gamma}aUa^{-1}$ as $U \supseteq \caO$. But by \Cref{lem:intersAp}, the latter intersection is the kernel of the action of $n$ on $\Sigma_0$. Hence $n\Sigma_0=\Sigma_0$.
\end{proof}

\begin{cor}\label{cor:intvalley}
Assume that $\phee(\caO) \subseteq \caO$. Suppose furthermore that $U\cap aUa^{-1}\subseteq \caO$ for all $a\in A_\Gamma$ with $a\neq 1$. Then, for all $n\in N$, the intersection of apartments $\Sigma_0\cap n\Sigma_0$ is either empty, or a single point (in fact a vertex), or a union of valleys $\bigcup_{a\in \caB}\caV_{a,n}$ for some non-empty subset $\caB\subseteq A_\Gamma$. 
\end{cor}

\begin{proof}
Note that parts~\eqref{thm:intersApart1} and~\eqref{thm:intersApart3} of \Cref{thm:intersApart} already cover most of the stated corollary. We need only improve on part~\eqref{thm:intersApart2}. 

By \Cref{lem:HDnotes2}, for all $a,b\in A_\Gamma$ with $a\neq b$ one has
\begin{align*}
aUa^{-1}\cap bUb^{-1}=a\caO a^{-1}\cap b\caO b^{-1}.
\end{align*}
Hence, given $n\in \bigcup_{a\in A_\Gamma}aUa^{-1}$, the intersection $\Sigma_0\cap n\Sigma_0$ satisfies \Cref{thm:intersApart}\eqref{thm:intersApart2} if and only if $\Sigma_0\cap n\Sigma_0$ consists only of one vertex.
\end{proof}

We briefly note that the situation described in \Cref{cor:intvalley} occurs, for instance, when $\caA_\Gamma(\varphi)\cong \caH_S(\varphi)$ --- so that $\wtil{S}_\Gamma(\phee)$ is a Bass--Serre tree --- or when $\phee(\caO) = \caO$ by \Cref{lem:UaUa-1}. 

Having described the intersection of apartments $\Sigma_0 \cap n\Sigma_0$ in general, we now discuss a special case where the filtration of~$\Sigma_0 \cap n\Sigma_0$ by valleys resembles sublevel sets of height functions. Note that 
\begin{equation*}
    N =\LL U\GG_{\caA_\Gamma(\varphi)}\supseteq \bigcup_{a\in A_\Gamma}a\caO a^{-1}. 
\end{equation*}
In case~$\Gamma$ is connected and~$\varphi(\caO)\subseteq \caO$, we have seen in \Cref{cor:expU2} that 
\begin{equation*}
    \bigcup_{a\in A_\Gamma}a\caO a^{-1}=\bigcup_{m\in \Z}s^m\caO s^{-m}
\end{equation*}
for an arbitrary $s \in S$, and clearly 
\begin{equation*}\label{eq:Uincr}
            \cdots\supseteq s^{-2}\caO s^2\supseteq s^{-1}\caO s\supseteq \caO\supseteq s\caO s^{-1}\supseteq s^{2}\caO s^{-2}\supseteq \cdots.
        \end{equation*}
This leads us to introduce a function that, roughly speaking, measures the smallest subgroup within $(s^m\caO s^{-m})_{m\in \Z}$ containing $n \in N$, if such a subgroup exists.

\begin{defn}\label{defn:epsilon}
Suppose that the finite graph~$\Gamma=(S,E)$ is connected and that $\varphi(\caO)\subseteq \caO$. Choosing any $s \in S$, define 
\begin{equation}\label{eq:epsilon}
    \varepsilon\colon N\to \Z\cup\{\pm \infty\}, \quad \varepsilon(n)=\sup\{m\in \Z \mid n\in s^m \caO s^{-m}\},
\end{equation}
where we set $\sup\emptyset = -\infty$ and $\sup\Z=+\infty$. 
\end{defn}

\begin{rem} \label{obs:epsilon} 
The map~$\veps$ from \Cref{defn:epsilon} has the following basic properties. 
    \begin{enumerate}
        \item By \Cref{cor:expU2}, the definition of $\veps$ does not depend on the choice of Artin generator~$s \in S$. 
        \item Since $s^m\caO s^{-m}$ is a subgroup of $\mc{A}_\Gamma(\phee)$ for every $m\in \Z$, one has $\veps(n^{-1})=\veps(n)$ for all $n\in N$.
    \end{enumerate}
\end{rem}

\begin{lem} \label{lem:Valleysconnectedcase}
Let $\Gamma = (S,E)$ be a finite connected graph and assume $\phee(\caO) \subseteq \caO$. Then, given~$a \in A_\Gamma$ and~$n \in a\caO a^{-1}$, the valley $\caV_{a,n}$ through~$aU$ at~$n$ is the subcomplex of the standard apartment~$\Sigma_0$ given by 
\[ \caV_{a,n} = \underset{e(b)+|T|\leq \veps(n)}{\bigcup_{b\in A_\Gamma, \, T\in \caV^f,}} bQ_T.\]
\end{lem}

\begin{proof}
Since $n \in a \caO a^{-1}$ by hypothesis and $a \caO a^{-1} = s^{e(a)} \caO s^{-e(a)}$ by \Cref{cor:expU2}, we have that $\veps(n) \geq e(a)$ and $n \in s^{\veps(n)} \caO s^ {-\veps(n)} \subseteq s^{e(a)} \caO s^{-e(a)}$. Recalling \Cref{rem:valley} and using \Cref{cor:expU2} again, a cube $bQ_T$ belongs to the valley $\caV_{a,n}$ if and only if 
\begin{align*}
n \in \, b \phee^{|T|}(\caO) b^{-1} \cap a \caO a^{-1} 
& = s^{e(b)+|T|} \caO s^{-e(b)-|T|} \cap s^{e(a)} \caO s^{e(a)} \\
& = s^{m} \caO s^{-m},
\end{align*}
where $m=\max \{ e(b)+|T|, e(a) \}$. The lemma follows. 
\end{proof}

The special case of $\Gamma$ connected and $\phee(\caO) \subseteq \caO$ should have now justified our nomenclature for valleys. To be more precise, in a valley $\caV_{a,n} \subseteq \Sigma_0$ through $aU$ at $n$, we think of $aU$ as a base vertex through which $\caV_{a,n}$ passes. Then, assigning to $n \in N$ a numerical value (by means of the map $\veps$ from \Cref{defn:epsilon}) that we think of as a ``height" or ``latitude", the valley $\caV_{a,n}$ ``at latitude $\veps(n)$" consists of all cells of the standard apartment $\Sigma_0$ that lie ``at latitude" at most $\veps(n)$.

This analogy will be revisited later on when determining connectivity properties of intersections of apartments in $\wtil{S}_\Gamma(\phee)$. 

\begin{cor}\label{cor:intersApartconn}
Let $\Gamma=(S,E)$ be a finite connected graph and assume $\varphi(\caO)\subseteq \caO$. Then, for all $n\in N$, the intersection $\Sigma_0\cap n\Sigma_0$ is either
\begin{enumerate}
\item empty, or 
\item a non-empty collection of vertices $\{aU\mid a\in \caB\}$ for some subset $\caB\subseteq A_\Gamma$ all of whose elements $b \in \caB$ have the exact same exponent,
or 
\item the subcomplex 
    \begin{align*}
		\Sigma_0\cap n\Sigma_0 = \underset{e(a)+|T|\leq \veps(n)}{\bigcup_{a\in A_\Gamma, \, T\in \caV^f,}} aQ_T.
		\end{align*}
\end{enumerate}
\end{cor}

\begin{proof}
Let~$s\in S$ and recall that $\varphi(\caO)=s\caO s^{-1}$ and $a \caO a^{-1} = s^{e(a)} \caO s^{-e(a)}$ by \Cref{cor:expU2}. The statement now follows from \Cref{thm:intersApart} and \Cref{lem:Valleysconnectedcase}. 
\end{proof}

We now address the question of whether geodesics are unique in the generalised universal Salvetti complex. This may happen (see \Cref{ex:basicExSalv}) but is not always the case, as shown in \Cref{ex:pockets}.

By default we have endowed $\wtil{S}_\Gamma(\varphi)$ with the piecewise Euclidean metric inherited from the cubical structure. If the defining graph $\Gamma=(S,\emptyset)$ has no edges, we have observed in \Cref{ex:basicExSalv}\eqref{ex:basicExSalv2} that $\wtil{S}_\Gamma(\varphi)$ is canonically isomorphic to the Bass--Serre tree of $\caH_S(\varphi)$, in which case $\wtil{S}_\Gamma(\varphi)$ is uniquely geodesic regardless of $\phee : \caO \into U$. 

We shall see that, for $\Gamma$ with at least one edge, the situation is vastly different. We start with some preliminary lemmas. 

\begin{lem}\label{lem:uniG1}
Let $\Gamma = (S,E)$ be a finite graph and assume that $\varphi(\caO)\subseteq \caO$. Suppose that the intersection $\Sigma_0\cap n\Sigma_0$ consists of only one vertex whenever $n\in N$ satisfies both $\Sigma_0\cap n\Sigma_0\neq \emptyset$ and $\Sigma_0\neq n\Sigma_0$. Then $\varphi(\caO)=\caO$.
\end{lem}

\begin{proof}
   We claim that, for all $a\in A_\Gamma$, 
   \begin{equation}\label{eq:uniG1}
       a\caO a^{-1}\subseteq \bigcap_{b\in A_\Gamma}b\caO b^{-1}.
   \end{equation}
Once we have proved~\eqref{eq:uniG1}, we see that  $s^{-1}\caO s\subseteq \bigcap_{b\in A_\Gamma}b\caO b^{-1}\subseteq \caO$ for every Artin generator $s \in S$, which implies that $\varphi(\caO)=s\caO s^{-1}=\caO$ because $\caO \supseteq \phee(\caO) =s\caO s^{-1}$ by hypothesis.

It remains to show~\eqref{eq:uniG1}. Let~$n\in a\caO a^{-1}\subseteq N$. By~\Cref{thm:intersApart}, $\Sigma_0\cap n\Sigma_0\supseteq \caV_{a,n}$. However, $\caV_{a,n}$ has infinitely many vertices. 
Hence $n\Sigma_0=\Sigma_0$ by hypothesis. Using \Cref{cor:nS=S} we conclude that $n\in \bigcap_{b\in A_\Gamma}b\caO b^{-1}$.
\end{proof}

\begin{lem}\label{lem:uniG2}
    Let $\Gamma = (S,E)$ be a finite graph and suppose that $\varphi(\caO)\subseteq \caO$. Then the following conditions are equivalent.
    \begin{enumerate}
        \item \label{lem:uniG2-1} For all~$n\in N$ such that~$\Sigma_0\cap n\Sigma_0\neq\emptyset$, either $\Sigma_0\cap n\Sigma_0$ is a point or~$\Sigma_0\cap n\Sigma_0\neq (\Sigma_0\cap n\Sigma_0)^{(0)}$;
        \item \label{lem:uniG2-2} $U\cap aUa^{-1}=\caO\cap a\caO a^{-1}$ for all $a\in A_\Gamma$ with $a\neq 1$.
    \end{enumerate}
\end{lem}

\begin{proof}
By \Cref{lem:HDnotes2}, part~\eqref{lem:uniG2-2} of the statement is equivalent to asking that
        \begin{equation*} 
            U\cap aUa^{-1}\subseteq \caO \qquad \forall\,a\in A_\Gamma \setminus\{1\}.
        \end{equation*}
Thus the implication \eqref{lem:uniG2-2} $\implies$ \eqref{lem:uniG2-1} follows from \Cref{cor:intvalley} and \Cref{thm:intersApart}. Assume conversely that item~\eqref{lem:uniG2-1} holds. Let $n\in U\cap aUa^{-1} \subseteq N$  for some $a\in A_\Gamma \setminus\{1\}$. By \Cref{lem:intersAp} the cosets $1U$ and $aU$ thus lie in $(\Sigma_0\cap n\Sigma_0)^{(0)}$. 
Moreover by \Cref{thm:intersApart}\eqref{thm:intersApart3} and by our assumed item~\eqref{lem:uniG2-1}, 
    \begin{equation}\label{eq:uniG2.1}
        \Sigma_0\cap n\Sigma_0 = \underset{n \in b\caO b^{-1}}{\bigcup_{b\in A_\Gamma}} \caV_{b,n}.
    \end{equation}

Now we claim that $n\in b\caO b^{-1}$ for every $b\in A_\Gamma$ such that $bU\in \Sigma_0\cap n\Sigma_0$. Indeed, by Equality~\eqref{eq:uniG2.1} there is some $c\in A_\Gamma$ such that $bU\in \caV_{c,n}$ and hence $n\in bUb^{-1}\cap c\caO c^{-1}$. If $e(b)\neq e(c)$, then \Cref{lem:UwU} implies that $n\in b\caO b^{-1}\cap c\caO c^{-1}$. In case $e(b)=e(c)$, choose $s\in S$ and observe, as a consequence of $\varphi(\caO)=s\caO s^{-1}\subseteq \caO$, that $n\in bUb^{-1}\cap c\caO c^{-1}\subseteq bUb^{-1}\cap cs^{-1}\caO sc^{-1}$. But then $n\in b\caO b^{-1}$ by \Cref{lem:UwU}. 
    
Since $1U$, $aU\in(\Sigma_0\cap n\Sigma_0)^{(0)}$ as seen earlier, 
we deduce from the previous paragraph that $n\in \caO \cap a\caO a^{-1}$. Thus $U\cap aUa^{-1}\subseteq \caO\cap a\caO a^{-1}$ and part~\eqref{lem:uniG2-2} follows. 
\end{proof}

\begin{thm}\label{thm:uniqgeod}
    Let $\Gamma = (S,E)$ be a finite connected graph with $E\neq\emptyset$, and assume that $\varphi(\caO)\subseteq \caO$. If $\wtil{S}_\Gamma(\varphi)$ is uniquely geodesic with respect to the piecewise Euclidean metric from \Cref{def:Salvetti}, then $\varphi(\caO)=\caO$.
\end{thm}

\begin{proof}
We shall prove the contrapositive: assuming $\varphi(\caO)\subsetneq \caO$, we aim to find two vertices in $\wtil{S}_\Gamma(\phee)$ connected by two distinct geodesics. 

Since $\varphi(\caO)\subsetneq \caO$, there is some $n\in N$ for which both $n\Sigma_0\neq \Sigma_0$ and $\Sigma_0\cap n\Sigma_0\neq \emptyset$. (For instance, we may take any $n\in \varphi(\caO)\setminus \varphi^2(\caO)$, which is non-empty as $\varphi(\caO)\subsetneq \caO$.) 
Again using that $\varphi(\caO)\subsetneq \caO$, \Cref{lem:uniG1} implies that $\Sigma_0\cap n\Sigma_0$ has at least two vertices. Due to $\Gamma$ being connected we can be more precise using \Cref{cor:intersApartconn}, which states that for every $n\in N$ satisfying both $n\Sigma_0\neq \Sigma_0$ and $\Sigma_0\cap n\Sigma_0\neq \emptyset$, exactly one of the following conditions must hold:
    \begin{itemize}
        \item[(a)] $\Sigma_0\cap n\Sigma_0=\{bU\mid b\in \caB\}$ for some $\caB\subseteq A_\Gamma$ with $|\caB|\geq 2$ and having all elements with the exact same exponent; or 
        \item[(b)] The set of all cells of $\Sigma_0\cap n\Sigma_0$ is equal to $\{aQ_T\mid a\in A_\Gamma,\,T\in \caV^f,\,e(a)+|T|\leq \veps(n)\}\subsetneq \Sigma_0$.
    \end{itemize}
Note conversely that if either~(a) or~(b) holds, then $\Sigma_0\neq n\Sigma_0$ and $\Sigma_0\cap n\Sigma_0$ has at least two vertices.

We now split the argument to construct our geodesics according to whether condition~(a) is fulfilled by some~$n\in N$ or not. 
To do so, recall that every apartment is uniquely geodesic and embeds isometrically in~$\wtil{S}_\Gamma(\varphi)$ (see~\Cref{prop:basicS}\eqref{prop:basicS1}). In particular, any geodesic contained in an apartment is also a geodesic in~$\wtil{S}_\Gamma(\varphi)$.

\underline{Case~(a) holds for some~$n\in N$:} We claim that there are two distinct vertices $aU$, $bU\in \Sigma_0\cap n\Sigma_0$ such that the unique geodesic~$\gamma$ in $\Sigma_0$ connecting them intersects $\Sigma_0\cap n\Sigma_0$ only at the extrema~$aU$ and~$bU$. Once we have found $\gamma$, it is clear that $\gamma$ and $n\gamma$ are two distinct geodesics in $\wtil{S}_\Gamma(\varphi)$ connecting~$aU$ and~$bU$.

So choose two distinct elements $a$, $b_1 \in \caB$ with $\caB$ as in condition~(a). There is a unique geodesic $\gamma_1$ in $\Sigma_0$ with starting point $\gamma_1(0) = a U$ and endpoint $\gamma_1(1) = b_1 U$. (Note that $a U \neq b_1 U$ as $U\cap A_\Gamma=\{1\}$ by \Cref{prop:artin}.) Since $\Sigma_0^{(0)}$ is discrete and (the image of) $\gamma_1$ is compact, $\gamma_1$ passes through only finitely many (but at least two) vertices of the form $cU$ with $c \in \caB$. Since any subsegment of $\gamma_1$ is also a geodesic segment, we may replace $b_1 U$ by the first vertex $\gamma_1(t) = bU \neq aU$ with $b \in \caB$ through which $\gamma_1$ passes. Thus $\gamma = \gamma_1\mid_{[0,t]} : [0,t] \to \Sigma_0$ is a geodesic in $\wtil{S}_\Gamma(\varphi)$ that intersects $\Sigma_0\cap n\Sigma_0=\{bU\mid b\in \caB\}$ only at $aU$ and $bU$. 
Therefore $\wtil{S}_\Gamma(\phee)$ is not uniquely geodesic in case~(a). 

\underline{Case~(b) holds for every~$n\in N$ satisfying~$n\Sigma_0\neq \Sigma_0$ and~$\Sigma_0\cap n\Sigma_0\neq\emptyset$:} Let $n\in N$ satisfy~$n\Sigma_0\neq \Sigma_0$ and~$\Sigma_0\cap n\Sigma_0\neq\emptyset$. We claim that there is a square (two-dimensional cubical cell) $Q \subseteq \Sigma_0$ such that the intersection $Q \cap nQ$ consists of exactly two adjacent edges of $Q$. Since $\Sigma_0$ has the piecewise Euclidean metric induced from its closed cells, the diagonals of $Q$ are geodesics in $\Sigma_0$. Thus there will be distinct geodesics $\gamma$ in $\Sigma_0$ and $n\gamma$ in $n\Sigma_0$ (hence in $\wtil{S}_\Gamma(\phee)$) that intersect only at their endpoints. This is the exact same situation as depicted back in \Cref{fig:nonuniquegeodesics}. 

By \Cref{lem:uniG2}, for all~$a\in A_\Gamma$ with~$a\neq 1$ we have $U\cap aUa^{-1}=\caO\cap a\caO a^{-1}$. Recalling the map $\veps : N \to \Z \cup \{\pm \infty\}$ from \Cref{defn:epsilon}, we note that $\veps(n) < + \infty$. For otherwise we would have $n \in \bigcap_{m \in \Z} s^m \caO s^{-m}$ and hence $n \in \bigcap_{a \in A_\Gamma} a\caO a^{-1}$ by \Cref{cor:expU2}, thus implying $n\Sigma_0 = \Sigma_0$ by \Cref{cor:nS=S}. 

Now let $\{s,t\} \in E$ be an edge of $\Gamma$ and consider $a = s^{\veps(n)} \in A_\Gamma$. By \Cref{rem:valley}\eqref{rem:valley1} and condition~(b), the vertices $aU$, $at^{-1}U$ and $at^{-1}sU$ belong to $\Sigma_0\cap n\Sigma_0$ whereas $at^{-1}stU$ belongs to $\Sigma_0\setminus n\Sigma_0$. Moreover, $aU=at^{-1}tU$, $at^{-1}U$, $at^{-1}sU$, $at^{-1}stU$ are vertices of a $2$-dimensional cubical cell~$Q$ of $\Sigma_0$. The vertices~$aU$ and~$at^{-1}sU$ are opposite in the square~$Q$, and the diagonal~$\gamma$ of~$Q$ connecting them is the unique geodesic connecting them in~$\Sigma_0$. 

Using \Cref{lem:intersAp}, since $aU$, $at^{-1}U$, $at^{-1}sU \in \Sigma_0\cap n\Sigma_0$ we have 
\begin{equation*}
    naU=aU,\quad nat^{-1}U=at^{-1}U\quad \text{and}\quad  nat^{-1}sU=at^{-1}sU
\end{equation*}
and, since $at^{-1}stU\in\Sigma_0\setminus n\Sigma_0$,
\begin{equation*}
    nat^{-1}stU\neq at^{-1}stU.
\end{equation*}
In turn, $naU$, $nat^{-1}U$, $nat^{-1}sU$, $nat^{-1}stU$ are the vertices of the $2$-dimensional cell $nQ$, which is entirely contained in $n\Sigma_0$. By design, $Q\cap nQ$ consists of the two adjacent edges $\{aU,at^{-1}U\}$ and $\{at^{-1}U,at^{-1}sU\}$. Note that $n\gamma\subseteq nQ$ is the geodesic in $n\Sigma_0$ connecting $aU$ to $at^{-1}sU$. Moreover, $n\gamma$ intersects the geodesic $\gamma$ precisely at $aU$ and $at^{-1}sU$, which are the extrema of both $\gamma$ and $n\gamma$. Therefore, $\wtil{S}_\Gamma(\varphi)$ is not uniquely geodesic.
\end{proof}

The assumption that $\Gamma$ is connected is reasonable for our purposes --- for disconnected graphs, $\mc{A}_\Gamma(\phee)$ itself splits as a (topological, amalgamated) free product; see \Cref{ex:ArtinSplit}. Thus, global information about $\mc{A}_\Gamma(\phee)$ obtained through geometry may be gained by looking at actions of topological RAAGs corresponding to connected components of~$\Gamma$ on their respective Salvetti complexes. 

\subsection{The complex $\wtil{S}_\Gamma(\varphi)$ if $\varphi(\caO)=\caO$}\label{s:phi(O)=O} 

So far, the analysis of intersection of apartments has allowed us to establish a necessary condition for~$\wtil{S}_\Gamma(\phee)$ to be uniquely geodesic, under two mild conditions: $\phee(\caO)\subseteq \caO$ and~$\Gamma$ connected. As seen earlier (\Cref{prop:stabs}), the former hypothesis is quite natural as it gives us a reasonable description of cell stabilisers of the action~$\mc{A}_\Gamma(\phee) \curvearrowright \wtil{S}_\Gamma(\phee)$. 

Our goal for this and the next section is to obtain further information about connectivity properties of the generalised universal Salvetti complex~$\wtil{S}_\Gamma(\phee)$ upon strengthening requirements in two different directions. In the present section we investigate what happens by asking for $\phee(\caO) = \caO$ instead of mere containment.

\begin{prop}\label{prop:intersApO=phiO}
Let $\Gamma = (S,E)$ be a finite graph. Assume $\varphi(\caO)=\caO$ and let $n\in N$ be such that $\Sigma_0\neq n\Sigma_0$. Then either $\Sigma_0\cap n\Sigma_0=\emptyset$ or $\Sigma_0\cap n\Sigma_0=\{aU\}$, where~$a$ is the unique element of $A_\Gamma$ for which $n\in aUa^{-1}\setminus\, a\caO a^{-1}=a(U\setminus \caO)a^{-1}$.
\end{prop}

\begin{proof}
If $\Sigma_0\cap n\Sigma_0\neq \emptyset$, \Cref{thm:intersApart} gives an $a\in A_\Gamma$ such that $n\in aUa^{-1}$. Assume 
for a contradiction 
that there are $a,b\in A_\Gamma$ with $a\neq b$ and~$n\in aUa^{-1}\cap bUb^{-1}$. By \Cref{lem:UaUa-1}, we have
    \begin{equation*}
        n\in a(U\cap a^{-1}bUb^{-1}a)a^{-1}=a\caO a^{-1}.
    \end{equation*}
By \Cref{obs:Oconj},
    \begin{equation*}
        n\in \bigcap_{c\in A_\Gamma}c\caO c^{-1}=a\caO a^{-1}.
    \end{equation*}
In turn, by \Cref{prop:stabs}, $\bigcap_{c\in A_\Gamma}c\caO c^{-1}$ is the pointwise stabiliser of the vertex set of~$\Sigma_0$. But every closed cell in $\wtil{S}_\Gamma(\varphi)$ is uniquely determined by its vertices. Thus 
$\bigcap_{c\in A_\Gamma}c\caO c^{-1}$ is the pointwise stabiliser of $\Sigma_0$ in $\caA_\Gamma(\varphi)$, which yields $\Sigma_0=n\Sigma_0$ contradicting our hypotheses. 
We conclude that $n\in aUa^{-1}$ for a unique $a\in A_\Gamma$ such that $n\in aUa^{-1}\setminus a\caO a^{-1}$. By \Cref{thm:intersApart}\eqref{thm:intersApart2} and~\Cref{cor:intvalley}, we also have $\Sigma_0\cap n\Sigma_0=\{aU\}$.
\end{proof}

\begin{cor}\label{cor:intApphi(O)=O}
Under the hypotheses of \Cref{prop:intersApO=phiO}, for every subset $M\subseteq N$  the intersection $\bigcap_{n\in M}n\Sigma_0$ is either empty or a point.
\end{cor}

Recall that $\wtil{S}_\Gamma(\varphi)$ is covered by the family $\mathscr{A}=\{n\Sigma_0\mid n\in N\}$ of apartments (see~\Cref{lem:basicApartm}). By \Cref{cor:intApphi(O)=O}, finite non-empty intersections of apartments are either empty or contractible.

We henceforth denote by $\caK=\caK_\Gamma(\varphi)$ the {nerve of the covering $\mathscr{A} = \{ n\Sigma_0 \mid n \in N\}$} of $\wtil{S}_\Gamma(\phee)$ by apartments; see~\Cref{sus:basicCellComplx} 

\begin{prop}\label{prop:Kconn}
Let $\Gamma = (S,E)$ be a finite graph. The $1$-skeleton~$\caK^{(1)}$ of $\caK = \caK_{\Gamma}(\phee)$ is connected.
\end{prop}

\begin{proof}
Note that~$N$ acts simplicially on~$\caK^{(1)}$. Thus it suffices to prove that the vertex $\Sigma_0$ can be connected to an arbitrary vertex $n\Sigma_0$ by a path in~$\caK^{(1)}$. Since~$N = \LL U \RR_{\mc{A}_\Gamma(\phee)}$ and because $\mc{A}_\Gamma(\phee)$ itself is generated by~$U$ and~$A_\Gamma$, for every~$n\in N$ we can find $a_1$, $\ldots$, $a_k\in A_\Gamma$ and $u_1$, $\ldots$, $u_k\in U$ such that
\[n=(a_1u_1a_1^{-1})\cdots (a_k u_k a_k^{-1}).\]
Now set~$n_0=1$ and, for every $i\in \{1,\ldots,k\}$, let $n_i=(a_1 u_1 a_1^{-1})\cdots (a_i u_i a_i^{-1})$. Note that $n_{i-1}^{-1}n_i=a_i u_i a_i^{-1}$ and, by \Cref{thm:intersApart}\eqref{thm:intersApart1}, 
\[n_{i-1}\Sigma_0\cap n_i\Sigma_0=n_{i-1}(\Sigma_0\cap a_i u_i a_i^{-1}\Sigma_0)\neq \emptyset.\]
Hence the sequence $(n_0\Sigma_0, n_1\Sigma_0,\ldots, n_k\Sigma_0)$ determines a path in~$\caK^{(1)}$ connecting $\Sigma_0$ to $n\Sigma_0$. 
\end{proof}

The next crucial result relies on normal forms established back in \Cref{sus:NF}.

\begin{prop}\label{prop:K1}
Let $\Gamma = (S,E)$ be a finite graph and suppose $\phee(\caO) = \caO$. Then the $1$-skeleton $\caK^{(1)}$ of $\caK = \caK_\Gamma(\phee)$ is a chordal graph.
\end{prop}

\begin{proof}
Suppose, on the contrary, that there exists an induced\footnote{``Induced'' means that the vertices and edges of $\gamma$ determine an induced subgraph of $\caK^{(1)}$.} cycle $\gamma$ of length at least $4$ in $\caK^{(1)}$. Say $\gamma=(n_0\Sigma_0, n_1\Sigma_0,\ldots, n_\ell\Sigma_0)$ for some $\ell\geq 3$ with $n_i\Sigma_0\neq n_j\Sigma_0$ for all $i\neq j$. Since $N$ acts by simplicial automorphisms on $\caK^{(1)}$, we may assume that $n_0=1$.
  
Let $i \in \{1,\ldots,\ell\}$ and consider the edge $\{n_{i-1}\Sigma_0,n_i\Sigma_0\}$ in~$\caK$. 
By \Cref{prop:intersApO=phiO}, there is a unique $a_i\in A_\Gamma$ such that $n_{i-1}^{-1}n_i\in a_iUa_i^{-1}\setminus \caO=a_i(U\setminus \caO)a_i^{-1}$; see \Cref{obs:Oconj}. In other words, there is some $u_i\in U\setminus \caO$ such that
   \begin{equation*}
       n_i=n_{i-1}a_iu_ia_i^{-1}.
   \end{equation*}
  Inductively for $1\leq i\leq \ell$ one has
  \begin{equation}\label{eq:ni}
      n_i=(a_1u_1a_1^{-1})\cdot \ldots \cdot (a_iu_ia_i^{-1}).
  \end{equation}
Since the vertices of~$\gamma$ are all distinct and there are no shorter cycles among its vertices, one has $n_{i-1}\Sigma_0\cap n_{i+1}\Sigma_0=\emptyset$ and $n_{i-1}\Sigma_0\neq n_{i+1}\Sigma_0$ for all $i \in\{1,\ldots, \ell-1\}$. \Cref{thm:intersApart} thus implies that
  \begin{equation*}
      n_{i-1}^{-1}n_{i+1}=(a_iu_ia_i^{-1})(a_{i+1}u_{i+1}a_{i+1}^{-1})\not\in \textstyle{\bigcup_{c\in A_\Gamma}}cUc^{-1}
  \end{equation*}
  and then, by \Cref{lem:prechor}, 
   \begin{equation*} 
       a_i\neq a_{i+1}.
   \end{equation*}
In addition, because $\Sigma_0\cap n_\ell\Sigma_0\neq \emptyset$ and $\Sigma_0\neq n_\ell\Sigma_0$, by \Cref{prop:intersApO=phiO}, there exist $b\in A_\Gamma$ and $u\in U\setminus \caO$ such that
   \begin{equation}\label{eq:nl2}
       n_\ell=bub^{-1}.
   \end{equation}
   Since $n_\ell\Sigma_0\cap n_1\Sigma_0=\emptyset$ and $n_\ell\Sigma_0\neq n_1\Sigma_0$, \Cref{thm:intersApart} yields 
   \begin{equation*} 
       n_\ell^{-1}n_1=(bu^{-1}b^{-1})(a_1u_1a_1^{-1})\not\in \textstyle{\bigcup_{c\in A_\Gamma}}cUc^{-1}
   \end{equation*}
  and again by \Cref{lem:prechor} 
   \begin{equation*}
       b\neq a_1.
   \end{equation*}
   Combining~\eqref{eq:ni} and~\eqref{eq:nl2} we have 
   \begin{equation*}
       (a_1 u_1 a_1^{-1})\cdot \ldots \cdot (a_\ell u_\ell a_\ell^{-1})=n_\ell=bub^{-1}.
   \end{equation*}
  
   Putting $c_1=a_1$, $c_i=a_{i-1}^{-1}a_i$ for $1\leq i\leq \ell$, and $c_{\ell+1}=a_\ell^{-1}$, we obtain 
   \begin{equation}\label{eq:x1}
       n_\ell= c_1 u_1 c_2 u_2 \cdot \ldots \cdot c_\ell u_\ell c_{\ell+1}.
   \end{equation} 
Since the $a_i$ are all distinct, no element $c_i$ with $i \in \{2,\ldots,\ell\}$ is trivial. We now claim that the normal form of $n_\ell$ is of length at least~$3$. Indeed, let $u_\ell=\omega_\ell\check{u}_\ell$ for some $\omega_\ell\in \caO$ and $\check{u}_\ell\in \caR\setminus\{1\}$. (Recall that $u_\ell\in U\setminus \caO$.) By \Cref{obs:Oconj}, there is some $\omega_\ell'\in\caO$ such that
   \begin{equation*}
       n_\ell=c_1 u_1 \cdot \ldots \cdot c_{\ell-1}(u_{\ell-1}\omega_\ell')c_\ell\check{u}_\ell c_{\ell+1}
   \end{equation*}
   and, since $u_{\ell-1}\in U\setminus \caO$, we have $u_{\ell-1}\omega_\ell'\in U\setminus \caO$.  Iterating this process we eventually find $u_0\in \caO$ and $\check{u}_1$, $\ldots$, $\check{u}_\ell\in \caR\setminus\{1\}$ such that
   \begin{equation*} 
       n_\ell = u_0 c_1\check{u}_1 c_2 \cdot \ldots \cdot c_\ell \check{u}_\ell c_{\ell+1}.
   \end{equation*}
   
By construction, $n_\ell = g_\sigma$ where $\sigma\in \caW$ has length at least $\ell-1\geq 2$ because $c_i \neq 1$ for $i \in \{2,\ldots,\ell\}$. Moreover, if the length of~$\sigma$ is~$2$ then its last entry is~$\check{u}_\ell\in \caR\setminus\{1\}$.

On the other hand we also have Equation~\eqref{eq:nl2}. So let $u=\omega\check{u}$ for some $\omega\in \caO$ and $\check{u}\in \caR\setminus\{1\}$. (Recall that $u\in U\setminus \caO$.) Hence $k:=b\omega b^{-1}\in \caO$ by~\Cref{obs:Oconj}, and also 
   \begin{equation*} 
       n_\ell=bub^{-1}=kb\check{u}b^{-1}. 
   \end{equation*}
   But then $n_\ell=g_\tau$, where $\tau\in \caW$ is $(k\check{u})$ if $b=1$ or $(k,b,\check{u},b^{-1},1)$ otherwise. Either way, $\tau$ has length at most~$2$. Moreover, if~$\tau$ has length~$2$ then its last entry is~$1$.
   But this is a contradiction because, by the Normal Form \Cref{thm:NF}, the equality 
	\[g_\sigma = n_k = g_\tau\] 
should imply that $\sigma=\tau$. We conclude that there are no {induced} cycles of length at least $4$ in $\caK$, i.e., $\caK^{(1)}$ is chordal. 
\end{proof}

Given a simplicial graph $\Lambda = (V, F)$, recall that its \emph{clique complex} is the simplicial complex whose $1$-skeleton is $\Lambda$ itself and whose $k$-simplices are precisely the $(k+1)$-element subsets of $V$ that form a clique (i.e., an induced subgraph) in $\Lambda$. 

\begin{prop}\label{prop:K2}
Let $\Gamma = (S,E)$ be a finite graph and suppose that $\phee(\caO) = \caO$. Then the simplicial complex $\caK$ is the clique complex of its $1$-skeleton. In particular, $\caK$ is contractible.
\end{prop}

\begin{proof}
It is clear that (the vertex set of) every simplex $\{n_0\Sigma_0, n_1\Sigma_0,\ldots, n_k\Sigma_0\}$ of $\caK$ forms a clique in $\caK^{(1)}$. Conversely, suppose $n_0\Sigma_0$, $n_1\Sigma_0$, $\ldots$, $n_k\Sigma_0$ form a clique in $\caK^{(1)}$. We need to argue that all such $n_i\Sigma_0$ contain some common element. 

Up to translation by an element of $N$, we may assume that $n_0 = 1$. We thus have that $\Sigma_0\cap n_i\Sigma_0\neq \emptyset$ and $\Sigma_0\neq n_i\Sigma_0$ for all $i \in \{1,\ldots, k\}$. By \Cref{prop:intersApO=phiO}, there are $a_i\in A_\Gamma$ and $u_i\in U\setminus \caO$ such that
    \begin{equation*}
        n_i=a_iu_ia_i^{-1}
    \end{equation*}
    and
    \begin{equation*}
        \Sigma_0\cap n_i\Sigma_0=\{a_iU\}.
    \end{equation*}
Again because $n_0\Sigma_0$, $n_1\Sigma_0$, $\ldots$, $n_k\Sigma_0$ is a clique, fixing $i$ and choosing another $j \in \{1,\ldots,k\} \setminus\{i\}$ we have $n_i\Sigma_0 \cap n_j\Sigma_0 \neq \leer$ and $n_i\Sigma_0 \neq n_j\Sigma_0$. Using \Cref{prop:intersApO=phiO} once more we find $b_{ij} \in A_\Gamma$ and $v_{ij} \in U \setminus \caO$ such that $n_j^{-1}n_i = b_{ij} v_{ij} b_{ij}^{-1}$ and thus 
\[ n_j^{-1}n_i = (a_j u_j^{-1} a_j^{-1})(a_i u_i a_i^{-1}) \in \bigcup_{b \in A_\Gamma} bUb^{-1}. \]
We may hence apply \Cref{lem:prechor} to conclude that 
\[a_j=a_i \qquad \forall i,j \in \{1,\ldots,k\} \text{ with } i \neq j.\] 
Therefore
\begin{equation*}
       \bigcap_{i=0}^kn_i\Sigma_0=\bigcap_{i=1}^k(\Sigma_0\cap n_i\Sigma_0)=\{a_1U\} \neq \leer, 
\end{equation*}
as desired.

It is a folklore result that the geometric realisation of the clique complex of a connected chordal graph is contractible, which is the case of our $\caK^{(1)}$ by Propositions~\ref{prop:K1} and~\ref{prop:K2}. This is well-known for finite complexes; see, e.g., \cite[Lemma~3.1]{docheng}. In the arbitrary case, 
if $f : S^n \to \caK$ and $g : S^n \to \caK$ represent two simplicial $n$-spheres we may find a finite full subcomplex $\caH \subseteq \caK$ containing $f(S^n)\cup g(S^n)$ by compactness; see \cite[Proposition~1.2.13]{GeoBook}. As $\caK^{(1)}$ is chordal and $\caH \subseteq \caK$ is a full subcomplex, then $\caH^{(1)}$ is chordal as well and $\caH$ is the clique complex of $\caH^{(1)}$. Thus $\caH$ is contractible and $f$, $g$ are homotopy-equivalent. Hence $\caK$ has trivial homotopy in every dimension and contractibility follows by Whitehead's theorem.
\end{proof}

\begin{cor}\label{thm:Salvcontr}
Let $\Gamma = (S,E)$ be a finite graph and suppose that $\phee(\caO)=\caO$. Then the generalised universal Salvetti complex $\wtil{S}_\Gamma(\varphi)$ is contractible. 
\end{cor}

\begin{proof}
The universal Salvetti complex $\wtil{\Sigma}_\Gamma$ of the abstract RAAG~$A_\Gamma$ is contractible \cite{CharneyDavisKpi1s}. But $\wtil{S}_\Gamma(\phee)=\bigcup_{n\in N}n\Sigma_0$ by \Cref{lem:basicApartm} and any apartment is isomorphic to $\wtil{\Sigma}_\Gamma$ by \Cref{prop:basicS}\eqref{prop:basicS1}. Now combine \Cref{thm:nerve}\eqref{thm:nerve1} (which applies because~$\wtil{S}_\Gamma(\varphi)$ is a cubed and hence regular cell complex) and \Cref{prop:K2}.
\end{proof}

In fact, $\wtil{S}_\Gamma(\varphi)$ has a much stronger geometry.

\begin{thm}\label{thm:CAT0}
Let $\Gamma = (S,E)$ be a finite graph and assume that $\varphi(\caO)=\caO$. Then $\wtil{S}_\Gamma(\varphi)$, endowed with the piecewise Euclidean metric from \Cref{def:Salvetti}, is a $\CAT(0)$ cube complex.
\end{thm}

\begin{proof}
Recall that $\wtil{S}_\Gamma(\phee)=\bigcup_{n\in N} n\Sigma_0$ by \Cref{lem:basicApartm} and that each $n\Sigma_0 \cong \wtil{\Sigma}_\Gamma$ is a CAT$(0)$ cube complex by~\cite[Theorem~3.1.1]{CharneyDavisKpi1s}. Any two apartments $n_1\Sigma_0$, $n_2\Sigma_0$ intersect at most at a unique vertex by \Cref{prop:intersApO=phiO}.

Hence any two cubes $\sigma$, $\tau$ of $\wtil{S}_\Gamma(\phee)$ with non-empty intersection meet at a vertex (which is a face) or lie in a common apartment and hence also meet at a common face (for this is true within apartments). This ensures that the cubed complex $\wtil{S}_\Gamma(\varphi)$ (see~\Cref{prop:basicS}\eqref{prop:basicS2}) is in fact a cube complex, i.e., any two (closed, cubical) cells with non-empty intersection meet at a common face of both. 

The first paragraph above also implies that $\wtil{S}_\Gamma(\phee)$ satisfies Gromov's link condition; we refer the reader to \cite[Chapter~4]{schwer:CUB} for the geometric and combinatorial background. In effect, given any vertex $aU \in \wtil{S}_\Gamma(\phee)^{(0)}$, all apartments $\{n\Sigma_0 \mid n \in I(a)\}$ that contain $aU$ intersect only at $aU$ by \Cref{prop:intersApO=phiO}. Thus the link~\cite[Definition~4.38]{schwer:CUB} of $aU$ in $\wtil{S}_\Gamma(\phee)$ is a disjoint union of its links in the apartments $n\Sigma_0$, $n \in I(a)$. But every such link is a flag complex because this is true in $\wtil{\Sigma}_\Gamma \cong n\Sigma_0$. Since $\wtil{S}_\Gamma(\phee)$ is finite-dimensional by \Cref{prop:basicS}\eqref{prop:basicS2} and simply-connected by \Cref{thm:Salvcontr}, Gromov's theorem~\cite[Theorem~4.43]{schwer:CUB} implies that $\wtil{S}_\Gamma(\phee)$ is CAT$(0)$.
\end{proof}

\begin{cor}\label{cor:uniqgeod}
Let $\Gamma = (S,E)$ be a finite connected graph with~$E\neq\emptyset$ and let~$\varphi(\caO)\subseteq \caO$. Then $\wtil{S}_\Gamma(\varphi)$ is uniquely geodesic with the piecewise Euclidean metric from \Cref{def:Salvetti} if and only if $\varphi(\caO)=\caO$. 
\end{cor}

\begin{proof}
Immediate from Theorems~\ref{thm:uniqgeod} and~\ref{thm:CAT0} and the fact that CAT$(0)$ spaces are uniquely geodesic \cite[Proposition~3.8]{schwer:CUB}.
\end{proof}

\subsection{The complex $\wtil{S}_\Gamma(\varphi)$ if $\caO=U$}\label{s:O=U} 

Here we tackle the case where $\Gamma$ is connected and $\phee:U\hookrightarrow U$ with $\phee(U) \neq U$. Recall that by \Cref{thm:uniqgeod} $\wtil{S}_\Gamma(\phee)$ is no longer uniquely geodesic (unless~$\Gamma$ has no edges). Hence there is no hope to get contractibility by CAT$(0)$-ness. Interestingly enough, the requirement $\caO=U$ lets us use Bestvina and Brady's famous results on RAAGs~\cite{BestvinaBrady} to show that $\wtil{S}_\Gamma(\varphi)$ is highly connected. This is due to a favourable description of the intersection of apartments (see~\Cref{lem:intersApO=U}) giving control on the homological and homotopical connectivity of $\wtil{S}_\Gamma(\varphi)$ (see~\Cref{thm:connPhi(U)=U}). 

As in the previous sections, we set~$N:=\LL U\GG_{\caA_\Gamma(\varphi)}$. 
We remind the reader that, if $\caO=U$ and $\Gamma$ is connected, then the normal subgroup $N$ can be expressed by a union of conjugates of $U$ by powers of an arbitrary (but fixed) Artin generator; see \Cref{prop:normUconn}. 

\begin{rem} \label{obs:basepointsnotneeded}
Recalling the definition of a valley $\caV_{a,n}$ (\Cref{defn:valley}) and the map $\veps : N \to \Z \cup \{\pm \infty\}$ (\Cref{defn:epsilon}), the element $n \in N$ is required to lie in $a\caO a^{-1}$. Combining \Cref{cor:expU2} with the assumptions $\caO=U$ and $\Gamma$ connected, one has 
    \begin{equation*}
        n\in a\caO a^{-1} \, \Longleftrightarrow \, n\in s^{e(a)}Us^{-e(a)} \, \Longleftrightarrow\,e(a)\leq \veps(n)
    \end{equation*}
for any~$s \in S$. Moreover, by \Cref{prop:normUconn}, any~$n \in N$ satisfies~$\veps(n) \in \Z$ because 
\[ N = \LL U\GG_{\caA_\Gamma(\varphi)} = \bigcup_{k\in\Z} s^k U s^{-k}. \]
\end{rem}

As already observed in a slightly more general setting in~\Cref{lem:Valleysconnectedcase}, in the favourable case~$\caO=U$ and~$\Gamma$ connected, for all~$a\in A_\Gamma$ and~$n\in aUa^{-1}$ the valley~$\caV_{a,n}$ is described as 
    \begin{equation*}
		\caV_{a,n} = \underset{e(b)+|T|\leq \veps(n)}{\bigcup_{b\in A_\Gamma, \, T\in \caV^f,}} bQ_T.
    \end{equation*}
Heuristically, $\caV_{a,n}$ does not quite depend on the choice of vertex $aU$ through which it passes, but rather is determined by the numerical value~$\veps(n)$ attached to the chosen $n \in N$.

In view of independence of ``basepoints" for valleys, we shall adopt the following convention throughout this section. 

\begin{conv}\label{conv:valley}
Under the assumptions $\caO=U$ and $\Gamma$ connected, we refer to a \emph{valley in $\wtil{S}_\Gamma(\varphi)$ at latitude $\epsilon \in \Z$} as any subcomplex of the form $g\caV_\epsilon$, where $g\in \caA_\Gamma(\varphi)$ and 
\begin{equation}\label{eq:Vepsilon}
    \caV_\epsilon := \underset{e(a)+|T|\leq \epsilon}{\bigcup_{a\in A_\Gamma, \, T\in \caV^f,}} aQ_T.
\end{equation}
We refer to $\caV_\epsilon$ as \emph{the valley in $\Sigma_0$ at latitude $\epsilon$} and, given $g\in \caA_\Gamma(\varphi)$, to $g\caV_\epsilon$ as \emph{the valley in $g\Sigma_0$ at latitude $\epsilon$}.
\end{conv}

\begin{rem}\label{rem:intUvalley}
It is immediate from their Definition in~\eqref{eq:Vepsilon} that:
    \begin{enumerate}
        \item Finite unions and finite intersections of valleys in $\Sigma_0$ are again valleys in $\Sigma_0$. More precisely, $\caV_\epsilon \cap \caV_\eta = \caV_{\min\{\epsilon,\eta\}}$ and $\caV_{\epsilon}\cup \caV_\eta=\caV_{\max\{\epsilon,\eta\}}$; 
        \item Every union or intersection of valleys in $\Sigma_0$ is either empty, or a valley of $\Sigma_0$, or all of $\Sigma_0$.
    \end{enumerate}
\end{rem}

We can easily extend part of \Cref{rem:intUvalley}:

\begin{lem}\label{lem:intersApO=U}
Assume $\caO=U$ and that $\Gamma=(S,E)$ is connected. Then, for every finite family $\{n_i\Sigma_0\mid i\in I\}$ of pairwise distinct apartments of $\wtil{S}_\Gamma(\varphi)$ with $2\leq|I|<\infty$ and for every $i\in I$, the intersection $\bigcap_{i\in I}n_i\Sigma_0$ is a valley in ${n_i}\Sigma_0$. In particular, $\bigcap_{i\in I}n_i\Sigma_0\neq\emptyset$. 
\end{lem}

\begin{proof}
Given any $i\in I$ note that
    \begin{equation*}
        \bigcap_{j\in I}n_j\Sigma_0 = \bigcap_{j\in I\setminus\{i\}} (n_{i}\Sigma_0\cap n_j\Sigma_0) = n_i \cdot \bigcap_{j\in I\setminus\{i\}} (\Sigma_0\cap n_{i}^{-1}n_j\Sigma_0).
    \end{equation*}
It suffices to prove that $\Sigma_0\cap n\Sigma_0$ is a valley in $\Sigma_0$ for every $n\in N$ with $n\Sigma_0\neq\Sigma_0$. 

Recalling \Cref{cor:intersApartconn}, we have to argue that $\Sigma_0\cap n\Sigma_0$ contains two vertices~$aU$ and~$bU$ with $a,b\in A_\Gamma$ satisfying $e(a) \neq e(b)$. Since $\caO=U$ we have $n\in s^kU s^{-k}$ for every $k\leq \veps(n)$ by definition of $\veps$ (see~\Cref{defn:epsilon}). It follows from \Cref{lem:intersAp} that 
    \begin{equation*}
        s^{k}U\in \Sigma_0\cap n\Sigma_0 \quad\forall k \leq \veps(n). \qedhere
    \end{equation*}
\end{proof}

The main technical observation in this section is that valleys are homotopically the same as sublevel sets studied by Bestvina and Brady in their seminal paper~\cite{BestvinaBrady}. Let us recall some relevant terminology.

Let $X$ be a cell complex, acted on by a group $G$ (by cell-permuting homeomorphisms). A height function on $X$ is just a continuous map $f : X \to \R$. We say that a height function $f$ on a $G$-complex $X$ is $G$-equivariant when an action $G \curvearrowright_\chi \R$ is given so that $f(g\cdot x) = \chi(g)+f(x)$ for all $g\in G$, $x\in X$. 

If $X$ is a polyhedral complex, then a \emph{Morse function} on $X$ (in the sense of Bestvina--Brady \cite[Definition~2.2]{BestvinaBrady}) is a height function $f : X \to \R$ that is affine on each cell, nonconstant on positive-dimensional cells, and such that $f(X^{(0)}) \subseteq \R$ is discrete. Given $t \in \R$, the corresponding \emph{sublevel set} of $X$ (with respect to the Morse function $f$) is the preimage $X_{(-\infty,t]} := f^{-1}((-\infty, t])$. 

The generalised universal Salvetti complex has a natural height function.

\begin{rem} \label{obs:exponentisMorsefunction}
Suppose $\Gamma$ is a finite graph. Recall once again the exponent map $e : A_\Gamma \to \Z$ from \Cref{defn:exponent}. Note that $e$ extends to a \emph{character}~\cite[Section~3]{BHQhomotopical}, i.e., a continuous group homomorphism 
\[\quer{e} : \mc{A}_\Gamma(\phee) \onto \Z,\] 
by setting $\quer{e}(u) = 0$ for all $u \in U$. (Alternatively, $\quer{e}$ can be defined using \Cref{thm:retract}\eqref{thm:retracta}.) This allows us to define the \emph{exponent of a vertex} $gU \in \wtil{S}_\Gamma(\phee)^{(0)}$ simply as $\quer{e}(gU) := \quer{e}(g) \in \Z$, abusing notation. This is well-defined as~$\quer{e}\vert_U=0$.

Since the complex $\wtil{S}_\Gamma(\phee)$ is a polyhedral complex whose cells are Euclidean cubes, we may (and do) extend the exponent to a continuous map 
\[\quer{e} : \wtil{S}_\Gamma(\phee) \to \R\]
by extending~$\quer{e} \colon\wtil{S}_\Gamma(\phee)^{(0)} \to \Z$ affinely within each cubical cell of~$\wtil{S}_\Gamma(\phee)$. We abuse terminology and call~$\quer{e} : \wtil{S}_\Gamma(\phee) \to \R$ the \emph{exponent (height) map} on~$\wtil{S}_\Gamma(\phee)$. 
\end{rem}

\begin{lem} \label{lem:Morsefunction} 
Let~$\Gamma$ be a finite graph. Given any~$g \in \mc{A}_\Gamma(\phee)$, the conjugate of the abstract RAAG~$A_\Gamma^{g} = gA_\Gamma g^{-1}\leq \caA_{\Gamma}(\varphi)$ has a well-defined cocompact action induced by left multiplication on the apartment~$g \Sigma_0$. More precisely, for any vertex~$gbU \in g\Sigma_0$, we set 
\[a^{g} \cdot gbU := gabU \quad \text{ for each } a \in A_\Gamma,\] 
and extend this action piecewise linearly to the whole apartment $g\Sigma_0$. 

Moreover, the exponent $\quer{e} : \wtil{S}_\Gamma(\phee) \to \R$ from \Cref{obs:exponentisMorsefunction} is a Morse function and its restriction to any apartment~$g\Sigma_0$ is $A_\Gamma^g$-equivariant. In case $g = 1$, this $A_\Gamma$-equivariant Morse function $\quer{e}|_{\Sigma_0}$ agrees with the Morse function $f : \wtil{\Sigma}_\Gamma \to \R$ studied by Bestvina--Brady in \cite[Section~5]{BestvinaBrady} after the natural cubical isomorphism $\Sigma_0 \cong \wtil{\Sigma}_\Gamma$ (see~\Cref{prop:basicS}\eqref{prop:basicS1} and \Cref{defn:apart}). 
\end{lem}

\begin{proof}
The given action is clearly well-defined.
The action of $\mc{A}_\Gamma(\phee)$ on $\wtil{S}_\Gamma(\phee)$ by left multiplication is by cellular isometries (see~\Cref{prop:stabs}), hence so is the given action $A_\Gamma^g \curvearrowright g\Sigma_0$. Conjugating back to $A_\Gamma \curvearrowright \Sigma_0$, we see that the action is also cocompact because the action of $A_\Gamma$ on $\Sigma_0 \cong \wtil{\Sigma}_\Gamma$ is so. 

For the second paragraph of the statement, note that the attaching maps of cubical cells of $\wtil{S}_\Gamma(\phee)$, as in \Cref{def:Salvetti}, 
naturally are characteristic functions in the sense of \cite[Definition~2.1]{BestvinaBrady}. The exponent, by construction, assumes integer values on vertices of cells, and is otherwise extended piecewise linearly to the interior of cells. Hence both conditions of \cite[Definition~2.2]{BestvinaBrady} are satisfied. Therefore $\quer{e}$ is a Morse function. 

To check that the exponent is $A_\Gamma^g$-equivariant, observe that the map~$\quer{e}$ 
defines an action~$A_\Gamma^g \curvearrowright \R$ by setting~$a^g \cdot r = r + \quer{e}(gag^{-1}) = r + {e}(a)$. Thus  
\begin{align*}
\quer{e}(a^g\cdot gbU) & = \quer{e}(gag^{-1}gbU) = \quer{e}(gag^{-1}U)+\quer{e}(gbU) \\
& = \quer{e}(aU)+\quer{e}(gbU) = e(a) + \quer{e}(gbU),
\end{align*}
whence equivariance. 

To clarify the last part, in Bestvina and Brady's notation we have~$G_L = A_\Gamma$~\cite[Definition~5.2]{BestvinaBrady} and $Q_L$~is the usual Salvetti complex of $G_L$; see~\cite[Section~5, proof of Theorem~5.12]{BestvinaBrady} and also~\cite[Section~3.6]{CharneySurvey}. (Here we mean the Salvetti complex ``downstairs", i.e., $Q_L$ is the usual Eilenberg--MacLane space for~$G_L$.) They then take~$X$ to be the universal cover of~$Q_L$~\cite[p.~458]{BestvinaBrady}, hence~$X = \wtil{\Sigma}_\Gamma$. Their map~$f : X \onto \R$ is the $G_L$-equivariant Morse function induced by the character~$\phi : G_L \onto \Z$ that sends every Artin generator of~$G_L$ to $1 \in \Z$ \cite[p.~454]{BestvinaBrady}, hence their~$\phi$ coincides with our exponent map~$e : A_\Gamma \onto \Z$ and thus~$f : X \to \R$ matches~$\quer{e}$ when restricted to the standard apartment~$\Sigma_0 \cong \wtil{\Sigma}_\Gamma = X$.
\end{proof}
 
We can now determine how highly connected valleys are. Recall that the \emph{homotopical connectivity} $\conn_\pi(X) \in \Z_{\geq 0} \cup \{\infty\}$ of a non-empty path-connected space $X$ is the supremum over all $n\in \Z_{\geq 0}$ for which $X$ is $n$-connected. Similarly, its \emph{(integral) homological connectivity} $\conn_h(X) \in \Z_{\geq 0} \cup \{\infty\}$ is the supremum over all $n\in \Z_{\geq 0}$ such that $X$ is $n$-acyclic over~$\Z$. By the Hurewicz theorem, if $X$ is simply-connected then $\conn_\pi(X)=\conn_h(X)$. By the Whitehead theorem, if $X$ is a CW-complex then $X$ is contractible if and only if $\conn_\pi(X)=\conn_h(X)=\infty$. 

\begin{prop}[Connectivity of valleys] \label{prop:BB} 
Consider valleys~$\caV_\epsilon$ as the subcomplexes of~$X := \Sigma_0$ defined by~\eqref{eq:Vepsilon} in \Cref{conv:valley}. 
Denote by~$L(\Gamma)$ the \emph{clique complex} of~$\Gamma$, i.e., the simplicial complex whose simplices are cliques of~$\Gamma$. Write~$X = \Sigma_0$ and consider its Morse function~$\quer{e} : X \to \R$ induced by the exponent; see \Cref{lem:Morsefunction}.

For every~$t\in \Z$, the valley~$\caV_t\subseteq X=\Sigma_0$ is a strong deformation retract of the sublevel set~$X_{(-\infty,t]} := \quer{e}^{-1}((-\infty,t])$. 
In particular, $\caV_t$~is path-connected and 
\begin{equation*} 
   \conn_h(\caV_t)=\conn_h(L(\Gamma)).
\end{equation*}
\end{prop}

\begin{proof}
Let~$t\in \Z$. Note that 
\begin{equation} \label{eq:VallSublevs}
    X_{(-\infty,t-1]}\subseteq \caV_t\subseteq X_{(-\infty,t]}
\end{equation} 
by definition of $\quer{e}$ and by \Cref{conv:valley}. 
We claim that, {up to homeomorphisms}, 
\begin{equation}\label{eq:BBpf1}
    \caV_t=X_{(-\infty,t-1]}\cup \bigcup_{v\in X^{(0)}\colon \overline{e}(v)=t}v\ast \Lk^{\downarrow}_X(v),
\end{equation}
where $v\ast \Lk^{\downarrow}_X(v)$ denotes the \emph{join}\footnote{Following~\cite[p.~9]{hatcher}, given a topological space~$X$ and a $1$-point space $\{v\}$, the set underlying the join $v\ast X$ can be described as the collection of all formal convex combinations $tv+(1-t)x$ with $x\in X$ and $t\in [0,1]$ subject the conventions that $1v+0x=v$ and $0v+1x=x$ for every $x\in X$. In particular, if $X$ is a union $X = \cup_{i \in I} K_i$ of subspaces~$K_i$, then $v\ast X = \cup_{i\in I} v \ast K$.} (or \emph{cone}) of the $1$-point space $\{v\}$ over the \emph{descending link} $\Lk_X^{\downarrow}(v)$ of~$v$ in~$X$ with respect to~$\overline{e}$. That is, 
\begin{equation}\label{eq:BBpf2}
    \Lk_X^{\downarrow}(v):= \underset{\overline{e}\vert_Q \text{ has a maximum in } v}{\bigcup_{Q\text{ cell in }X, \, Q \, \ni \, v,}} \Lk_Q(v), 
\end{equation}
where $\Lk_Q(v)$ denotes the (cellular) \emph{link} of~$v$ in the cell~$Q$, i.e., the union of all the closed cells $F\subseteq Q$ not containing~$v$. 

Once we have proved Equality~\eqref{eq:BBpf1}, we can apply \cite[Proposition~8.3.3]{GeoBook} to deduce that~$\caV_t$ is a strong deformation retract of~$X_{(-\infty,t]}$ because $(-\infty,t] \setminus (-\infty,t-1]$ contains only one point of $\overline{e}(X^{(0)})=\Z$ (namely, $t$).

To prove~\eqref{eq:BBpf1}, for every $v\in \caV_t^{(0)}$ with $\overline{e}(v)=t$ we observe the following. Firstly, by definition of $\Lk_X^{\downarrow}(v)$, we have
\begin{equation*}
    v\ast \Lk_X^\downarrow(v)=\underset{\overline{e}\vert_Q \text{ has a maximum in } v}{\bigcup_{Q\text{ cell in }X, \, Q \, \ni \, v,}} v\ast\Lk_Q(v).
\end{equation*}
Secondly, since $X$ is a cube complex, for every cubical cell~$Q$ of~$X$ containing~$v$ it is straightforward to check (up to homeomorphisms) that
\begin{equation*} 
    Q=\mathrm{Star}_Q(v)=v\ast \Lk_Q(v), 
\end{equation*}
where $\mathrm{Star}_Q(v)$ is the union of all closed cells $F\subseteq Q$ containing~$v$. 
On the other side, note that 
\begin{equation}\label{eq:BBpf5}
    \caV_t = \caV_{t-1} \cup \underset{\overline{e}(v)=t}{\bigcup_{v \in X^{(0)},}} \underset{\overline{e}\vert_Q \text{ has a maximum in } v}{\bigcup_{Q\text{ cell in }X, \, Q \, \ni \, v,}} Q.
\end{equation}
By the containments in~\eqref{eq:VallSublevs}, we conclude that 
\begin{equation*}
\begin{split}
    \caV_t & =X_{(-\infty,t-1]} \cup \underset{\overline{e}(v)=t}{\bigcup_{v \in X^{(0)},}} \, \underset{\overline{e}\vert_Q \text{ has a maximum in } v}{\bigcup_{Q\text{ cell in }X, \, Q \, \ni \, v,}} Q \\
     & =X_{(-\infty,t-1]}\cup \underset{\overline{e}(v)=t}{\bigcup_{v \in X^{(0)},}} v\ast \Lk^{\downarrow}_X(v),
\end{split}
\end{equation*}
as claimed in~\eqref{eq:BBpf1}.

The fact that $\caV_t$ is a strong deformation retract of $X_{(-\infty,t]}$ implies the remaining parts of the statement. As seen in \Cref{lem:Morsefunction}, our map $\quer{e} : X \to \R$ agrees with the Morse function of Bestvina and Brady. And in \cite[Corollary~7.2(1) and Theorem~8.6]{BestvinaBrady}, they fully calculated the homotopy type of such sublevel sets. Concretely, by~\cite[Theorem~8.6]{BestvinaBrady} the sublevel set $X_{(-\infty,t]}$ is homotopy equivalent to a suitable non-empty wedge product of copies of the clique complex~$L(\Gamma)$. Since~$\Gamma$ (hence~$L(\Gamma)$) is path-connected, so is~$X_{(-\infty,t]}$. Thus the valley $\caV_t$ is path-connected as well. 
Moreover, the inclusion~$\caV_t\hookrightarrow X_{(-\infty,t]}$ induces isomorphisms~$\wtil{\hH}_k(\caV_t)\to \wtil{\hH}_k(X_{(-\infty,t]})$ for every $k\geq 0$. Therefore 
\begin{equation*} 
		\conn_h(\caV_t)=\conn_h(X_{(-\infty,t]}), 
\end{equation*}
and \cite[Corollary~7.2(1)]{BestvinaBrady} in fact gives 
\begin{equation*} 
    \conn_h(X_{(-\infty,t]})=\conn_h(L(\Gamma)) \qquad \forall \,t\in \R. \qedhere
\end{equation*}
\end{proof}

\begin{thm}[A lower bound for the connectivity of~$\wtil{S}_\Gamma(\varphi)$]\label{thm:connPhi(U)=U}
    Let~$\Gamma$ be a finite connected graph with clique complex $L(\Gamma)$, and assume $\caO=U$. Then the generalised universal Salvetti complex $\wtil{S}_\Gamma(\phee)$ is simply-connected and 
    \begin{equation*}
\conn_\pi(\wtil{S}_\Gamma(\varphi))=\conn_h(\wtil{S}_\Gamma(\varphi))\geq\conn_h(L(\Gamma))+1.
    \end{equation*}
\end{thm}

\begin{proof}
Throughout the proof let $l=\conn_h(L(\Gamma))$. Consider the covering $\mathscr{A}:=\{n\Sigma_0\mid n\in \LL U\GG\}$ of~$\wtil{S}_\Gamma(\varphi)$ by apartments as in~\Cref{lem:basicApartm}. 
Note that every intersection $n_{1}\Sigma_0\cap \ldots \cap n_{t}\Sigma_0$ of finitely many (pairwise distinct) apartments is always non-empty: if~$t=1$ this is just an apartment; if~$t\geq 2$ this intersection is a valley by \Cref{lem:intersApO=U}. In particular, any finite collection of vertices in the nerve complex~$\caN(\mathscr{A})$ forms a simplex. Hence $\caN(\mathscr{A})$ is $k$-connected for all~$k \geq 0$ (see the last paragraph of the proof of~\Cref{prop:K2}). 

Recall that apartments are contractible and, by~\Cref{prop:BB}, valleys are path-connected (i.e., $0$-connected) and have homological connectivity~$l$. 
Thus every non-empty intersection $n_{1}\Sigma_0\cap \ldots \cap n_{t}\Sigma_0$ of $t$~apartments is $(\max\{2-t,-1\})$-connected and $(\max\{l+2-t,-1\})$-acyclic for every~$t\geq 1$. 
 
We may thus apply the Nerve~\Cref{thm:nerve}\eqref{thm:nerve2}: taking~$k=1$, we conclude that $\wtil{S}_\Gamma(\varphi)$ is simply-connected; setting~$k=l+1$, we obtain $\conn_\pi(\wtil{S}_\Gamma(\varphi))=\conn_h(\wtil{S}_\Gamma(\varphi)) \geq l+1$. Notice that~\Cref{thm:nerve} is applicable because~$\wtil{S}_\Gamma(\varphi)$ is a cubed and hence regular cell complex.
\end{proof}

\begin{rem}
Note that the inequality in \Cref{thm:connPhi(U)=U} is not always an equality. In the trivial case where $\phee(\caO) = \caO = U$, then $\wtil{S}_\Gamma(\phee) \cong \wtil{\Sigma}_\Gamma$~is contractible, even though~$\Gamma$ may be freely chosen so that~$L(\Gamma)$ is not contractible. 

It is unclear to us at this point whether the inequality in \Cref{thm:connPhi(U)=U} 
becomes an equality when $\caO = U \neq \phee(U)$. Note that, since valleys have homological connectivity \textbf{exactly}~$l$, we cannot apply the homological version of \Cref{thm:nerve}\eqref{thm:nerve2} for $k > l$ using the covering of $\wtil{S}_\Gamma(\phee)$ by apartments. 
\end{rem}

\begin{cor}\label{cor:connPhi(U)=U} 
Let $\Gamma$ be a finite connected graph and assume that $\caO=U$ and $\phee$ is not surjective. If the clique complex~$L(\Gamma)$ is acyclic, then $\wtil{S}_\Gamma(\varphi)$ is contractible.
\end{cor}

\begin{proof}
If $L(\Gamma)$ is acyclic, \Cref{thm:connPhi(U)=U} yields $\conn_\pi(\wtil{S}_\Gamma(\varphi))=\infty$, i.e., all the homotopy groups of $\wtil{S}_\Gamma(\varphi)$ are trivial. 

Contractibility of $\wtil{S}_\Gamma(\varphi)$ then follows from the Whitehead theorem. 
\end{proof}

\Cref{cor:connPhi(U)=U} applies, for instance, when $\Gamma$ is complete or, more weakly, a chordal graph. In these case, $L(\Gamma)$ is itself be contractible; see, e.g., \cite[Lemma~3.1]{docheng}.

\begin{rem} \label{obs:BBgps}
As usual let $\wtil{\Sigma}_\Gamma$ denote the universal Salvetti complex for the abstract RAAG~$A_\Gamma$, with the $A_\Gamma$-equivariant Morse function induced by the exponent $e: A_\Gamma \onto \Z$. The reader familiar with the work of Bestvina--Brady \cite{BestvinaBrady} recalls that connectivity properties of 
both the ascending and descending links ($\Lk^{\uparrow}_{\wtil{\Sigma}_\Gamma}(v)$ and $\Lk^{\downarrow}_{\wtil{\Sigma}_\Gamma}(v)$, respectively) of vertices $v \in \wtil{\Sigma}^{(0)}_\Gamma$ 
play a major role in their main results. Whether such links are highly connected depends on $L(\Gamma)$, the clique complex of $\Gamma$; see \cite[Theorem~5.12(4)]{BestvinaBrady}.

Here we remark that it is not feasible to transfer the Bestvina--Brady strategy to our setting, replacing $\wtil{\Sigma}_\Gamma$ by $\wtil{S}_\Gamma(\phee)$. This is because ascending links $\Lk^{\uparrow}_{\wtil{S}_\Gamma(\phee)}(v)$ of vertices $v$ in the generalised universal Salvetti complex $\wtil{S}_\Gamma(\phee)$ have, in general, worst connectivity properties. 

For instance, consider the case $\phee(\caO) = \caO$, where $\wtil{S}_\Gamma(\phee)$ is a CAT$(0)$ cube complex; see \Cref{s:phi(O)=O}. As argued during the proof of \Cref{thm:CAT0}, the (usual) link of any vertex $v$ of $\wtil{S}_\Gamma(\phee)$ is just the disjoint union of the links of $v$ in the apartments containing $v$. One readily deduces that $\Lk^{\uparrow}_{\wtil{S}_\Gamma(\phee)}(v)$ is thus disconnected, being the disjoint union of ascending links of $v$ in said apartments. This holds regardless of how well-behaved $L(\Gamma)$ is.

The case $\phee(\caO) \subsetneq \caO = U$ studied in \Cref{s:O=U} also provides further examples. For an illustrative case, consider the TDLC RAAGs whose universal Salvetti complexes form ``pockets'', as in~\Cref{ex:pockets}. These exhibit ascending links in $\wtil{S}_\Gamma(\phee)$ that have worse connectivity properties compared to their counterparts in $\wtil{\Sigma}_\Gamma$. Take e.g. $\Gamma$ to be a single edge connecting two vertices, $\caO = U = \Z_p$ and $\phee$ being multiplication by~$p$ (see~\Cref{fig:Raphael} for~$p=2$). Then ascending links in $\wtil{\Sigma}_\Gamma$ are homeomorphic to edges (and are thus contractible), whereas ascending links $\Lk^{\uparrow}_{\wtil{S}_\Gamma(\phee)}(v)$ in $\wtil{S}_\Gamma(\phee)$ are homeomorphic to a wedge of finitely many circles (formed by gluing edges at their endpoints, with one edge per apartment containing the vertex $v$). Thus $\Lk^{\uparrow}_{\wtil{S}_\Gamma(\phee)}(v)$ is connected but not even simply-connected, despite $L(\Gamma)$ being contractible.
\end{rem}


\subsection{A covered space structure for~$\wtil{S}_\Gamma(\varphi)$}\label{sus:build}
In this section, we spell out the ``building-like'' structure given by the covering of~$\wtil{S}_\Gamma(\varphi)$ by ``apartment'' (i.e., embedded copies of the universal Salvetti complex of~$A_\Gamma$).

To make the analogy more tangible, we will refer to the more general notion of \emph{covered spaces}, introduced in~\cite[\S 3.1.18.5]{rouss:euclid} and that simultaneously generalise, for instance, buildings and symmetric spaces. We are grateful to Auguste H\'ebert for pointing this out to us.

A \emph{covered space} consists of a non-empty set~$X$, a non-empty family~$\mathscr{A}$ of subsets of~$X$ whose elements are called \emph{apartments}, and a reflexive and symmetric binary relation~$\sim$ on~$X$ called \emph{friendliness relation} such that the following hold:
\begin{enumerate}
    \item[(CS1)] \label{axiomCS1} All apartments of~$\mathscr{A}$ are endowed with a given common structure (e.g., Euclidean space, simplicial complex, Coxeter complex, etc ...); 
    \item[(CS2)] \label{axiomCS2} (\emph{Cohabitation}) For all~$x,y\in X$ with~$x\sim y$, there is~$\Sigma\in \mathscr{A}$ such that~$x,y\in \Sigma$;
    \item[(CS3)] \label{axiomCS3} (\emph{Isomorphism}) Let~$\Sigma_1,\Sigma_2\in \mathscr{A}$ be such that~$\Sigma_1\cap \Sigma_2$ contains~$x,y\in X$ with~$x\sim y$. Then there is an isomorphism~$\varphi\colon \Sigma_1\to \Sigma_2$ preserving the structure given in~(CS1) such that~$\varphi(x)=x$ and~$\varphi(y)=y$;
    \item[(CS4)] \label{axiomCS4} (\emph{Connectedness}) For all~$x,y\in X$, there are~$x_1,\ldots, x_n\in X$, $n\geq 1$, such that~$x\sim x_1\sim \ldots \sim x_n=y$.
\end{enumerate}

In this paper, while referring to a \emph{covered space structure} on a CW-complex, the set~$X$ will be always the set of all closed cells of the given CW-complex, and apartments will be always the sets of all closed cells of the given subcomplexes.

\begin{thm}[A covered space structure to~$\wtil{S}_\Gamma(\varphi)$]\label{thm:covSalv}
    Let~$\Gamma$ be a finite graph and~$\varphi\colon \caO\hookrightarrow U$ be a continuous open monomorphisms of topological groups~$\caO\leq_o U$. Let~$G=\caA_\Gamma(\varphi)$ and~$N=\LL U\GG_{\caA_\Gamma(\varphi)}$. Then~$\wtil{S}_\Gamma(\varphi)$ has a covered space structure with respect to:
    \begin{enumerate}
        \item the set of apartments~$\mathscr{A}=\{g\Sigma_0\mid g\in G\}=\{n\Sigma_0\mid n\in N\}$ (see~\Cref{lem:basicApartm}), where each apartment is endowed with the $\CAT(0)$ cube complex structure inherited from~$\wtil{\Sigma}_\Gamma$ (see~\Cref{prop:basicS}); and,
        \item the friendliness relation~$\sim$ defined, for all fundamental cubes~$Q_T,Q_W$ in~$\Sigma_0$ and~$h,g\in G$, as
    \begin{equation}\label{eq:friends}
        h\cdot Q_T \sim g\cdot Q_W\,\Longleftrightarrow\,h^{-1}g\in G_{Q_T}\cdot A_\Gamma\cdot G_{Q_W}
    \end{equation}
    and~$G_{Q_T}$ and~$G_{Q_W}$ are the setwise stabilisers of~$Q_T$ and~$Q_W$ in~$G$, respectively.
    \end{enumerate} 
\end{thm}

\begin{proof}
   Firstly, \eqref{eq:friends} is a reflexive and symmetric binary relation on all the closed cubes of~$\wtil{S}_\Gamma(\varphi)$, and
   axiom~\hypertarget{axiomCS1}{(CS1)} is clearly satisfied. To show \hypertarget{axiomCS3}{(CS3)}, note that for all~$n_1,n_2\in N$ the left action by $n_2n_1^{-1}$ induces a cubical isomorphism~$n_1\Sigma_0\to n_2\Sigma_0$ fixing each closed cube of~$n_1\Sigma_0\cap n_2\Sigma_0$ pointwise (see~\Cref{lem:Nact}).
   
    To prove~\hypertarget{axiomCS2}{(CS2)}, for all~$T,W\in \caV^f$ and~$h,g\in G$, we claim that
    \begin{equation}\label{eq:towardsCS2}
        hQ_T,gQ_W\subseteq \Sigma\text{ for some }\Sigma\in \caA\,\Longleftrightarrow\,h^{-1}g\in G_{Q_T}\cdot A_\Gamma\cdot G_{Q_W}.
    \end{equation}
   To show~($\Leftarrow$), if~$g=hu_1au_2$ for some~$u_1\in G_{Q_T}$, $a\in A_\Gamma$ and~$u_2\in G_{Q_W}$, then~$hQ_T=hu_1Q_T$ and~$gQ_W=hu_1aQ_W$ both belong to the apartment~$hu_1\Sigma_0$. To prove~($\Rightarrow$), since~$G=N\cdot A_\Gamma$ by~\Cref{prop:artin}, we may write~$h=n_1a_1$ and~$g=n_2a_2$ for some~$n_1,n_2\in N$ and~$a_1,a_2\in A_\Gamma$. Let~$n\in N$ satisfy~$hQ_T,gQ_W\subseteq n\Sigma_0$. Observe that
   \begin{equation}\label{eq:covering0}
       n_1a_1Q_T=hQ_T=nb_1Q_{T'}\quad\text{and}\quad n_2a_2Q_W=gQ_W=nb_2Q_{W'}
   \end{equation}
   for some~$b_1,b_2\in A_\Gamma$ and~$T',W'\in \caV^f$.
   By~\Cref{lem:Nact}, 
   \begin{equation}\label{eq:covering00}
       a_1Q_T=b_1Q_{T'}\quad\text{and}\quad a_2Q_W=b_2Q_{W'}.
   \end{equation}
   We claim that~$a_1=b_1$ and~$a_2=b_2$, and prove only~$a_1=b_1$ as the remaining part is analogous. In detail, since~$a_1Q_T=b_1Q_{T'}$, one has
   \begin{equation}\label{eq:covering1}
       a_1U=b_1s_1^{\veps_1}\cdots s_h^{\veps_h}U\quad\text{and}\quad a_1t_1^{\eta_1}\cdots t_k^{\eta_k}U=b_1U
   \end{equation}
    for some~$s_1,\ldots, s_h,t_1,\ldots, t_k\in S=V(\Gamma)$ and~$\veps_1,\ldots, \veps_h,\eta_1,\ldots, \eta_k\in\{0,1\}$. Having~$A_\Gamma\cap U=\{1\}$ by~\Cref{prop:artin}, we deduce that 
   \begin{equation}\label{eq:covering2}
       a_1=b_1s_1^{\veps_1}\cdots s_h^{\veps_h}\quad\text{and}\quad a_1t_1^{\eta_1}\cdots t_k^{\eta_k}=b_1.
   \end{equation}
   Applying the exponent map~$e\colon A_\Gamma\to \Z$ on~$A_\Gamma$ (see~\Cref{defn:exponent}) to both the equalities in~\eqref{eq:covering2}, we get~$e(a_1)=e(b_1)$ and~$\sum_{i=1}^h\veps_i=0$. This yields~$\veps_1=\ldots=\veps_h=0$ and hence~$a_1=b_1$. 

   Having~$a_1=b_1$ and~$a_2=b_2$, from~\eqref{eq:covering00} we deduce that~$Q_T=Q_{T'}$ and~$Q_W=Q_{W'}$.
   This and~\eqref{eq:covering0} now yield
   \begin{equation}\label{eq:covering4}
       n_1a_1Q_T=na_1Q_T\quad\text{and}\quad n_2a_2Q_W=na_2Q_W.
   \end{equation}
  By~\Cref{lem:Nact}, $n^{-1}n_1\in a_1G_{Q_T}a_1^{-1}$ and~$n^{-1}n_2\in a_2G_{Q_W}a_2^{-1}$. Therefore, $h=n_1a_1=na_1u_1$ and~$g=n_2a_2=na_2u_2$ for some~$u_1\in G_{Q_T}$ and~$u_2\in G_{Q_W}$. We conclude that~$h^{-1}g=u_1^{-1}a_1^{-1}a_2u_2\in G_{Q_T}\cdot A_\Gamma \cdot G_{Q_W}$, and this completes the proof of~\eqref{eq:towardsCS2}.

  It remains to verify~\hypertarget{axiomCS4}{(CS4)}. Since every closed cell of~$\wtil{S}_\Gamma(\varphi)$ and each of its vertex belong to a common apartment in~$\mathscr{A}$, it suffices to verify that~\hypertarget{axiomCS4}{(CS4)} holds for every pair of vertices in~$\wtil{S}_\Gamma(\varphi)$. Given any two vertices~$hU$ and~$gU$, notice that
  \begin{equation*}
      hU\sim gU \,\Longleftrightarrow\, h^{-1}g\in UA_\Gamma U.
  \end{equation*}
  Hence, for all~$h,g\in G$, since~$G=\langle U\cup A_\Gamma\rangle$, we have~$g=hu_1a_1\cdots u_na_n$ for some~$u_1,\ldots, u_n\in U$ and~$a_1,\ldots, a_n\in A_\Gamma$,~$n\geq 1$. Taking~$x_i=hu_1a_1\cdots u_ia_iU$ for every~$0\leq i\leq n$, we get~$hU=x_0\sim x_1\sim \ldots \sim x_n=gU$.
\end{proof}

\begin{rem}
    Unlike what happens for buildings, the relation in~\eqref{eq:friends} is \textbf{not} the universal relation, i.e., there might be cells in~$\wtil{S}_\Gamma(\varphi)$ that are mutually unrelated. 
\end{rem}


\section{Compactness properties and classifying spaces for locally compact RAAGs}\label{s:FPRAAGs}

\subsection{Compactness properties for locally compact RAAGs}\label{sus:FPRAAGs}

For the section, set once and for all a finite non-empty graph~$\Gamma=(S,E)$ and a continuous open monomorphism~$\varphi\colon \caO\hookrightarrow U$ of TDLC groups~$\caO\leq_o U$. 

The results of this section will give answers to~\Cref{ques:compactness} for RAAG presentations using various proof strategies. We shall freely use notation established in Sections~\ref{sus:basicRAAGs} and~\ref{s:Salvetti}. 

Recall that, thanks to \Cref{lem:FPpropTDLC}, all the statements below can be immediately extended to the more general case when~$U$ and hence~$\caO$ and~$\caA_\Gamma(\varphi)$ are LC~groups.

\begin{thm}\label{thm:FPa}
 Let~$U$ be a TDLC group and assume that~$\varphi(\caO)=\caO$. Let~$R$ be a unital ring and~$n\geq 1$. If~$\caO$ is of type~$\FP_{n-1}(R)$ and~$U$ is of type~$\FP_n(R)$, then~$\caA_\Gamma(\varphi)$ is of type~$\FP_n(R)$.
 Moreover, if~$\caO$ is type~$\FP_{n}(R)$ then
 \begin{equation}\label{eq:FPa}
     U \text{ is of type }\FP_n(R)\,\Longleftrightarrow\,\caA_\Gamma(\varphi)\text{ is of type }\FP_n(R).
 \end{equation}
 Analogous statements hold by replacing~$\FP_\bullet(R)$ with~$\tF_\bullet$.
\end{thm}
\begin{proof}
  By~\Cref{prop:stabs}, $\caA_\Gamma(\varphi)$ has a continuous cellular action on~$\wtil{S}_\Gamma(\varphi)$ with finitely many orbits, with vertex stabilisers topologically isomorphic to~$U$, and with pointwise stabilisers of positive dimensional cells topologically isomorphic to~$\caO$. Moreover, by~\Cref{thm:CAT0} the complex~$\wtil{S}_\Gamma(\varphi)$ is contractible.
  
  If~$\caO$ is of type~$\FP_{n-1}(R)$ and~$U$ is of type~$\FP_n(R)$, from~\cite[Proposition~4.5]{ccc} we deduce that~$\caA_\Gamma(\varphi)$ is of type~$\FP_n(R)$. If in particular~$\caO$ is of type~$\FP_n(R)$, \eqref{eq:FPa} now follows from~\Cref{lem:FPpropRetr}.
  
  It remains to prove the claim for the~$\tF_\bullet$-conditions. If~$\caO$ is of type~$\tF_1$ and~$U$ is of type~$\tF_2$, then~$\caA_\Gamma(\varphi)$ is of type~$\tF_2$ by~\cite[Theorem~4.9]{ccc}. For~$n\geq 3$, among compactly presented TDLC groups, being ``of type~$\tF_n$'' is equivalent to being ``of type~$\FP_n(\Z)$''; see~\cite[Proposition~3.13]{ccc}. Hence, if~$\caO$ is of type~$\tF_{n-1}$ and~$U$ is of type~$\tF_n$, then~$\caA_\Gamma(\varphi)$ is of type~$\tF_n$.
  Again, if in particular~$\caO$ is of type~$\tF_n$, \eqref{eq:FPa} now follows from~\Cref{lem:FPpropRetr}.
\end{proof}

\begin{thm}\label{thm:FPb}
    Let~$\Gamma$ be a finite connected graph with associated flag complex~$L(\Gamma)$, and assume that~$\caO=U$ is a TDLC group. 
   Let~$R$ be any unital ring and let~$1\leq n\leq \conn_h(L(\Gamma))+2$ where $\conn_h(L(\Gamma))$ is the integral homological connectivity of the geometric realisation of~$L(\Gamma)$.
   
   If~$U$ is of type~$\FP_n(R)$ then~$\caA_\Gamma(\varphi)$ is of type~$\FP_n(R)$. Similarly, if~$U$ is of type~$\tF_n$ then~$\caA_\Gamma(\varphi)$ is of type~$\tF_n$.
\end{thm}

\begin{proof}
    One proceeds analogously as in the proof of Theorem~\ref{thm:FPa}. The only difference is that, in this case, $\wtil{S}_\Gamma(\varphi)$ is simply connected and its integral reduced homology groups satisfy
    $$\wtil{\hH}_k(\wtil{S}_\Gamma(\varphi),\Z)=\{0\}, \qquad \forall\,k\leq \conn_h(L(\Gamma))+1,$$ 
    see~\Cref{thm:connPhi(U)=U}. 
    Hence, by the universal coefficient theorem in homology, for every unital ring~$R$ we get 
    $$\wtil{\hH}_k(\wtil{S}_\Gamma(\varphi),R)=\{0\}, \qquad \forall\,k\leq \conn_h(L(\Gamma))+1.$$ 
    Notice that, for~$1\leq n\leq \conn_h(L(\Gamma))+2$, the connectivity properties of~$\wtil{S}_\Gamma(\varphi)$ are sufficient to let us apply~\cite[Proposition~4.5, Theorem~4.9]{ccc} and proceed as in the proof of Theorem~\ref{thm:FPa}. 
\end{proof}
\begin{thm}\label{thm:FPc}
    Let~$\caO=U$ be a TDLC group. Assume there is a partition  
    \begin{equation*}
        S=V_1\sqcup \ldots \sqcup V_m, \quad m\geq 1
    \end{equation*}
   such that~$\ind_\Gamma(V_1)$ is connected and~$[V_i,V_j]=\{1\}$ for all~$1\leq i,j\leq n$ with~$i\neq j$. Then
    \begin{equation*}
        \Gamma=\ind_{\Gamma}(V_1)\vee \ldots \vee \ind_\Gamma(V_m)
    \end{equation*}
    (recall that the join of graphs is associative) and in particular~$\Gamma$ is connected. Moreover, for every unital ring~$R$ and every~$n\geq 1$, the following conditions are equivalent:
    \begin{enumerate}
        \item $\caA_{\ind_\Gamma(V_1)}(\varphi)$ is of type~$\FP_n(R)$;
        \item $\caA_{\ind_\Gamma(V_1\sqcup \ldots \sqcup V_i)}(\varphi)$ is of type~$\FP_n(R)$ for every~$1\leq i\leq m$;
        \item $\caA_\Gamma(\varphi)$ is of type~$\FP_n(R)$.
      \end{enumerate}
      An analogous equivalence holds by replacing~$\FP_n(R)$ with~$\tF_n$.
\end{thm}

\begin{proof}
 By design, for every~$1\leq i\leq m$ we have
 \begin{equation}\label{eq:FPc1}
     \Lambda_i:=\ind_\Gamma(V_1\sqcup \ldots \sqcup V_i)=\ind_\Gamma(V_1)\vee \ldots \vee \ind_\Gamma(V_i)
 \end{equation}
 with~$\Lambda_m=\Gamma$. In particular, for every~$i\geq 2$, we have~$\Lambda_i=\Lambda_{i-1}\vee \ind_\Gamma(V_i)$. Being~$\Lambda_1$ connected, we deduce inductively on~$i$ that~$\Lambda_i$ is connected for every~$i\geq 1$. 

 For~$1\leq i\leq m$, let~$T_i:=V_1\sqcup \ldots \sqcup V_i$. Hence~$[T_{i-1},T_i\setminus T_{i-1}]=\{1\}$ and~$\Lambda_i=\ind_\Gamma(T_i)$ for every~$i\geq 1$. By \Cref{prop:chainTopParab} we get a series of open subgroups of~$\caA_\Gamma(\varphi)$
 \begin{equation*}
     \caA_{\Lambda_1}(\varphi) \leq \caA_{\Lambda_2}(\varphi)\leq \ldots\leq \caA_{\Lambda_m}(\varphi)=\caA_\Gamma(\varphi)
 \end{equation*}
 such that, for every~$i\geq 2$,~$\caA_{\Lambda_{i-1}}(\varphi)$ normal in~$\caA_{\Lambda_i}(\varphi)$ and 
 \begin{equation}\label{eq:FPc2}
     \caA_{\Lambda_i}(\varphi)/\caA_{\Lambda_{i-1}}(\varphi)\cong A_{\ind_\Gamma(V_i)}.
 \end{equation}
 Finitely generated RAAGs are of type~$\tF_\infty$ (and hence of type~$\FP_\infty(R)$): indeed, they act freely and cocompactly on their universal Salvetti complex, which is contractible. From~\eqref{eq:FPc2} and~\cite[Theorem~3.20]{ccc} we deduce that, for every~$i\geq 2$,~$\caA_{\Lambda_i}(\varphi)$ is of type~$\FP_n(R)$  if and only if~$\caA_{\Lambda_{i-1}}(\varphi)$ is of type~$\FP_n(R)$. An analogous reasoning holds for~$\tF_n$. 
This proves the claimed equivalences.
\end{proof}

One can readily produce examples of topological RAAGs for which \Cref{thm:FPc} holds. Recall that a \emph{cluster graph} is a graph such that the connected components of its complement are complete graphs. The \emph{complement graph} of a graph~$\Gamma = (S,E)$ is the graph~$\Gamma^\ast = (S,E^\ast)$ with same vertex set~$S$ and, for all~$v,w\in S$ with $v\neq w$, such that~$\{v,w\} \in E^\ast$ if and only if~$\{v,w\} \notin E$. 

\begin{cor}\label{cor:FPd}
    Let~$\Gamma = (S,E)$ be a connected cluster graph, $\caO=U$~be a TDLC group, $R$~be a unital ring, and~$n\geq 1$. 

    If~$U$ is of type~$\FP_n(R)$ then~$\caA_\Gamma(\varphi)$ is of type~$\FP_n(R)$. Similarly, if~$U$ is of type~$\tF_n$ then~$\caA_\Gamma(\varphi)$ is of type~$\tF_n$.
\end{cor}

\begin{proof}
    Let~$S=S_0\sqcup \ldots \sqcup S_m$ be the partition induced on~$S$ by the vertices of the various connected components of the complement graph~$\Gamma^\ast$ of~$\Gamma$. 
    If~$m=0$, then~$\Gamma^\ast$ is complete and hence~$E=\emptyset$. In this case, $\caA_\Gamma(\varphi)\cong \caH_S(\varphi)$ (and~$|S|=1$, since~$\Gamma$ is both connected and totally disconnected in this case) and the statement follows from~\Cref{prop:graphofgrpsFP}. 
    
    Assume now~$m\geq 1$. For every~$i\neq j$, there are no edges in~$\Gamma^\ast$ connecting any vertex in~$S_i$ to any vertex in~$S_j$, i.e.,~$[S_i,S_j]=\{1\}$. Let
    $$V_1=S_0\sqcup S_1\quad\text{and}\quad V_i=S_i\quad\forall\,2\leq i\leq m.$$
    Then~$\ind_\Gamma(V_1)=\ind_\Gamma(S_0)\vee \ind_\Gamma(S_1)$ is connected and~$[V_i,V_j]=\{1\}$ for all~$i\neq j$. By Theorem~\ref{thm:FPc}, $\caA_\Gamma(\varphi)$ is of type~$\FP_n(R)$ (resp.~$\tF_n$) if and only if~$\caA_{\ind_\Gamma(V_1)}(\varphi)$ is of type~$\FP_n(R)$ (resp.~$\tF_n$).

    However,~$\ind_\Gamma(V_1)$ is the join of two totally disconnected subgraphs, namely~$\ind_\Gamma(V_0)$ and~$\ind_\Gamma(V_1)$. Hence, $\ind_\Gamma(V_1)$ is connected and chordal. In particular, its flag complex is contractible~\cite[Lemma~3.1]{docheng} and hence has infinite integral homological connectivity. By Theorem~\ref{thm:FPb}, if~$U$ is of type~$\FP_n(R)$ (resp.~$\tF_n$) then so is~$\caA_{\ind_\Gamma}(V_1)$, and this concludes the argument.
\end{proof}

\subsection{Classifying spaces for locally compact RAAGs \mbox{with~$\varphi(\caO)=\caO$}}\label{sus:EFRAAGs}
Classifying spaces for continuous actions of LC groups with prescribed isotropy play a prominent role in topology and homological algebra, and are relevant for instance in the Baum--Connes conjecture or in the Farrell--Jones conjecture. We refer the reader to L\"uck's survey~\cite{luck:survey} for an introduction to the topic. Below we only recall some basic notions from~\cite[\S 1]{luck:survey} that are necessary for the section.

\smallskip

Let~$G$ be an LC group. A \emph{family of subgroups} of~$G$ is a non-empty collection~$\caF$ of closed subgroups of~$G$ that is closed under conjugation and under intersection of finitely many of its members. Prominent examples are the family~$\mathcal{TRIV}$ consisting only of the trivial subgroup~$\{1\}$, the family~$\caK$ of all compact subgroups of~$G$ or, if~$G$ is TDLC, the family~$\caC\caO$ of all compact open subgroups. In this latter case, the non-emptyness of~$\caC\caO$ is due to van Dantzig's theorem.

Given a family~$\caF$ of subgroup of~$G$, a \emph{model for the classifying $G$-CW-complex~$\spE_\caF G$ for $G$ with respect to~$\caF$} is any $G$-CW complex (in the sense of~\cite[Definition~1.1]{luck:survey}) whose $G$-pointwise cell stabilisers are all in~$\caF$ and such that, for every~$H\in \caF$, the set of $H$-fixed points is weakly contractible (hence non-empty), see~\cite[Theorem~1.9]{luck:survey}. 
As usual, if~$\caF$ is the family of all compact subgroups of~$G$, we write~$\underline{\operatorname{E}}G$ in place of~$\spE_{\caF}(G)$.

\begin{thm}[A model for $\spE_\caF\caA_\Gamma(\varphi)$]\label{thm:EFGspace}
    Let~$\Gamma$ be a finite graph and~$\varphi\colon \caO\hookrightarrow U$ be a continuous open monomorphism of LC groups~$\caO\leq_oU$ such that~$\varphi(\caO)=\caO$ and~$|U:\caO|<\infty$. 
    
    Let~$\caF$ be any family for~$\caA_\Gamma(\varphi)$ containing~$U$ and consisting of closed subgroups~$H\leq \caA_\Gamma(\varphi)$ such that~$|H:H\cap gUg^{-1}|<\infty$ for some~$g\in \caA_\Gamma(\varphi)$.
    Then~$\wtil{S}_\Gamma(\varphi)$ is a model for~$\spE_\caF\caA_\Gamma(\varphi)$.
\end{thm}
\begin{proof}
    Firstly, by~\Cref{lem:UaUa-1} we have~$\varphi(\caO)=\caO\in \caF$.
    By~\Cref{prop:stabs} and~\Cref{thm:CAT0}, $X:=\wtil{S}_\Gamma(\varphi)$ is a $\CAT(0)$ cube complex and~$\caF$ contains all the pointwise cell stabilisers of the~$\caA_\Gamma(\varphi)$-action on~$\wtil{S}_\Gamma(\varphi)$. 
    
 Moreover, by Lemma~\ref{lem:Ufinorbs}, $U$ (and hence all its $\caA_\Gamma(\varphi)$-conjugates) acts on~$\wtil{S}_\Gamma(\varphi)$ with bounded orbits. We deduce that every~$H\in \caF$ has bounded orbits on~$\wtil{S}_\Gamma(\varphi)$. Indeed, there is~$g\in \caA_\Gamma(\varphi)$ such that $|H:H\cap gUg^{-1}|<\infty$, and~$H\cap gUg^{-1}\leq gUg^{-1}$ has bounded orbits on~$\wtil{S}_\Gamma(\varphi)$.
 
 By~\cite[Corollary~II.2.8(1)]{BridsonHaefliger}, for every~$H\in \caF$ the set of fixed point~$X^H$ is a non-empty convex subset of~$X$. In particular, $X^H$ is contractible. The statement now follows from~\cite[Theorem~1.9]{luck:survey}.
\end{proof}

By~\Cref{lem:Ufinorbs}, if~$|U:\caO|<\infty$ then the smallest family for which Theorem~\ref{thm:EFGspace} applies is
 \begin{equation}\label{eq:Ffamily}
       \Big\{\textstyle{\bigcap_{g\in F}}gUg^{-1}\,\Big\vert\, F\subseteq \caA_\Gamma(\varphi)\text{ finite non-empty  subset}\Big\}.
    \end{equation}

If~$U$ is compact, another family for which~\Cref{thm:EFGspace} applies is the one of all compact subgroups of~$\caA_\Gamma(\varphi)$, as stated in what follows.

\begin{cor}\label{cor:ECOGspace}
     Let~$\Gamma$ be a finite graph,~$U$ be a compact group, and~$\varphi\colon \caO\hookrightarrow U$ be a continuous open monomorphism such that~$\varphi(\caO)=\caO$. Then~$\wtil{S}_\Gamma(\varphi)$ is a model for $\underline{\spE}\caA_\Gamma(\varphi)$.
\end{cor}
\begin{proof}
Since~$U$ is open in~$\caA_\Gamma(\varphi)$, every compact subgroup~$K\leq\caA_\Gamma(\varphi)$ satisfies~$|K:K\cap gUg^{-1}|<\infty$ for every~$g\in \caA_\Gamma(\varphi)$. Similarly, $|U:\caO|<\infty$. Hence~\Cref{thm:EFGspace} applies.
\end{proof}

Note that by \Cref{thm:connPhi(U)=U} there are many more examples where $\wtil{S}_\Gamma(\varphi)$ is contractible without requiring $\phee(\caO)=\caO$, though it is unclear to us whether fixed point sets are contractible.

\begin{ques}
Under which general conditions on~$\varphi$ and~$\Gamma$ can we say that~$\wtil{S}_\Gamma(\varphi)$ is a model for~$\spE_\caF G$ for some family~$\caF$?
\end{ques}
\subsection{Rational discrete cohomological dimension for TDLC RAAGs}\label{sus:cdQRAAGs}
By \Cref{prop:artin} and the Lyndon--Hochschild--Serre spectral sequence, the following rough estimates for the rational discrete  cohomological dimension hold: $$\mathrm{max}\{\ccd_\Q(A_\Gamma),\ccd_\Q(\LL U\RR_{G_\phee})\}\leq \ccd_\Q(\caA_\Gamma(\phee))\leq \ccd_\Q(A_\Gamma)+\ccd_\Q(\LL U\RR_{G_\phee}).$$ 
 In the case $\Gamma$ is connected and $\caO=U$, the normal closure is an increasing union of subgroups isomorphic to $U$ (see~\Cref{prop:normUconn}) and then we can deduce the better estimates:
    $$ \mathrm{max}\{\ccd_\Q(A_\Gamma),\ccd_\Q(U)\}\leq \ccd_\Q(\caA_\Gamma(\phee))\leq \ccd_\Q(A_\Gamma)+\ccd_\Q(U)+1.$$
 If~$U$ is profinite, the generalised universal Salvetti complex allows us to compute precisely the rational discrete cohomological dimension of $\caA_\Gamma(\phee)$ as follows.
\begin{thm}\label{thm:cdQRAAG}
    Let~$\Gamma$ be a finite graph,~$U$ be a profinite group and~$\varphi\colon \caO\hookrightarrow U$ be a continuous open monomorphism with~$\caO\leq_o U$. If~$\wtil{S}_\Gamma(\varphi)$ is $\Q$-acyclic\footnote{I.e., all its reduced homology groups with coefficients in~$\Q$ are trivial.}, then
    \begin{equation}\label{eq:ccdQRAAG}
    \ccd_\Q(\caA_\Gamma(\varphi))=\dim\wtil{S}_\Gamma(\varphi)=\ccd_\Q(A_\Gamma).
    \end{equation}
    If~$\wtil{S}_\Gamma(\varphi)$ is a model for $\underline{\spE}\caA_\Gamma(\varphi)$, then it is a model of minimal dimension.
\end{thm}
\begin{proof}
   By~\cite[Fact~2.7]{ccc}, the augmented cellular chain complex of every  $\Q$-acyclic $G$-CW-complex~$X$ with coefficients in~$\Q$ is a projective resolution of~$\Q$ in~$\QGdis$, where~$G=\caA_\Gamma(\varphi)$. (Note that in~\cite[Fact~2.7]{ccc} the hypothesis that the complex is contractible can be weakened with no changes by asking that it is $\Q$-acyclic.)
   Hence~$\ccd_\Q(G)\leq \dim X$. It follows that if~$\wtil{S}_\Gamma(\varphi)$ is $\Q$-acyclic then
   \begin{equation}\label{eq:cdQ1}
       \ccd_\Q(G)\leq \dim\wtil{S}_\Gamma(\varphi).
   \end{equation}
 Similarly, for every model~$X$ for~$\underline{\spE}G$, since~$X$ is contractible (because~$X$ is the set of fixed points of the trivial group, which is compact in~$G$) we have
 \begin{equation}\label{eq:cdQ2}
     \ccd_\Q(G)\leq \dim X.
 \end{equation}

  It remains to show that the inequality in~\eqref{eq:cdQ1} is an equality. By \Cref{prop:artin}, $A_\Gamma$ is a discrete and hence closed subgroup of~$\caA_\Gamma(\varphi)$. By~\cite[Proposition~3.7(c)]{CastellanoWeigel0} and~\cite[Corollary~3.2.2]{CharneyDavisKpi1s} (adapting the proof to~$\Q$ as ring of coefficients), we also have
   \begin{equation*}
       \dim \wtil{S}_\Gamma(\varphi)=\ccd_\Q(A_\Gamma)\leq \ccd_\Q(G).
   \end{equation*}
   This yields~\eqref{eq:ccdQRAAG}.

   The second part of the statement follows from the fact that every model~$X$ for~$\underline{\spE}G$ satisfies~$\ccd_\Q(G)\leq \dim X$. 
\end{proof}


\section{The Bieri--Stallings groups}\label{s:BiSta}

In this section we give a TDLC version of a prominent result due to 
John Stallings (case~$n=3$) and Robert Bieri (arbitrary~$n\geq 1$), which we now recall. Consider the free group of rank two,~$F_2$. Looking at the exponent map $e_1 : F_2 \onto \Z$ sending both free generators of~$F_2$ to~$1$, it is clear that~$\ker(e_1)$ is infinitely generated, hence of homological type~$\FP_0(\Z)$ but not of type~$\FP_1(\Z)$. 
It is an exercise to check that $\ker(e_2)$ is finitely generated but not of type $\FP_2(\Z)$, where 
\[e_n : \underbrace{F_2 \times F_2 \times \ldots \times F_2}_{n \text{ factors}} \onto \Z\] 
is the exponent map again sending the free generators of (each copy of)~$F_2$ to~$1\in \Z$. The first recorded example of a group of type $\FP_2(\Z)$ (in fact, finitely presented) but not $\FP_3(\Z)$ was precisely $\ker(e_3)$, as shown by Stallings \cite{StallingsSB2}. Bieri later extended this, proving that $\ker(e_n)$ --- which is sometimes denoted $SB_{n-1}$ in the literature --- is of type $\FP_{n-1}(\Z)$ but not of type $\FP_n(\Z)$; see~\cite[pp.~37--40]{bieri}. 

Note that $F_2 \times \ldots \times F_2$ is a RAAG. Bestvina and Brady generalised the results of Stallings and Bieri, showing that one can fully control finiteness properties of the kernel of the exponent $e : A_\Gamma \onto \Z$ of a RAAG $A_\Gamma$. The exact finiteness properties of $\ker(e)$ are calculated from the homological and homotopical connectivity of $L(\Gamma)$, where $L(\Gamma)$ is the flag complex of $\Gamma$ \cite[Main Theorem]{BestvinaBrady}. 

One might thus be rightfully tempted to replicate the Bestvina--Brady strategy into our setting to create further examples of LC groups with varying compactness properties. A major step to turn connectivity of $L(\Gamma)$ into good finiteness properties of $\ker(e)$ is to show that ascending and descending links in the universal Salvetti complex are as highly connected as $L(\Gamma)$; see \cite[Theorems~4.1 and~5.12]{BestvinaBrady}. However, as seen in \Cref{obs:BBgps}, this strategy cannot be reproduced in our context since (ascending) links in our generalised universal Salvetti complex typically exhibit worse connectivity than $L(\Gamma)$.

Thus, to construct further examples of groups with controlled finiteness properties, we revisit the Bieri--Stallings machinery. To extend Stalling's approach, Bieri used a more explicit description of~$SB_{n-1} = \ker(e_n)$ that identifies a free group of countable rank as a normal subgroup. We shall adopt this perspective further below to build generalised versions of~$SB_{n-1}$ over continuous open monomorphisms~$\phee : U \into U$. The restriction to the case $\caO=U$ is an artefact of our proof, see~\Cref{rem:O=Unec}. 
Specialising to $U$ a TDLC group, we obtain an analogue of the Bieri--Stallings theorem (see~\Cref{thm:BieriStallings}), noting that the base group~$U$ can also be chosen discrete.

\subsection{Bieri--Stallings groups over~$\phee$}
\begin{defn}[{Generalised Bieri--Stallings groups}] \label{def:BieriStallings}
Let~$\phee : U \into U$ be a continuous open monomorphism of a topological group~$U$ into itself, and consider the free group of countable rank~$F_\infty \cong F(\Z)$ with explicit free basis
\begin{equation}\label{eq:Xbasis}
    X = \{ \ldots, x_{-2}, x_{-1}, x_0, x_1, x_2, \ldots\}
\end{equation}
and endowed with the discrete topology. For each~$i \in \Z$, consider a copy of the free group~$F_2 \cong \spans{s_i, t_i \vert \leer}$ of rank two and the corresponding generalised presentation~$\mc{H}_{\{s_i, t_i\}}(\phee)$ over~$\phee$. 
Given~$n \in \Z_{\geq 1}$, we let 
\[D_n(\phee) := \prod_{i=1}^n \mc{H}_{\{s_i, t_i\}}(\phee)\]
denote the $n$-fold product of the groups $\mc{H}_{\{s_i, t_i\}}(\phee)$ endowed with the product topology. For~$n\in \Z_{\geq 1}$, the \emph{$n$-th (generalised) Bieri--Stallings group over $\phee$} is the 
topological semi-direct product (see~\Cref{sec:basicTopGroups})
\begin{equation}\label{eq:SBn}
    SB_n(\phee):= F_\infty \rtimes D_n(\phee),
\end{equation}
where each $s_i$ and $t_i$ act on the basis $X$ of $F_\infty$ by a $+1$ shift and the canonical images of the base group $U$ in each topological HNN-extension $\mc{H}_{\{s_i, t_i\}}(\phee)$ act trivially on $F_\infty$. We also set~$SB_0(\varphi):= F_\infty$.
\end{defn}

Let us mention some facts about generalised Bieri--Stallings groups. First note that the given action of~$D_n(\phee)$ on~$F_\infty$ is indeed continuous since the pointwise stabilisers of this action contain the $n$-fold product~$U \times \ldots \times U \leq D_n(\phee)$, which is open by construction and by~\Cref{fact:HNNtop}\eqref{fact:HNNtopTopology}. Moreover, each factor~$\mc{H}_{\{s_i,t_i\}}(\phee)$ embeds as open subgroup of~$SB_n(\phee)$ as it does so in~$D_n(\phee)$.

Since we are interested in the finiteness properties of~$SB_n(\varphi)$, the case where~$\varphi(U)=U$ turns out to be completely understood.
\begin{lem}[Generalised Bieri--Stallings for~$\varphi(U)=U$]\label{lem:gSBsurj}
 Let~$n\geq 1$. If the relevant continuous open monomorphism~$\varphi\colon U\hookrightarrow U$ is surjective, then
    \begin{equation}\label{eq:gSBsurj}
        SB_n(\varphi)\cong U^n\rtimes_{\Phi} SB_n,
    \end{equation}
    where~$U^n$ is the $n$-fold direct product of~$U$ (with the product topology), $SB_n\cong SB_n(\iid_{\{1\}})\cong F_\infty\rtimes \prod_{i=1}^n\langle s_i,t_i\rangle$ is the standard Bieri--Stallings group, and the action~$\Phi$ of~$SB_n$ on~$U^n$ is defined by~$\Phi\vert_{F_\infty}\equiv \iid_{F_\infty}$ and, for every~$1\leq i\leq n$, by~$\Phi(s_i)=\Phi(t_i)=\varphi^n$ where~$\varphi^n$ is the $n$-fold product of~$\varphi$.
\end{lem}
\begin{proof}
  Since~$\varphi$ is a topological isomorphism, for every~$1\leq i\leq n$ we have
  \begin{equation*}
      \caH_{\{s_i,t_i\}}(\varphi)\cong U\rtimes_\varphi\langle s_i,t_i\rangle,
  \end{equation*}
  see~\Cref{prop:artin}. The subscript~``$\varphi$" before reminds that both~$s_i$ and~$t_i$ act on~$U$ via~$\varphi$.
  Hence, 
  \begin{equation*}
     D_n(\varphi)\cong \prod_{i=1}^n (U\rtimes_\varphi\langle s_i,t_i\rangle)\cong U^n\rtimes_{\Psi} \Big(\prod_{i=1}^n\langle s_i,t_i\rangle\Big),
  \end{equation*}
  where~$\Psi(s_i)=\Psi(t_i)=\varphi^n$ for every~$i$. We conclude that
  \begin{equation*}
  \begin{split}
      SB_n(\varphi) & =F_\infty\rtimes D_n(\varphi)\cong F_\infty\rtimes \Big(U^n\rtimes_{\Psi} \Big(\prod_{i=1}^n\langle s_i,t_i\rangle\Big)\Big) \\
       & \cong U^n\rtimes_{\Phi}\Big(F_\infty\rtimes \prod_{i=1}^n\langle s_i,t_i\rangle\Big)\cong U^n\rtimes_\Phi SB_n. \qedhere
      \end{split}
  \end{equation*}
\end{proof}

Our aim is to establish the following: 

\begin{thm}[A generalised Bieri--Stallings Theorem] \label{thm:BieriStallings}
Let~$R \in \{\Z, \Q\}$ and $\phee : U \into U$ be a continuous open monomorphism of TDLC groups. If~$U$ is $\Q$-acyclic and of type~$\FP_n(R)$, then $SB_n(\phee)$ is of type $\FP_n(R)$ but not of type $\FP_{n+1}(R)$.
\end{thm}
\begin{rem}\label{rem:easyBSthm}
    If~$\varphi$ is surjective, \Cref{thm:BieriStallings} follows directly from the analogue result for~$SB_n$~\cite[p.~37-40]{bieri}, \Cref{lem:gSBsurj} and~\cite[Theorem~3.20]{ccc}. Note however that, apart from this case, it is not clear whether we can express~$SB_n(\varphi)$ as a topological group extension of two TDLC groups with ``good" finiteness properties (and then apply~\cite[Theorem~3.20]{ccc}).
\end{rem}

We shall split the proof of~\Cref{thm:BieriStallings} into two parts, not dissimilar from Bieri's~\cite{bieri} original methods. 

To simplify notation towards the proof of~\Cref{thm:BieriStallings}, for any set~$T=\{t_1,...,t_n\}$ of size~$n$, we reserve the notation~$\caH_T(\varphi)$ when we consider continuous open monomorphisms~$\varphi: U \into U.$ However, when we use a different base group~$G$ and continuous open monomorphisms~$\phi: G \into G$, we will 
use the more standard HNN-notation~$G\ast _\phi^{t_1,...,t_n}$. 

The following observations will be needed later on. 
\begin{lem}\label{lem:semidirect}
Let $G$ be a topological group on which $\caH_{\{s,t\}}(\phee)$ acts in the following way: $U$~acts trivially on~$G$ and $s$~and~$t$ both act by the same topological automorphism~$\phi' \in \Aut(G)$ (see~\Cref{fact:HNNtop}\eqref{fact:HNNtopUniversalproperty}). Then there is a topological group isomorphism 
\[G \rtimes_{\phi'} \caH_{\{s,t\}}(\phee) \cong (G \times U)\ast_\phi^{s,t},\]
where~$\phi(g)=\phi'(g)$ and~$\phi(u)=\varphi(u)$ for all~$g \in G, u\in U$.
\end{lem}

\begin{proof}
  Let~$\LL U \GG$ be the normal closure of the base group~$U$ in~$\caH_{\{s,t\}}(\phee)$. By~\Cref{prop:artin}, $\caH_{\{s,t\}}(\phee) \cong \LL U \GG \rtimes_\varphi \langle s,t\rangle$. 
  Furthermore, since~$U$ (and hence~$\LL U\GG$) acts trivially on~$G$, we have 
  \[G \rtimes_{\phi'} \caH_{\{s,t\}}(\phee) \cong \big(G \times \LL U \GG\big) \rtimes_{\phi'} \langle s,t\rangle.\]
  Finally, by~\Cref{prop:artin} we have 
  \[\big(G \times \LL U \GG\big) \rtimes_{\phi'} \langle s,t\rangle \cong (G \times U)\ast_\phi^{s,t}\]
  as required.
\end{proof}

\begin{rem}\label{rem:O=Unec}
    For the definition of~$\phi$ on~$U$ in~\Cref{lem:semidirect}, it is crucial that the domain of~$\varphi$ is~$U$.
\end{rem}

\begin{prop}\label{prop:B-groupstructure} 
Let~$n\in\Z_{\geq 1}$ and let~$\phee : U \into U$ be a continuous open monomorphism of topological groups. Then one has the topological group isomorphism 
\[SB_n(\varphi)\cong \big(SB_{n-1}(\varphi)\times U\big)\ast_\phi^{s_n,t_n},\]
where the endomorphism~$\phi\colon SB_{n-1}(\varphi)\times U\to SB_{n-1}(\varphi)\times U$ is recursively defined as 
   $\varphi$ on the factor~$U$, and on~$SB_{n-1}(\varphi) = F_\infty\rtimes D_{n-1}(\varphi)$ 
     as the map~$\phi': SB_{n-1}(\varphi) \to SB_{n-1}(\varphi)$ given by: 
        \begin{itemize}
        \item the shift~$\sigma\colon F_\infty\to F_\infty$, $\sigma(x_j)=x_{j+1}$ for every~$j\in \Z$; and 
        \item the identity on $D_{n-1}(\varphi)$. 
        \end{itemize}
\end{prop}
\begin{proof} One has 
\begin{equation*}
\begin{split}
    SB_n(\varphi) &=  F_\infty \rtimes D_n(\varphi) \\
    & \cong F_\infty \rtimes \left(\prod_{i=1}^{n-1} \caH_{\{s_i,t_i\}}(\phee)  \times \caH_{\{s_n,t_n\}}(\phee)\right) \\
    & \cong \left(F_\infty \rtimes \prod_{i=1}^{n-1} \caH_{\{s_i,t_i\}}(\phee) \right) \rtimes \caH_{\{s_n,t_n\}}(\phee) \\
    & \cong SB_{n-1}(\varphi) \rtimes_{\phi'} \caH_{\{s_n,t_n\}}(\phee),
    \end{split}
\end{equation*}
where the last semidirect product in line three is given by the action of~$\caH_{\{s_n,t_n\}}(\varphi)$ on~$F_\infty$ via the shift~$\sigma$ and a trivial action on~$\prod_{i=1}^{n-1}\caH_{\{s_i,t_i\}}(\phee)$. 
 Since~$\sigma \in \Aut(F_\infty)$, note that~$\phi' \in \mathrm{Aut}(SB_{n-1}(\varphi))$. Hence we can apply~\Cref{lem:semidirect}.
\end{proof}

\subsection{Generalising Bieri's theorem} 

To prove~\Cref{thm:BieriStallings}, one could adapt Robert~Bieri's original method to our setting. The positive part of the result requires our analogue of the Mayer--Vietoris sequence (\Cref{thm:mv}), \Cref{lem:cohomologyiso} as well as a TDLC analogue of the Bieri--Eckmann criterion for~$\FP_n(\Q)$~\cite[Theorem 4.2]{ccc}. The proof of this latter criterion for~TDLC groups is based on the fact that $\RGdis$ has enough projectives, a property that is guaranteed if~$R=\Q$~\cite[Proposition~3.2]{CastellanoWeigel0}.
Hence, to obtain a general result, we adapt a version of Bieri's result based on an approach shown to us by Ged Corob~Cook avoiding the use of the Bieri--Eckmann criterion.

\begin{prop}\label{prop:bieri-Z} 
Let~$R$ be a commutative ring with unit. Let~$U$ be a TDLC group of type~$\FP_n(R)$ together with a continuous open monomorphism $\varphi: U \hookrightarrow U$. Then~$SB_n(\varphi)$ is of type~$\FP_n(R)$.
\end{prop}

\begin{proof}
We proceed by induction on~$n\geq 1$. In order to simplify the notation, we write~$SB_n$ in place of~$SB_n(\varphi)$. The TDLC~group~$SB_1$ is generated by~$U\cup\{x_0,s_1,t_1\}$, where $x_0$ is a member of the given basis~$X$ of~$F_\infty$ (see~\eqref{eq:Xbasis}).
Hence, if~$U$ is of type~$\FP_1(R)$ (i.e., compactly generated), then~$D_1(\varphi)$ is compactly generated and hence of type~$\FP_1(R)$.
We therefore assume~$n\geq 2$. 

Following \Cref{prop:B-groupstructure} we may consider the following diagram.
For every pair of groups connected by an edge, the lower one is subgroup of the upper one. Sometimes, when the relevant subgroup is normal in the bigger group, the connected edge is labelled by the quotient.
\begin{equation}\label{eq:diag}
    \xymatrix{
    &{}^{(SB_{n-1}\times U)\ast_\phi^{t_n}}\ar@{-}[dd]^{D_{n-1}(\varphi)}\\
    {}^{SB_{n-1}\times U}\ar@{-}[ur]\ar@{-}[dd]&\\
    &{}^{(F_\infty\times U)\ast^{t_n}_{(\sigma,\varphi)}}\\
    {}^{F_\infty\times U}\ar@{-}[d]\ar@{-}[ur]\\
    1\ar@{-}[uur]_{F_2\times (U\ast_\varphi^{t_n})}&}
\end{equation}
Here $(F_\infty\times U)\ast_{(\sigma,\varphi)}^{t_n}$ denotes the ascending HNN-extension with base group $F_\infty\times U$, stable letter~$t_n$ and continuous open monomorphism~$(\sigma,\varphi)\colon F_\infty\times U\hookrightarrow F_\infty\times U$ given by the product of the shift map~$\sigma\colon F_\infty\to F_\infty$ (see~\Cref{prop:B-groupstructure}) and~$\varphi\colon U\hookrightarrow U$.

The diagram~\eqref{eq:diag} gives us a short exact sequence of topological groups 
\[(F_\infty\times U)\ast_{(\sigma,\varphi)}^{t_n} \hookrightarrow (SB_{n-1}\times U)\ast_\phi^{t_n} \twoheadrightarrow D_{n-1}(\varphi).\]
 By an argument analogous to \Cref{prop:B-groupstructure}, $(F_\infty\times U)\ast_{(\sigma,\varphi)}^{t_n}$ is of type $\FP_n(R).$ 

Since $D_{n-1}(\varphi)$ is a finite product of groups of type~$\FP_n(R)$ (see~\Cref{cor:HpheeisFPn}), it is also of type $\FP_n(R)$ (see~\cite[Theorem~3.20]{ccc}). 
An application of~\cite[Theorem 3.20]{ccc} yields that~$(SB_{n-1}\times U)\ast_\phi^{t_n}$ is of type~$\FP_n(R)$ as well. 
The result follows from~\Cref{prop:B-groupstructure} and~\Cref{prop:graphofgrpsFP}.
\end{proof}

The negative part of the Bieri--Stallings result is slightly more straightforward. We do the same as Stallings, who showed the stronger result that his group in fact has third homology of infinite rank.

\begin{prop}\label{prop:notFPn+1}
    Let~$U$ be a $\Q$-acyclic TDLC~group together with a continuous open monomorphism~$\varphi\colon U\hookrightarrow U$. Then~$SB_n(\varphi)$ is not of type~$\FP_{n+1}(\Q)$.
\end{prop}

\begin{proof}
    Again we write~$SB_n$ instead of~$SB_n(\phee)$. By~\Cref{lem:fin dim}, it suffices to prove that $\dH_{n+1}(SB_n,\Q)$ has infinite dimension. We first use induction to prove that 
    \begin{align}\label{eq:ind_hyp}
    \begin{split}
      \dH_{k+2}(SB_n \times U, \Q) = & \, 0 = \dH_{k+2}(SB_n,\Q) \quad \text{ for all } k\geq n, \text{ and} \\ 
      \dH_{n+1}(SB_n,\Q) \cong & \dH_{n+1}(SB_n\times U,\Q) \cong \bigoplus_{m\in\Z} \Q.
      \end{split}
   \end{align}
To this end we will use the K\"unneth Theorem~\ref{thm:kunneth}. For $SB_0=F_\infty$ the claims in~\eqref{eq:ind_hyp} are clear. For $n\geq1$, as $SB_n\cong(SB_{n-1}\times U)\ast_\phi^{t_n,s_n}$, \Cref{thm:mv} yields 
\[\xymatrix{\cdots\ar[r] & \dH_{n+1}(SB_{n-1}\times U,\Q)\ar[r]&\dH_{n+1}(SB_n,\Q)\ar[d]\\
  & & \oplus_{i=1}^2\dH_{n}(SB_{n-1}\times U,\Q)\ar[d]^{(f_1,f_2)}\\
    & \cdots& \dH_{n}(SB_{n-1}\times U,\Q)\ar[l]\\
    }\]
 Here $\dH_{n+1}(SB_{n-1}\times U,\Q)=0$ and $\dH_{n}(SB_{n-1}\times U,\Q)\cong\oplus_{m\in\Z}\Q$, hence also $\dH_{n+1}(SB_n,\Q)$ must be isomorphic to $\oplus_{m\in\Z}\Q$ since $f_1$ and $f_2$ are the same morphism $\dH_{n+1}(SB_n,\Q)\to \dH_{n+1}(SB_n,\Q)$.
\end{proof}
\begin{proof}[Proof of Theorem~\ref{thm:BieriStallings}]
Immediate from Propositions~\ref{prop:bieri-Z} and~\ref{prop:notFPn+1}.
\end{proof}

\begin{prop}\label{prop:admi+cdQSBn}
    Let~$\varphi\colon U\hookrightarrow U$ be a continuous open monomorphism of a compactly generated TDLC group~$U$ and let~$n\geq 1$. Then the compactly generated TDLC group~$SB_n(\varphi)$ has infinite asymptotic dimension and its rational discrete cohomological dimension satisfies:
    \begin{equation}\label{eq:cdQSBn}
        \ccd_\Q(SB_n(\varphi))\in\{n\cdot \ccd_\Q(\caH_{\{s,t\}}(\varphi)), n\cdot \ccd_\Q(\caH_{\{s,t\}}(\varphi))+1\}.
    \end{equation}
\end{prop}
\begin{proof}
    By design, $SB_n(\varphi)$ contains~$F_\infty$ as a closed subgroup. Hence, the asymptotic dimension of~$SB_n(\varphi)$ is at least the asymptotic dimension of~$F_\infty$~\cite[Corollary~4.A.7]{cdlh:metric}, which is infinite.

    Moreover, by~\cite[Proposition~3.7(b)]{CastellanoWeigel0} we deduce that
    \begin{equation*}
       \ccd_\Q(SB_n(\varphi))\leq \ccd_\Q(F_\infty)+\ccd_\Q(D_n(\varphi)).
    \end{equation*}
         By Stallings--Swan theorem we have $\ccd_\Q(F_\infty)=1$.  Moreover, by additivity of $\ccd_\Q(\argu)$~\cite[Theorem~5.5]{bieri}, 
    \begin{equation*}
        \ccd_\Q(D_n(\varphi))=n\cdot \ccd_\Q(U\ast_\varphi^{s,t}).
    \end{equation*}
    Finally, as $D_n(\varphi)$ embeds as a closed subgroup in~$SB_n(\varphi)$, by~\cite[Proposition~3.7(c)]{CastellanoWeigel0} we have
    \begin{equation*}
        \ccd_\Q(D_n(\varphi))\leq \ccd_\Q(SB_n(\varphi)).\qedhere
    \end{equation*}
\end{proof}
In view of~\eqref{eq:cdQSBn}, by \Cref{thm:mv} and~\cite[Proposition~3.7(c)]{CastellanoWeigel0} it is worth noting that
\begin{equation*}
    \ccd_\Q(\caH_{\{s,t\}}(\varphi))\in\{\ccd_\Q(U),\ccd_\Q(U)+1\}.
\end{equation*}

\subsection{Examples}\label{sus:exMonosO=U}
\Cref{thm:BieriStallings} gives us another recipe to construct TDLC groups~$SB_n(\varphi)$ with prescribed finiteness properties. 
As long as~$U$ is non-discrete,  the group~$D_n(\varphi)\cong SB_n(\varphi)/F_\infty$ is non-discrete and \emph{a fortiori}~$SB_n(\varphi)$ will be non-discrete. 
Moreover, in light of~\Cref{lem:gSBsurj} and \Cref{rem:easyBSthm}, the key objective is to construct continuous open monomorphisms~$\varphi\colon U\hookrightarrow U$ that fail to be surjective. We will call such a $\varphi$ an {\em shrinking monomorphism} of $U$.

Recall that a group $G$ is called \emph{co-Hopfian} if every monomorphism~$\varphi\colon G \hookrightarrow G$ is surjective. We introduce a variant for topological groups:

\begin{defn}\label{def:opencoHopfian} 
A topological group~$U$ is called \emph{open co-Hopfian} if every continuous open monomorphism~$U\hookrightarrow U$ is surjective.
\end{defn}
Examples of open co-Hopfian groups are provided by any non-compact group $U$  such that every proper open
subgroup of $U$ is compact. E.g., $\mathrm{SL}_n(\Q_p)$ and $\mathrm{SL}_n(\R)$. Other examples are provided in~\cite{CdM}, where the authors prove that  every proper open
subgroup of the simple Burger--Mozes group $U(F)^+$ is compact if and only if  $F \leq \mathrm{Sym}({1,\ldots, d})$ is primitive. 
The latter condition, i.e., that every proper open subgroup
is compact, is well-known to be satisfied by any locally compact group having the Howe--Moore property; we refer the reader to \cite{HM} for the definition and more information about the Howe--Moore property. The following list contains standard examples of LC~groups having the Howe--Moore. 
\begin{enumerate}
\item Connected, non-compact, simple real Lie groups with finite centre~\cite{HM};
\item Isotropic simple algebraic groups over non-Archimedean local fields~\cite{HM};
\item Closed, topologically simple, strongly transitive and type-preserving subgroup of~$\Aut(\Delta)$, where $\Delta$ is a locally finite thick affine building~\cite{ciob}.

This generalises the claim that closed, topologically simple subgroups of $\Aut(T)$ with a $2$-transitive action on the boundary of the bi-regular tree~$T$ that has valence $\geq 3$ at every vertex~\cite{BM}. 
\end{enumerate}
\subsubsection{Profinite groups}\label{ex:BSUcompact}
The following are examples of profinite groups that fail to be open co-Hopfian:

    \begin{enumerate}
        \item $U=\Z_p$, the additive group of~$p$-adic integers, and~$\varphi\colon x\mapsto p^kx$ for some~$k\geq 1$ (see~\Cref{ex:monoOU}\eqref{ex:monoOU3});
        \item $U=\prod_{n\geq 0}\caO_n$ and~$\varphi=\prod_{n\geq 0}\varphi_n$, where~$(\caO_n)_{n\geq 0}$ is a sequence of profinite groups with~$\caO_1\neq\{1\}$ that comes with a sequence~$(\varphi_n\colon \caO_n\hookrightarrow \caO_{n+1})_{n\geq 0}$ of continuous open monomorphisms. 
        \item The following well-known construction is based on Willis's theory; see for instance~\cite{will94, will:scale}. Let~$G$ be a TDLC~group and~$\alpha\colon G\to G$ be a topological automorphism with scale~$s(\alpha)\in \Z_{\geq 1}$. Let~$U\leq G$ be a compact open subgroup such that~$s(\alpha)=|\alpha(U):U\cap \alpha(U)|$. Such a subgroup always exists and is said to be \emph{minimising for~$\alpha$}; see~\cite[Definition~2]{will:scale}. In turn, the subgroup
        \begin{equation*}
            U_+:=\bigcap_{n\geq 0}\alpha^n(U)
        \end{equation*}
        is closed in~$U$ and hence is a profinite group, it satisfies~$U_+\subseteq \alpha(U_+)$ and~$s(\alpha)=|\alpha(U_+):U_+|=|U_+:\alpha^{-1}(U_+)|$; see~\cite[Lemma~1, Theorem~2]{will:scale}. Therefore, the restriction
        \begin{equation*}
            \varphi:=(\alpha^{-1})\vert_{U_+}\colon U_+\hookrightarrow U_+
        \end{equation*}
        is an continuous open monomorphism and~$|U_+:\varphi(U_+)|=s(\alpha)$.
        If~$s(\alpha)>1$ then~$U_+=\bigcap_{n\geq 0}\alpha^n(U)$ is not open co-Hopfian. 
        
        Notice that the existence of an (inner) topological automorphism~$\alpha$ of~$G$ satisfying~$s(\alpha)>1$ is guaranteed if, for instance, $G$ is \emph{not} uniscalar.
    \end{enumerate}

       Given a profinite group~$U$ with a shrinking monomorphism~$\varphi\colon U\to U$, the group $SB_n(\varphi)$ is a non-discrete TDLC~group (see~\Cref{lem:topologyofHNNextensions}) of type~$\FP_{n}(R)$ but not of type~$\FP_{n+1}(R)$ for~$R\in\{\Z,\Q\}$ (see~\Cref{thm:BieriStallings}). Here the $\mathbb Q$-acyclicity of $U$ is guaranteed by~\Cref{fact:ProfQacyc}.

\subsubsection{Products of suitable groups}
Let~${R \in \{\Z,\Q\}}.$ Assume that $H$ is a TDLC~group of type $\FP_\infty(R)$ that is $\Q$-acyclic. Even in the case that $H$ is open co-Hopfian, the following result gives an alternative way to use $H$ to generate examples of TDLC~groups with prescribed finiteness properties. To this end, let $U$ be a non-discrete profinite group with a shrinking monomorphism~$\varphi\colon U\hookrightarrow U$ (see~\Cref{ex:BSUcompact}).
\begin{prop}\label{prop:product}
The topological product~$U\times H$ is a $\Q$-acyclic TDLC-group of type $\FP_\infty(R)$ that fails to be open co-Hopfian.
\end{prop}

\begin{proof}
The topological product $G:=U\times H$  is a compactly generated TDLC~group  because~$U$ is profinite and~$H$ is compactly generated TDLC. Since~$U$ is compact, $G$ is quasi-isometric to~$H$~\cite[Proposition~4.C.11]{cdlh:metric} and, consequently, $G$ is of type~$\FP_\infty(R)$~\cite[Corollary~5.7]{ccc}.

Since both $H$ and~$U$ are $\Q$-acyclic (see~\Cref{prop:neretin} and~\Cref{fact:ProfQacyc}), $G$ is $\Q$-acyclic by~\Cref{cor:Qacyc}.  

Let $\alpha\colon H\to H$ be a continuous open monomorphism. 
The map $\varphi':=\varphi\times\alpha\colon G\to G$ is a non-surjective continuous open monomorphism, i.e., $G$ fails to be open co-Hopfian. 
\end{proof}

By~\Cref{thm:BieriStallings}, for every~$n\geq 1$ one has that~$SB_n(\varphi')$ is of type~$\FP_n(R)$ but not of type~$\FP_{n+1}(R)$. Here $\varphi'\colon U\times H\to U\times H$ is the shrinking monomorphism constructed in the latter proof.

\subsubsection{The semirestricted wreath product}\label{sus:semiWreath}

Let~$B$ and~$H$ be groups and $X$ an $H$-set; the group $H$ acts on $B^X$
by $h \cdot f(x) = f(h^{-1}x)$. Consider the semirestricted power
\[B^{X,A} = \{f \in B^X \mid f(x) \in A \text{ for all but finitely many $x\in X$}\}.\]

The {\em semirestricted wreath product} is the semidirect product $B^{X,A} \rtimes H$ with the obvious action of $H$ on $B^{X,A}$. 

\begin{prop}[\protect{\cite[Proposition 2.4]{corn:lcwp}}]
Suppose that~$B$ is a locally compact group, $A$ a compact
open subgroup and $X$ a discrete set. There is a unique structure of topological
group on $B^{X,A}$ that makes the embedding of $A^X$ a topological isomorphism to an
open subgroup. In this setup, $B^{X,A}$~is locally compact. 

Suppose in addition that $H$ is a locally compact group and that the $H$-action
on $X$ has open point stabilisers. Then there is a unique
structure of topological group on $B^{X,A}\rtimes H$ that makes it a topological semidirect product.
\end{prop}
The locally compact group topology on $B^{X,A}$  is generated by sets of the form $\prod_{x\in X}C_x$, where each
 $C_x$ is open in $B$ and $C_x=A$ for all but finitely many $x$.
As a topological space, $B^{X,A} \rtimes H\cong B^{X,A} \times H$. Moreover, if $B$ and $H$ are TDLC~groups, then  $B^{X,A} \rtimes H$ is so. 

 \begin{prop}\label{prop:wp} 
 Let~$B$ and~$H$ be TDLC~groups and let~$A\leq B$ be a compact open subgroup.   
 Suppose there is a continuous open monomorphism~$\phi\colon B\hookrightarrow B$ such that~$\phi(B)\neq B$, $\phi(A)= A$ and~$\phi^{-1}(A)= A$. 
 Then~$B^{X,A} \rtimes H$ is not open co-Hopfian.
 \end{prop}
 \begin{proof}  
 Following~\cite[proof of Proposition~6.1]{bff}, there is a non-surjective monomorphism $$\Phi\colon B^{X,A} \rtimes H\hookrightarrow B^{X,A} \rtimes H,\quad \Phi(f, h) = (\phi\, \circ f, h),\quad f\in B^{X,A}, h\in H.$$ We must check that~$\Phi$ is continuous and open. It suffices to prove it for $\Phi_1\colon B^{X,A}\to B^{X,A}$ mapping $f\mapsto \phi\,\circ f$. Consider the product~$\prod_{x\in X}C_x$ where each
 $C_x\subset B$ is open and $C_x=A$ for all but finitely many~$x$. Then $$\Phi_1\left(\prod_{x\in X}C_x\right)=\prod_{x\in X}\phi(C_x)\quad\text{and}\quad\Phi_1^{-1}\left(\prod_{x\in X}C_x\right)=\prod_{x\in X}\phi^{-1}(C_x)$$ 
where each $\phi(C_x)\subseteq B$ is open and, in particular, $\phi(C_x) =A$ for all but finitely many~$x$ (analogously for $\phi^{-1}(C_x)$).
 \end{proof}
 The previous result provides a method to build TDLC~groups that fail to be open co-Hopfian starting from discrete and profinite groups. For instance, one can set $B=U\times\Gamma$ with$~U$ profinite and~$\Gamma$ a non co-Hopfian discrete group, and consider the compact open subgroup $A=U\times\{1\}$. Therefore, for every shrinking monomorphism $\varphi\colon \Gamma\hookrightarrow\Gamma$, $\phi:=\mathrm{id_B}\times \varphi$ satisfies the hypothesis of Proposition~\ref{prop:wp}.

 The finiteness conditions of semirestricted wreath products are well-understood in the TDLC~case~\cite[\S 6]{ccc}. Recall also that, as a consequence of the Lyndon--Hochschild--Serre spectral sequence in homology, $B^{X,A} \rtimes H$ is rationally $\Q$-acyclic provided that both~$B^{X,A}$ and $H$ are rationally $\Q$-acyclic.
 Therefore, one could use them to build further examples of~$SB_n(\varphi)$ groups of type~$\FP_n(R)$ but not~$\FP_{n+1}(R)$, {for~$R\in \{\Z,\Q\}$}.

\subsubsection{Thompson's groups}\label{ex:ThompsonV} 
Thompson's group~$V$ can be considered as the group of left-continuous piecewise linear homeomorphisms of $(0,1]$ with breaks in $\Z[\frac12]$ and slopes powers of $2.$ For detail of this viewpoint of $V$, see~\cite{CannonFloydParry}. Thompson's group~$V$ contains a copy of~$F$ and has a standard infinite presentation, see for example~\cite{CannonFloydParry}. 
However, for our purpose the following infinite presentation is easier to work with. This presentation was pointed out to us by Pep Burillo and is well known, but we do not know of a reference. Here we adapt the presentation given by Brin~\cite[Proposition~6.1]{Brin2V}. The generators are 
$$\{x_i,\pi_i,\bar{\pi}_i\mid i\in \Z_{\geq 0}\}$$
and the relations are as follows:
\begin{enumerate}
    \item  $x_ix_j=x_jx_{i+1}$ if $j<i.$
    \item $\pi_ix_j = x_j\pi_{i+1}$ if $i>j.$
    \item $\pi_ix_i=x_{i+1}\pi_i\pi_{i+1}$ for $i \geq 0.$
    \item $\pi_ix_i = x_i\pi_j$ for $i>j+1.$
    \item $\bar\pi_ix_j = x_j\bar\pi_{i+1}$ if $i>j.$
    \item $\bar\pi_ix_i=\pi_i\bar\pi_{i+1}$ for $i \geq 0.$
    \item $\pi_i\pi_j = \pi_j\pi_i$ for $|i-j| \geq 2.$
    \item $\pi_i\pi_{i+1}\pi_i=\pi_{i+1}\pi_i\pi_{i+1}$ for $i \geq 0.$
    \item $\bar\pi_i\pi_j = \pi_j\bar\pi_i$ for $i \geq j+1.$
    \item $\pi_i\bar\pi_{i+1}\pi_i=\bar\pi_{i+1}\pi_i\bar\pi_{i+1}$ for $i \geq 0.$
    \item $\pi_i^2=1$ for $i \geq 0.$
    \item $\bar\pi_i^2=1$ for $i \geq 0.$
    \end{enumerate}
For a comparison, recall that the canonical presentation of Thompson's group~$F$ is
\begin{equation*}
    \langle \{x_i\}_{i\geq 0}\mid x_ix_j=x_jx_{i+1} \text{ for all }j<i\rangle.
\end{equation*}
Hence, the embedded copy of~$F$ in~$V$ may be identified with the subgroup of~$V$ generated by~$\{x_i\}_{i\geq 0}$.

Analogously to the well-known shift-endomorphism~$\varphi_F$ of~$F$, we can define an monomorphism given by  sending a map on~$(0,1]$ to a map on~$(\frac12,1]$ by shifting and rescaling.  
In particular, this monomorphism $\varphi_V: V \to V$ is defined on the generators as follows, for every~$i\geq 0$:
\begin{equation}
    \varphi_V(x_i) = x_{i+1},\quad\varphi_V(\pi_i)= \pi_{i+1},\quad
   \varphi_V(\bar\pi_i) = \bar\pi_{i+1}.
\end{equation}

Since $\varphi_V$ is a monomorphism, $\varphi_V(V):=V[1]$ is isomorphic to~$V$. However, since~$V[1]\cap F=\varphi_F(F)$ is strictly contained in~$F$,~$\varphi_V$ is not surjective.

\begin{cor}[A Bieri--Stallings--Thompson group] \label{cor:BieriStallingsThompson}
    Let~$\varphi_V : V \into V$ be the non-surjective self-monomorphism on Thompson's group~$V$ as in~\Cref{ex:ThompsonV}. Then the Bieri--Stallings--Thompson group~$SB_n(\varphi_V)$ is of type~$\tF_n$ but not of type~$\FP_{n+1}(R)$ {for~$R\in\{\Z,\Q\}$}.
\end{cor}

\begin{proof}
    The group~$V$ is finitely presented, of type~$\FP_\infty(R)$, and $R$-acyclic for $R\in \{\Q,\Z\}$; see \cite{BrownFP}, \cite{Brown89} and \cite{szymik-wahl} respectively. \Cref{thm:BieriStallings} thus implies that $SB_n(\varphi_V)$ is of type~$\FP_n(R)$ but not~$\FP_{n+1}(R)$. As a consequence of \Cref{lem:topologyofHNNextensions}\eqref{lem:topologyofHNNextensions6}, the group~$D_n(\varphi_V)$ as in~\Cref{def:BieriStallings} is finitely presented as well. With its given action on~$F_\infty$, it is easy to construct a finite presentation for~$SB_n(\varphi_V) = F_\infty \rtimes D_n(\varphi_V)$. Since, for~$n\geq 3$, being of type $\tF_n$ is equivalent to finite presentability plus being of type~$\FP_n(\Z)$, the result follows.
\end{proof}
The previous result exhibits variants of 
Thompson's group~$V$ that can be distinguished by their finiteness conditions. The reader might compare them with the simple groups containing Thompson's groups and separated by finiteness properties constructed in~\cite{skipper-witzel-zaremsky}.

\subsubsection{Neretin's groups}

For~$d\in \Z_{\geq 2}$, Neretin's group~$\mathcal N_d$ is the group all almost automorphism of the rooted $d$-regular tree. It can be viewed as a TDLC analogue of Thompson's groups~$V_{d,2}$, and shares many of its properties. For example,~$\mathcal N_d$ is 
of type~$\tF_\infty$~\cite{satu:top}.

\begin{prop}\label{prop:neretin}
    Neretin's group $\mathcal N_d$ is a $\Q$-acyclic TDLC~group.
\end{prop}
\begin{proof}
   By~\cite[Corollary~1.4 and Proposition~4.1]{BonnSauer}, the vanishing of~$\dH_n(\mathcal N_d,\Q)$ follows from~\Cref{lem:cohomologyiso}.
\end{proof}

\begin{rem}
    The results in~\cite{BonnSauer} are general facts about Schlichting completions. Hence~\Cref{prop:neretin} 
    generalises to any Schlichting completion of a $\Q$-acyclic group of type~$\FP_\infty(\Q)$.
\end{rem}
Note that, since $\mathcal N_d$ is (abstractly) simple~\cite{kap}, any nontrivial 
endomorphism has trivial kernel and so, in particular, $\mathcal N_d$ is Hopfian.
In~\cite{willi15}, it is shown the existence of a continuous self-endomorphism~$\alpha \colon \mathcal N_d \to \mathcal N_d$ that is one-to-one but not onto, i.e.,~$\mathcal N_d$ is not co-Hopfian. 
However,  $\mathcal N_d$ is open co-Hopfian as pointed out to us by Waltraud Lederle. Lederle's argument can be generalised as follows.

\begin{prop}\label{prop:simCoHopf}
    Let~$G$ be a topologically simple TDLC group. 
    Suppose that $G$ is hyperrigid (in the sense of~\cite[\S 3.2]{BEW}). Then every continuous open monomorphism~$\varphi\colon G\hookrightarrow G$ is an inner automorphism of~$G$. In particular, $G$ is open co-Hopfian.
\end{prop}

\begin{proof} By \cite[Corollary 3.5]{BEW}, there exists an infinite profinite group~$U$ such that $G$ is topologically isomorphic to the abstract commensurator of~$U$ endowed with the strong topology~\cite[Section 3.1]{BEW}. In particular, $U$ is a compact open subgroup of~$G$.
   The restriction~$\alpha:=\varphi\vert_U$ of~$\varphi$ on~$U$ is a topological isomorphism~$U\to\varphi(U)$ between compact open subgroups of~$G$. By hyperrigidity of~$G$, there is~$g\in G$ such that~$\alpha(u)=gug^{-1}$ for every~$u\in U$. Denote by~$\mathrm{Ad}(g)$ the inner automorphism on~$G$ induced by~$g$, and note that the kernel of~$\varphi\circ \mathrm{Ad}(g)^{-1}\colon G\to G$, which is a closed normal subgroup of~$G$, contains~$U$. Since~$U\neq \{1\}$ and~$G$ is topologically simple, we conclude that~$\varphi=\mathrm{Ad}(g)$. 
\end{proof}

\Cref{prop:simCoHopf} applies in the case~$G=\caN_d$ because $\caN_d$ is abstractly simple~\cite{kap} and hyperrigid~\cite[Corollary~E]{CdM}. 
 Nevertheless, by Proposition~\ref{prop:product}, we are still able to use Neretin's group to get new groups with prescribed finiteness properties.

\begin{cor}[{Bieri--Stallings--Neretin groups}] \label{cor:prodner}
Let~${R \in \{\Z,\Q\}}$. Let $U$ be profinite group with a shrinking monomorphism $\varphi\colon U\hookrightarrow U$. There exist a TDLC group $N(\varphi)$ quasi-isometric to Neretin's group~$\caN_d$ and a shrinking monomorphism~$\phee' : N(\varphi) \into N(\varphi)$ such that, for each~$n \in \Z_{\geq 1}$, the TDLC~group $SB_n(\varphi')$ is of type $\FP_n(R)$ but not of type $\FP_{n+1}(R)$.
\end{cor}
\begin{proof}
     As argued in the proof of Proposition~\ref{prop:product}, $N(\varphi):=U\times\mathcal N_d$ is quasi-isometric to~$\mathcal N_d$ and admits a shrinking monomorphism $\varphi'$. By~\Cref{cor:Qacyc}, \Cref{fact:ProfQacyc} and~\Cref{prop:neretin}, $N(\varphi)$
     is $\Q$-acyclic. Moreover, $N(\varphi)$ is of type~$\FP_{\infty}(R)$ because it is quasi isometric to~$\mathcal{N}_d$, which is of type~$\FP_\infty(R)$; see~\cite[Theorem~1.5]{satu:top} and~\cite[Corollary~5.7]{ccc}. Hence~\Cref{thm:BieriStallings} applies.
\end{proof}

\subsubsection{Universal Burger--Mozes groups} Let~$d\geq 3$ and~$T_d$ be the (unrooted) $d$-regular tree\footnote{Note that in this section, following~\cite{bumo}, trees are in the sense of Jean-Pierre Serre.}.  Burger-Mozes groups~$U(F)$~\cite[\S 3.2]{bumo} are closed subgroups of~$\text{Aut}(T_d)$ that are constructed on the basis of a ``local action'' as follows.

 A \emph{legal colouring} is a map
$l\colon E(T_d)\to\{1,\ldots,d\}$ 
such that:
\begin{enumerate}
\item $l(e)=l(\bar e)$ for every edge $e\in E(T_d)$, and
\item the restriction $l_v:=l|_{E(v)}\colon E(v)\to\{1,\ldots,d\}$ to the star~$E(v):=\{e\in E(T_d)\mid v\text{ is the origin of }e\}$ of each vertex $v$ is a bijection.
\end{enumerate}

Given a legal colouring, one defines 
$$\sigma_l\colon\Aut(T_d)\times V(T_d)\to\mathrm{Sym}(\Omega),\quad (g,x)\mapsto l_{gx}\circ g\circ l_x^{-1}$$
and, for every subgroup~$F\leq S_d$,
$$U(F):=\{g\in\Aut(T_d)\mid \forall v\in V(T_d),\,\sigma_l(g,v)\in F\}.$$
Any reference to the legal colouring $l$ is omitted in the notation $U(F)$ because different legal colourings produce conjugate (and hence isomorphic) groups in~$\Aut(T_d)$.

Being a closed subgroup of~$\Aut(T_d)$, the group~$U(F)$ is a TDLC group with the induced permutation topology. By~$U(F)^+$ one denotes the subgroup of $U(F)$ generated by the pointwise edge-stabilisers of~$U(F)$. It is known that $U(F)^+$ is trivial or simple~\cite[Proposition~3.2.1]{bumo}.
\begin{prop}\label{prop:UFacyc}
$U(F)$ and $U(F)^+$ are $\Q$-acyclic. 
\end{prop}
\begin{proof} 
Consider the induced action of~$U(F)$ (and $U(F)^+$ respectively) on the barycentric subdivision~$T_d'$ of the regular tree~$T_d$. The latter action is without edge inversions.
In particular, they act with compact open stabilisers. Since profinite groups are $\Q$-acyclic (see~\Cref{fact:ProfQacyc}), the Mayer--Vietoris sequence in homology (see~\Cref{thm:mv}) yields the claim.
\end{proof}

If $F$ is primitive we know that $U(F)^+$ is open co-Hopfian because all proper open subgroups of $U(F)^+$ are compact; see~\cite[Corollary~3.10 and Proposition~4.1]{CdM}. The following question arises naturally:
\begin{ques}
    For which~$F\leq S_d$ does~$U(F)$ or $U(F)^+$ fail to be open co-Hopfian?
\end{ques}

\printbibliography

\end{document}